\documentclass[11pt, twoside]{article}
\usepackage{amsmath,amsthm}
\usepackage{amssymb,latexsym}
\usepackage{mathrsfs}
\usepackage{enumerate}
\usepackage[colorlinks,citecolor=blue,dvipdfm]{hyperref}
\usepackage[numbers,sort&compress]{natbib}
\usepackage{indentfirst}

\newcommand{\e}{\ensuremath{\epsilon}}

\newcommand{\nn}{\ensuremath{\textbf{n}}}
\newcommand{\NN}{\ensuremath{\textbf{N}}}

\newcommand{\rto}{\ensuremath{\rightarrow}}
\newcommand{\lem}{\ensuremath{\lesssim}}

\newtheorem{theorem}{Theorem}[section]

\newtheorem{lemma}{Lemma}[section]
\newtheorem{proposition}{Proposition}[section]

\newtheorem{remark}{Remark}[section]

\theoremstyle{definition} \theoremstyle{remark}
\numberwithin{equation}{section}
\allowdisplaybreaks

\date{}

\begin{document}


\date{}

\title{\bf Inviscid Limit for the Free-Boundary problems of MHD
Equations with or without Surface Tension
\thanks{Ding is partially supported by the NSF of China (Grant No.11571117, 11371152, 11871005 and 11771155)
 and NSF of Guangdong Province (Grant No.2017A030313003).}\\
}

\author{Pengfei Chen
and ~Shijin Ding
\thanks{Corresponding author: dingsj@scnu.edu.cn(S. Ding), cpfmath@163.com(P. Chen).}
\\[1.8mm]
\footnotesize  {South China Research Center for Applied Mathematics and Interdisciplinary Studies}\\
\footnotesize  {South China Normal University, Guangzhou 510631 ,China }\\
}

\baselineskip 0.23in

\maketitle

\begin{abstract}
In this paper,
we investigate the convergence rates of inviscid limits for the free-boundary problems
of the incompressible magnetohydrodynamics (MHD) with or without surface tension in $\mathbb{R}^3$,
where the magnetic field is identically constant on the surface and outside of the domain.
First, we establish the vorticity,
the normal derivatives and the regularity structure of the solutions,
and develop a priori co-norm estimates including time derivatives by the vorticity system.
Second, we obtain two independent sufficient conditions for the existence of strong vorticity layers:
(I) the limit of the difference between the initial MHD vorticity of velocity or magnetic field
and that of the ideal MHD equations is nonzero.
(II) The cross product of tangential projection on the free surface of
the ideal MHD strain tensor of velocity or magnetic field with
the normal vector of the free surface is nonzero.
Otherwise, the vorticity layer is weak.
Third, we prove high order convergence rates of tangential derivatives and
the first order normal derivative in standard Sobolev space,
where the convergence rates depend on the ideal MHD boundary value.
\\[2mm]
\\[1.2mm]
\end{abstract}

\tableofcontents

\baselineskip 0.234in
\section{Introduction}
In this paper,
we study the inviscid limit
for the free boundary problems of the incompressible
Magnetohydrodynamics (MHD) equations with or without surface tension in $\mathbb{R}^3$
\begin{equation}\label{MHD1}
\left\{\begin{aligned}
&\partial_{t}u-\epsilon\Delta u+u\cdot\nabla u+\nabla P=H\cdot\nabla H-\frac{1}{2}\nabla|H|^2,\
&{\rm in}\ \Omega(t),\\
&\partial_{t}H-\lambda\Delta H+u\cdot\nabla  H=H\cdot\nabla u,\
&{\rm in}\ \Omega(t),\\
&\nabla\cdot u=0,  \nabla\cdot H=0,\
&{\rm in}\ \Omega(t),\\
&Pn-2\epsilon Su\nn=ghn+(H\otimes H-\frac{1}{2}I |H|^2)\nn-\sigma \mathcal{M}\nn,\
&{\rm on}\ S_F(t),\\
&\partial_th=u\cdot \NN,\
&{\rm on}\ S_F(t),\\
&H=0,\
&{\rm on}\ S_F(t),\\
&(u,H,h)|_{t=0} = (u_0^{\e},H_0^{\e},h_0^{\e}),
&{\rm in}\ \Omega(t),
\end{aligned}\right.
\end{equation}
which describe the motion of conducting fluids in an electromagnetic field,
where $\epsilon>0$, $\lambda>0$, $\sigma\geq0$ be the kinematic viscosity of the MHD equation,
the magnetic diffusivity of Faraday's law, and the surface tension coefficient in the dynamical
boundary condition, respectively.
The surface tension in the dynamical boundary condition,
the fourth equation in \eqref{MHD1},
namely $\mathcal{M}=\nabla\cdot\frac{\nabla h}{\sqrt{1+|\nabla h|^2}}$,
is twice the mean curvature of the free surface $S_t$,
$g>0$ is a gravitational constant.
The kinematic boundary condition,
the fifth equation in \eqref{MHD1},
implies that the free surface
is adverted with the fluid.
Denote $x=(y, z)$, y is the horizontal variable, z is the vertical variable,
the initial data satisfies the
compatibility condition $\Pi Su^\epsilon_0\nn=0$ and $\Pi SH^\epsilon_0\nn=0$ on $S_F(t)$,
where the projection on the tangent space of the boundary denoted by
$\Pi=I d-\mathbf{n} \otimes \mathbf{n}$.

We neglect the Coriolis effect generated by the planetary rotation, then
there is no Ekman layer near the free surface even if Rossby number is small.
We shall focus on the three dimensional equation in the domain $-\infty<z<h(t,y)$,
which are defined as follows:
\begin{equation}
\begin{array}{ll}
\Omega_t =\{x\in\mathbb{R}^3|\, -\infty< z<h(t,y)\},
\\[8pt]
\Sigma_t = \{x\in\mathbb{R}^3|\, z =h(t,y)\},
\\[8pt]
\NN=(-\nabla h, 1)^{\top},\quad  \nn=\frac{\NN}{|\NN|},
\\[8pt]
\mathcal{S}u =\frac{1}{2}(\nabla u +(\nabla u)^{\top}),
\end{array}
\end{equation}
where the symbol $\top$ means the transposition of matrices or vectors.
We suppose $h(t,y)\rightarrow0$ as $|y|\rightarrow+\infty$ for any $t\geq0$.

Let us introduce the boundary conditions of the magnetic fields on the free boundary
or outside the conducting fluid.
First, it is important to point out that for the classical plasma-vacuum interface problem in \cite{Goedbloed14},
where in the vacuum, the magnetic field
satisfies the div-curl system
\begin{equation}\label{pvi}
\nabla\times \mathcal{H}=0, \nabla\cdot \mathcal{H}=0,
\end{equation}
which is a special from the pre-Maxwell dynamics
by neglecting the displacement current $(1/c)\partial_tE$.
Meanwhile,
we easily get normal continuity by the divergence-free condition
and divergence theorem
\[H\cdot n=\mathcal{H}\cdot n,\ {\rm on}\ S_F(t).
\]
Second, Lee \cite{Lee17} considered the uniform estimate of free boundary problem
with the constant magnetic value on the free boundary and in the vacuum region,
where the constant magnetic can be treated as special case of the plasma-vacuum interface problem.
The local well-posedness theorem can be seem in Lee \cite{Lee18}.

In system \eqref{MHD1},
it is convenient to define the total pressure $p$ as the sum of the fluid pressure and
magnetic pressure,
\begin{align}
p:= P+\frac{1}{2}|H|^2.
\end{align}
Then, the free boundary problems for MHD equations can be rewritten as
\begin{equation}\label{MHD2}
\left\{
\begin{aligned}
&\partial_{t}u-\epsilon\Delta u+u\cdot\nabla u+\nabla p=H\cdot\nabla H,
&{\rm in}\ \Omega(t),\\
&\partial_{t}H-\lambda\Delta H+u\cdot\nabla  H=H\cdot\nabla u,
&{\rm in}\ \Omega(t),\\
&\nabla\cdot u=0,  \nabla\cdot H=0,
&{\rm in}\ \Omega(t),\\
&p\nn-2\epsilon Su\nn=gh\nn-\sigma \mathcal{M}\nn,
&{\rm on}\ S_F(t),\\
&\partial_th=u\cdot \NN,
&{\rm on}\ S_F(t),\\
&H=0,
&{\rm on}\ S_F(t),\\
&(u,H,h)|_{t=0} = (u_0^{\e},H_0^{\e},h_0^{\e}),
&{\rm in}\ \Omega(t).
\end{aligned}
\right.
\end{equation}

What we are interested in this paper are the convergence rates of inviscid limit
of the free boundary MHD equations,
either $\sigma=0$ or $\sigma>0$ is fixed.
However, we do not study the zero surface tension limit here.
Let $\epsilon, \lambda\to0$,
we formally get the following ideal MHD equations
\begin{equation}
\left\{
\begin{aligned}
&\partial_{t}u+u\cdot\nabla u+\nabla p=H\cdot\nabla H,
&{\rm in}\ \Omega(t),\\
&\partial_{t}H+u\cdot\nabla  H=H\cdot\nabla u,
&{\rm in}\ \Omega(t),\\
&\nabla\cdot u=0,  \nabla\cdot H=0,
&{\rm in}\ \Omega(t),\\
&p\nn=gh\nn-\sigma \mathcal{M}\nn,
&{\rm on}\ S_F(t),\\
&\partial_th=u\cdot N,
&{\rm on}\ S_F(t),\\
&H=0,
&{\rm on}\ S_F(t),\\
&(u,H,h)|_{t=0}=(u_0,H_0,h_0):=\lim_{\epsilon,\lambda\to0}
(u^{\epsilon}_0,H^{\epsilon}_0,h^{\epsilon}_0).
\end{aligned}
\right.
\end{equation}
where $(u_0,H_0,h_0)=\lim\limits_{\epsilon,\lambda\to0}(u^{\epsilon}_0,H^{\epsilon}_0,h^{\epsilon}_0)$
is in the $L^\infty$ sense or even in the $L^2$ sense,
$(u_0,H_0,h_0)$ are independent of $\epsilon,\lambda$.

The investigation of free surface motions is an important topic in fluid dynamics,
which has attracted a lot of attentions during the last thirty years.
For free surface problems of Navier-Stokes equation, in 1981, Beale considered the local existence result in \cite{Beale81} without surface tension.
Some similar results on global well-posedness for the free boundary problems of
incompressible Navier-Stokes equations,
have been obtained, for instance, by Hataya \cite{Hataya09}, Padula \cite{Padula10}.
Guo and Tice \cite{Guo131,Guo132,Guo133}
obtained a series of results by the so called two-tier energy method,
which combines the boundedness of high-order energy with the decay of low-order energy.
For the case of Navier-Stokes equations with surface tension,
we refer the reader to Beale \cite{Beale84}, Nishida, Teramoto and Yoshihara \cite{Nishida04},
Tan and Wang \cite{Tan14}, Tanaka and Tani \cite{Tani95}, Tani \cite{Tani96}.

However, the free boundary problems for Euler equations are much harder and interesting.
For the irrotational case, we refer to Wu \cite{Wu09,Wu11}, Germain,
Masmoudi and Shatah \cite{Germain12}, Ionescu and Pusateri \cite{Ionescu15},
Alazard and Delort \cite{Alazard13}, for the water waves without surface tension.
We also refer to Beyer and G\"{u}nther \cite{Beyer98},
Germain, Masmoudi and Shatah \cite{Germain15} for the water waves with surface tension.
For the general rotational case,
it is still not clear whether the free boundary problems of incompressible Euler equations
for the general small initial data admits a global unique solution or not, even in 2D.
As to local-in-time results,
we refer to Lindblad \cite{Lindblad05}, Coutand and Shkiller \cite{Coutand07},
Shatah and Zeng \cite{Shatah08}, Zhang and Zhang \cite{Zhang08} for the zero surface tension case.

As to the global well-posedness for the free boundary problems of incompressible MHD equations,
we refer the reader to Padula and Solonnikov \cite{Padula11},
Solonnikov \cite{Solonnikov13},
Lee \cite{Lee18} for the viscous and resistive case, Craig\cite{Craig85},
Wang and Xin \cite{Wang18} for the inviscid and resistive case.
For the inviscid case, Hao and Luo \cite{Hao14} established an a priori estimate in the
spirit of \cite{Christodoulou00}.
The plasma-vacuum problem for ideal MHD equations was studied by
Morando, Trakhinin and Trebeschi \cite{Morando14}, Hao \cite{Hao17},
Sun, Wang and Zhang \cite{Sun17}.
The compressible case be refered to Secchi and Trakhinin \cite{Secchi13}.

Another classical and interesting problem in the mathematical theory
of fluid mechanics is to study the asymptotic limit of
the viscous solutions to that of small viscosity solutions.
Of course, it is natural to expect that the limit is given
by a solution of the Euler equation.
Generally, different regions and boundary conditions should
be taken into account when studying inviscid limit problem.
For the case of the whole spaces where the domain has no boundaries,
see for instances \cite{Constantin86,Constantin88,Kato84,Masmoudi07}.
However, in the presence of physical boundaries, the problems become much more
complicated due to the formation of boundary layers.
For the case of Dirichlet boundary condition in the fixed domain,
the inviscid limit is not rigorously verified except for the following special cases, i.e.,
the analytic setting (see \cite{Sammartino981,Sammartino982,Wang17})
and the case where the vorticity is located away from the boundary
in 2D half plane (see \cite{Maekawa13,Maekawa14}).
While, for the Navier-slip boundary condition in a fixed domain,
the $H^2$ convergence rete estimates have been obtained by Xiao and Xin \cite{Xiao07}
for complete slip boundary condition and flat boundary,
which are generalized in \cite{Berselli12,Xiao09}.
For the inviscid limit results with Prandtl expansion,  we refer to \cite{Iftimie11,Liu19,Wang10}.
Recently, the conormal uniform estimates have been widely used to estimate the normal derivatives of first
order of the velocity field, for example, in \cite{Masmoudi12, Wang16}.
Based on such uniform estimates,the authors of \cite{Masmoudi12},
\cite{Xiao13} have proved the convergence rates estimate in Sobolev space $H^1$.

For the free surface with kinetical and dynamical boundary conditions in the moving domain,
the uniform regularity estimates and inviscid limit have been greatly developed
by the co-normal Sobolev spaces for which the normal differential operators vanish on the free surface.
For the Navier-Stokes equations with free surface,
Masmoudi and Rousset \cite{Masmoudi17} first established the local existence of solutions to the incompressible Navier-Stokes system without surface tension by uniform in $\epsilon$
estimates in conormal Sobolev spaces.
Wang and Xin \cite{Wang15}, Elgindi and Lee \cite{Elgindi17}
extend the convergence results with surface tension.
Later, Wu \cite{Wu16} study the regularity structure,
the vorticity layer and the convergence
rates of the inviscid limits for the cases with or without surface tension,
where the author proved that not only tangential derivatives and
standard normal derivative have different convergence rates,
and obtained the convergence rates of high order tangential derivatives and
the first order standard normal derivative in energy norms.
Recently, Mei, Wang and Xin \cite{Mei18} proved the case of
compressible Navier-Stokes equations with or without surface tension.

For the discussions on the inviscid limit of free boundary problems of MHD equations, we refer to
Lee \cite{Lee17} for the case without surface tension, Mei \cite{Mei16} for the case with surface tension.
On the other hand, it is also interesting to investigate the zero surface
tension limit of free boundary problems.
The zero surface tension limit of the free surface
Navier-Stokes equation and Euler equations with damping
has been established by Tan and Wang \cite{Tan14}, Lian \cite{Lian18}, respectively.
For the compressible viscous surface-internal wave problem,
the reader be referred to Jang, Tice and Wang \cite{Jang16}.
However, the convergence rates have not been discussed in these papers.
Moreover, a clear description of the generation principle of strong
and weak vorticity layers for free boundary problems of MHD is still needed.
For the relevant description for Navier-Stokes equations, we refer to Wu \cite{Wu16}.

In this paper,
we are interested in the inviscid limit theory of the
free boundary problems for the MHD equations.
Our main goal is to establish the convergence rate of the vanishing viscous limit,
and to analyze two independent sufficient conditions for the existence of strong
vorticity layers.
Our approach here is motivated by Wu \cite{Wu16} which studies
the same problem for the Navier-Stokes equations
and is based on the following observations:

First,
the estimates of normal derivatives are based on the estimates of
vorticity rather than those of $\Pi S^\varphi v\nn$ and $\Pi S^\varphi b\nn$.
Here we establish the relationship between the vorticity,
the normal derivatives and its regularity structure of MHD equations,
where the boundary value of the vorticity only depend on
its tangential derivatives.

Second,
we show that there are two independent sufficient conditions for the existence of strong vorticity layer. We
note that these two conditions are almost independent.
One condition is that the ideal MHD boundary data satisfies
$\Pi S^\varphi v\nn|_{z=0}=0$ or
$\Pi S^\varphi b\nn|_{z=0}=0$ in $(0,T]$,
and the initial vorticity layer of velocity or magnetic field is strong,
then there exist a strong vorticity layer.
Another condition is that the ideal MHD boundary data satisfies
$\Pi S^\varphi v\nn|_{z=0}\neq0$ or
$\Pi S^\varphi b\nn|_{z=0}\neq0$ in $(0,T]$,
then the MHD solution has a strong vorticity layer too.
Otherwise,
we show that the vorticity layer is weak.
By the following two Lagrangian maps $Y_1$ and $Y_2$:
\begin{equation}
\begin{array}{ll}
\partial_tY_1=u-b,\,\partial_tY_2=u+b,
\end{array}
\end{equation}
the moving domain $\Omega(t)$ is transformed into two fixed domains
$\Omega_1$ and $\Omega_2$.
We combine two $PDE_s$ to get the heat equations with coupled damping,
where it is required $\epsilon=\lambda$:
\begin{equation}\label{1.8}
\begin{array}{ll}
a_0\partial_tW_+-\epsilon\partial_i(a_{i,j}\partial_jW_+)+\gamma a_0W_+-(f^7_v-f^7_b)W_-
=I_{v},
\\[7pt]
b_0\partial_tW_--\epsilon\partial_i(b_{i,j}\partial_jW_-)+\gamma b_0W_--(f^7_v+f^7_b)W_+
=I_{b},
\end{array}
\end{equation}
where $W_\pm =e^{-\gamma t}(\hat{\omega}_{vh}\pm\hat{\omega}_{bh})(t,\Phi^{-1}\circ Y_i)$,
for more detail, see section 3 and section 4.
Since
the maximum principle can not be directly applied to system \eqref{1.8},
we use Duhamel's principle to get the $L^\infty$ estimate of \eqref{1.8},
where $(f^7_v-f^7_b)W_-$, $(f^7_v+f^7_b)W_+$ be treated as force term.
For the discrepancy of the vorticity on the free boundary,
we translate the problem into a symbolic version of ODE system with force term
by Fourier transformation.
By the scaling analysis for the solution of the ODE,
we can make it clear when the strong vorticity layers or weak vorticity layers appear.

Last, we obtain the convergence rates for tangential derivatives of high order and for the normal derivative of first order in Sobolev norms
by the difference equations between viscous MHD equations and ideal MHD equations.
We notice that the convergence rates for tangential derivatives and normal derivative are different.

\subsection{Parametrization into a fixed domain}
In this subsection, we rewrite the free-boundary problem \eqref{MHD2} with $\sigma=0$
into the fixed domain, the lower half space in $\mathbb{R}^3$.
Similar to Masmoudi and Rousset \cite{Masmoudi17},
we define the diffeomorphism between $R_-^3$ and the moving domain $\Omega_t$:
\begin{align}
\Phi(t,\cdot):\mathbb{R}_-^3=\mathbb{R}^2\times(-\infty,0)&\rightarrow\Omega_t,\\
x=(y,z)&\rightarrow(y,\varphi(t,y,z)).
\end{align}

There are many ways to take $\varphi$ and we have to decide which one is optimal for our purposes.
One easy option is to set $\varphi(t,y,z)=z+h(t,y)$.
However, it is more useful to take a function $\Phi$ which is actually more regular than $h$.
If one takes a harmonic extension,
then $\varphi$ gains an aditional $\frac{1}{2}$ derivative, which is more regular than $h$.
We define $\varphi$ as
\begin{align}
\varphi(t,y,z)=Az+\eta(t,y,z),
\end{align}
where $A > 0$ is to be chosen,
$\eta$ is given by the extension of h to the domain $\mathbb{R}^3_-$,
defined by $\hat{\eta}(\xi, z)=\chi(z \xi) \hat{h}(\xi)$,
where $\chi$ is
a smooth, even, compactly supported function such that $\chi=1 \text { on } B(0,1)$.

The constant $A > 0$ is suitably chosen such that $\Phi(0, \cdot)$ is a diffeomorphism,
namely
\begin{align}
\partial_z\varphi(0,y,z)\geq1, \forall x\in \mathbb{R}^3_-.
\end{align}

Now we can rewrite equations \eqref{MHD2} with $\sigma=0$ in the domain $\mathbb{R}^3_-$ by change of variables, i.e
\begin{equation}
\left\{
\begin{aligned}
&v(t,x) = u(t,y, \varphi(t,y,z)), \partial_i^\varphi v(t,x)=\partial_iu(t,y, \varphi(t,y,z)),\\
&b(t,x) = H(t,y, \varphi(t,y,z)),\partial_i^\varphi b(t,x)=\partial_iH(t,y, \varphi(t,y,z)),\\
&q(t,x) = p(t,y, \varphi(t,y,z)),\partial_i^\varphi q(t,x)=\partial_ip(t,y, \varphi(t,y,z)),
\end{aligned}
\right.
\end{equation}
for all $x\in R^3_-,i=t,1,2,3$, while $h(t,y)$ does not change.
Hence, it is convenient to define the following operator:
\begin{align*}
\partial_i^\varphi:=\partial_i-\frac{\partial_i\varphi}{\partial_z\varphi}\partial_z, \, \mbox{for}\,
i=t,1,2, \, \mbox{and}\, \partial_i^\varphi:=\frac{1}{\partial_z\varphi}\partial_z, \, \mbox{for}\,
i=3.
\end{align*}

Then the free boundary problem with $\sigma= 0$ is equivalent to the following system:
\begin{equation}\label{MHDF}
\left\{
\begin{aligned}
&\partial_{t}^\varphi v-\epsilon\Delta^\varphi v+v\cdot\nabla^\varphi v+\nabla^\varphi q
=b\cdot\nabla^\varphi b,
&{\rm in}\ \mathbb{R}^3_-,\\
&\partial_{t}^\varphi b-\lambda\Delta^\varphi b+v\cdot\nabla^\varphi b=b\cdot\nabla^\varphi v,
&{\rm in}\ \mathbb{R}^3_-,\\
&\nabla^\varphi\cdot v=0,  \nabla^\varphi\cdot b=0,
&{\rm in}\ \mathbb{R}^3_-,\\
&q\nn-2\epsilon S^\varphi v\nn=gh\nn,
&{\rm on}\ z=0,\\
&\partial_th=v(t,y,0)\cdot \NN,
&{\rm on}\ z=0,\\
&b=0,
&{\rm on}\ z=0\cup \mathbb{R}^3_+,\\
&(v,b,h)|_{t=0}=(v^{\epsilon}_0,b^{\epsilon}_0,h^{\epsilon}_0),
\end{aligned}
\right.
\end{equation}
where
\begin{equation}
\begin{array}{ll}
\NN=(-\nabla h(t,y), 1)^{\top},\quad  \nn=\frac{\NN}{|\NN|}, \\[7pt]
\mathcal{S}^{\varphi}v =\frac{1}{2}(\nabla^{\varphi} v +\nabla^{\varphi} v^{\top}).
\end{array}
\end{equation}

Letting $\epsilon=\lambda\to0$ in \eqref{MHDF},
we formally get:
\begin{equation}\label{IMHDF}
\left\{
\begin{aligned}
&\partial_{t}^\varphi v+v\cdot\nabla^\varphi v+\nabla^\varphi q
=b\cdot\nabla^\varphi b,
&{\rm in}\ \mathbb{R}^3_-,\\
&\partial_{t}^\varphi b+v\cdot\nabla^\varphi b=b\cdot\nabla^\varphi v,
&{\rm in}\ \mathbb{R}^3_-,\\
&\nabla^\varphi\cdot v=0,  \nabla^\varphi\cdot b=0,
&{\rm in}\ \mathbb{R}^3_-,\\
&q\nn=gh\nn,
&{\rm on}\ z=0,\\
&\partial_th=v(t,y,0)\cdot \NN,
&{\rm on}\ z=0,\\
&b=0,
&{\rm on}\ z=0\cup \mathbb{R}^3_+,\\
&(v,b,h)|_{t=0}=(v_0,b_0,h_0),
\end{aligned}
\right.
\end{equation}
where $v_0, b_0, h_0$ is the limit of $v^{\epsilon}_0, b^{\epsilon}_0, h^{\epsilon}_0$
in the $L^2$ sense for $\sigma=0$, respectively.
The following Taylor sign condition should be imposed when $\sigma=0$,
\begin{eqnarray}\label{1.17}
g-\partial_zq|_{z=0}\geq\delta_q>0.
\end{eqnarray}
The well-posedness of \eqref{IMHDF} under the condition \eqref{1.17} has been obtained by
Hao and Luo in \cite{Hao14}. We state their results as follows:
\begin{lemma}
Let $h_0\in H^s(\mathbb{R}^2)$, $v_0,b_0\in H^s(\mathbb{R}^3_{-})$,
where $s\leq n+1$, and the Taylor sign condition \eqref{1.17} holds at $t=0$. Then there exists $T>0$ and a unique solution $(v,b,q,h)$ of \eqref{IMHDF}
with $v,b\in L^{\infty}([0,T],H^3(\mathbb{R}^s_{-}))$,
$\nabla q\in L^{\infty}([0,T],H^2(\mathbb{R}^{s-1}_{-}))$,
$h\in L^{\infty}([0,T],H^{s}(\mathbb{R}^2))$.
\end{lemma}

Though co-normal derivatives of the MHD solutions and co-normal
derivatives of ideal MHD solutions vanish on the free boundary, their differences oscillate
dramatically in the vicinity of the free boundary, thus the conormal
functional spaces are not suitable for studying the convergence rates of inviscid
limit. Thus, we define the following functional spaces
\begin{alignat*}{2}
&\|f\|^2_{X^{m,s}}:=\sum_{l\leq m,|\alpha|\leq m+s-l}\|\partial_t^lZ^\alpha f\|^2_{L^2(\mathbb{R}^3_-)},
&\|f\|^2_{X^m}:=\|f\|^2_{X^{m,0}},\\
&\|f\|^2_{X^{m,s}_{tan}}:=\sum_{l\leq m,|\alpha|\leq m+s-l}\|\partial_t^l\partial_y^\alpha f\|^2_{L^2(\mathbb{R}^3_-)},
&\|f\|^2_{X^m_{tan}}:=\|f\|^2_{X^{m,0}_{tan}},\\
&|h|^2_{X^{m,s}}:=\sum_{l\leq m,|\alpha|\leq m+s-l}|\partial_t^lZ^\alpha h|^2_{L^2(\mathbb{R}^2)},
&|h|^2_{X^m}:=|h|^2_{X^{m,0}},\\
&\|f\|^2_{Y^{m,s}_{tan}}:=\sum_{l\leq m,|\alpha|\leq m+s-l}\|\partial_t^l\partial_y^\alpha f\|^2_{L^\infty(\mathbb{R}^3_-)},
&\|f\|^2_{Y^m_{tan}}:=\|f\|^2_{X^{m,0}_{tan}},\\
&|h|^2_{Y^{m,s}}:=\sum_{l\leq m,|\alpha|\leq m+s-l}|\partial_t^lZ^\alpha h|^2_{L^\infty(\mathbb{R}^2)},
&|h|^2_{Y^m}:=|h|^2_{X^{m,0}},
\end{alignat*}
where the differential operators are defined as $\mathcal{Z}_1=\partial_{y_1}, \mathcal{Z}_2=\partial_{y_2}, \mathcal{Z}_3 =\frac{z}{1-z}\partial_z$.
Also, we use $|\cdot|_m$ to denote the standard Sobolev norm defined in the
horizontal space $\mathbb{R}^2$.

\subsection{Main Results for MHD Equations without Surface Tension}
Let $\sigma=0$, the following proposition concerns the uniform regularity of time derivatives
of the free boundary problems for MHD when $\epsilon=\lambda\in(0,1]$.

\begin{proposition}\label{Proposition1.1}
For $m\geq6$, assume the initial data $(v_0^{\epsilon},b_0^{\epsilon},h_0^{\epsilon})$
 satisfy the compatibility conditions $\Pi S^\varphi v^\epsilon_0n|_{z=0}=0$
 and $\Pi S^\varphi b^\epsilon_0n|_{z=0}=0$ in $(0,T]$ and the regularities
 \begin{align}\label{p1}
\sup_{\epsilon\in(0,1]}&(|h^\epsilon_0|_{X^{m-1,1}}
+\epsilon^{\frac{1}{2}}|h^\epsilon_0|_{X^{m-1,\frac{3}{2}}}
+\|v^\epsilon_0\|_{X^{m-1,1}}+\|b^\epsilon_0\|_{X^{m-1,1}}
+\|\omega^\epsilon_{v0}\|_{X^{m-1}}+\|\omega^\epsilon_{b0}\|_{X^{m-1}}\nonumber
\\
&+\|\omega^\epsilon_{v0}\|_{1,\infty}+\|\omega^\epsilon_{b0}\|_{1,\infty}
+\epsilon^{\frac{1}{2}}(\|\partial_z\omega^\epsilon_{v0}\|_{L^\infty}
+\|\partial_z\omega^\epsilon_{b0}\|_{L^\infty}))\leq C_0,
 \end{align}
 where $C_0\geq0$ is suitably small such that the Taylor sign condition $g-\partial_z^{\varphi^\epsilon} q^\epsilon|_{z=0}\geq c_0>0$ holds. Then the unique solution to \eqref{MHDF} satisfies
\begin{align}
\sup_{t\in[0,T]}&(|h^\epsilon|^2_{X^{m-1,1}}
+\epsilon^{\frac{1}{2}}|h^\epsilon|^2_{X^{m-1,\frac{3}{2}}}
+\|v^\epsilon\|^2_{X^{m-1,1}}+\|b^\epsilon\|^2_{X^{m-1,1}}
+\|\partial_zv^\epsilon\|^2_{X^{m-2}}+\|\partial_zb^\epsilon\|^2_{X^{m-2}}
\nonumber\\
&+\|\omega^\epsilon_{v}\|_{X^{m-2}}+\|\omega^\epsilon_{b}\|_{X^{m-2}}
+\|\partial_zv^\epsilon\|^2_{1,\infty}+\|\partial_zb^\epsilon\|^2_{1,\infty}
+\epsilon^{\frac{1}{2}}(\|\partial_{zz}v^\epsilon\|^2_{L^\infty}
+\|\partial_{zz}b^\epsilon\|^2_{L^\infty}))
\nonumber\\
&+\|\partial_t^mh\|^2_{L^4([0,T],L^2)}+\epsilon\|\partial_t^mh\|^2_{L^4([0,T],H^{\frac{1}{2}})}
+\epsilon\int_0^T\|\nabla v^\epsilon\|^2_{X^{m-1,1}}+\|\nabla b^\epsilon\|^2_{X^{m-1,1}}
\nonumber\\
&+\|\nabla\partial_zv^\epsilon\|^2_{X^{m-2}}+\|\nabla\partial_zb^\epsilon\|^2_{X^{m-2}}dt
\leq C.
 \end{align}

As $\epsilon,\lambda\to0$,
the solution to \eqref{IMHDF} satisfies the following regularities
\begin{align}
\sup_{t\in[0,T]}&(|h|^2_{X^{m-1,1}}
+\|v\|^2_{X^{m-1,1}}+\|b\|^2_{X^{m-1,1}}
+\|\partial_zv\|^2_{X^{m-2}}+\|\partial_zb\|^2_{X^{m-2}}
\nonumber\\
&+\|\omega_{v}\|_{X^{m-2}}+\|\omega_{b}\|_{X^{m-2}}
+\|\partial_zv\|^2_{1,\infty}+\|\partial_zb\|^2_{1,\infty})
+\|\partial_t^mh\|^2_{L^4([0,T],L^2)}
\leq C,
 \end{align}
where the Taylor sign condition $g-\partial_zp|_{z=0}\geq c_0>0$ holds.
\end{proposition}

Since we can not obtain the estimates of $\|\partial_t^m v^{\e}\|_{L^4([0,T],L^2)}$
and $\|\partial_t^m b^{\e}\|_{L^4([0,T],L^2)}$ by \eqref{p1}, we need some additional conditions
$\partial_t^m v^{\e}|_{t=0}$, $\partial_t^m b^{\e}|_{t=0}$,
$\partial_t^m h^{\e}|_{t=0}\in L^2(\mathbb{R}^3_{-})$.
For convenience, we denote $\partial_t^m f|_{t=0}$ and $\|f|_{t=0}\|_{X^{m,s}}$ by $\partial_t^m f_0$ and $\|f_0\|_{X^{m,s}}$
in the following parts of this paper.

We give some remarks on Proposition \ref{Proposition1.1}:
\begin{remark}
In the following viewpoints, the proof of this Proposition is obviously different from that of Lee \rm\cite{Lee17}:

(i)\emph{ Alinhac's good unknown}.
When $\sigma=0$, let $0<l+|\alpha|\leq m,l\leq m-1$,
we need the following Alinhac's good unknown to estimate the tangential derivatives
\begin{equation}
\left\{
\begin{aligned}
&V^{l,\alpha}=\partial_t^lZ^\alpha v-\partial_z^\varphi v\partial_t^lZ^\alpha \eta,
\\
&B^{l,\alpha}=\partial_t^lZ^\alpha b-\partial_z^\varphi v\partial_t^lZ^\alpha \eta,
\\
&Q^{l,\alpha}=\partial_t^lZ^\alpha q-\partial_z^\varphi v\partial_t^lZ^\alpha \eta.
\end{aligned}
\right.
\end{equation}
The divergence free property and the zero boundary condition of $b$
play critical roles in cancelling the
nontransport-type nonlinear terms involving $b\cdot\nabla^\varphi B^{l,\alpha},\, b\cdot\nabla^\varphi V^{l,\alpha}$ when we make higher order energy estimates.

(ii)\emph{Normal derivative estimate}.
The authors of \cite{Elgindi17,Lee17,Masmoudi12,Mei16,Mei18} estimated the normal derivatives
$\|\partial_zv\|_{m-1}$ and $\|\partial_zb\|_{m-1}$
by $\Pi S^\varphi vn$ and $\Pi S^\varphi bn$ and its evolution equations.
Motivated by Wu \cite{Wu16},
we analyze the relationship between the vorticity and the normal derivatives on the free boundary,
and estimate the normal derivatives
by controlling the vorticity and the equations
\begin{equation}\label{vorticity equations}
\left\{
\begin{aligned}
&\partial_t^\varphi\omega_{vh}-\epsilon\Delta^\varphi\omega_{vh}
+v\cdot\nabla^\varphi\omega_{vh}-b\cdot\nabla^\varphi\omega_{bh}
=F_v^0[\nabla\varphi](\omega_{vh},\omega_{bh},\partial_jv^i,\partial_jb^i)\\
&\partial_t^\varphi\omega_{bh}-\epsilon\Delta^\varphi\omega_{bh}
+v\cdot\nabla^\varphi\omega_{bh}-b\cdot\nabla^\varphi\omega_{vh}
=F^0_b[\nabla\varphi](\omega_{vh},\omega_{bh},\partial_jv^i,\partial_jb^i)\\
&\omega^1_{vh}|_{z=0}=F^1[\nabla\varphi](\partial_jv^i),
\omega^1_{bh}|_{z=0}=F^1[\nabla\varphi](\partial_jb^i),\\
&\omega^2_{vh}|_{z=0}=F^2[\nabla\varphi](\partial_jv^i),
\omega^2_{bh}|_{z=0}=F^2[\nabla\varphi](\partial_jb^i).
\end{aligned}
\right.
\end{equation}
where $j=1,2,i=1,2,3$,
$F_v^0=\omega_v\cdot\nabla^\varphi v_h-\omega_b\cdot\nabla^\varphi b_h$ and
$F_b^0=[\nabla^\varphi\times,b\cdot\nabla^\varphi]v-[\nabla^\varphi\times,v\cdot\nabla^\varphi]b$
are the quadratic polynomial vector with respect to
$\omega_{vh},\omega_{bh},\partial_jv^i,\partial_jb^i$,
respectively.
$F^1[\nabla\varphi](\partial_jv^i)$,
$F^2[\nabla\varphi](\partial_jv^i)$,
$F^1[\nabla\varphi](\partial_jb^i)$,
$F^2[\nabla\varphi](\partial_jb^i)$
are polynomials with respect to $\partial_jv^i,\partial_jb^i$,
respectively,
all the coefficients are fractions of $\nabla\varphi$.
Since $b$ is zero on the free boundary, we have $F^1[\nabla\varphi](\partial_jb^i)=F^2[\nabla\varphi](\partial_jb^i)=0$.

(iii) \emph{Pressure estimates and Taylor sign condition}.
In \cite{Lee17},
the Taylor sign condition $g-\partial_z^{\varphi^\epsilon} q^{\epsilon,E}|_{z=0}\geq c_0>0$ is imposed on the ideal MHD part of the pressure $q^\epsilon$.
In fact,
$q^\epsilon$ has a decomposition $q^\epsilon=q^{\epsilon,E}+q^{\epsilon,NS}$,
where $q^{\epsilon,E}$ satisfies
\begin{equation}
\left\{\begin{array}{ll}
\Delta^{\varphi^\epsilon}q^{\epsilon,E}
=-\partial_i^{\varphi^\epsilon}v^{\epsilon,j}\partial_i^{\varphi^\epsilon}v^{\epsilon,j}
+\partial_i^{\varphi^\epsilon}b^{\epsilon,j}\partial_i^{\varphi^\epsilon}b^{\epsilon,j},
\\[6pt]
q^{\epsilon,E}|_{z=0}=gh^\epsilon,
\end{array}\right.
\end{equation}
and $q^{\epsilon,NS}$ satisfies
\begin{equation}
\left\{\begin{array}{ll}
\Delta^{\varphi^\epsilon}q^{\epsilon,NS}=0,\\[6pt]
q^{\epsilon,NS}|_{z=0}=2\epsilon S^{\varphi^\epsilon}v\nn\cdot \nn.
\end{array}\right.
\end{equation}
In \cite{Lee17}, the author do not discuss whether $\partial_z^{\varphi^\epsilon}q^{\epsilon,E}|_{z=0}$
converges pointwisely to  $\partial_z^{\varphi}q|_{z=0}$ or not,
since $q^{\epsilon,E}$ has boundary layer in the vicinity of the free
boundary in general,
thus $\partial_z^{\varphi^\epsilon}q^{\epsilon,E}|_{z=0}$
may also has boundary layer.
Therefore, the Taylor sign condition in \cite{Lee17} is imposed on the Euler part of the pressure $q^\epsilon$.

However, we have proved
$\|\partial_{zz} v\|_{L^{\infty}},\sqrt{\e}\|\partial_{zz} v\|_{L^{\infty}}$ are bounded.
Due to the fact that $\partial_z^{\varphi^{\e}} q^{\e}|_{z=0}
=\e\Delta^{\varphi^{\e}} v^3-\partial_t v^3
-v_y^{\e} \cdot\nabla_y v^{\e,3}$,
$\partial_z^{\varphi^{\e}} q^{\e}|_{z=0}$
converges to $\partial_z^{\varphi} q|_{z=0}$ pointwisely.
In this paper,
our Taylor sign condition is
$g-\partial_z^{\varphi^{\e}} q^{\e}|_{z=0} \geq c_0 >0$.

Moreover, $\|\partial_t^{\ell} q\|_{L^2}$ has no bound in general.
When $|\alpha|=0$, let $0\leq \ell \leq m-1$,
we estimate $V^{\ell,0}$, $B^{\ell,0}$
and $\nabla\partial_t^{\ell}q$,
where the dynamical boundary condition can not be used.
\end{remark}

Denote by $\hat{\omega}_v=\omega^\epsilon_v-\omega_v$
and $\hat{\omega}_b=\omega^\epsilon_b-\omega_b$,
where $\omega^\epsilon_v, \omega^\epsilon_b$, $\omega_v, \omega_b$
are the vorticity of MHD equations and ideal MHD equations, respectively.
It follows that $\hat{\omega}_v$, $\hat{\omega}_b$ satisfy
the following vorticity difference equations
\begin{equation}\label{1.25}
\left\{\begin{array}{ll}
\partial_t^{\varphi^\epsilon}\hat{\omega}_{vh}
-\epsilon\Delta^{\varphi^\epsilon}\hat{\omega}_{vh}
+v^\epsilon\cdot\nabla^{\varphi^\epsilon}\hat{\omega}_{vh}
-b^\epsilon\cdot\nabla^{\varphi^\epsilon}\hat{\omega}_{bh}
\\[8pt]\quad

=F_v^{0,\epsilon}[\nabla\varphi^\epsilon](\omega_{vh}^{\varphi^\epsilon},
\omega_{bh}^{\varphi^\epsilon},\partial_jv^{\epsilon,i},\partial_jb^{\epsilon,i})
-F_v^0[\nabla\varphi](\omega_{vh},\omega_{bh},\partial_jv^i,\partial_jb^i)
+\epsilon\Delta^{\varphi^\epsilon}\omega_{vh}
\\[8pt]\quad

+\partial_z^\varphi\omega_{vh}\partial_t^{\varphi^\epsilon}\hat{\eta}
+\partial_z^\varphi\omega_{vh}v\cdot\nabla^{\varphi^\epsilon}\hat{\eta}
-\hat{v}\cdot\nabla^\varphi\omega_{vh}
-\partial_z^\varphi\omega_{bh}b\cdot\nabla^{\varphi^\epsilon}\hat{\eta}
+\hat{b}\cdot\nabla^\varphi\omega_{bh},
\\[8pt]

\partial_t^{\varphi^\epsilon}\hat{\omega}_{bh}-\epsilon\Delta^{\varphi^\epsilon}\hat{\omega}_{bh}
+v^\epsilon\cdot\nabla^{\varphi^\epsilon}\hat{\omega}_{bh}
-b^\epsilon\cdot\nabla^{\varphi^\epsilon}\hat{\omega}_{vh}
\\[8pt]\quad

=F_b^{0,\epsilon}[\nabla\varphi^\epsilon](\omega_{vh}^{\varphi^\epsilon},
\omega_{bh}^{\varphi^\epsilon},\partial_jv^{\epsilon,i},\partial_jb^{\epsilon,i})
-F_b^0[\nabla\varphi](\omega_{vh},\omega_{bh},\partial_jv^i,\partial_jb^i)
+\epsilon\Delta^{\varphi^\epsilon}\omega_{bh}
\\[8pt]\quad

+\partial_z^\varphi\omega_{bh}\partial_t^{\varphi^\epsilon}\hat{\eta}
+\partial_z^\varphi\omega_{bh}v\cdot\nabla^{\varphi^\epsilon}\hat{\eta}
-\hat{v}\cdot\nabla^\varphi\omega_{bh}
-\partial_z^\varphi\omega_{vh}b\cdot\nabla^{\varphi^\epsilon}\hat{\eta}
+\hat{b}\cdot\nabla^\varphi\omega_{vh},
\\[5pt]

\hat{\omega}_{vh}|_{z=0}=F^{1,2}[\nabla\varphi](\partial_jv^i)
-\omega_{vh}|_{z=0},
\\[8pt]

\hat{\omega}_{bh}|_{z=0}=F^{1,2}[\nabla\varphi](\partial_jb^i)
-\omega_{bh}|_{z=0},
\\[8pt]

(\hat{\omega}_{vh}|_{t=0},\hat{\omega}_{bh}|_{t=0})=(\hat{\omega}_{v0},\hat{\omega}_{b0})^\top.
\end{array}\right.
\end{equation}
where
\begin{align*}
&F_v^{0,\epsilon}[\nabla\varphi^\epsilon](\omega_{vh}^{\varphi^\epsilon},
\omega_{bh}^{\varphi^\epsilon},\partial_jv^{\epsilon,i},\partial_jb^{\epsilon,i})
-F_v^0[\nabla\varphi](\omega_{vh},\omega_{bh},\partial_jv^i,\partial_jb^i)
\\[4pt]
=&\omega_v^\epsilon\cdot\nabla^\varphi v_h^\epsilon-\omega_b^\epsilon\cdot\nabla^\varphi b_h^\epsilon
-\omega_v\cdot\nabla^\varphi v_h+\omega_b\cdot\nabla^\varphi b_h,
\\[4pt]
&F_b^{0,\epsilon}[\nabla\varphi^\epsilon](\omega_{vh}^{\varphi^\epsilon},
\omega_{bh}^{\varphi^\epsilon},\partial_jv^{\epsilon,i},\partial_jb^{\epsilon,i})
-F_b^0[\nabla\varphi](\omega_{vh},\omega_{bh},\partial_jv^i,\partial_jb^i)
\\[4pt]
=&[\nabla^\varphi\times,b^\varphi\cdot\nabla^\varphi]v^\varphi
-[\nabla^\varphi\times,v^\varphi\cdot\nabla^\varphi]b^\varphi
-[\nabla^\varphi\times,b\cdot\nabla^\varphi]v+[\nabla^\varphi\times,v\cdot\nabla^\varphi]b,
\end{align*}
in which $[\cdot,\cdot]$ stands for the commutator.
$F^{1,2}[\nabla\varphi](\partial_jv^i)$,
$F^{1,2}[\nabla\varphi](\partial_jb^i)$
are defined in \eqref{vorticity equations}.

Based on the analysis of the equations \eqref{1.25},
the following theorem implies two independent sufficient conditions for the formation of strong or weak vorticity layer.
The two conditions include the initial vorticity layer in the vicinity of the initial time
and the discrepancy between boundary value of MHD vorticity and that of ideal MHD vorticity.

\begin{theorem}\label{Theorem1.1}
Assume $T>0$ is finite,
$(v^\epsilon,b^\epsilon,h^\epsilon)$ is the solution in $[0,T]$ of MHD equations \eqref{MHDF}
with initial data $(v_0^\epsilon,b_0^\epsilon,h_0^\epsilon)$ satisfying \eqref{p1},
and $(v, b, h)$ is the solution in $[0,T]$ of ideal MHD
equations\eqref{IMHDF} with initial data
$(v_0,b_0,h_0)\in X^{m-1,1}(\mathbb{R}^3_-)\times X^{m-1,1}(\mathbb{R}^2)$.

(1) If the ideal MHD boundary data satisfies
$\Pi S^\varphi vn|_{z=0}=0$ and $\Pi S^\varphi bn|_{z=0}=0$ in $(0,T]$,
the initial MHD data satisfies
$\lim\limits_{\epsilon\to0}(\nabla^{\varphi^\epsilon}\times v_0^\epsilon)-\nabla^\varphi\times\lim\limits_{\epsilon\to0}v_0^\epsilon\neq0$
or $\lim\limits_{\epsilon\to0}(\nabla^{\varphi^\epsilon}\times b_0^\epsilon)-\nabla^\varphi\times\lim\limits_{\epsilon\to0}b_0^\epsilon\neq0$
in the initial set $\mathcal{A}_0=\{x|-\sqrt{\epsilon}\leq z<0\}$,
then the MHD solution has a strong vorticity layer.
Assume
\[\lim\limits_{\epsilon\to0}(\nabla^{\varphi^\epsilon}\times v_0^\epsilon)-\nabla^\varphi\times\lim\limits_{\epsilon\to0}v_0^\epsilon\neq0\,\,
{\rm and}\,\,
\lim\limits_{\epsilon\to0}(\nabla^{\varphi^\epsilon}\times b_0^\epsilon)-\nabla^\varphi\times\lim\limits_{\epsilon\to0}b_0^\epsilon\neq0\]
in the initial set $\mathcal{A}_0$,
then the MHD solution has a strong vorticity layer satisfying
\begin{equation}\label{1.26}
\begin{array}{ll}
\lim\limits_{\epsilon\to0}\|\omega_v^\epsilon-\omega_v\|_{L^\infty(\mathcal{X}(\mathcal{A}_0)\times(0,T]}\neq0,
\\[8pt]
\lim\limits_{\epsilon\to0}\|\omega_b^\epsilon-\omega_b\|_{L^\infty(\mathcal{X}(\mathcal{A}_0)\times(0,T]}\neq0,
\\[8pt]
\lim\limits_{\epsilon\to0}\|\partial_z^{\varphi^\epsilon}v^\epsilon
-\partial_z^\varphi v\|_{L^\infty(\mathcal{X}(\mathcal{A}_0)\times(0,T]}\neq0,
\\[8pt]
\lim\limits_{\epsilon\to0}\|\partial_z^{\varphi^\epsilon}b^\epsilon
-\partial_z^\varphi b\|_{L^\infty(\mathcal{X}(\mathcal{A}_0)\times(0,T]}\neq0,
\\[8pt]
\lim\limits_{\epsilon\to0}\|S^{\varphi^\epsilon}v^\epsilon
-S^\varphi v\|_{L^\infty(\mathcal{X}(\mathcal{A}_0)\times(0,T]}\neq0,
\\[8pt]
\lim\limits_{\epsilon\to0}\|S^{\varphi^\epsilon}b^\epsilon
-S^\varphi b\|_{L^\infty(\mathcal{X}(\mathcal{A}_0)\times(0,T]}\neq0,
\\[8pt]
\lim\limits_{\epsilon\to0}\|\nabla^{\varphi^\epsilon}q^\epsilon
-\nabla^\varphi q\|_{L^\infty(\mathcal{X}(\mathcal{A}_0)\times(0,T]}\neq0,
\end{array}
\end{equation}
where $\mathcal{X}(\mathcal{A}_0)=\{\mathcal{X}(t,x)|\mathcal{X}(0,x)\in\mathcal{A}_0,
\partial_t\mathcal{X}(t,x)=v(t,\Phi^{-1}\circ X)\}$.

(2) If the initial MHD data satisfies
$\lim\limits_{\epsilon\to0}(\nabla^{\varphi^\epsilon}\times v_0^\epsilon)-\nabla^\varphi\times\lim\limits_{\epsilon\to0}v_0^\epsilon=0$
and
$\lim\limits_{\epsilon\to0}(\nabla^{\varphi^\epsilon}\times b_0^\epsilon)-\nabla^\varphi\times\lim\limits_{\epsilon\to0}b_0^\epsilon=0$,
 the ideal MHD boundary data satisfies
 $\Pi S^\varphi vn|_{z=0}\neq0$ or $\Pi S^\varphi bn|_{z=0}\neq0$
 in $(0,T]$,
then the MHD solution has a strong vorticity layer.
Assume
\[\Pi S^\varphi vn|_{z=0}\neq0 \,\,
{\rm and}\,\,
\Pi S^\varphi bn|_{z=0}\neq0\]
in $(0,T]$,
then the MHD solution has a strong vorticity layer satisfying
\begin{equation}\label{1.23}
\begin{array}{ll}
\lim\limits_{\epsilon\to0}|\omega_v^\epsilon|_{z=0}-\omega_v|_{z=0}|_{L^\infty(\mathbb{R}^2)\times(0,T]}\neq0,
\\[8pt]
\lim\limits_{\epsilon\to0}|\omega_b^\epsilon|_{z=0}-\omega_b|_{z=0}|_{L^\infty(\mathbb{R}^2)\times(0,T]}\neq0,
\\[8pt]
\lim\limits_{\epsilon\to0}\|\omega_v^\epsilon-\omega_v\|_{L^\infty(\mathbb{R}^2
\times(0,O(\epsilon^{\frac{1}{2}-\delta_z}))\times(0,T])]}\neq0,
\\[8pt]
\lim\limits_{\epsilon\to0}\|\omega_b^\epsilon-\omega_b\|_{L^\infty(\mathbb{R}^2
\times(0,O(\epsilon^{\frac{1}{2}-\delta_z}))\times(0,T])]}\neq0,
\\[8pt]
\lim\limits_{\epsilon\to0}\|\partial_z^{\varphi^\epsilon}v^\epsilon
-\partial_z^\varphi v\|_{L^\infty(\mathbb{R}^2
\times(0,O(\epsilon^{\frac{1}{2}-\delta_z}))\times(0,T])]}\neq0,
\\[8pt]
\lim\limits_{\epsilon\to0}\|\partial_z^{\varphi^\epsilon}b^\epsilon
-\partial_z^\varphi b\|_{L^\infty(\mathbb{R}^2
\times(0,O(\epsilon^{\frac{1}{2}-\delta_z}))\times(0,T])]}\neq0,
\\[8pt]
\lim\limits_{\epsilon\to0}\|S^{\varphi^\epsilon}v^\epsilon
-S^\varphi v\|_{L^\infty(\mathbb{R}^2
\times(0,O(\epsilon^{\frac{1}{2}-\delta_z}))\times(0,T])]}\neq0,
\\[8pt]
\lim\limits_{\epsilon\to0}\|S^{\varphi^\epsilon}b^\epsilon
-S^\varphi b\|_{L^\infty(\mathbb{R}^2
\times(0,O(\epsilon^{\frac{1}{2}-\delta_z}))\times(0,T])]}\neq0,
\\[8pt]
\lim\limits_{\epsilon\to0}\|\nabla^{\varphi^\epsilon}q^\epsilon
-\nabla^\varphi q\|_{L^\infty(\mathbb{R}^2
\times(0,O(\epsilon^{\frac{1}{2}-\delta_z}))\times(0,T])]}\neq0,
\end{array}
\end{equation}
for some constant $\delta_z>0$.

(3) If the initial MHD data satisfies
$\lim\limits_{\epsilon\to0}(\nabla^{\varphi^\epsilon}\times v_0^\epsilon)-\nabla^\varphi\times\lim\limits_{\epsilon\to0}v_0^\epsilon=0$ and
$\lim\limits_{\epsilon\to0}(\nabla^{\varphi^\epsilon}\times b_0^\epsilon)-\nabla^\varphi\times\lim\limits_{\epsilon\to0}b_0^\epsilon=0$,
the ideal MHD boundary data satisfies $\Pi S^\varphi vn|_{z=0}=0$ and $\Pi S^\varphi bn|_{z=0}=0$ in $(0,T]$,
then the MHD solution has a weak vorticity layer satisfying
\begin{equation}
\begin{array}{ll}
\lim\limits_{\epsilon\to0}\|\omega_v^\epsilon-\omega_v\|_{L^\infty(\mathfrak{A}(R^3_-)\times(0,T])}=0,
\\[7pt]
\lim\limits_{\epsilon\to0}\|\omega_b^\epsilon-\omega_b\|_{L^\infty(\mathfrak{A}(R^3_-)\times(0,T])}=0,
\\[7pt]
\lim\limits_{\epsilon\to0}\|\partial_z^{\varphi^\epsilon}v^\epsilon
-\partial_z^\varphi v\|_{L^\infty(\mathfrak{A}(R^3_-)\times(0,T])}=0,
\\[7pt]
\lim\limits_{\epsilon\to0}\|\partial_z^{\varphi^\epsilon}b^\epsilon
-\partial_z^\varphi b\|_{L^\infty(\mathfrak{A}(R^3_-)\times(0,T])}=0,
\\[7pt]
\lim\limits_{\epsilon\to0}\|S^{\varphi^\epsilon}v^\epsilon
-S^\varphi v\|_{L^\infty(\mathfrak{A}(R^3_-)\times(0,T])}=0,
\\[7pt]
\lim\limits_{\epsilon\to0}\|S^{\varphi^\epsilon}b^\epsilon
-S^\varphi b\|_{L^\infty(\mathfrak{A}(R^3_-)\times(0,T])}=0,
\\[7pt]
\lim\limits_{\epsilon\to0}\|\nabla^{\varphi^\epsilon}q^\epsilon
-\nabla^\varphi q\|_{L^\infty(\mathfrak{A}(R^3_-)\times(0,T])}=0,
\end{array}
\end{equation}
where $\mathfrak{A}(R^3_-)=\mathbb{R}^3_-\cup \{x|z=0\}$ is the closure of $\mathbb{R}^3_-$.
\end{theorem}

\begin{remark}
(i) Denote that $\textsf{S}_n^v =\Pi \mathcal{S}^{\varphi} v \nn$
 and $\textsf{S}_n^v =\Pi \mathcal{S}^{\varphi} v \nn$ which satisfy the forced transport equations:
\begin{equation}\label{1.29}
\begin{array}{ll}
\partial_t^{\varphi} \textsf{S}_n^v + v\cdot\nabla^{\varphi} \textsf{S}_n^v
-b\cdot\nabla^{\varphi} \textsf{S}_n^b =
-\frac{1}{2}\Pi\big((\nabla^{\varphi} v)^2+((\nabla^{\varphi} v)^{\top})^2\big) \nn
\\[8pt]\quad
+ (\partial_t^{\varphi}\Pi + v\cdot\nabla^{\varphi}\Pi)\mathcal{S}^{\varphi} v\nn
+ \Pi\mathcal{S}^{\varphi} v(\partial_t^{\varphi}\nn + v\cdot\nabla^{\varphi}\nn)
- \Pi((\mathcal{D}^{\varphi})^2 q)\nn
\\[8pt]\quad
+\frac{1}{2}\Pi\big((\nabla^{\varphi} b)^2+((\nabla^{\varphi} b)^{\top})^2\big) \nn
- b\cdot\nabla^{\varphi}\Pi\mathcal{S}^{\varphi} b\nn
- \Pi\mathcal{S}^{\varphi} bb\cdot\nabla^{\varphi}\nn,
\\[8pt]

\partial_t^{\varphi} \textsf{S}_n^b + v\cdot\nabla^{\varphi} \textsf{S}_n^b
-b\cdot\nabla^{\varphi} \textsf{S}_n^v =
(\partial_t^{\varphi}\Pi + v\cdot\nabla^{\varphi}\Pi)\mathcal{S}^{\varphi} b\nn
\\[8pt]\quad
+ \Pi\mathcal{S}^{\varphi} b(\partial_t^{\varphi}\nn + v\cdot\nabla^{\varphi}\nn)
- b\cdot\nabla^{\varphi}\Pi\mathcal{S}^{\varphi} v\nn
- \Pi\mathcal{S}^{\varphi} vb\cdot\nabla^{\varphi}\nn,
\end{array}
\end{equation}
where $\big((\mathcal{D}^{\varphi})^2 q\big)$ is the Hessian matrix of $q$.
The equation \eqref{1.29} implies that even if $\textsf{S}_n^v|_{t=0} = 0$
and $\textsf{S}_n^b|_{t=0} = 0$,
it is still possible that $\textsf{S}_n^v \neq 0$ or $\textsf{S}_n^b \neq 0$ in $(0,T]$
due to the force terms of \eqref{1.29}.

(ii) The estimate of $\lim\limits_{\e\rto 0}\|\partial_z^{\varphi^{\e}} f^{\e}
-\partial_z^{\varphi} f\|_{L^{\infty}}\neq 0$
results from that of $\lim\limits_{\e\rto 0}\|\partial_z f^{\e} - \partial_z f\|_{L^{\infty}}\neq 0$
and $\lim\limits_{\e\rto 0}\|\partial_z(\eta^{\e} -\eta)\|_{L^{\infty}} =0$ due to
the formula:
\begin{equation}\label{1.30}
\begin{array}{ll}
\partial_z^{\varphi^{\e}} f^{\e} - \partial_z^{\varphi} f
= \partial_z^{\varphi^{\e}}(f^{\e} -f) - \partial_z^{\varphi} f \, \partial_z^{\varphi^{\e}} (\eta^{\e} -\eta) \\[5pt]\hspace{1.95cm}

= \frac{1}{\partial_z\varphi^{\e}} \cdot\partial_z(v^{\e} -v)
- \partial_z^{\varphi} v \, \frac{1}{\partial_z\varphi^{\e}} \cdot \partial_z (\eta^{\e} -\eta).
\end{array}
\end{equation}
\end{remark}

Now we show our strategy of the proofs of Theorem \ref{Theorem1.1}.

First, we prove that the strong initial vorticity layer is one of sufficient
conditions for the existence of strong vorticity layer of the free boundary MHD equations in \eqref{1.26}.
By introducing  Lagrangian coordinates map, the equations \eqref{1.25} can be
transformed into two coupling  heat equations with damping and force terms.
If $\lim\limits_{\epsilon\to0}\|\widehat{\omega}_{uh}|_{t=0}\|_{L^\infty}(\mathcal{A}_0)\neq0$
or
$\lim\limits_{\epsilon\to0}\|\widehat{\omega}_{bh}|_{t=0}\|_{L^\infty}(\mathcal{A}_0)\neq0$,
then it follows that $\lim\limits_{\epsilon\to0}\|W_+|_{t=0}\|_{X(\mathcal{A}_0)\times(0,T])}\neq0$
or
$\lim\limits_{\epsilon\to0}\|W_-|_{t=0}\|_{X(\mathcal{A}_0)\times(0,T])}\neq0$.
Hence, we deduce that $\lim\limits_{\epsilon\to0}\|W_+\|_{X(\mathcal{A}_0)\times(0,T])}\neq0$
or $\lim\limits_{\epsilon\to0}\|W_-\|_{X(\mathcal{A}_0)\times(0,T])}\neq0$,
which implies there is at least one of
 $\lim\limits_{\epsilon\to0}\|\widehat{\omega}_{uh}\|_{X(\mathcal{A}_0)\times(0,T])}\neq0$,
$\lim\limits_{\epsilon\to0}\|\widehat{\omega}_{bh}\|_{X(\mathcal{A}_0)\times(0,T])}\neq0$
hold.

Second, \eqref{1.23} implies that the discrepancy between boundary value of
MHD vorticity and boundary value of ideal MHD vorticity is also one of sufficient conditions for
the existence of strong vorticity layer for the free boundary MHD equations.
If $\lim\limits_{\e\rto 0}\big\|\hat{\omega}_{vh}|_{t=0}\big\|_{L^{\infty}}=0$
and $\lim\limits_{\e\rto 0}\big\|\hat{\omega}_{bh}|_{t=0}\big\|_{L^{\infty}}=0$,
and the boundary value satisfies
$\hat{\omega}_{vh}|_{z=0} \neq 0$ or $\hat{\omega}_{bh}|_{z=0} \neq 0$,
by splitting into \eqref{fo1} and \eqref{bdy1},
we can prove that
$\lim\limits_{\e\rto 0}\|W_{+}\|_{L^{\infty}(\mathbb{R}^2\times
[0, O(\e^{\frac{1}{2} -\delta_z}))\times (0,T])}\neq 0$
or $\lim\limits_{\e\rto 0}\|W_{-}\|_{L^{\infty}(\mathbb{R}^2\times
[0, O(\e^{\frac{1}{2} -\delta_z}))\times (0,T])}\neq 0$,
which implies that there is at least one of
$\lim\limits_{\e\rto 0}\|\omega^{\e}_v-\omega_v\|_{L^{\infty}
(\mathbb{R}^2\times [0, O(\e^{\frac{1}{2} -\delta_z}))\times (0,T])}\neq 0$,
 $\lim\limits_{\e\rto 0}\|\omega^{\e}_b-\omega_b\|_{L^{\infty}
(\mathbb{R}^2\times [0, O(\e^{\frac{1}{2}-\delta_z}))\times(0,T])}\neq 0$ holds.

Before estimating the the convergence rates,
we give the difference equations between viscous MHD system and ideal MHD system.
Denote by $\hat{v}= v^\epsilon-v$, $\hat{b}= b^\epsilon-b$,
$\hat{h} = h^\epsilon-h$, then $\hat{v}$, $\hat{b}$, $\hat{h}$
satisfy the following equations:
\begin{equation}\label{difference1}
\left\{\begin{aligned}
&\partial_t^{\varphi^{\e}}\hat{v} -\partial_z^{\varphi} v \partial_t^{\varphi^{\e}}\hat{\eta}
+ v^{\e} \cdot\nabla^{\varphi^{\e}} \hat{v}-b^{\e} \cdot\nabla^{\varphi^{\e}} \hat{b}
 -v^{\e}\cdot \nabla^{\varphi^{\e}}\hat{\eta}\, \partial_z^{\varphi} v
+b^{\e}\cdot \nabla^{\varphi^{\e}}\hat{\eta}\, \partial_z^{\varphi} b
 \\
&\quad= 2\e \nabla^{\varphi^{\e}} \cdot\mathcal{S}^{\varphi^{\e}} \hat{v}
+ \e \Delta^{\varphi^{\e}} v
- \nabla^{\varphi^{\e}} \hat{q} +\partial_z^{\varphi} q\nabla^{\varphi^{\e}}\hat{\eta}
- \hat{v}\cdot\nabla^{\varphi} v+ \hat{b}\cdot\nabla^{\varphi} b,
&  x\in\mathbb{R}^3_{-},
\\
&\partial_t^{\varphi^{\e}}\hat{b} -\partial_z^{\varphi} b \partial_t^{\varphi^{\e}}\hat{\eta}
+ v^{\e} \cdot\nabla^{\varphi^{\e}} \hat{b}-b^{\e} \cdot\nabla^{\varphi^{\e}} \hat{v}
 -v^{\e}\cdot \nabla^{\varphi^{\e}}\hat{\eta}\, \partial_z^{\varphi} b
+b^{\e}\cdot \nabla^{\varphi^{\e}}\hat{\eta}\, \partial_z^{\varphi} v
 \\
&\quad= 2\e \nabla^{\varphi^{\e}} \cdot\mathcal{S}^{\varphi^{\e}} \hat{v}
+ \e \Delta^{\varphi^{\e}} b
- \hat{v}\cdot\nabla^{\varphi} b+ \hat{b}\cdot\nabla^{\varphi} v,
&  x\in\mathbb{R}^3_{-},
\\
&\nabla^{\varphi^{\e}}\cdot \hat{v} = \partial_z^{\varphi}v \cdot\nabla^{\varphi^{\e}}\hat{\eta},
\nabla^{\varphi^{\e}}\cdot \hat{b} = \partial_z^{\varphi}b \cdot\nabla^{\varphi^{\e}}\hat{\eta},
& x\in\mathbb{R}^3_{-},\\
&\partial_t \hat{h} + v_y\cdot \nabla_y \hat{h} = \hat{v}\cdot\NN^{\e},
&\{z=0\},\\
&(\hat{q} -g \hat{h})\NN^{\e} -2\e \mathcal{S}^{\varphi^\e}\hat{v}\,\NN^\e = 2\e \mathcal{S}^{\varphi^\e}v\,\NN^\e,
&\{z=0\},\\
&\hat{b} =0,
&\{z=0\},\\
&(\hat{v},\hat{b},\hat{h})|_{t=0} = (v_0^\e -v_0,b_0^\e -b_0,h_0^\e -h_0).
\end{aligned}\right.
\end{equation}

Now, we define the following variables, which is similar to Alinhac's good unknown:
\begin{equation}\label{Alinhac's good unknown}
\left\{
\begin{aligned}
\hat{V}^{l,\alpha}=\partial_t^lZ^\alpha\hat{v}-\partial_z^\varphi v\partial_t^lZ^\alpha\hat{\eta},\\
\hat{B}^{l,\alpha}=\partial_t^lZ^\alpha\hat{b}-\partial_z^\varphi v\partial_t^lZ^\alpha\hat{\eta},\\
\hat{Q}^{l,\alpha}=\partial_t^lZ^\alpha\hat{q}-\partial_z^\varphi v\partial_t^lZ^\alpha\hat{\eta}.
\end{aligned}
\right.
\end{equation}
Similarly to the estimates of $V^{l,\alpha}$, $B^{l,\alpha}$, $Q^{l,\alpha}$,
we estimate $\hat{V}^{l,\alpha}$, $\hat{B}^{l,\alpha}$, $\hat{Q}^{l,\alpha}$ when $|\alpha|>0$,
and estimate $\hat{V}^{l,0},\hat{B}^{l,0},
\partial_t^l\nabla \hat{q}$ when $|\alpha| = 0$.
By the divergence free property and the zero boundary condition of $b^\epsilon$,
we can cancel the nontransport type nonlinear terms involving $ b\cdot\nabla^\varphi \hat{B}^{l,\alpha},\, b\cdot\nabla^\varphi \hat{V}^{l,\alpha}$
by combining the two energy estimates.

The following theorem concerns the convergence rates of the inviscid limits
of \eqref{MHDF}.

\begin{theorem}\label{Theorem1.2}
Assume $T>0$ is finite,
$(v^\epsilon,b^\epsilon,h^\epsilon)$ is the solution in $[0,T]$ of MHD equations \eqref{MHDF}
with initial data $(v_0^\epsilon,b_0^\epsilon,h_0^\epsilon)$ satisfies \eqref{p1},
and $(v, b, h)$ is the solution in $[0,T]$ of ideal MHD equations \eqref{IMHDF}
 with initial data  satisfy $(v_0,b_0,h_0)\in X^{m-1,1}(\mathbb{R}^3_-)\times X^{m-1,1}(\mathbb{R}^2)$.
Assume there exists an integer k where $1\leq k\leq m-2$, such that
$\|v_0^\epsilon-v_0\|_{X^{k-1,1}}(\mathbb{R}^3_-)=O(\epsilon^{\lambda^v})$,
$\|b_0^\epsilon-b_0\|_{X^{k-1,1}}(\mathbb{R}^3_-)=O(\epsilon^{\lambda^b})$,
$|h_0^\epsilon-h_0|_{X^{k-1,1}}(\mathbb{R}^2)=O(\epsilon^{\lambda^h})$,
$\|\omega_{u0}^\epsilon-\omega_{u0}\|_{X^{k-1,1}}(\mathbb{R}^3_-)=O(\epsilon^{\lambda_1^{\omega_v}})$,
$\|\omega_{b0}^\epsilon-\omega_{b0}\|_{X^{k-1,1}}(\mathbb{R}^3_-)=O(\epsilon^{\lambda_1^{\omega_b}})$,
where $\lambda^v>0$,$\lambda^b>0$,$\lambda^h>0$, $\lambda_1^{\omega_u}>0$,$\lambda_1^{\omega_b}>0$,
and $g-\partial_z^{\varphi^\epsilon} q^\epsilon|_{z=0}\geq c_0>0$ and
$g-\partial_z^{\varphi} q|_{z=0}\geq c_0>0$ hold.

$(1)$ If the Euler boundary data satisfies $\Pi S^\varphi vn|_{z=0}\neq0$ and $\Pi S^\varphi bn|_{z=0}\neq0$ in $(0,T]$,
then the convergence rates of the inviscid limit satisfy
\begin{equation}\label{ConvergenceRates1}
\begin{array}{ll}
\|v^\epsilon-v\|_{X_{tan}^{k-1,1}}+\|b^\epsilon-b\|_{X_{tan}^{k-1,1}}+|h^\epsilon-h|_{X^{k-1,1}}
=O(\epsilon^{min\{\frac{1}{4},\lambda^v,\lambda^b,\lambda^h,\lambda_1^{\omega_v},\lambda_1^{\omega_b}\}}),
\\[5pt]
\|N^\epsilon\cdot\partial_z^{\varphi^\epsilon}v^\epsilon-N\cdot\partial_zv\|_{X_{tan}^{k-1,1}}
+\|N^\epsilon\cdot\partial_z^{\varphi^\epsilon}b^\epsilon-N\cdot\partial_zb\|_{X_{tan}^{k-1,1}}
+\|N^\epsilon\cdot\omega_v^\epsilon-N\cdot\omega_v\|_{X_{tan}^{k-1,1}}
\\[5pt]
\quad+\|N^\epsilon\cdot\omega_b^\epsilon-N\cdot\omega_b\|_{X_{tan}^{k-1,1}}
=O(\epsilon^{min\{\frac{1}{4},\lambda^v,\lambda^b,\lambda^h,\lambda_1^{\omega_v},\lambda_1^{\omega_b}\}}),
\\[5pt]
\|\partial_z^{\varphi^\epsilon}v^\epsilon-\partial_z^\varphi v\|_{X_{tan}^{k-2}}
+\|\partial_z^{\varphi^\epsilon}b^\epsilon-\partial_z^\varphi b\|_{X_{tan}^{k-2}}
+\|\omega^\epsilon_v-\omega_v\|_{X_{tan}^{k-2}}+\|\omega^\epsilon_b-\omega_b\|_{X_{tan}^{k-2}}
\\[5pt]
\quad=O(\epsilon^{min\{\frac{1}{8},\frac{\lambda^v}{2},\frac{\lambda^b}{2},\frac{\lambda^h}{2},
\frac{\lambda_1^{\omega_v}}{2},\frac{\lambda_1^{\omega_b}}{2}\}}),
\\[5pt]
\|\nabla^{\varphi^\epsilon}q^\epsilon-\nabla^\varphi q\|_{X_{tan}^{k-2}}
+\|\Delta^{\varphi^\epsilon}q^\epsilon-\Delta^\varphi q\|_{X_{tan}^{k-2}}
=O(\epsilon^{min\{\frac{1}{8},\frac{\lambda^v}{2},\frac{\lambda^b}{2},\frac{\lambda^h}{2},
\frac{\lambda_1^{\omega_v}}{2},\frac{\lambda_1^{\omega_b}}{2}\}}),
\\[5pt]
\|v^\epsilon-v\|_{Y_{tan}^{k-3}}+\|b^\epsilon-b\|_{Y_{tan}^{k-3}}+|h^\epsilon-h|_{Y^{k-3}}
=O(\epsilon^{min\{\frac{1}{8},\frac{\lambda^v}{2},\frac{\lambda^b}{2},\frac{\lambda^h}{2},
\frac{\lambda_1^{\omega_v}}{2},\frac{\lambda_1^{\omega_b}}{2}\}}),
\\[5pt]
\|N^\epsilon\cdot\partial_z^{\varphi^\epsilon}v^\epsilon-N\cdot\partial_zv\|_{Y_{tan}^{k-4}}
+\|N^\epsilon\cdot\partial_z^{\varphi^\epsilon}b^\epsilon-N\cdot\partial_zb\|_{Y_{tan}^{k-4}}
+\|N^\epsilon\cdot\omega_v^\epsilon-N\cdot\omega_v\|_{Y_{tan}^{k-4}}
\\[5pt]
\quad+\|N^\epsilon\cdot\omega_b^\epsilon-N\cdot\omega_b\|_{Y_{tan}^{k-4}}
=O(\epsilon^{min\{\frac{1}{8},\frac{\lambda^v}{2},\frac{\lambda^B}{2},\frac{\lambda^h}{2},
\frac{\lambda_1^{\omega_v}}{2},\frac{\lambda_1^{\omega_b}}{2}\}}).
\end{array}
\end{equation}

$(2)$ If the ideal MHD boundary data satisfies
$\Pi S^\varphi vn|_{z=0}=0$ and $\Pi S^\varphi bn|_{z=0}=0$ in $(0,T]$,
assume $\|\omega^\epsilon_{v0}-\omega_{v0}\|_{X^{k-2}(\mathbb{R}^3_-)}=O(\epsilon^{\lambda_2^{\omega_u}})$
and $\|\omega^\epsilon_{b0}-\omega_{b0}\|_{X^{k-2}(\mathbb{R}^3_-)}=O(\epsilon^{\lambda_2^{\omega_b}})$,
where $\lambda_2^{\omega_u}>0,\lambda_2^{\omega_b}>0$,
then the convergence rates of the inviscid limit satisfy
\begin{equation}\label{ConvergenceRates2}
\begin{array}{ll}
\|v^\epsilon-v\|_{X_{tan}^{k-2,1}}+\|b^\epsilon-b\|_{X_{tan}^{k-2,1}}+|h^\epsilon-h|_{X^{k-2,1}}
=O(\epsilon^{min\{\frac{1}{2},\lambda^v,\lambda^b,\lambda^h,\lambda_2^{\omega_v},\lambda_2^{\omega_b}\}}),
\\[5pt]
\|N^\epsilon\cdot\partial_z^{\varphi^\epsilon}v^\epsilon-N\cdot\partial_zv\|_{X_{tan}^{k-2}}
+\|N^\epsilon\cdot\partial_z^{\varphi^\epsilon}b^\epsilon-N\cdot\partial_zb\|_{X_{tan}^{k-2}}
+\|N^\epsilon\cdot\omega_v^\epsilon-N\cdot\omega_v\|_{X_{tan}^{k-2}}
\\[5pt]
\quad+\|N^\epsilon\cdot\omega_b^\epsilon-N\cdot\omega_b\|_{X_{tan}^{k-2}}
=O(\epsilon^{min\{\frac{1}{2},\lambda^v,\lambda^b,\lambda^h,\lambda_2^{\omega_v},\lambda_2^{\omega_b}\}}),
\\[5pt]
\|\partial_z^{\varphi^\epsilon}v^\epsilon-\partial_z^\varphi v\|_{X_{tan}^{k-3}}
+\|\partial_z^{\varphi^\epsilon}b^\epsilon-\partial_z^\varphi b\|_{X_{tan}^{k-3}}
+\|\omega^\epsilon_v-\omega_v\|_{X_{tan}^{k-2}}+\|\omega^\epsilon_b-\omega_b\|_{X_{tan}^{k-3}}
\\[5pt]
\quad=O(\epsilon^{min\{\frac{1}{4},\frac{\lambda^v}{2},\frac{\lambda^b}{2},\frac{\lambda^h}{2},
\frac{\lambda_2^{\omega_v}}{2},\frac{\lambda_2^{\omega_b}}{2}\}}),
\\[5pt]
\|\nabla^{\varphi^\epsilon}q^\epsilon-\nabla^\varphi q\|_{X_{tan}^{k-3}}
+\|\Delta^{\varphi^\epsilon}q^\epsilon-\Delta^\varphi q\|_{X_{tan}^{k-3}}
=O(\epsilon^{min\{\frac{1}{4},\frac{\lambda^v}{2},\frac{\lambda^b}{2},\frac{\lambda^h}{2},
\frac{\lambda_2^{\omega_v}}{2},\frac{\lambda_2^{\omega_b}}{2}\}}),
\\[5pt]
\|v^\epsilon-v\|_{Y_{tan}^{k-4}}+\|b^\epsilon-b\|_{Y_{tan}^{k-4}}+|h^\epsilon-h|_{Y^{k-4}}
=O(\epsilon^{min\{\frac{1}{4},\frac{\lambda^v}{2},\frac{\lambda^b}{2},\frac{\lambda^h}{2},
\frac{\lambda_2^{\omega_v}}{2},\frac{\lambda_2^{\omega_b}}{2}\}}),
\\[5pt]
\|N^\epsilon\cdot\partial_z^{\varphi^\epsilon}v^\epsilon-N\cdot\partial_zv\|_{Y_{tan}^{k-5}}
+\|N^\epsilon\cdot\partial_z^{\varphi^\epsilon}b^\epsilon-N\cdot\partial_zb\|_{Y_{tan}^{k-5}}
+\|N^\epsilon\cdot\omega_v^\epsilon-N\cdot\omega_v\|_{Y_{tan}^{k-5}}
\\[5pt]
\quad+\|N^\epsilon\cdot\omega_b^\epsilon-N\cdot\omega_b\|_{Y_{tan}^{k-5}}
=O(\epsilon^{min\{\frac{1}{4},\frac{\lambda^v}{2},\frac{\lambda^b}{2},\frac{\lambda^h}{2},
\frac{\lambda_2^{\omega_v}}{2},\frac{\lambda_2^{\omega_b}}{2}\}}).
\end{array}
\end{equation}
\end{theorem}
\begin{remark}
Due to formula \eqref{1.30},
the estimate of $\|\partial_z^{\varphi^{\e}} v^{\e}-\partial_z^{\varphi} v\|_{X_{tan}^{k-2}}$
and $\|\partial_z^{\varphi^{\e}} b^{\e}-\partial_z^{\varphi} b\|_{X_{tan}^{k-2}}$
results from the estimate of $\|\partial_z v^{\e}-\partial_z v\|_{X_{tan}^{k-2}}$,
$\|\partial_z b^{\e}-\partial_z b\|_{X_{tan}^{k-2}}$
and $\|\partial_z(\eta^{\e} -\eta)\|_{X_{tan}^{k-2}}$.
The $L^{\infty}$ type estimates are based on the formula \eqref{1.47}.
$\|v^{\e} -v\|_{X_{tan}^k}$, $\|b^{\e} -b\|_{X_{tan}^k}$
and $|h^{\e}-h|_{X^k}$ can not be estimated because
we can not control $\|\partial_t^k(q^{\e} -q)\|_{L^2}$.
In general,
the convergence rates of the inviscid limit for the free boundary problem
are slower than those of the Navier-slip boundary case.
\end{remark}

\subsection{Main Results for MHD Equations with Surface Tension}
For fixed $\sigma > 0$,
we study the free boundary problem with surface tension when $\epsilon=\lambda$,
which is equivalent to
\begin{equation}\label{MHDFS}
\left\{
\begin{aligned}
&\partial_{t}^\varphi v-\epsilon\Delta^\varphi v+v\cdot\nabla^\varphi v+\nabla^\varphi q
=b\cdot\nabla^\varphi b,
&{\rm in}\ \mathbb{R}^3_-,\\
&\partial_{t}^\varphi b-\lambda\Delta^\varphi b+v\cdot\nabla^\varphi b=b\cdot\nabla^\varphi v,
&{\rm in}\ \mathbb{R}^3_-,\\
&\nabla^\varphi\cdot v=0,  \nabla^\varphi\cdot b=0,
&{\rm in}\ \mathbb{R}^3_-,\\
&q\nn-2\epsilon S^\varphi vn=gh\nn-\sigma \mathcal{M}\nn,
&{\rm on}\ z=0,\\
&\partial_th=v(t,y,0)\cdot \NN,
&{\rm on}\ z=0,\\
&b=0,
&{\rm on}\ z=0\cup \mathbb{R}^3_+,\\
&(v, b, h)|_{t=0}=(v^{\epsilon}_0, b^{\epsilon}_0, h^{\epsilon}_0),
\end{aligned}
\right.
\end{equation}
where $\mathcal{M}=\nabla\cdot\frac{\nabla h}{\sqrt{1+|\nabla h|^2}}
=-\nabla_x\cdot(\frac{(-\nabla_y h,1)}{\sqrt{1+|\nabla h|^2}})
=-\nabla_y\cdot(\frac{(\nabla_y h)}{\sqrt{1+|\nabla h|^2}})$.

As $\epsilon\to0$, \eqref{MHDFS}
reduces to the following free boundary problems for ideal MHD system:
\begin{equation}\label{IMHDFS}
\left\{
\begin{aligned}
&\partial_{t}^\varphi v+v\cdot\nabla^\varphi v+\nabla^\varphi q
=b\cdot\nabla^\varphi b,
&{\rm in}\ \mathbb{R}^3_-,\\
&\partial_{t}^\varphi b+v\cdot\nabla^\varphi b=b\cdot\nabla^\varphi v,
&{\rm in}\ \mathbb{R}^3_-,\\
&\nabla^\varphi\cdot v=0,  \nabla^\varphi\cdot b=0,
&{\rm in}\ \mathbb{R}^3_-,\\
&q\nn=gh\nn-\sigma \mathcal{M}\nn,
&{\rm on}\ z=0,\\
&\partial_th=v(t,y,0)\cdot \NN,
&{\rm on}\ z=0,\\
&b=0,
&{\rm on}\ z=0\cup \mathbb{R}^3_+,\\
&(v, b, h)|_{t=0}=(v_0, b_0, h_0),
\end{aligned}
\right.
\end{equation}
where $(v_0,b_0,h_0)=\lim\limits_{\epsilon\to0}(v^{\epsilon}_0,b^{\epsilon}_0,h^{\epsilon}_0)$.
In this subsection,
 we do not need the Taylor sign condition.

The following proposition establishes the regularity structures for
\eqref{MHDFS} and \eqref{IMHDFS}.
\begin{proposition}\label{Proposition1.2}
Fix $\sigma>0$, let $\epsilon=\lambda\in(0,1]$. For $m\geq6$, assume the initial data $(v_0^{\epsilon},B_0^{\epsilon},h_0^{\epsilon})$
 satisfies the compatibility conditions $\Pi S^\varphi v^\epsilon_0\nn|_{z=0}=0$
and $\Pi S^\varphi b^\epsilon_0\nn|_{z=0}=0$ in $(0,T]$ and the regularities
 \begin{align}
\sup_{\epsilon\in(0,1]}&(|h^\epsilon_0|_{X^{m}}
+\epsilon^{\frac{1}{2}}|h^\epsilon_0|_{X^{m,\frac{1}{2}}}+\sigma|h^\epsilon_0|_{X^{m,1}}
+\|v^\epsilon_0\|_{X^{m}}+\|b^\epsilon_0\|_{X^{m}}
+\|\omega^\epsilon_{v0}\|_{X^{m-1}}+\|\omega^\epsilon_{b0}\|_{X^{m-1}}\nonumber
\\
&+\epsilon(\|\nabla v_0\|^2_{X^{m-1,1}}+\|\nabla b_0\|^2_{X^{m-1,1}})
+\epsilon(\|\nabla \omega_{u0}\|^2_{X^{m-1,1}}+\|\nabla \omega_{b0}\|^2_{X^{m-1,1}})
\nonumber\\
&+\|\omega^\epsilon_{v0}\|^2_{X^{1,\infty}}+\|\omega^\epsilon_{b0}\|^2_{X^{1,\infty}}
+\epsilon^{\frac{1}{2}}(\|\partial_z\omega^\epsilon_{v0}\|_{L^\infty}
+\|\partial_z\omega^\epsilon_{b0}\|_{L^\infty}))\leq C_0,
 \end{align}
 then the unique solution to \eqref{MHDFS} satisfies
\begin{align}
\sup_{\epsilon\in[0,T]}&(|h^\epsilon|^2_{X^{m-1,1}}
+\epsilon^{\frac{1}{2}}|h^\epsilon|^2_{X^{m-1,\frac{3}{2}}}
+\sigma|h^\epsilon|^2_{X^{m-1,2}}
+\|v^\epsilon\|^2_{X^{m-1,1}}+\|b^\epsilon\|^2_{X^{m-1,1}}
+\|\partial_zv^\epsilon\|^2_{X^{m-2}}
\nonumber\\
&+\|\partial_zb^\epsilon\|^2_{X^{m-2}}
+\|\omega^\epsilon_{v}\|_{X^{m-2}}+\|\omega^\epsilon_{b}\|_{X^{m-2}}
+\|\partial_zv^\epsilon\|^2_{1,\infty}+\|\partial_zb^\epsilon\|^2_{1,\infty}
+\epsilon^{\frac{1}{2}}(\|\partial_{zz}v^\epsilon\|^2_{L^\infty}
\nonumber\\
&+\|\partial_{zz}b^\epsilon\|^2_{L^\infty}))
+\|\partial_{z}v\|^2_{L^4([0,T],X^{m-1})}+\|\partial_{z}b\|^2_{L^4([0,T],X^{m-1})}
+\|\partial_t^mv\|^2_{L^4([0,T],L^2)}
\nonumber\\
&+\|\partial_t^mb\|^2_{L^4([0,T],L^2)}
+\|\partial_t^mh\|^2_{L^4([0,T],L^2)}+\epsilon\|\partial_t^mh\|^2_{L^4([0,T],X^{0,\frac{1}{2}})}
+\sigma\|\partial_t^mh\|^2_{L^4([0,T],X^{0,1})}
\nonumber\\
&+\epsilon\int_0^T\|\nabla v^\epsilon\|^2_{X^{m-1,1}}+\|\nabla b^\epsilon\|^2_{X^{m-1,1}}
+\|\nabla\partial_zv^\epsilon\|^2_{X^{m-2}}+\|\nabla\partial_zb^\epsilon\|^2_{X^{m-2}}dt
\leq C.
 \end{align}

As $\epsilon=\lambda\to0$,
the solution to \eqref{IMHDFS} satisfies
\begin{align}
\sup_{\epsilon\in[0,T]}&(|h^\epsilon|^2_{X^{m-1,1}}
+\sigma|h^\epsilon|^2_{X^{m-1,2}}
+\|v^\epsilon\|^2_{X^{m-1,1}}+\|b^\epsilon\|^2_{X^{m-1,1}}
+\|\partial_zv^\epsilon\|^2_{X^{m-2}}
+\|\partial_zb^\epsilon\|^2_{X^{m-2}}
\nonumber\\
&+\|\omega^\epsilon_{v}\|_{X^{m-2}}+\|\omega^\epsilon_{b}\|_{X^{m-2}}
+\|\partial_{z}v\|^2_{L^4([0,T],X^{m-1})}+\|\partial_{z}b\|^2_{L^4([0,T],X^{m-1})}
+\|\partial_t^mv\|^2_{L^4([0,T],L^2)}
\nonumber\\
&+\|\partial_t^mb\|^2_{L^4([0,T],L^2)}
+\|\partial_t^mh\|^2_{L^4([0,T],L^2)}
+\sigma\|\partial_t^mh\|^2_{L^4([0,T],X^{0,1})}
\leq C.
 \end{align}
\end{proposition}

\begin{remark}
We estimate the pressure by the elliptic equation of the pressure subject to the Neumann boundary condition as follows
\begin{equation}\label{1.40}
\left\{\begin{array}{ll}
\Delta^{\varphi} q = -\partial_j^{\varphi} v^i \partial_i^{\varphi} v^j
+\partial_j^{\varphi} b^i \partial_i^{\varphi} b^j, \\[5pt]

\nabla^{\varphi} q\cdot\NN|_{z=0}
= - \partial_t^{\varphi} v\cdot\NN - v\cdot\nabla^{\varphi} v\cdot\NN
+ \e\Delta^{\varphi} v\cdot\NN.
\end{array}\right.
\end{equation}
If we couple $\triangle^{\varphi} q = -\partial_j^{\varphi} v^i \partial_i^{\varphi} v^j
+\partial_j^{\varphi} b^i \partial_i^{\varphi} b^j$ with
its nonhomogeneous Dirichlet boundary condition $q|_{z=0} = g h - \sigma H + 2\e\mathcal{S}^{\varphi} v\nn\cdot\nn$,
the estimates can not be closed due to the less regularity of $h$.
In $(\ref{1.40})$, we have to prove $\partial_t^m v\in L^4([0,T],L^2)$
and $\partial_t^m b\in L^4([0,T],L^2)$,
 due to the nontransport-type nonlinear terms.
We use the following Hardy's inequality to overcome the difficulties generated by $\partial_t^m q$,
\begin{equation}\label{Sect1_HardyIneq}
\begin{array}{ll}
\|\frac{1}{1-z} \partial_t^{\ell} q\|_{L^2(\mathbb{R}^3_{-})}
\lem \big|\partial_t^{\ell} q|_{z=0}\big|_{L^2(\mathbb{R}^2)} + \|\partial_z \partial_t^{\ell} q\|_{L^2(\mathbb{R}^3_{-})},
\quad 0\leq\ell\leq m-1.
\end{array}
\end{equation}
In \rm\cite{Mei16}, the depth of the fluid is finite, then $\|\partial_t^{\ell} q\|_{L^2(\mathbb{R}^2\times [-L,0])}$ is bounded.
While we consider the infinite fluid depth in this paper, therefore $\|\partial_t^{\ell} q\|_{L^2}$ has no bound in general.
\end{remark}

Since the equations of the vorticity and its boundary condition
$\Pi S^\varphi vn|_{z=0}=0$, $\Pi S^\varphi bn|_{z=0}=0$ are the same as the $\sigma=0$ case,
Theorem \ref{Theorem1.1} is also valid for the
equations \eqref{MHDF} with $\sigma>0$.
Thus, in this case we only need to discuss the convergence
rates of the inviscid limit. The results are stated in the following theorem:

\begin{theorem}\label{Theorem1.3}
Assume $T>0$ is finite,
$(v^\epsilon,b^\epsilon,h^\epsilon)$ is the solution of MHD equations in $[0,T]$
with initial data $(v_0^\epsilon,b_0^\epsilon,h_0^\epsilon)$,
and $(v, b, h)$ is the solution of ideal MHD equations in $[0,T]$ with initial data
$(v_0,b_0,h_0)\in X^{m-1,1}(\mathbb{R}^3_-)\times X^{m-1,1}(\mathbb{R}^2)$.
Assume there exists an integer k where $1\leq k\leq m-2$, such that
$\|v_0^\epsilon-v_0\|_{X^{k-1,1}}(\mathbb{R}^3_-)=O(\epsilon^{\lambda^v})$,
$\|b_0^\epsilon-b_0\|_{X^{k-1,1}}(\mathbb{R}^3_-)=O(\epsilon^{\lambda^b})$,
$|h_0^\epsilon-h_0|_{X^{k-1,1}}(\mathbb{R}^2)=O(\epsilon^{\lambda^h})$,
$\|\omega_{v0}^\epsilon-\omega_{v0}\|_{X^{k-1,1}}(\mathbb{R}^3_-)=O(\epsilon^{\lambda_1^{\omega_v}})$,
$\|\omega_{b0}^\epsilon-\omega_{b0}\|_{X^{k-1,1}}(\mathbb{R}^3_-)=O(\epsilon^{\lambda_1^{\omega_b}})$,
where $\lambda^v>0$,$\lambda^b>0$,$\lambda^h>0$, $\lambda_1^{\omega_v}>0$,$\lambda_1^{\omega_b}>0$.

(1) If the ideal MHD boundary data satisfies $\Pi S^\varphi vn|_{z=0}\neq0$ and
$\Pi S^\varphi bn|_{z=0}\neq0$ in $(0,T]$,
then the convergence rates satisfy
\begin{align*}
&\|v^\epsilon-v\|_{X_{tan}^{k-1,1}}+\|b^\epsilon-b\|_{X_{tan}^{k-1,1}}+|h^\epsilon-h|_{X^{k-1,1}}
=O(\epsilon^{min\{\frac{1}{4},\lambda^v,\lambda^b,\lambda^h,\lambda_1^{\omega_v},\lambda_1^{\omega_b}\}}),\\
&\|N^\epsilon\cdot\partial_z^{\varphi^\epsilon}v^\epsilon-N\cdot\partial_zv\|_{X_{tan}^{k-1,1}}
+\|N^\epsilon\cdot\partial_z^{\varphi^\epsilon}b^\epsilon-N\cdot\partial_zb\|_{X_{tan}^{k-1,1}}
+\|N^\epsilon\cdot\omega_v^\epsilon-N\cdot\omega_v\|_{X_{tan}^{k-1,1}}
\nonumber\\
&\qquad+\|N^\epsilon\cdot\omega_b^\epsilon-N\cdot\omega_b\|_{X_{tan}^{k-1,1}}
=O(\epsilon^{min\{\frac{1}{4},\lambda^v,\lambda^b,\lambda^h,\lambda_1^{\omega_v},\lambda_1^{\omega_b}\}}),\\
&\|\partial_z^{\varphi^\epsilon}v^\epsilon-\partial_z^\varphi v\|_{X_{tan}^{k-2}}
+\|\partial_z^{\varphi^\epsilon}b^\epsilon-\partial_z^\varphi b\|_{X_{tan}^{k-2}}
+\|\omega^\epsilon_v-\omega_v\|_{X_{tan}^{k-2}}+\|\omega^\epsilon_b-\omega_b\|_{X_{tan}^{k-2}}
\nonumber\\
&\qquad=O(\epsilon^{min\{\frac{1}{8},\frac{\lambda^v}{2},\frac{\lambda^b}{2},\frac{\lambda^h}{2},
\frac{\lambda_1^{\omega_v}}{2},\frac{\lambda_1^{\omega_b}}{2}\}}),\\
&\|\nabla^{\varphi^\epsilon}q^\epsilon-\nabla^\varphi q\|_{X_{tan}^{k-2}}
+\|\Delta^{\varphi^\epsilon}q^\epsilon-\Delta^\varphi q\|_{X_{tan}^{k-2}}
=O(\epsilon^{min\{\frac{1}{8},\frac{\lambda^v}{2},\frac{\lambda^b}{2},\frac{\lambda^h}{2},
\frac{\lambda_1^{\omega_v}}{2},\frac{\lambda_1^{\omega_b}}{2}\}}),
\\
&\|v^\epsilon-v\|_{Y_{tan}^{k-3}}+\|b^\epsilon-b\|_{Y_{tan}^{k-3}}+|h^\epsilon-h|_{Y^{k-3}}
=O(\epsilon^{min\{\frac{1}{8},\frac{\lambda^v}{2},\frac{\lambda^b}{2},\frac{\lambda^h}{2},
\frac{\lambda_1^{\omega_v}}{2},\frac{\lambda_1^{\omega_b}}{2}\}}),
\\
&\|N^\epsilon\cdot\partial_z^{\varphi^\epsilon}v^\epsilon-N\cdot\partial_zv\|_{Y_{tan}^{k-4}}
+\|N^\epsilon\cdot\partial_z^{\varphi^\epsilon}b^\epsilon-N\cdot\partial_zb\|_{Y_{tan}^{k-4}}
+\|N^\epsilon\cdot\omega_u^\epsilon-N\cdot\omega_v\|_{Y_{tan}^{k-4}}
\nonumber\\
&\qquad+\|N^\epsilon\cdot\omega_b^\epsilon-N\cdot\omega_b\|_{Y_{tan}^{k-4}}
=O(\epsilon^{min\{\frac{1}{8},\frac{\lambda^v}{2},\frac{\lambda^b}{2},\frac{\lambda^h}{2},
\frac{\lambda_1^{\omega_u}}{2},\frac{\lambda_1^{\omega_b}}{2}\}}).
\end{align*}

(2) If the ideal MHD boundary data satisfies
$\Pi S^\varphi vn|_{z=0}=0$ and $\Pi S^\varphi bn|_{z=0}=0$ in $(0,T]$,
assume $\|\omega^\epsilon_{v0}-\omega_{v0}\|_{X^{k-2}(R^3_-)}=O(\epsilon^{\lambda_2^{\omega_v}})$
and $\|\omega^\epsilon_{b0}-\omega_{b0}\|_{X^{k-2}(R^3_-)}=O(\epsilon^{\lambda_2^{\omega_b}})$,
where $\lambda_2^{\omega_v}>0,\lambda_2^{\omega_b}>0$,
then the convergence rates satisfy
\begin{align*}
&\|v^\epsilon-v\|_{X_{tan}^{k-2,1}}+\|b^\epsilon-b\|_{X_{tan}^{k-2,1}}+|h^\epsilon-h|_{X^{k-2,1}}
=O(\epsilon^{min\{\frac{1}{2},\lambda^v,\lambda^b,\lambda^h,\lambda_2^{\omega_u},\lambda_2^{\omega_b}\}}),\\
&\|N^\epsilon\cdot\partial_z^{\varphi^\epsilon}v^\epsilon-N\cdot\partial_zv\|_{X_{tan}^{k-2}}
+\|N^\epsilon\cdot\partial_z^{\varphi^\epsilon}b^\epsilon-N\cdot\partial_zb\|_{X_{tan}^{k-2}}
+\|N^\epsilon\cdot\omega_v^\epsilon-N\cdot\omega_v\|_{X_{tan}^{k-2}}
\nonumber\\
&\qquad+\|N^\epsilon\cdot\omega_b^\epsilon-N\cdot\omega_b\|_{X_{tan}^{k-2}}
=O(\epsilon^{min\{\frac{1}{2},\lambda^v,\lambda^b,\lambda^h,\lambda_2^{\omega_v},\lambda_2^{\omega_b}\}}),\\
&\|\partial_z^{\varphi^\epsilon}v^\epsilon-\partial_z^\varphi v\|_{X_{tan}^{k-3}}
+\|\partial_z^{\varphi^\epsilon}b^\epsilon-\partial_z^\varphi b\|_{X_{tan}^{k-3}}
+\|\omega^\epsilon_v-\omega_v\|_{X_{tan}^{k-2}}+\|\omega^\epsilon_b-\omega_b\|_{X_{tan}^{k-3}}
\nonumber\\
&\qquad=O(\epsilon^{min\{\frac{1}{4},\frac{\lambda^v}{2},\frac{\lambda^b}{2},\frac{\lambda^h}{2},
\frac{\lambda_2^{\omega_v}}{2},\frac{\lambda_2^{\omega_b}}{2}\}}),
\\
&\|\nabla^{\varphi^\epsilon}q^\epsilon-\nabla^\varphi q\|_{X_{tan}^{k-3}}
+\|\Delta^{\varphi^\epsilon}q^\epsilon-\Delta^\varphi q\|_{X_{tan}^{k-3}}
=O(\epsilon^{min\{\frac{1}{4},\frac{\lambda^v}{2},\frac{\lambda^b}{2},\frac{\lambda^h}{2},
\frac{\lambda_2^{\omega_v}}{2},\frac{\lambda_2^{\omega_b}}{2}\}}),
\\
&\|v^\epsilon-v\|_{Y_{tan}^{k-4}}+\|b^\epsilon-b\|_{Y_{tan}^{k-4}}+|h^\epsilon-h|_{Y^{k-4}}
=O(\epsilon^{min\{\frac{1}{4},\frac{\lambda^v}{2},\frac{\lambda^b}{2},\frac{\lambda^h}{2},
\frac{\lambda_2^{\omega_v}}{2},\frac{\lambda_2^{\omega_b}}{2}\}}),
\\
&\|N^\epsilon\cdot\partial_z^{\varphi^\epsilon}v^\epsilon-N\cdot\partial_zv\|_{Y_{tan}^{k-5}}
+\|N^\epsilon\cdot\partial_z^{\varphi^\epsilon}b^\epsilon-N\cdot\partial_zb\|_{Y_{tan}^{k-5}}
+\|N^\epsilon\cdot\omega_v^\epsilon-N\cdot\omega_v\|_{Y_{tan}^{k-5}}
\nonumber\\
&\qquad+\|N^\epsilon\cdot\omega_b^\epsilon-N\cdot\omega_b\|_{Y_{tan}^{k-5}}
=O(\epsilon^{min\{\frac{1}{4},\frac{\lambda^v}{2},\frac{\lambda^b}{2},\frac{\lambda^h}{2},
\frac{\lambda_2^{\omega_v}}{2},\frac{\lambda_2^{\omega_b}}{2}\}}).
\end{align*}
\end{theorem}

\subsection{Preliminaries}
In this subsection,
we collect some necessary notations, propositions
and preliminary estimates,
the function is defined in the fixed domain $\mathbb{R}_-^3$.

We denote the commutator,
\begin{align}
[\partial_t^l\mathcal{Z}^\alpha,\partial_i^\varphi]f=-\partial_z^\varphi f\partial_i^\varphi(\partial_t^l\mathcal{Z}^\alpha \eta)
+{\text b.t.}, i = t, 1, 2, 3.
\end{align}
Here the abbreviation b.t. represents bounded terms in this paper
\begin{align*}
b.t.=\mathcal{C}_i^\alpha(f)=\mathcal{C}_{i,1}^\alpha(f)+\mathcal{C}_{i,2}^\alpha(f)
+\mathcal{C}_{i,3}^\alpha(f),
\end{align*}
$i=1,2$, and
\begin{equation}
\left\{\begin{aligned}
&\mathcal{C}_{i,1}:=[\partial_t^l\mathcal{Z}^\alpha,\frac{\partial_i\varphi}{\partial_z\varphi},\partial_zf]\\
&\mathcal{C}_{i,2}:=-\partial_zf[\partial_t^l\mathcal{Z}^\alpha,\partial_i\varphi,\frac{1}{\partial_z\varphi}]
-\partial_i\varphi(\partial_t^l\mathcal{Z}^\alpha(\frac{1}{\partial_z\varphi})
+\frac{Z^\alpha\partial_z\eta}{(\partial_z\varphi)^2})\partial_zf,\\
&\mathcal{C}_{i,3}^\alpha:=-\frac{\partial_i\varphi}{\partial_z\varphi}[\partial_t^l\mathcal{Z}^\alpha,\partial_z]f
+\frac{\partial_i\varphi}{(\partial_z\varphi)^2}\partial_zf[\partial_t^l\mathcal{Z}^\alpha,\partial_z]\eta.
\end{aligned}\right.
\end{equation}
For $i = 3$, the result is very similar, and it suffices to replace $\partial_i\varphi$
 by 1 in the above terms.
We need the following lemma to estimate the commutators.
\begin{lemma}[\cite{Masmoudi17}]\label{Lemma1.1}
For $1\leq|\alpha|\leq m, i = 1, 2, 3$, we have
\begin{align*}
\|\mathcal{C}_i^\alpha(f)\|\leq \Lambda(\frac{1}{c_0},|h|_{2,\infty}+\|\nabla f\|_{1,\infty})
(\|\nabla f\|_{m-1}+|h|_{m-\frac{1}{2}}).
\end{align*}
\end{lemma}

Now, we state the integration by parts for $\int_{\mathbb{R}^3_{-}}\partial_i^{\varphi}fg\mathrm{d}\mathcal{V}_t$.
\begin{lemma}
In $\mathbb{R}^3_{-}$, we have the following integration by parts rules:
\begin{equation}
\begin{array}{ll}
\int_{\mathbb{R}^3_{-}}\partial_i^{\varphi}fg\mathrm{d}\mathcal{V}_t
=-\int_{\mathbb{R}^3_{-}}f\partial_i^{\varphi}g\mathrm{d}\mathcal{V}_t
+\int_{z=0}fg\NN_i\mathrm{d}y, \\[12pt]

\frac{\mathrm{d}}{\mathrm{d}t}\int_{\mathbb{R}^3_{-}} f \mathrm{d}\mathcal{V}_t
=\int_{\mathbb{R}^3_{-}} \partial_t^{\varphi} f \mathrm{d}\mathcal{V}_t
+ \int_{z=0} f v\cdot\NN \mathrm{d}y,  \\[12pt]

\int_{\mathbb{R}^3_{-}} \vec{a} \cdot\nabla^{\varphi} f \mathrm{d}\mathcal{V}_t
= \int_{z=0} \vec{a}\cdot \NN\, f  \mathrm{d}y
- \int_{\mathbb{R}^3_{-}} \nabla^{\varphi}\cdot \vec{a}\, f \mathrm{d}\mathcal{V}_t, \\[12pt]

\int_{\mathbb{R}^3_{-}} \vec{a} \cdot (\nabla^{\varphi}\times \vec{b}) \,\mathrm{d}\mathcal{V}_t
= \int_{z=0} \vec{a} \cdot (\NN \times \vec{b}) \,\mathrm{d}y
+ \int_{\mathbb{R}^3_{-}} (\nabla^{\varphi}\times \vec{a}) \cdot \vec{b} \,\mathrm{d}\mathcal{V}_t,
\end{array}
\end{equation}
where $\mathrm{d}\mathcal{V}_t = \partial_z\varphi \mathrm{d}y\mathrm{d}z$ is defined on $\mathbb{R}^3_{-}$.
\end{lemma}

The following lemma states the Korn inequality in $\mathbb{R}^3_{-}$.
\begin{lemma}[\cite{Masmoudi17}]
If $\partial_z\varphi\geq c_0$, $\|\nabla\varphi\|_{L^\infty}+\|\nabla^2\varphi\|_{L^\infty}\leq \frac{1}{c_0}$ for some $c_0>0$, then there exists $\Lambda_0=\Lambda(\frac{1}{c_0})>0$,
such that for every $v\in H^1(\mathbb{R}^3_{-})$, one has
\begin{equation}
\begin{array}{ll}
\|\nabla v\|_{L^2}^2 \lem \Lambda_0(\int_{\mathbb{R}^3_{-}}|\mathcal{S}^{\varphi} v|^2 \,\mathrm{d}\mathcal{V}_t + \|v\|_{L^2}^2).
\end{array}
\end{equation}
\end{lemma}

Denote the viscous terms by
\begin{align*}
\Delta^\varphi f
&=2\nabla^\varphi\cdot S^\varphi f-\nabla^\varphi(\nabla^\varphi\cdot f)\\
&=\nabla^\varphi(\nabla^\varphi\cdot f)-\nabla^\varphi\times(\nabla^\varphi\times f).
\end{align*}

The following estimate allows us to control the gradient of $\omega$ by its vorticity and $|\omega\cdot \nn|_{\frac{1}{2}}$.
\begin{equation}
\begin{array}{ll}
\|\nabla \omega\|_{L^2}^2 \lem \int_{\mathbb{R}^3_{-}}|\nabla^{\varphi}\times \omega|^2 \,\mathrm{d}\mathcal{V}_t + \|\omega\|_{L^2}^2
+ |\omega\cdot \nn|_{\frac{1}{2}},
\end{array}
\end{equation}
which is proved by the
Hodge decomposition and $\nabla^\varphi\cdot\omega=0$.

We shall also need the following embedding and trace estimates for these spaces:
\begin{lemma}[\cite{Masmoudi17}]
For $s_{1} \geq 0, s_{2} \geq 0$ such that $s_{1}+s_{2}>2$,
$f \in H_{t a n}^{s_{1}}$ and $\partial_{z} f \in H_{t a n}^{s_{2}}$,
we have the anisotropic Sobolev embedding
\begin{equation}\label{1.47}
\|f\|_{L^{\infty}}^{2} \lesssim\left\|\partial_{z} f\right\|_{H_{t a n}^{s_{2}}}\|f\|_{H_{t a n}^{s_{1}}}.
\end{equation}

For $f \in H^{1}(\mathcal{S})$,
we have the trace estimates
\begin{equation}\label{1.48}
|f(\cdot, 0)|_{H^{s}\left(\mathbb{R}^{2}\right)}
\leqq C\left\|\partial_{z} f\right\|_{H_{t a n}}^{\frac{1}{2}}\|f\|_{H_{t a n}^{s_{1}}}^{\frac{1}{2}},
\end{equation}
with $s_{1}+s_{2}=2 s \geqq 0$.
\end{lemma}

The rest of the paper is organized as follows:
In Section 2, we investigate the relationship between the vorticity,
the normal derivative and regularity structure of MHD solutions with $\sigma=0$,
and obtain the well-posedness which includes the time derivatives estimates.
In Section 3, we study the strong initial vorticity layer is one of sufficient conditions for the existence
strong vorticity layer for the free boundary MHD equations.
In Section 4, we show the discrepancy between boundary value of MHD vorticity and boundary value of
ideal MHD vorticity is also one of sufficient conditions for the existence strong vorticity layer for
the free boundary MHD equations.
In Section 5, we establish the convergence rates estimates for $\sigma= 0$.
In Section 6, we estimate the regularity structure of MHD solutions with $\sigma> 0$.
In Section 7, we estimate the convergence rates of the inviscid limit for $\sigma> 0$.

\section{Regularity Structure of MHD Solutions for $\sigma=0$}
In this section, let $\sigma=0$, we first analyze the relationship between the vorticity
and the normal derivatives on the free boundary,
and establish the vorticity equations with  boundary conditions.
Next we give an a priori estimates for Proposition \ref{Proposition1.1}
on the regularities with respect to the time derivatives.
For simplicity, we omit the superscript $\epsilon=\lambda$ in this section.

\subsection{Vorticity and Normal Derivatives}
The following lemma shows that the normal derivatives $(\partial_zv^1, \partial_zv^2,
\partial_zb^1, \partial_zb^2)$ can be
estimated by the tangential vorticity $(\omega_{vh},\omega_{bh})$.
\begin{lemma}\label{Lemma2.1}
Assume $\omega_v$, $\omega_b$ are the vorticity of $v$, $b$ of the MHD equations, respectively.
If $\|v\|_{X^{m-1,1}}+\|b\|_{X^{m-1,1}}+\|h\|_{X^{m-1,1}}<+\infty$, then
\begin{align}
&\|\partial_zv^1\|_{X^k}+\|\partial_zv^2\|_{X^k}+\|\partial_zb^1\|_{X^k}+\|\partial_zb^2\|_{X^k}
\nonumber\\
\leq &\|\omega_{vh}\|_{X^k}+\|\omega_{bh}\|_{X^k}+\|v\|_{X^{k,1}}+\|b\|_{X^{k,1}}+|h|_{X^{k,\frac{1}{2}}}
\end{align}
for $k\leq m-1$.
\end{lemma}

\begin{proof}
According to the definition of the vorticity, $\omega_v$, $\omega_b$ can be rewritten as:
\begin{equation}\label{2.2}
\left\{\begin{aligned}
&\omega^1_v=\partial_2^\varphi v^3-\partial_z^\varphi v^2
=\partial_2v^3-\frac{\partial_2\varphi}{\partial_z\varphi}\partial_zv^3
-\frac{1}{\partial_z\varphi}\partial_z v^2,
\\
&\omega^2_v=\partial_z^\varphi v^1-\partial_1^\varphi v^3
=-\partial_1v^3+\frac{\partial_1\varphi}{\partial_z\varphi}\partial_zv^3
+\frac{1}{\partial_z\varphi}\partial_z v^1,
\\
&\omega^3_v=\partial_1^\varphi v^2-\partial_2^\varphi v^1
=\partial_1v^2-\frac{\partial_1\varphi}{\partial_z\varphi}\partial_zv^2
-\frac{\partial_2\varphi}{\partial_z\varphi}\partial_z v^1,
\end{aligned}\right.
\end{equation}
and
\begin{equation}\label{2.3}
\left\{\begin{aligned}
&\omega^1_b=\partial_2^\varphi b^3-\partial_z^\varphi b^2
=\partial_2b^3-\frac{\partial_2\varphi}{\partial_z\varphi}\partial_zb^3
-\frac{1}{\partial_z\varphi}\partial_z b^2,
\\
&\omega^2_b=\partial_z^\varphi b^1-\partial_1^\varphi b^3
=-\partial_1b^3+\frac{\partial_1\varphi}{\partial_z\varphi}\partial_zb^3
+\frac{1}{\partial_z\varphi}\partial_z b^1,
\\
&\omega^3_b=\partial_1^\varphi b^2-\partial_2^\varphi b^1
=\partial_1b^2-\frac{\partial_1\varphi}{\partial_z\varphi}\partial_zb^2
-\frac{\partial_2\varphi}{\partial_z\varphi}\partial_z b^1.
\end{aligned}\right.
\end{equation}
Pluging the following divergence free condition
\begin{equation}\label{divergence free condition}
\begin{aligned}
&\partial_zv^3=\partial_1\varphi\partial_z v^1+\partial_2\varphi\partial_z v^2-\partial_z\varphi(\partial_1 v^1+\partial_2 v^2),\\
&\partial_zb^3=\partial_1\varphi\partial_z b^1+\partial_2\varphi\partial_z b^2-\partial_z\varphi(\partial_1 b^1+\partial_2 b^2),
\end{aligned}
\end{equation}
into \eqref{2.2} and \eqref{2.3}, respectively, one has
\begin{equation}
\left\{\begin{aligned}
&\omega^1_v=-\frac{\partial_1\varphi\partial_2\varphi}{\partial_z\varphi}\partial_z v^1
-\frac{1+(\partial_2\varphi)^2}{\partial_z\varphi}\partial_z v^2+\partial_2 v^3
+\partial_2\varphi(\partial_1 v^1+\partial_2 v^2),
\\
&\omega^2_v=\frac{1+(\partial_2\varphi)^2}{\partial_z\varphi}\partial_z v^1
+\frac{1+\partial_1\varphi\partial_2\varphi}{\partial_z\varphi}\partial_z v^2-\partial_1 v^3
-\partial_1\varphi(\partial_1 v^1+\partial_2 v^2),
\end{aligned}\right.
\end{equation}
and
\begin{equation}
\left\{\begin{aligned}
&\omega^1_b=-\frac{\partial_1\varphi\partial_2\varphi}{\partial_z\varphi}\partial_z b^1
-\frac{1+(\partial_2\varphi)^2}{\partial_z\varphi}\partial_z b^2+\partial_2 b^3
+\partial_2\varphi(\partial_1 b^1+\partial_2 b^2),
\\
&\omega^2_b=\frac{1+(\partial_2\varphi)^2}{\partial_z\varphi}\partial_z b^1
+\frac{1+\partial_1\varphi\partial_2\varphi}{\partial_z\varphi}\partial_z b^2-\partial_1 b^3
-\partial_1\varphi(\partial_1 b^1+\partial_2 b^2).
\end{aligned}\right.
\end{equation}
It follows that
\begin{equation}\label{2.8}
\left\{\begin{aligned}
&\frac{\partial_1\varphi\partial_2\varphi}{\partial_z\varphi}\partial_z v^1
+\frac{1+(\partial_2\varphi)^2}{\partial_z\varphi}\partial_z v^2=-\omega^1_v+\partial_2 v^3
+\partial_2\varphi(\partial_1 v^1+\partial_2 v^2),
\\
&\frac{1+(\partial_2\varphi)^2}{\partial_z\varphi}\partial_z v^1
+\frac{1+\partial_1\varphi\partial_2\varphi}{\partial_z\varphi}\partial_z v^2
=\omega^2_v+\partial_1 v^3
+\partial_1\varphi(\partial_1 v^1+\partial_2 v^2),
\end{aligned}\right.
\end{equation}
and
\begin{equation}\label{2.9}
\left\{\begin{aligned}
&\frac{\partial_1\varphi\partial_2\varphi}{\partial_z\varphi}\partial_z b^1
+\frac{1+(\partial_2\varphi)^2}{\partial_z\varphi}\partial_z b^2
=-\omega^1_b+\partial_2 b^3
+\partial_2\varphi(\partial_1 b^1+\partial_2 b^2),
\\
&\frac{1+(\partial_2\varphi)^2}{\partial_z\varphi}\partial_z b^1
+\frac{1+\partial_1\varphi\partial_2\varphi}{\partial_z\varphi}\partial_z b^2
=\omega^2_b+\partial_1 b^3
+\partial_1\varphi(\partial_1 b^1+\partial_2 b^2).
\end{aligned}\right.
\end{equation}
Therefore, the determinant of the coefficient matrix of \eqref{2.8}, \eqref{2.9} is
\begin{align}
\begin{vmatrix}
\frac{\partial_1\varphi\partial_2\varphi}{\partial_z\varphi}
& \frac{1+(\partial_2\varphi)^2}{\partial_z\varphi}\\
\frac{1+(\partial_2\varphi)^2}{\partial_z\varphi}
& \frac{1+\partial_1\varphi\partial_2\varphi}{\partial_z\varphi}\\
\end{vmatrix}
=-\frac{1+(\partial_1\varphi)^2+(\partial_2\varphi)^2}{(\partial_z\varphi)^2}\neq0.
\end{align}

Thus we can solve $\partial_z v^1$,$\partial_z v^2$,$\partial_z b^1$ and
$\partial_z b^2$ from \eqref{2.8} and \eqref{2.9},
respectively.
It exists four homogeneous polynomials
$f^k[\nabla\varphi](\partial_jv_i)$, $f^k[\nabla\varphi](\partial_jb_i)$, k=1,2,3,4,
for $j = 1, 2, i = 1, 2, 3,$
such that
\begin{equation}\label{2.10}
\left\{\begin{aligned}
&\partial_z v^1
=f^1[\nabla\varphi](\omega^1_v,\omega^2_v)+f^2[\nabla\varphi](\partial_jv_i),
\\
&\partial_z v^2
=f^3[\nabla\varphi](\omega^1_v,\omega^2_v)+f^4[\nabla\varphi](\partial_jv_i),
\end{aligned}\right.
\end{equation}
and
\begin{equation}\label{2.11}
\left\{\begin{aligned}
&\partial_z b^1
=f^1[\nabla\varphi](\omega^1_b,\omega^2_b)+f^2[\nabla\varphi](\partial_jb_i),
\\
&\partial_z b^2
=f^3[\nabla\varphi](\omega^1_b,\omega^2_b)+f^4[\nabla\varphi](\partial_jb_i).
\end{aligned}\right.
\end{equation}
It follows that
\begin{align}
&\|\partial_zv^1\|_{X^k}+\|\partial_zv^2\|_{X^k}+\|\partial_zb^1\|_{X^k}+\|\partial_zb^2\|_{X^k}
\nonumber\\
\leq &\|\omega_{vh}\|_{X^k}+\|\omega_{bh}\|_{X^k}+
\mathop{\Sigma}\limits_{i,j}\|\partial_jv_i\|_{X^k}+\mathop{\Sigma}\limits_{i,j}\|\partial_jb_i\|_{X^k}+|h|_{X^{k,\frac{1}{2}}}
\nonumber\\
\leq &\|\omega_{vh}\|_{X^k}+\|\omega_{bh}\|_{X^k}+\|v\|_{X^{k,1}}+\|b\|_{X^{k,1}}+|h|_{X^{k,\frac{1}{2}}}.
\end{align}
Thus, Lemma \eqref{Lemma2.1} is proved.
\end{proof}

The following lemma claims that the tangential vorticities $\omega_{vh}$, $\omega_{bh}$
satisfy vorticity equations \eqref{vorticity equations}.

\begin{lemma}
Assume $\omega_v,\omega_b$ are the vorticities of v, b, respectively.
Then there exist polynomials $F_v^0[\nabla\varphi](\omega_{vh},\omega_{bh},\partial_jv_i,\partial_jb_i)$,
$F_b^0[\nabla\varphi](\omega_{vh},\omega_{bh},\partial_jv_i,\partial_jb_i)$
such that $\omega_{vh},\omega_{bh}$ satisfy \eqref{vorticity equations}
where $F_v^0[\nabla\varphi](\omega_{vh},\omega_{bh},\partial_jv_i,\partial_jb_i)$,
$F^0_b[\nabla\varphi](\omega_{vh},\omega_{bh},\partial_jv_i,\partial_jb_i)$
are the quadratic polynomial vector with respect to $\omega_{vh},\omega_{bh},\partial_jv_i,\partial_jb_i$.
\end{lemma}

\begin{proof}
Firstly, we consider the following stress tensor on the free boundary:
\begin{equation}
S^\varphi v\nn=
\left(
  \begin{array}{c}
n^1\partial_1^\varphi v^1+\frac{n^2}{2}(\partial_1^\varphi v^2+\partial_2^\varphi v^1)
+\frac{n^3}{2}(\partial_1^\varphi v^3+\partial_z^\varphi v^1) \\
\frac{n^1}{2}(\partial_1^\varphi v^2+\partial_2^\varphi v^1)+n^2\partial_2^\varphi v^2
+\frac{n^3}{2}(\partial_2^\varphi v^3+\partial_z^\varphi v^2)\\
\frac{n^1}{2}(\partial_1^\varphi v^3+\partial_z^\varphi v^1)+\frac{n^2}{2}(\partial_2^\varphi v^3+\partial_z^\varphi v^2)
-n^3\partial_1^\varphi v^1-n^3\partial_2^\varphi v^2)\\
  \end{array}
\right)
\end{equation}
and
\begin{equation}
S^\varphi b\nn=
\left(
  \begin{array}{c}
n^1\partial_1^\varphi b^1+\frac{n^2}{2}(\partial_1^\varphi b^2+\partial_2^\varphi b^1)
+\frac{n^3}{2}(\partial_1^\varphi b^3+\partial_z^\varphi b^1) \\
\frac{n^1}{2}(\partial_1^\varphi b^2+\partial_2^\varphi b^1)+n^2\partial_2^\varphi b^2
+\frac{n^3}{2}(\partial_2^\varphi b^3+\partial_z^\varphi b^2)\\
\frac{n^1}{2}(\partial_1^\varphi b^3+\partial_z^\varphi b^1)+\frac{n^2}{2}(\partial_2^\varphi B^3+\partial_z^\varphi b^2)
-n^3\partial_1^\varphi b^1-n^3\partial_2^\varphi b^2)\\
  \end{array}
\right).
\end{equation}
Since $\Pi S^\varphi v\nn=0$, $\Pi S^\varphi b\nn=0$,
then $S^\varphi vn\times \nn=0$, $S^\varphi bn\times \nn=0$,
by formula $\mathbf{A}\times \mathbf{B}=\varepsilon_{ijk}\mathbf{A}_j\mathbf{B}_k$, one has
\begin{equation}
\left\{
\begin{aligned}
&n^3[n^1\partial_1^\varphi v^1+\frac{n^2}{2}(\partial_1^\varphi v^2+\partial_2^\varphi v^1)
+\frac{n^3}{2}(\partial_1^\varphi v^3+\partial_z^\varphi v^1)]\\
&\qquad=n^1[\frac{n^1}{2}(\partial_1^\varphi v^3+\partial_z^\varphi v^1)+\frac{n^2}{2}(\partial_2^\varphi v^3+\partial_z^\varphi v^2)
-n^3\partial_1^\varphi v^1-n^3\partial_2^\varphi v^2)],\\
&n^3[\frac{n^1}{2}(\partial_1^\varphi v^2+\partial_2^\varphi v^1)+n^2\partial_2^\varphi v^2
+\frac{n^3}{2}(\partial_2^\varphi v^3+\partial_z^\varphi v^2)]\\
&\qquad=n^2[\frac{n^1}{2}(\partial_1^\varphi v^3+\partial_z^\varphi v^1)+\frac{n^2}{2}(\partial_2^\varphi v^3+\partial_z^\varphi v^2)
-n^3\partial_1^\varphi v^1-n^3\partial_2^\varphi v^2)],\\
&n^2[n^1\partial_1^\varphi v^1+\frac{n^2}{2}(\partial_1^\varphi v^2+\partial_2^\varphi v^1)
+\frac{n^3}{2}(\partial_1^\varphi v^3+\partial_z^\varphi v^1)]\\
&\qquad=n^1[\frac{n^1}{2}(\partial_1^\varphi v^2+\partial_2^\varphi v^1)+n^2\partial_2^\varphi v^2
+\frac{n^3}{2}(\partial_2^\varphi v^3+\partial_z^\varphi v^2)],
\end{aligned}
\right.
\end{equation}
and
\begin{equation}
\left\{
\begin{aligned}
&n^3[n^1\partial_1^\varphi b^1+\frac{n^2}{2}(\partial_1^\varphi b^2+\partial_2^\varphi b^1)
+\frac{n^3}{2}(\partial_1^\varphi b^3+\partial_z^\varphi b^1)]\\
&\qquad=n^1[\frac{n^1}{2}(\partial_1^\varphi b^3+\partial_z^\varphi b^1)+\frac{n^2}{2}(\partial_2^\varphi b^3+\partial_z^\varphi b^2)-n^3\partial_1^\varphi b^1-n^3\partial_2^\varphi b^2)],\\
&n^3[\frac{n^1}{2}(\partial_1^\varphi b^2+\partial_2^\varphi b^1)+n^2\partial_2^\varphi b^2
+\frac{n^3}{2}(\partial_2^\varphi b^3+\partial_z^\varphi b^2)]\\
&\qquad=n^2[\frac{n^1}{2}(\partial_1^\varphi b^3+\partial_z^\varphi b^1)+\frac{n^2}{2}(\partial_2^\varphi b^3+\partial_z^\varphi b^2)-n^3\partial_1^\varphi b^1-n^3\partial_2^\varphi b^2)],\\
&n^2[n^1\partial_1^\varphi b^1+\frac{n^2}{2}(\partial_1^\varphi b^2+\partial_2^\varphi b^1)
+\frac{n^3}{2}(\partial_1^\varphi b^3+\partial_z^\varphi b^1)]\\
&\qquad=n^1[\frac{n^1}{2}(\partial_1^\varphi b^2+\partial_2^\varphi b^1)+n^2\partial_2^\varphi b^2
+\frac{n^3}{2}(\partial_2^\varphi b^3+\partial_z^\varphi b^2)].
\end{aligned}
\right.
\end{equation}
Due to the fact that $\partial_i\varphi=-\frac{n^i}{n^3}, i=1,2$,
it follows that $\partial_zv^i, i=1,2$,
\begin{align*}
&[(n^1)^2+\frac{(n^3)^2}{2}+\frac{1}{2}\frac{(n^1)^4}{(n^3)^2}-\frac{1}{2}\frac{(n^2)^4}{(n^3)^2}]
\partial_zv^1+[n^1n^2+\frac{(n^1)^3n^2}{(n^3)^2}+\frac{^1(n^2)^3}{(n^3)^2}]\partial_zv^2\\
=&-[\frac{(n^3)^2}{2}-\frac{(n^1)^2}{2}-\frac{(n^2)^2}{2}][\partial_z\varphi\partial_1v^3
-\partial_1\varphi[-\partial_z\varphi(\partial_1v^1+\partial_2v^2)]]\\
&+(\frac{n^1(n^2)^2}{n^3}-2n^1n^3)(\partial_z\varphi\partial_1v^1)
-(n^1n^3+\frac{n^1(n^2)^2}{n^3})(\partial_z\varphi\partial_1v^2)\\
&+[\frac{(n^2)^2}{n^3}\frac{n^2}{2}-\frac{n^1n^2}{n^3}\frac{n^1}{2}-\frac{n^2n^3}{2}]
(\partial_z\varphi\partial_1v^2+\partial_z\varphi\partial_1v^1),
\end{align*}
\begin{align*}
&[n^1n^2+\frac{(n^1)^3n^2}{(n^3)^2}+\frac{n^1(n^2)^3}{(n^3)^2}]\partial_zv^1
+[(n^2)^2+\frac{(n^3)^2}{2}+\frac{1}{2}\frac{(n^2)^4}{(n^3)^2}-\frac{1}{2}\frac{(n^1)^4}{(n^3)^2}]\partial_zv^2\\
=&-[\frac{(n^3)^2}{2}-\frac{(n^1)^2}{2}-\frac{(n^2)^2}{2}][\partial_z\varphi\partial_2v^3
-\partial_z\varphi[-\partial_z\varphi(\partial_1v^1+\partial_2v^2)]]\\
&+(n^2n^3+\frac{n^2(n^1)^2}{n^3})(\partial_z\varphi\partial_1v^1)
-(\frac{n^2(n^1)^2}{n^3}-2n^2n^3)(\partial_z\varphi\partial_2v^2)\\
&+[\frac{(n^1)^2}{n^3}\frac{n^1}{2}-\frac{n^1n^2}{n^3}\frac{n^2}{2}-\frac{n^1n^3}{2}]
(\partial_z\varphi\partial_1v^2+\partial_z\varphi\partial_2v^1),
\end{align*}
and $\partial_zb^i, i=1,2$,
\begin{align*}
&[(n^1)^2+\frac{(n^3)^2}{2}+\frac{1}{2}\frac{(n^1)^4}{(n^3)^2}-\frac{1}{2}\frac{(n^2)^4}{(n^3)^2}]
\partial_zb^1+[n^1n^2+\frac{(n^1)^3n^2}{(n^3)^2}+\frac{^1(n^2)^3}{(n^3)^2}]\partial_zb^2\\
=&-[\frac{(n^3)^2}{2}-\frac{(n^1)^2}{2}-\frac{(n^2)^2}{2}][\partial_z\varphi\partial_1b^3
-\partial_1\varphi[-\partial_z\varphi(\partial_1b^1+\partial_2b^2)]]\\
&+(\frac{n^1(n^2)^2}{n^3}-2n^1n^3)(\partial_z\varphi\partial_1b^1)
-(n^1n^3+\frac{n^1(n^2)^2}{n^3})(\partial_z\varphi\partial_1b^2)\\
&+[\frac{(n^2)^2}{n^3}\frac{n^2}{2}-\frac{n^1n^2}{n^3}\frac{n^1}{2}-\frac{n^2n^3}{2}]
(\partial_z\varphi\partial_1b^2+\partial_z\varphi\partial_1b^1),
\end{align*}
\begin{align*}
&[n^1n^2+\frac{(n^1)^3n^2}{(n^3)^2}+\frac{n^1(n^2)^3}{(n^3)^2}]\partial_zb^1
+[(n^2)^2+\frac{(n^3)^2}{2}+\frac{1}{2}\frac{(n^2)^4}{(n^3)^2}-\frac{1}{2}\frac{(n^1)^4}{(n^3)^2}]\partial_zb^2\\
=&-[\frac{(n^3)^2}{2}-\frac{(n^1)^2}{2}-\frac{(n^2)^2}{2}][\partial_z\varphi\partial_2b^3
-\partial_z\varphi[-\partial_z\varphi(\partial_1b^1+\partial_2b^2)]]\\
&+(n^2n^3+\frac{n^2(n^1)^2}{n^3})(\partial_z\varphi\partial_1b^1)
-(\frac{n^2(n^1)^2}{n^3}-2n^2n^3)(\partial_z\varphi\partial_2b^2)\\
&+[\frac{(n^1)^2}{n^3}\frac{n^1}{2}-\frac{n^1n^2}{n^3}\frac{n^2}{2}-\frac{n^1n^3}{2}]
(\partial_z\varphi\partial_1b^2+\partial_z\varphi\partial_2b^1),
\end{align*}
where the coefficient matrix of $(\partial_zv^1,\partial_zv^2)^\top$ and $(\partial_zb^1,\partial_zb^2)^\top$ is
\begin{equation}\label{M}
M=
\left(
  \begin{array}{cc}
(n^1)^2+\frac{(n^3)^2}{2}+\frac{1}{2}\frac{(n^1)^4}{(n^3)^2}-\frac{1}{2}\frac{(n^2)^4}{(n^3)^2}&
 n^1n^2+\frac{(n^1)^3n^2}{(n^3)^2}+\frac{^1(n^2)^3}{(n^3)^2}\\
n^1n^2+\frac{(n^1)^3n^2}{(n^3)^2}+\frac{n^1(n^2)^3}{(n^3)^2}  &(n^2)^2+\frac{(n^3)^2}{2}+\frac{1}{2}\frac{(n^2)^4}{(n^3)^2}-\frac{1}{2}\frac{(n^1)^4}{(n^3)^2} \\
  \end{array}
\right).
\end{equation}
Assume $|\nabla h|_\infty$ is suitably small, then $n^3$ is suitably large
and $|n^1| + |n^2|$ is suitably small, such that M is strictly diagonally dominant matrix.
Hence, M is nondegenerate. We can solve $\partial_zv^1$, $\partial_zv^2$, $\partial_zb^1$ and $\partial_zb^2$,
namely, there exist four homogeneous polynomials $f^5[\nabla\varphi](\partial_jv^i)$, $f^6[\nabla\varphi](\partial_jv^i)$, $f^5[\nabla\varphi](\partial_jb^i)$ and $f^6[\nabla\varphi](\partial_jb^i)$,
which are linear functions of $\partial_jv^i$ with the coefficients being fractions of $\nabla\varphi$.
Precisely, for $j = 1, 2, i = 1, 2, 3,$ we have
\begin{equation}\label{2.18}
\left\{
\begin{aligned}
&\partial_zv^1=f^5[\nabla\varphi](\partial_jv^i),
&\partial_zv^2=f^6[\nabla\varphi](\partial_jv^i),\\
&\partial_zb^1=f^5[\nabla\varphi](\partial_jb^i),
&\partial_zb^2=f^6[\nabla\varphi](\partial_jb^i).
\end{aligned}
\right.
\end{equation}

We have the boundary values of $\omega_{vh}=(\omega^1_v,\omega^2_v)$ as follows
\begin{align*}
\omega^1_v&=-\frac{\partial_1\varphi\partial_2\varphi}{\partial_z\varphi}\partial_z v^1
-\frac{1+(\partial_2\varphi)^2}{\partial_z\varphi}\partial_z v^2+\partial_2 v^3
+\partial_2\varphi(\partial_1 v^1+\partial_2 v^2)
\\
&=-\frac{\partial_1\varphi\partial_2\varphi}{\partial_z\varphi}f_v^5[\nabla\varphi](\partial_jv^i)
-\frac{1+(\partial_2\varphi)^2}{\partial_z\varphi}f_v^6[\nabla\varphi](\partial_jv^i)
+\partial_2^\varphi v^3+\partial_2\varphi(\partial_1 v^1+\partial_2 v^2)
\\
&=F^1[\nabla\varphi](\partial_jv^i),
\end{align*}
\begin{align*}
\omega^2_v&=\frac{1+(\partial_2\varphi)^2}{\partial_z\varphi}\partial_z v^1
+\frac{1+\partial_1\varphi\partial_2\varphi}{\partial_z\varphi}\partial_z v^2-\partial_1 v^3
-\partial_1\varphi(\partial_1 v^1+\partial_2 v^2)\\
&=\frac{1+(\partial_2\varphi)^2}{\partial_z\varphi}f_v^5[\nabla\varphi](\partial_jv^i)
+\frac{1+\partial_1\varphi\partial_2\varphi}{\partial_z\varphi}f_v^6[\nabla\varphi](\partial_jv^i)
-\partial_1^\varphi v^3-\partial_1\varphi(\partial_1 v^1+\partial_2 v^2)
\\
&=F^2[\nabla\varphi](\partial_jv^i).
\end{align*}
Similarly, we have the boundary value of $\omega_{bh}=(\omega^1_b,\omega^2_b)$
\begin{align*}
\omega^1_b&=-\frac{\partial_1\varphi\partial_2\varphi}{\partial_z\varphi}\partial_z b^1
-\frac{1+(\partial_2\varphi)^2}{\partial_z\varphi}\partial_z b^2+\partial_2 b^3
+\partial_2\varphi(\partial_1 b^1+\partial_2 b^2)
\\
&=-\frac{\partial_1\varphi\partial_2\varphi}{\partial_z\varphi}f_b^5[\nabla\varphi](\partial_jb^i)
-\frac{1+(\partial_2\varphi)^2}{\partial_z\varphi}f_b^6[\nabla\varphi](\partial_jb^i)
+\partial_2^\varphi b^3+\partial_2\varphi(\partial_1 b^1+\partial_2 b^2)
\\
&=F^1[\nabla\varphi](\partial_jb^i),
\end{align*}
\begin{align*}
\omega^2_b&=\frac{1+(\partial_2\varphi)^2}{\partial_z\varphi}\partial_z b^1
+\frac{1+\partial_1\varphi\partial_2\varphi}{\partial_z\varphi}\partial_z b^2-\partial_1 b^3
-\partial_1\varphi(\partial_1 b^1+\partial_2 b^2)\\
&=\frac{1+(\partial_2\varphi)^2}{\partial_z\varphi}f_b^5[\nabla\varphi](\partial_jb^i)
+\frac{1+\partial_1\varphi\partial_2\varphi}{\partial_z\varphi}f_b^6[\nabla\varphi](\partial_jb^i)
-\partial_1^\varphi b^3-\partial_1\varphi(\partial_1 b^1+\partial_2 b^2)
\\
&=F^2[\nabla\varphi](\partial_jb^i).
\end{align*}
Since the zero boundary value of magnetic field, one has $F^1[\nabla\varphi](\partial_jb^i)=F^2[\nabla\varphi](\partial_jb^i)=0$
on $z=0$ for $j = 1, 2, i = 1, 2, 3$.
In addition, $\omega_{vh},\omega_{bh}$ satisfy the equations
\begin{equation*}
\left\{\begin{aligned}
&\partial_t^\varphi\omega_{vh}-\epsilon\Delta^\varphi\omega_{vh}+v\cdot\nabla^\varphi\omega_{vh}
-b\cdot\nabla^\varphi\omega_{bh}=\omega_{vh}\cdot\nabla_h^\varphi v_h
+\omega^3_v\partial_z^\varphi v_h-\omega_{bh}\cdot\nabla_h^\varphi b_h
-\omega^3_b\partial_z^\varphi b^h,
\\
&\partial_t^\varphi\omega_{bh}-\epsilon\Delta^\varphi\omega_{bh}+v\cdot\nabla^\varphi\omega_{bh}
-b\cdot\nabla^\varphi\omega_{vh}
=[\nabla^\varphi\times,b\cdot\nabla^\varphi]v_h-[\nabla^\varphi\times,v\cdot\nabla^\varphi]b_h,
\end{aligned}\right.
\end{equation*}
where the force term can be transformed as follows:
\begin{align*}
&\omega_{vh}\cdot\nabla_h^\varphi v_h
+\omega^3_v\partial_z^\varphi v_h
-\omega_{bh}\cdot\nabla_h^\varphi b_h
-\omega^3_b\partial_z^\varphi b^h
\\
=&\omega_{v 1}(\partial_{1} v_{h}-\frac{\partial_{1}\varphi}{\partial_{z}\varphi}\partial_{z} v_{h})
+\omega_{v 2}(\partial_{2} v_{h}-\frac{\partial_{2}\varphi}{\partial_{z}\varphi}\partial_{z} v_{h})
\\
&
+(\partial_{1} v_{2}-\frac{\partial_{1}\varphi}{\partial_{z}\varphi}\partial_{z} v_{2}
-\partial_{2} v_{1}+\frac{\partial_{2}\varphi}{\partial_{z}\varphi}\partial_{z}v_{1})\frac{1}{\partial_{z} \varphi} \partial_{z} v_{h}
\\
&-\omega_{b1}(\partial_1b_h-\frac{\partial_1\varphi}{\partial_z\varphi}\partial_zb_h)
-\omega_{b2}(\partial_2b_h-\frac{\partial_2\varphi}{\partial_z\varphi}\partial_zb_h)
\\
&
-(\partial_1b_2-\frac{\partial_1\varphi}{\partial_z\varphi}\partial_zb_2-\partial_2b_1
+\frac{\partial_2\varphi}{\partial_z\varphi}\partial_zb_1)\frac{1}{\partial_z\varphi}\partial_zb_h.
\end{align*}
By \eqref{2.10} and \eqref{2.11}, one has
\begin{align*}
\omega_{vh}\cdot\nabla_h^\varphi v_h
+\omega^3_v\partial_z^\varphi v_h
-\omega_{bh}\cdot\nabla_h^\varphi b_h
+\omega^3_b\partial_z^\varphi b^h
=F_v^0[\nabla\varphi](\omega_{vh},\omega_{bh},\partial_jv^i,\partial_jb^i),j = 1, 2, i = 1, 2, 3.
\end{align*}
It follows from the definition of curl that
\begin{align*}
&([\nabla^\varphi\times,b\cdot\nabla^\varphi]v_h-[\nabla^\varphi\times,v\cdot\nabla^\varphi]b_h)_1
\\
=&\Sigma_{i=1}^3(\partial_2^\varphi b^i\partial_i^\varphi v^3
-\partial_z^\varphi b^i\partial_i^\varphi v^2
-\partial_2^\varphi v^i\partial_i^\varphi b^3
+\partial_z^\varphi v^i\partial_i^\varphi b^2)\\
=&(\partial_2b^1-\frac{\partial_2\varphi}{\partial_z\varphi}\partial_zb^1)
(\partial_1 v^3-\frac{\partial_1\varphi}{\partial_z\varphi}\partial_zv^3)
+(\partial_2b^2-\frac{\partial_2\varphi}{\partial_z\varphi}\partial_zb^2)
(\partial_2 v^3-\frac{\partial_2\varphi}{\partial_z\varphi}\partial_zv^3)
\\
&+(\partial_2b^3-\frac{\partial_2\varphi}{\partial_z\varphi}\partial_zb^3)
\frac{\partial_zv^3}{\partial_z\varphi}
-(\partial_2v^1-\frac{\partial_2\varphi}{\partial_z\varphi}\partial_zv^1)
(\partial_1 b^3-\frac{\partial_1\varphi}{\partial_z\varphi}\partial_zb^3)
\\
&
-(\partial_2v^2-\frac{\partial_2\varphi}{\partial_z\varphi}\partial_zv^1)
(\partial_2 b^3-\frac{\partial_2\varphi}{\partial_z\varphi}\partial_zb^3)
-(\partial_2v^3-\frac{\partial_2\varphi}{\partial_z\varphi}\partial_zv^1)
\frac{\partial_zb^3}{\partial_z\varphi}
\\
&
-\frac{\partial_zb^1}{\partial_z\varphi}(\partial_1 v^2-\frac{\partial_1\varphi}{\partial_z\varphi}\partial_zv^2)
-\frac{\partial_zb^2}{\partial_z\varphi}(\partial_2 v^2-\frac{\partial_2\varphi}{\partial_z\varphi}\partial_zv^2)
-\frac{\partial_zb^3}{\partial_z\varphi}\frac{\partial_zv^2}{\partial_z\varphi}
\\
&
+\frac{\partial_zv^1}{\partial_z\varphi}(\partial_1 b^2-\frac{\partial_1\varphi}{\partial_z\varphi}\partial_zb^2)
+\frac{\partial_zv^2}{\partial_z\varphi}(\partial_2 b^2-\frac{\partial_2\varphi}{\partial_z\varphi}\partial_zb^2)
+\frac{\partial_zv^3}{\partial_z\varphi}\frac{\partial_zb^2}{\partial_z\varphi}.
\end{align*}
Due to \eqref{2.10}, \eqref{2.11} and divergence free condition \eqref{divergence free condition},
we get $([\nabla^\varphi\times,b\cdot\nabla^\varphi]v_h-[\nabla^\varphi\times,v\cdot\nabla^\varphi]b_h)_1
=F_b^0[\nabla\varphi](\omega_{vh},\omega_{bh},\partial_jv^i,\partial_jb^i)$.
For $([\nabla^\varphi\times,b\cdot\nabla^\varphi]v_h-[\nabla^\varphi\times,v\cdot\nabla^\varphi]b_h)_2$,
one can deduce that $([\nabla^\varphi\times,b\cdot\nabla^\varphi]v_h-[\nabla^\varphi\times,v\cdot\nabla^\varphi]b_h)_2
=F_b^0[\nabla\varphi](\omega_{vh},\omega_{bh},\partial_jv^i,\partial_jb^i),
j = 1, 2, i = 1, 2, 3$.
where $F_v^0[\nabla\varphi](\omega_{vh},\omega_{bh},\partial_jv^i,\partial_jb^i),
F_v^0[\nabla\varphi](\omega_{vh},\omega_{bh},\partial_jv^i,\partial_jb^i)$
is a quadratic polynomial vector with
$\omega_{vh},\omega_{bh},\partial_jv^i,\partial_jb^i$, the
coefficients are fractions of $\nabla\varphi$.
\end{proof}

\subsection{Time Derivatives Estimates}
We establish a priori estimates for the free boundary problem for MHD equations \eqref{MHDF}
of the tangential derivatives and time derivatives.
The estimates for the normal derivatives are very different from Lee\cite{Lee17}.
Lee used the variable $S^v_n =\Pi S^\varphi vn$,
$S^b_n =\Pi S^\varphi bn$,
while we investigate the normal derivatives by the vorticity in this paper.

Before the estimates of tangential derivatives,
we give two preliminary lemmas for h by using the kinetical boundary condition.
\begin{lemma}[\cite{Wu16}]\label{Lemma2.3}
Assume $0\leq l\leq m-1$. Then $\|\partial_t^lh\|_{L^2}$ can be estimated as follows
\begin{align}
\|\partial_t^lh\|^2_{L^2}\leq |h_0|^2_{X^{m-1,1}}+\int_0^t|h|^2_{X^{m-1}}+\|v\|^2_{X^{m-1,1}}dt+
\|\partial_zv\|^2_{L^4(0,T],X^{m-1})}.
\end{align}
\end{lemma}

The second preliminary lemma concerns the estimate of $\sqrt{\epsilon}|\partial_t^lZ^\alpha h|_{\frac{1}{2}}$,
by which we can bound $\sqrt{\epsilon}\|S^\varphi\partial_t^lZ^\alpha \eta\|_{L^2}$
and then we can bound $\sqrt{\epsilon}\|S^\varphi\partial_t^lZ^\alpha v\|_{L^2}$
and $\sqrt{\epsilon}\|S^\varphi\partial_t^lZ^\alpha b\|_{L^2}$.

\begin{lemma}[\cite{Wu16}]\label{Lemma2.4}
Assume $0\leq l\leq m-1$. We have the estimate
\begin{align}
\epsilon|h|^2_{X^{m-1,\frac{3}{2}}}\leq\epsilon|h_0|^2_{X^{m-1,\frac{3}{2}}}+
\int_0^t|h|^2_{X^{m-1,1}}
+\epsilon\sum_{l\leq m-1,l+|\alpha|\leq m}|\nabla V^{l,\alpha}|^2_{L^2}dt.
\end{align}
\end{lemma}

Next we give the estimates for the tangential derivatives $\|\partial_t^lv\|_{L^2}$,
$\|\partial_t^lb\|_{L^2}$ and
$\sqrt{\epsilon}\|\nabla\partial_t^lZ^\alpha v\|_{L^2}$,
$\sqrt{\epsilon}\|\nabla\partial_t^lZ^\alpha b\|_{L^2}$.
However, the proof is different from Lee \cite{Lee17}
since $\|\partial_t^{\ell} q\|$ has no bound for infinite fluid depth.
The following lemma concerns the estimates of tangential derivatives.
\begin{lemma}\label{Lemma2.5}
Assume the conditions are the same with those of Proposition \ref{Proposition1.1},
then $v$, $b$ and $h$ satisfy the a priori estimate:
\begin{equation}\label{2.21}
\begin{aligned}
&\|v\|_{X^{m-1,1}}+\|b\|_{X^{m-1,1}}+|h|^2_{X^{m-1,1}}
+\epsilon|h|^2_{X^{m-1,\frac{3}{2}}}+\epsilon\int_0^t(\|\nabla v\|^2_{X^{m-1,1}}+\|\nabla b\|^2_{X^{m-1,1}})
\\
&\leq \|v_0\|_{X^{m-1,1}}+\|b_0\|_{X^{m-1,1}}+|h_0|^2_{X^{m-1,1}}
+\epsilon|h_0|^2_{X^{m-1,\frac{3}{2}}}
\\
&+\|\partial_zv\|_{L^4([0,T],X^{m-1})}+\|\partial_zb\|_{L^4([0,T],X^{m-1})}.
\end{aligned}
\end{equation}
\end{lemma}
\begin{proof}
We choose the Alinhac's good unknown as the variable.
Applying $\partial_t^lZ^\alpha$ to \eqref{MHDF},
then $V^{l,\alpha}$, $B^{l,\alpha}$ and $Q^{l,\alpha}$ satisfy
\begin{equation}\label{AlinhacE}
\left\{\begin{aligned}
&\partial_t^\varphi V^{l,\alpha}-2\epsilon\nabla^\varphi\cdot S^\varphi V^{l,\alpha}
+v\cdot\nabla^\varphi V^{l,\alpha}-b\cdot\nabla^\varphi B^{l,\alpha}+\nabla^\varphi Q^{l,\alpha}\\
&\quad=-\partial_t^\varphi\partial_z^\varphi v\partial_t^lZ^\alpha \eta
-v\cdot\nabla^\varphi\partial_z^\varphi v\partial_t^lZ^\alpha \eta
+b\cdot\nabla^\varphi\partial_z^\varphi b\partial_t^lZ^\alpha \eta
+\nabla^\varphi\partial_z^\varphi q\partial_t^lZ^\alpha \eta
\\
&\quad+2\epsilon\nabla^\varphi\cdot(S^\varphi\partial_z^\varphi v\partial_t^lZ^\alpha \eta)
-2\epsilon\partial_z^\varphi S^\varphi v_y\partial_t^lZ^\alpha \eta+{\text b.t.},
\\
&\partial_t^\varphi B^{l,\alpha}-2\epsilon\nabla^\varphi\cdot S^\varphi B^{l,\alpha}+v\cdot\nabla^\varphi\varphi B^{l,\alpha}-b\cdot\nabla^\varphi V^{l,\alpha}
\\
&\quad=-\partial_t^\varphi\partial_z^\varphi B\partial_t^lZ^\alpha \eta
-v\cdot\nabla^\varphi\partial_z^\varphi b\partial_t^lZ^\alpha \eta
+b\cdot\nabla^\varphi\partial_z^\varphi v\partial_t^lZ^\alpha \eta
\\
&\quad+2\epsilon\nabla^\varphi\cdot(S^\varphi\partial_z^\varphi b\partial_t^lZ^\alpha \eta)
-2\epsilon\partial_z^\varphi S^\varphi b_y\partial_t^lZ^\alpha \eta+{\text b.t.},
\\
&\nabla^\varphi\cdot V^{l,\alpha}=-(\nabla^\varphi\cdot\partial_z^\varphi v)\partial_t^lZ^\alpha \eta+{\text b.t.}={\text b.t.},
\\
&\nabla^\varphi\cdot B^{l,\alpha}=-(\nabla^\varphi\cdot\partial_z^\varphi b)\partial_t^lZ^\alpha \eta+{\text b.t.}={\text b.t.},
\end{aligned}\right.
\end{equation}
and the following initial boundary conditions
\begin{equation*}
\left\{\begin{aligned}
&Q^{l,\alpha}N-2\epsilon S^\varphi V^{l,\alpha}\NN
\\
&\quad=(g-\partial_z^\varphi q)\partial_t^lZ^\alpha h\NN
+2\epsilon(S^\varphi\partial_z^\varphi v\NN)\partial_t^lZ^\alpha h
-[\partial_t^lZ^\alpha, 2\epsilon S^\varphi vn\cdot n,\NN]
\\
&\quad+(2\epsilon S^\varphi v-2\epsilon S^\varphi vn\cdot n)\partial_t^lZ^\alpha N
+2\epsilon[\partial_t^lZ^\alpha, S^\varphi v,\NN]+{\text b.t.},
\\
&B^{l,\alpha}|_{z=0}=-(\partial_z^\varphi b\partial_t^lZ^\alpha\eta)|_{z=0},
\\
&(\partial_t^lZ^\alpha v,\partial_t^lZ^\alpha b,\partial_t^lZ^\alpha h)|_{t=0}
=(\partial_t^lZ^\alpha v_0),\partial_t^lZ^\alpha b_0,\partial_t^lZ^\alpha h_0).
\end{aligned}\right.
\end{equation*}

Step one:
When $|\alpha| \geq 1, 1 \leq l + |\alpha| \leq m$,
we establish the $L^2$ estimate of $V^{l,\alpha},B^{l,\alpha}$.
The estimates are similar to \cite{Lee17},
but we do not use $g-\partial_z^\varphi q^E>0$.
\begin{align*}
&\frac{1}{2}\frac{\mathrm{d}}{\mathrm{d}t}\int_{\mathbb{R}^3_-}|V^{l,\alpha}|^2+|B^{l,\alpha}|^2\mathrm{d}\mathcal{V}_t
+2\epsilon\int_{\mathbb{R}^3_-}|S^\varphi V^{l,\alpha}|^2+|S^\varphi B^{l,\alpha}|^2\mathrm{d}\mathcal{V}_t
-\int_{\mathbb{R}^3_-}Q^{l,\alpha}\nabla^\varphi\cdot V^{l,\alpha}\mathrm{d}\mathcal{V}_t
\\
&\leq \int_{z=0}(2\epsilon S^\varphi V^{l,\alpha}\NN-Q^{l,\alpha}\NN)V^{l,\alpha}\mathrm{d}y
+\int_{z=0}2\epsilon S^\varphi B^{l,\alpha}\NN B^{l,\alpha}\mathrm{d}y
\\
&\quad+\|\partial_z v\|^2_{X^{m-1}}+\|\partial_z b\|^2_{X^{m-1}}+\|\nabla q\|^2_{X^{m-1}}+b.t.
\\
&\leq-\frac{1}{2}\frac{\mathrm{d}}{\mathrm{d}t}\int_{z=0}(g-\partial_z^\varphi q)|\partial_t^lZ^\alpha h|^2\mathrm{d}y
+\|\partial_z v\|^2_{X^{m-1}}+\|\partial_z b\|^2_{X^{m-1}}
+\|\nabla q\|^2_{X^{m-1}}+\epsilon|h|_{X^{m,\frac{1}{2}}}+{\text b.t.}
\end{align*}
where we use $\int_{\mathbb{R}^3_-}b\cdot\nabla^\varphi B^{l,\alpha}V^{l,\alpha}\mathrm{d}\mathcal{V}_t
+\int_{R^3_-}b\cdot\nabla^\varphi B^{l,\alpha}V^{l,\alpha}\mathrm{d}\mathcal{V}_t=0$,
then
\begin{align*}
&\frac{1}{2}\frac{d}{dt}\int_{\mathbb{R}^3_-}|V^{l,\alpha}|^2+|B^{l,\alpha}|^2\mathrm{d}\mathcal{V}_t
+\frac{1}{2}\frac{d}{dt}\int_{z=0}(g-\partial_z^\varphi q)|\partial_t^lZ^\alpha h|^2\mathrm{d}y
\\
&\quad+2\epsilon\int_{\mathbb{R}^3_-}|S^\varphi V^{l,\alpha}|^2\mathrm{d}\mathcal{V}_t
+2\epsilon\int_{\mathbb{R}^3_-}|S^\varphi B^{l,\alpha}|^2\mathrm{d}\mathcal{V}_t
\\
&\leq\|\partial_z v\|^2_{X^{m-1}}+\|\partial_z b\|^2_{X^{m-1}}+\|\nabla q\|^2_{X^{m-1}}+{\text b.t.}
\end{align*}

Since $(g-\partial_z^\varphi q)_{z=0}\geq c_0>0$,
a priori estimates can be closed. Thus,
\begin{align}\label{2.23}
&\|\partial_t^lZ^\alpha v\|^2+\|\partial_t^lZ^\alpha b\|^2+|\partial_t^lZ^\alpha h|^2
+\epsilon |\partial_t^lZ^\alpha h|^2_{\frac{1}{2}}
+\epsilon\int_0^t\|\nabla\partial_t^lZ^\alpha v\|^2\mathrm{d}t
+\epsilon\int_0^t\|\nabla\partial_t^lZ^\alpha b\|^2\mathrm{d}t
\nonumber\\
&\leq\|v_0\|^2_{X^{m-1}}+\|b_0\|^2_{X^{m-1}}+|h_0|^2_{X^{m-1,1}}+\epsilon |h_0|^2_{X^{m-1},\frac{3}{2}}
\nonumber\\
&+\int_0^T\|\partial_zv\|^2_{X^{m-1}}+\int_0^T\|\partial_zb\|^2_{X^{m-1}}+\|\nabla q\|^2_{X^{m-1}}
\end{align}
where we have used the estimate of $\epsilon |\partial_t^lZ^\alpha h|^2_{\frac{1}{2}}$
that is proved by Lemma \ref{Lemma2.4}.

Step two:
when $|\alpha|=0$ and $0 \leq l\leq m-1$, we have no bounds of $q$ and $\partial_t^l q$.
Since we have developed the $L^2$ estimate of $V^{l,0},B^{l,0}$,
there is no need to use the variable $Q^{l,\alpha}$,
the divergence free condition and the dynamical
boundary condition. Then
\begin{align*}
&\frac{1}{2}\frac{\mathrm{d}}{\mathrm{d}t}\int_{\mathbb{R}^3_-}|V^{l,0}|^2+|B^{l,0}|^2\mathrm{d}\mathcal{V}_t
+2\epsilon\int_{\mathbb{R}^3_-}|S^\varphi V^{l,0}|^2\mathrm{d}\mathcal{V}_t
+2\epsilon\int_{\mathbb{R}^3_-}|S^\varphi B^{l,0}|^2\mathrm{d}\mathcal{V}_t
\\
&\leq -\int_{R^3_-}\partial_t^l\nabla^\varphi q\cdot V^{l,0}\mathrm{d}\mathcal{V}_t
+\int_{z=0}2\epsilon S^\varphi V^{l,0}\NN\cdot V^{l,0}\mathrm{d}y
+\int_{z=0}2\epsilon S^\varphi B^{l,0}\NN\cdot B^{l,0}\mathrm{d}y
\\
&\quad+\|\partial_z v\|^2_{X^{m-1}}+\|\partial_z b\|^2_{X^{m-1}}+{\text b.t.}
\\
&\leq \|\partial_t^l\nabla q\|^2_{L^2}+\epsilon\|\partial_t^lv|_{z=0}\|^2_{L^2}
+\epsilon|\partial_t^lh|_{z=0}|^2_{L^2}
+\epsilon\|\partial_t^l\partial_yv|_{z=0}\|^2_{L^2}
+\epsilon\|\partial_t^l\partial_zv|_{z=0}\|^2_{L^2}
\\
&\quad+\epsilon\|\partial_t^l\partial_zb|_{z=0}\|^2_{L^2}
+\|\partial_z v\|^2_{X^{m-1}}+\|\partial_z b\|^2_{X^{m-1}}+\|\nabla q\|^2_{X^{m-1}}+{\text b.t.}.
\end{align*}

Since $\partial_zv,\partial_zb$ can be expressed in terms of tangential derivatives,
then
\begin{align*}
&\frac{1}{2}\frac{\mathrm{d}}{\mathrm{d}t}\int_{\mathbb{R}^3_-}|V^{l,0}|^2+|B^{l,0}|^2\mathrm{d}\mathcal{V}_t
+2\epsilon\int_{\mathbb{R}^3_-}|S^\varphi V^{l,0}|^2\mathrm{d}\mathcal{V}_t
+2\epsilon\int_{\mathbb{R}^3_-}|S^\varphi B^{l,0}|^2\mathrm{d}\mathcal{V}_t
\\
&\leq \|\partial_t^l\nabla q\|^2_{L^2}
+\epsilon|\partial_t^l\partial_yv|_{z=0}\|^2_{L^2}
+\epsilon|\partial_t^lh|^2_{X^{0,\frac{1}{2}}}
+\|\partial_z v\|^2_{X^{m-1}}+\|\partial_z b\|^2_{X^{m-1}}+{\text b.t.}
\\
&\leq \|\partial_t^l\nabla q\|^2_{L^2}+\epsilon\|\nabla v\|^2_{X^{m-1,1}}
+\epsilon|\partial_t^lh|^2_{X^{0,\frac{1}{2}}}
+\|\partial_z v\|^2_{X^{m-1}}+\|\partial_z b\|^2_{X^{m-1}}+{\text b.t.}.
\end{align*}

Similarly, we have
\begin{equation}\label{2.24}
\begin{aligned}
&\|V^{l,0}\|^2+\|B^{l,0}\|^2+\epsilon\int_0^t\|\nabla V^{l,0}\|^2\mathrm{d}t
+\epsilon\int_0^t\|\nabla B^{l,0}\|^2\mathrm{d}t
\\
&\leq \|v_0\|^2_{X^{m-1,1}}+\|b_0\|^2_{X^{m-1,1}}+|h_0|^2_{X^{m-1,1}}
+\epsilon|h_0|^2_{X^{m-1,\frac{3}{2}}}+\|\nabla q\|^2_{L^4([0,T],X^{m-1})}
\\
&+\|\partial_zv\|^2_{L^4([0,T],X^{m-1})}+\|\partial_zb\|^2_{L^4([0,T],X^{m-1})}
+\epsilon\int_0^T\|\nabla v\|^2_{X^{m-1,1}}\mathrm{d}t.
\end{aligned}
\end{equation}
By \eqref{2.24} and Lemma \ref{Lemma2.3},
we obtain the estimate of $\partial_t^lv,\partial_t^lb$:
\begin{equation*}
\begin{aligned}
&\int_{\mathbb{R}^2}|\partial_t^lv|^2+|\partial_t^lb|^2\mathrm{d}y\leq \|v_0\|^2_{X^{m-1,1}}+\|b_0\|^2_{X^{m-1,1}}+|h_0|^2_{X^{m-1,1}}
+\epsilon|h_0|^2_{X^{m-1,\frac{3}{2}}}
\\
&+\|\nabla q\|^2_{L^4([0,T],X^{m-1})}+\|\partial_zv\|^2_{L^4([0,T],X^{m-1})}
+\|\partial_zb\|^2_{L^4([0,T],X^{m-1})}+\epsilon\int_0^T\|\nabla v\|^2_{X^{m-1}}\mathrm{d}t
\\
&+\epsilon\int_0^T\|\nabla b\|^2_{X^{m-1}}\mathrm{d}t
+\int_0^t|h|^2_{X^{m-1}}+\|v\|^2_{X^{m-1}}+\|v\|^2_{X^{m-1}}\mathrm{d}t.
\end{aligned}
\end{equation*}
While $q$ satisfies the elliptic equation with nonhomogeneous Dirichlet boundary condition
\begin{equation*}
\left\{
\begin{aligned}
&\Delta^\varphi q=-\nabla\cdot(v\cdot\nabla^\varphi v)+\nabla\cdot(b\cdot\nabla^\varphi b),\\
&q|_{z=0}=gh+2\epsilon S^\varphi vn\cdot n,
\end{aligned}
\right.
\end{equation*}
then it is easy to get the gradient estimate for $q$ as follows
\begin{align}\label{2.25}
\|\nabla q\|_{X^{m-1}}\leq& \|\nabla\cdot(v\cdot\nabla^\varphi v)\|_{X^{m-1}}
+\|\nabla\cdot(b\cdot\nabla^\varphi b)\|_{X^{m-1}}+|q|_{z=0}|_{X^{m-1},\frac{1}{2}}
\nonumber\\
\leq& \|v\|_{X^{m-1,1}}+\|b\|_{X^{m-1,1}}+\|\partial_zv\|_{X^{m-1}}+\|\partial_zb\|_{X^{m-1}}
\nonumber\\
&+g|h|_{X^{m-1},\frac{3}{2}}
+\epsilon|v_{z=0}|_{X^{m-1},\frac{3}{2}}
+\epsilon|h|_{X^{m-1},\frac{3}{2}}.
\end{align}
Takig summation for $l$ and $\alpha$ \eqref{2.23}, \eqref{2.24} and \eqref{2.25}, we get the estimate \eqref{2.21}.
Thus, Lemma \ref{Lemma2.5} is proved.
\end{proof}

\subsection{Proof Proposition 1.1}
In order to study $\partial_zv,\partial_zb$ on the right hand of \eqref{2.21} to close the energy estimates, we should estimate
$\|\partial_zv\|^2_{L^4([0,T],X^{m-1})}$, $\|\partial_zb\|^2_{L^4([0,T],X^{m-1})}$.

\begin{lemma}\label{Lemma2.6}
Assume $\omega_v$ and $\omega_b$  are the vorticities of $v$ and $b$ of the free boundary problems for
MHD equations. Then $\omega_{vh}$ and $\omega_{bh}$ satisfy the following estimate:
\begin{align}
&\|\omega_{vh}\|^2_{L^4([0,T],X^{m-1})}
+\|\omega_{bh}\|^2_{L^4([0,T],X^{m-1})}
+\|\partial_zv\|^2_{L^4([0,T],X^{m-1})}
+\|\partial_zb\|^2_{L^4([0,T],X^{m-1})}
\nonumber\\[6pt]
&\leq\|\omega_{vh}|_{t=0}\|^2_{X^{m-1}}+\|\omega_{bh}|_{t=0}\|^2_{X^{m-1}}
+\int_0^t\|v\|^2_{X^{m-1,1}}+\|b\|^2_{X^{m-1,1}}+|h|^2_{X^{m-1,1}}dt
\nonumber\\[6pt]
&+\epsilon\int_0^t\|\partial_zv\|^2_{X^{m-1,1}}+\|\partial_zb\|^2_{X^{m-1,1}}\mathrm{d}t.
\end{align}
\end{lemma}

\begin{proof}
We have the equations for $\omega_{vh},\omega_{bh}$:
\begin{equation}
\left\{
\begin{aligned}
&\partial_t^\varphi\omega_{vh}-\epsilon\Delta^\varphi\omega_{vh}+v\cdot\nabla^\varphi\omega_{vh}
-b\cdot\nabla^\varphi\omega_{bh}=F_v^0[\nabla\varphi](\omega_{vh},\omega_{bh},\partial_jv^i,\partial_jb^i),
\\[6pt]
&\partial_t^\varphi\omega_{bh}-\epsilon\Delta^\varphi\omega_{bh}+v\cdot\nabla^\varphi\omega_{bh}
-b\cdot\nabla^\varphi\omega_{vh}
=F_b^0[\nabla\varphi](\omega_{vh},\omega_{bh},\partial_jv^i,\partial_jb^i),\\[6pt]
&\omega_{vh}|_{z=0}=F^{1,2}[\nabla\varphi](\partial_jv^i),\\[6pt]
&\omega_{bh}|_{z=0}=F^{1,2}[\nabla\varphi](\partial_jb^i),\\[6pt]
&(\omega_{vh},\omega_{bh})|_{t=0}=(\omega^1_{v0},\omega^2_{v0},\omega^1_{b0},\omega^2_{b0}),
\end{aligned}\right.
\end{equation}
where $j=1,2, i=1,2,3$.

Similarly to \cite{Lee17}, we decompose $\omega_{vh}=\omega_{vh}^{nh}+\omega_{vh}^h$,
$\omega_{bh}=\omega_{bh}^{nh}+\omega_{bh}^h$,
such that $\omega_{vh}^{nh},\omega_{bh}^{nh}$ solve
the nonhomogeneous equations with zero boundary condition:
\begin{equation}
\left\{
\begin{aligned}
&\partial_t^\varphi\omega_{vh}^{nh}-\epsilon\Delta^\varphi\omega_{vh}^{nh}
+v\cdot\nabla^\varphi\omega_{vh}^{nh}
-b\cdot\nabla^\varphi\omega_{bh}^{nh}
=F_v^0[\nabla\varphi](\omega_{vh},\omega_{bh},\partial_jv^i,\partial_jb^i),
\\[6pt]
&\partial_t^\varphi\omega_{bh}^{nh}-\epsilon\Delta^\varphi\omega_{bh}^{nh}
+v\cdot\nabla^\varphi\omega_{bh}^{nh}
-b\cdot\nabla^\varphi\omega_{vh}^{nh}
=F_b^0[\nabla\varphi](\omega_{vh},\omega_{bh},\partial_jv^i,\partial_jb^i),
\\[6pt]
&\omega_{vh}^{nh}|_{z=0}=0,
\\[6pt]
&\omega_{bh}^{nh}|_{z=0}=0,
\\[6pt]
&(\omega_{vh}^{nh},\omega_{bh}^{nh})|_{t=0}=(\omega^1_{v0},\omega^2_{v0},\omega^1_{b0},\omega^2_{b0})^\top.
\end{aligned}
\right.
\end{equation}

For $j=1, 2$, $i=1, 2, 3$, $\omega_{vh}^{h}$, $\omega_{bh}^{h}$ solve
the homogeneous equations with general boundary condition
as follows:
\begin{equation}
\left\{
\begin{aligned}
&\partial_t^\varphi\omega_{vh}^h-\epsilon\Delta^\varphi\omega_{vh}^h+v\cdot\nabla^\varphi\omega_{vh}^h
-b\cdot\nabla^\varphi\omega_{Bh}^h=0,
\\[6pt]
&\partial_t^\varphi\omega_{bh}^h-\epsilon\Delta^\varphi\omega_{bh}^h+v\cdot\nabla^\varphi\omega_{bh}^h
-b\cdot\nabla^\varphi\omega_{vh}^h=0,
\\[6pt]
&\omega_{vh}^h|_{z=0}=F^{1,2}[\nabla\varphi](\partial_jv^i),\\[6pt]
&\omega_{bh}^h|_{z=0}=F^{1,2}[\nabla\varphi](\partial_jb^i),\\[6pt]
&(\omega_{vh}^h,\omega_{bh}^h)|_{t=0}=0.
\end{aligned}
\right.
\end{equation}
where $F^{1,2}[\nabla\varphi](\partial_jb^i)=0$,
which is derived from the zero magnetic field.

The nonhomogeneous equations are equivalent to
\begin{equation}\label{2.30}
\left\{
\begin{aligned}
&\partial_t\omega_{vh}^{nh}-\epsilon\Delta^\varphi\omega_{vh}^{nh}
+v_y\cdot\nabla_y\omega_{vh}^{nh}+V_z\partial_z\omega_{vh}^{nh}
-b_y\cdot\nabla_y\omega_{bh}^{nh}-B_z\partial_z\omega_{bh}^{nh}\\
&\quad=F_v^0[\nabla\varphi](\omega_{vh},\omega_{bh},\partial_jv^i,\partial_jb^i),
\\[6pt]
&\partial_t\omega_{bh}^{nh}-\epsilon\Delta^\varphi\omega_{bh}^{nh}
+v_y\cdot\nabla_y\omega_{bh}^{nh}+V_z\partial_z\omega_{bh}^{nh}
-b_y\cdot\nabla_y\omega_{vh}^{nh}-B_z\partial_z\omega_{vh}^{nh}
\\[6pt]
&\quad=F_B^0[\nabla\varphi](\omega_{vh},\omega_{bh},\partial_jv^i,\partial_jb^i).
\end{aligned}
\right.
\end{equation}
where $V_z=\frac{1}{\partial_z\varphi}(v\cdot N-\partial_t\varphi)=\frac{1}{\partial_z\varphi}(v^3-\partial_t\eta-v_y\cdot\nabla_y\eta)$,
$B_z=\frac{1}{\partial_z\varphi}b\cdot N=\frac{1}{\partial_z\varphi}(b^3-b_y\cdot\nabla_y\eta)$.

Applying $\partial_t^lZ^\alpha$, where $l+|\alpha|\leq m-1$, to the equations \eqref{2.30},
we have
\begin{equation}
\left\{
\begin{aligned}
&\partial_t\partial_t^lZ^\alpha\omega_{vh}^{nh}-\epsilon\Delta^\varphi\partial_t^lZ^\alpha\omega_{vh}^{nh}
+v_y\cdot\nabla_y\partial_t^lZ^\alpha\omega_{vh}^{nh}+V_z\partial_z\partial_t^lZ^\alpha\omega_{vh}^{nh}
-b_y\cdot\nabla_y\partial_t^lZ^\alpha\omega_{bh}^{nh}-B_z\partial_z\partial_t^lZ^\alpha\omega_{bh}^{nh}
\\[6pt]
&\qquad=\partial_t^lZ^\alpha F_v^0[\nabla\varphi](\omega_{vh},\omega_{bh},\partial_jv^i,\partial_jb^i)
-[\partial_t^lZ^\alpha,v_y\cdot\nabla_y]\omega_{vh}^{nh}
-[\partial_t^lZ^\alpha,V_z\partial_z]\omega_{vh}^{nh}
\\[6pt]
&\qquad+[\partial_t^lZ^\alpha,b_y\cdot\nabla_y]\omega_{bh}^{nh}
+[\partial_t^lZ^\alpha,B_z\partial_z]\omega_{bh}^{nh}
+\epsilon\nabla^\varphi\cdot[\partial_t^lZ^\alpha,\nabla^\varphi]\omega_{vh}^{nh}
+\epsilon[\partial_t^lZ^\alpha,\nabla^\varphi\cdot]\omega_{vh}^{nh}
\\[6pt]
&\partial_t\partial_t^lZ^\alpha\omega_{bh}^{nh}-\epsilon\Delta^\varphi\partial_t^lZ^\alpha\omega_{bh}^{nh}
+v_y\cdot\nabla_y\partial_t^lZ^\alpha\omega_{bh}^{nh}+V_z\partial_z\partial_t^lZ^\alpha\omega_{bh}^{nh}
-b_y\cdot\nabla_y\partial_t^lZ^\alpha\omega_{vh}^{nh}-B_z\partial_z\partial_t^lZ^\alpha\omega_{vh}^{nh}
\\[6pt]
&\qquad=\partial_t^lZ^\alpha F_B^0[\nabla\varphi](\omega_{vh},\omega_{bh},\partial_jv^i,\partial_jb^i)
-[\partial_t^lZ^\alpha,v_y\cdot\nabla_y]\omega_{bh}^{nh}
-[\partial_t^lZ^\alpha,V_z\partial_z]\omega_{bh}^{nh}
\\[6pt]
&\qquad+[\partial_t^lZ^\alpha,b_y\cdot\nabla_y]\omega_{vh}^{nh}
+[\partial_t^lZ^\alpha,B_z\partial_z]\omega_{vh}^{nh}
+\epsilon\nabla^\varphi\cdot[\partial_t^lZ^\alpha,\nabla^\varphi]\omega_{bh}^{nh}
+\epsilon[\partial_t^lZ^\alpha,\nabla^\varphi\cdot]\omega_{bh}^{nh}
\\[6pt]
&\partial_t^lZ^\alpha\omega_{vh}^{nh}|_{z=0}=0,\\[6pt]
&\partial_t^lZ^\alpha\omega_{bh}^{nh}|_{z=0}=0,\\[6pt]
&(\partial_t^lZ^\alpha\omega_{vh}^{nh},\partial_t^lZ^\alpha\omega_{bh}^{nh})|_{t=0}
=(\partial_t^lZ^\alpha\omega^1_{v0},\partial_t^lZ^\alpha\omega^2_{v0},
\partial_t^lZ^\alpha\omega^1_{b0},\partial_t^lZ^\alpha\omega^2_{b0})^\top.
\end{aligned}
\right.
\end{equation}

For the $L^2$ estimate, one has
\begin{equation}\label{2.32}
\begin{aligned}
&\frac{d}{dt}(\|\partial_t^lZ^\alpha\omega_{vh}^{nh}\|^2_{L^2}
+\|\partial_t^lZ^\alpha\omega_{bh}^{nh}\|^2_{L^2})
+2\epsilon\|\nabla^\varphi\partial_t^lZ^\alpha\omega_{vh}^{nh}\|^2_{L^2}
+2\epsilon\|\nabla^\varphi\partial_t^lZ^\alpha\omega_{bh}^{nh}\|^2_{L^2}
\\[6pt]
\leq &\|\partial_t^lZ^\alpha\omega_{vh}\|^2_{L^2}
+\|\partial_t^lZ^\alpha\omega_{bh}\|^2_{L^2}
+\|\partial_t^lZ^\alpha\partial_jv^i\|^2_{L^2}+\|\partial_t^lZ^\alpha\partial_jb^i\|^2_{L^2}
+\|\partial_t^lZ^\alpha\nabla\varphi\|^2_{L^2}
\\[6pt]
&+\epsilon\int_{\mathbb{R}^3_-}\nabla^\varphi\cdot[\partial_t^lZ^\alpha,\nabla^\varphi]
\omega_{vh}^{nh}\cdot\partial_t^lZ^\alpha\omega_{vh}^{nh}\mathrm{d}\mathcal{V}_t
+\epsilon\int_{\mathbb{R}^3_-}[\partial_t^lZ^\alpha,\nabla^\varphi\cdot]\omega_{vh}^{nh}
\cdot\partial_t^lZ^\alpha\omega_{vh}^{nh}\mathrm{d}\mathcal{V}_t
\\
&+\epsilon\int_{\mathbb{R}^3_-}\nabla^\varphi\cdot[\partial_t^lZ^\alpha,\nabla^\varphi]
\omega_{bh}^{nh}\cdot\partial_t^lZ^\alpha\omega_{bh}^{nh}\mathrm{d}\mathcal{V}_t
+\epsilon\int_{\mathbb{R}^3_-}[\partial_t^lZ^\alpha,\nabla^\varphi\cdot]\omega_{bh}^{nh}
\cdot\partial_t^lZ^\alpha\omega_{bh}^{nh}\mathrm{d}\mathcal{V}_t
\\
&
+\|[\partial_t^lZ^\alpha,V_z\partial_z]\omega_{vh}^{nh}\|^2_{L^2}
+\|[\partial_t^lZ^\alpha,V_z\partial_z]\omega_{bh}^{nh}\|^2_{L^2}
+\|[\partial_t^lZ^\alpha,B_z\partial_z]\omega_{vh}^{nh}\|^2_{L^2}
\\
&
+\|[\partial_t^lZ^\alpha,B_z\partial_z]\omega_{bh}^{nh}\|^2_{L^2}
+{\text b.t.}.
\end{aligned}
\end{equation}

Now we estimate the last eight terms on the right hand of \eqref{2.32}. It follows from \cite{Wu16} that
\begin{equation}
\begin{aligned}
&\|[\partial_t^lZ^\alpha,V_z\partial_z]\omega_{vh}^{nh}\|^2_{L^2}\leq |h|^2_{X^{m-1,1}}
+\|\partial_zv\|^2_{X^{m-1,1}}+b.t.,
\\
&\epsilon\int_{\mathbb{R}^3_-}\nabla^\varphi\cdot[\partial_t^lZ^\alpha,\nabla^\varphi]
\omega_{vh}^{nh}\cdot\partial_t^lZ^\alpha\omega_{vh}^{nh}\mathrm{d}\mathcal{V}_t
\\
&\leq |h|^2_{X^{m-1,1}}+\epsilon\|\nabla^\varphi\partial_t^lZ^\alpha\omega_{vh}^{nh}\|^2_{L^2}
+\epsilon\|\nabla^\varphi\partial_t^lZ^\alpha\omega_{bh}^{nh}\|^2_{L^2}+{\text b.t.}.
\end{aligned}
\end{equation}
Similarly, one has
\begin{align*}
&\|[\partial_t^lZ^\alpha,V_z\partial_z]\omega_{bh}^{nh}\|^2_{L^2}
+\|[\partial_t^lZ^\alpha,B_z\partial_z]\omega_{vh}^{nh}\|^2_{L^2}
+\|[\partial_t^lZ^\alpha,B_z\partial_z]\omega_{bh}^{nh}\|^2_{L^2}
\\
&\leq |h|^2_{X^{m-1,1}}+\|\partial_zv\|^2_{X^{m-1,1}}+\|\partial_zb\|^2_{X^{m-1,1}}+{\text b.t.},
\end{align*}
and
\begin{align*}
&\epsilon\int_{R^3_-}[\partial_t^lZ^\alpha,\nabla^\varphi\cdot]\omega_{vh}^{nh}
\cdot\partial_t^lZ^\alpha\omega_{vh}^{nh}\mathrm{d}\mathcal{V}_t
+\epsilon\int_{R^3_-}\nabla^\varphi\cdot[\partial_t^lZ^\alpha,\nabla^\varphi]
\omega_{bh}^{nh}\cdot\partial_t^lZ^\alpha\omega_{bh}^{nh}\mathrm{d}\mathcal{V}_t
\\
&+\epsilon\int_{R^3_-}[\partial_t^lZ^\alpha,\nabla^\varphi\cdot]\omega_{bh}^{nh}
\cdot\partial_t^lZ^\alpha\omega_{bh}^{nh}\mathrm{d}\mathcal{V}_t
\\
&\leq |h|^2_{X^{m-1,1}}+\epsilon\|\nabla^\varphi\partial_t^lZ^\alpha\omega_{vh}^{nh}\|^2_{L^2}
+\epsilon\|\nabla^\varphi\partial_t^lZ^\alpha\omega_{bh}^{nh}\|^2_{L^2}+{\text b.t.}.
\end{align*}
Therefore, one gets
\begin{equation}\label{2.34}
\begin{aligned}
&\frac{\mathrm{d}}{\mathrm{d}t}(\|\partial_t^lZ^\alpha\omega_{vh}^{nh}\|^2_{L^2}
+\|\partial_t^lZ^\alpha\omega_{bh}^{nh}\|^2_{L^2})
+2\epsilon\|\nabla^\varphi\partial_t^lZ^\alpha\omega_{vh}^{nh}\|^2_{L^2}
+2\epsilon\|\nabla^\varphi\partial_t^lZ^\alpha\omega_{bh}^{nh}\|^2_{L^2}
\\[6pt]
\leq &\|\partial_t^lZ^\alpha\omega_{vh}\|^2_{L^2}
+\|\partial_t^lZ^\alpha\omega_{bh}\|^2_{L^2}
+|h|^2_{X^{m-1,1}}+\|\partial_zv\|^2_{X^{m-1,1}}+\|\partial_zb\|^2_{X^{m-1,1}}
\\[6pt]
&+\epsilon\|\nabla^\varphi\partial_t^lZ^\alpha\omega_{vh}^{nh}\|^2_{L^2}
+\epsilon\|\nabla^\varphi\partial_t^lZ^\alpha\omega_{bh}^{nh}\|^2_{L^2}+{\text b.t.}.
\end{aligned}
\end{equation}
Taking Summation for $l$ and $\alpha$ in \eqref{2.34}, and integrating from 0 to $t$, we get
\begin{align}
&\|\omega_{vh}^{nh}\|^2_{X^{m-1}}+\|\omega_{bh}^{nh}\|^2_{X^{m-1}}
+\epsilon\int_0^t\|\nabla\omega_{vh}^{nh}\|^2_{X^{m-1}}\mathrm{d}t
+\epsilon\int_0^t\|\nabla\omega_{bh}^{nh}\|^2_{X^{m-1}}\mathrm{d}t
\nonumber\\
\leq&\|\omega_{vh}^{nh}|_{t=0}\|^2_{X^{m-1}}+\|\omega_{bh}^{nh}|_{t=0}\|^2_{X^{m-1}}
+\int_0^t\|\omega_{vh}\|^2_{X^{m-1}}dt+\int_0^t\|\omega_{bh}\|^2_{X^{m-1}}\mathrm{d}t
\nonumber\\
&+\int_0^t\|v\|^2_{X^{m-1,1}}+\|b\|^2_{X^{m-1,1}}+\|h\|^2_{X^{m-1,1}}\mathrm{d}t
\nonumber\\
\leq&\|\omega_{vh}|_{t=0}\|^2_{X^{m-1}}+\|\omega_{bh}|_{t=0}\|^2_{X^{m-1}}
+\sqrt{t}\|\omega_{vh}\|^2_{L^4([0,T],X^{m-1})}+\sqrt{t}\|\omega_{bh}\|^2_{L^4([0,T],X^{m-1})}
\nonumber\\
&+\int_0^t\|v\|^2_{X^{m-1,1}}dt+\int_0^t\|b\|^2_{X^{m-1,1}}dt+\int_0^t\|h\|^2_{X^{m-1,1}}\mathrm{d}t.
\end{align}
It follows that
\begin{align*}
&\|\omega_{vh}^{nh}\|^4_{X^{m-1}}+\|\omega_{bh}^{nh}\|^4_{X^{m-1}}
\leq\|\omega_{vh}|_{t=0}\|^4_{X^{m-1}}+\|\omega_{bh}|_{t=0}\|^4_{X^{m-1}}
+T\|\omega_{vh}\|^4_{L^4([0,T],X^{m-1})}
\nonumber\\
&\qquad+T\|\omega_{bh}\|^4_{L^4([0,T],X^{m-1})}
+(\int_0^t\|v\|^2_{X^{m-1,1}}dt)^2+(\int_0^t\|b\|^2_{X^{m-1,1}}dt)^2+(\int_0^t\|h\|^2_{X^{m-1,1}}\mathrm{d}t)^2.
\end{align*}
Then
\begin{align*}
&\int_0^T\|\omega_{vh}^{nh}\|^4_{X^{m-1}}+\int_0^T\|\omega_{bh}^{nh}\|^4_{X^{m-1}}
\leq T\|\omega_{vh}|_{t=0}\|^4_{X^{m-1}}+T\|\omega_{bh}|_{t=0}\|^4_{X^{m-1}}
\nonumber\\
&\qquad+T\int_0^T\|\omega_{vh}\|^4_{L^4([0,T],X^{m-1})}
+T\int_0^T\|\omega_{bh}\|^4_{L^4([0,T],X^{m-1})}
\nonumber\\
&\qquad+T(\int_0^t\|v\|^2_{X^{m-1,1}}dt)^2+T(\int_0^t\|b\|^2_{X^{m-1,1}}\mathrm{d}t)^2
+T(\int_0^t\|h\|^2_{X^{m-1,1}}\mathrm{d}t)^2.
\end{align*}

For the homogeneous equations,
similarly to Lee \cite{Lee17}(see Theorem 10.5 in \cite{Lee17}),
we can obtain the $H^\frac{1}{4}([0,T],L^2)$ estimate when $l+|\alpha|\leq m-1$,
\begin{equation}\label{2.36}
\begin{aligned}
&\|\partial_t^lZ^\alpha\omega_{vh}^{h}\pm
\partial_t^lZ^\alpha\omega_{bh}^{h}\|^2_{H^\frac{1}{4}([0,T],L^2(\mathbb{R}^3_-))}
\\
\leq &\sqrt{\epsilon}\int_0^T|\partial_t^lZ^\alpha\omega_{vh}^{h}|_{z=0}|^2_{L^2(R^2)}\mathrm{d}t
+\sqrt{\epsilon}\int_0^T|\partial_t^lZ^\alpha\omega_{bh}^{h}|_{z=0}|^2_{L^2(R^2)}\mathrm{d}t
\\
\leq&\sqrt{\epsilon}\int_0^T|\partial_t^lZ^\alpha(F^{1,2}[\nabla\varphi](\partial_jv^i))
|_{z=0}|^2_{L^2(R^2)}\mathrm{d}t
+\sqrt{\epsilon}\int_0^T|\partial_t^lZ^\alpha F^{1,2}[\nabla\varphi](\partial_jb^i)
|_{z=0}|^2_{L^2(R^2)}\mathrm{d}t
\\
\leq&\sqrt{\epsilon}\int_0^T|h|^2_{X^{m-1,1}}\mathrm{d}t
+\sqrt{\epsilon}\int_0^T|v|_{z=0}|^2_{X^{m-1,1}_{tan}}\mathrm{d}t
+\sqrt{\epsilon}\int_0^T|b|_{z=0}|^2_{X^{m-1,1}_{tan}}\mathrm{d}t
\\
\leq&\sqrt{\epsilon}\int_0^T|h|^2_{X^{m-1,1}}\mathrm{d}t
+\int_0^T\|v\|^2_{X^{m-1,1}}dt
+\epsilon\int_0^T\|\partial_zv\|^2_{X^{m-1,1}}\mathrm{d}t,
\end{aligned}
\end{equation}
where $\partial_jv^i|_{z=0}|_{X^{m-1}}=|v|_{z=0}|_{H^m}$,
$\partial_jb^i|_{z=0}|_{X^{m-1}}=|b|_{z=0}|_{H^m}$ for j = 1, 2.
Taking summation for $\alpha$ in \eqref{2.36}, we have
\begin{align*}
&\|\omega_{vh}^{h}\pm\omega_{bh}^{h}\|^2_{L^4([0,T],X^{m-1})}
\leq\int_0^T|h|^2_{X^{m-1,1}}\mathrm{d}t
+\int_0^T\|v\|^2_{X^{m-1,1}}\mathrm{d}t
+\epsilon\int_0^T\|\partial_zv\|^2_{X^{m-1,1}}\mathrm{d}t.
\end{align*}
Hence, one has
\begin{align*}
\|\omega_{vh}^{h}\pm\omega_{bh}^{h}\|^4_{L^4([0,t],X^{m-1})}
\leq(\int_0^T|h|^2_{X^{m-1,1}}\mathrm{d}t)^2
+(\int_0^T\|v\|^2_{X^{m-1,1}}\mathrm{d}t)^2
+(\epsilon\int_0^T\|\partial_zv\|^2_{X^{m-1,1}}\mathrm{d}t)^2.
\end{align*}
Therefore
\begin{align*}
&\|\omega_{vh}\|^4_{L^4([0,t],X^{m-1})}+\|\omega_{bh}\|^4_{L^4([0,t],X^{m-1})}
\\
\leq&\|\omega_{vh}^{h}\|^4_{L^4([0,T],X^{m-1})}
+\|\omega_{vh}^{nh}\|^4_{L^4([0,T],X^{m-1})}
+\|\omega_{bh}^{h}\|^4_{L^4([0,T],X^{m-1})}
+\|\omega_{bh}^{nh}\|^4_{L^4([0,T],X^{m-1})}
\\
\leq&\|\omega_{vh}|_{t=0}\|^4_{X^{m-1}}
+\|\omega_{bh}|_{t=0}\|^4_{X^{m-1}}
+\int_0^T\|\omega_{vh}\|^4_{L^4([0,T],X^{m-1})}
+\int_0^T\|\omega_{bh}\|^4_{L^4([0,T],X^{m-1})}
\\
&+(\int_0^T|h|^2_{X^{m-1,1}}\mathrm{d}t)^2
+(\int_0^T\|v\|^2_{X^{m-1,1}}\mathrm{d}t)^2
+(\int_0^T\|b\|^2_{X^{m-1,1}}\mathrm{d}t)^2
\\
&+(\epsilon\int_0^T\|\partial_zv\|^2_{X^{m-1,1}}\mathrm{d}t)^2
+(\epsilon\int_0^T\|\partial_zb\|^2_{X^{m-1,1}}\mathrm{d}t)^2.
\end{align*}
It follows from the Gronwall's inequality that
\begin{align*}
&\|\omega_{vh}\|^2_{L^4([0,t],X^{m-1})}+\|\omega_{bh}\|^2_{L^4([0,t],X^{m-1})}
\\
\leq&\|\omega_{vh}|_{t=0}\|^2_{X^{m-1}}
+\|\omega_{bh}|_{t=0}\|^2_{X^{m-1}}
+\int_0^T|h|^2_{X^{m-1,1}}\mathrm{d}t
+\int_0^T\|v\|^2_{X^{m-1,1}}\mathrm{d}t
\\
&+\int_0^T\|b\|^2_{X^{m-1,1}}\mathrm{d}t
+\epsilon\int_0^T\|\partial_zv\|^2_{X^{m-1,1}}\mathrm{d}t
+\epsilon\int_0^T\|\partial_zb\|^2_{X^{m-1,1}}\mathrm{d}t.
\end{align*}
Then, we have
\begin{align*}
&\|\partial_zv_h\|^2_{L^4([0,T],X^{m-1})}+\|\partial_zb_h\|^2_{L^4([0,T],X^{m-1})}
\\
\leq&\|\omega_{vh}\|^4_{L^4([0,t],X^{m-1})}+\|\omega_{bh}\|^4_{L^4([0,t],X^{m-1})}
+|h|^2_{L^4([0,T],X^{m-1,\frac{1}{2}})}\\
&+\|v\|^2_{L^4([0,T],X^{m-1,1})}
+\|b\|^2_{L^4([0,T],X^{m-1,1})}.
\end{align*}
By the divergence free condition,
one has
\begin{align*}
&\|\partial_zv^3\|^2_{L^4([0,T],X^{m-1})}+\|\partial_zb^3\|^2_{L^4([0,T],X^{m-1})}
\\
\leq&\|\partial_zv_h\|^2_{L^4([0,T],X^{m-1})}+\|\partial_zb_h\|^2_{L^4([0,T],X^{m-1})}
+\|\nabla\varphi\|^2_{L^4([0,T],X^{m-1})}
\\
&+\|\partial_jv_i\|^2_{L^4([0,T],X^{m-1})}
+\|\partial_jb_i\|^2_{L^4([0,T],X^{m-1})}
\\
\leq&\|\omega_{vh}\|^4_{L^4([0,t],X^{m-1})}+\|\omega_{bh}\|^4_{L^4([0,t],X^{m-1})}
+|h|^2_{L^4([0,T],X^{m-1,\frac{1}{2}})}\\
&+\|v\|^2_{L^4([0,T],X^{m-1,1})}
+\|b\|^2_{L^4([0,T],X^{m-1,1})}.
\end{align*}
The lemma is proved.
\end{proof}

Refer to \cite{Lee17} for the $L^\infty$ estimates, which implies
$\partial_zv$, $\partial_zb$, $Z^3\partial_zv$, $Z^3\partial_zb$
$\sqrt{\epsilon}\partial_{zz}v$, $\sqrt{\epsilon}\partial_{zz}b\in L^\infty$.
The following lemma establishes the estimates of $\|\partial_zv\|_{L^\infty([0,T],X^{m-2})}$,
$\|\partial_zb\|_{L^\infty([0,T],X^{m-2})}$.
Note that we can not have $\partial_zv,\partial_zb\in L^\infty([0,T],X^{m-1})$,
due to the fact that $\omega_v|_{z=0}=F^{1,2}_v[\nabla\varphi](\partial_jv^i))$.

\begin{lemma}\label{Lemma2.7}
Assume $\omega_v$ and $\omega_b$ are the vorticity of the v and b for the free boundary problems of
MHD equations, respectively. $\omega_v$, $\omega_b$ satisfy the following estimate:
\begin{align}
&\|\partial_zv\|^2_{X^{m-2}}+\|\partial_zb\|^2_{X^{m-2}}
+\epsilon\int_0^t\|\partial_{zz}v\|^2_{X^{m-2}}+\|\partial_{zz}b\|^2_{X^{m-2}}\mathrm{d}t
\nonumber\\
&+\|\omega_v\|^2_{X^{m-2}}+\|\omega_b\|^2_{X^{m-2}}
+\epsilon\int_0^t\|\nabla\omega_v\|^2_{X^{m-2}}+\|\nabla\omega_b\|^2_{X^{m-2}}\mathrm{d}t
\nonumber\\
\leq&\|\partial_zv_0\|^2_{X^{m-2}}+\|\partial_zb_0\|^2_{X^{m-2}}
+\int_0^t\|v\|^2_{X^{m-1,1}}+\|b\|^2_{X^{m-1,1}}+|h|^2_{X^{m-1}}\mathrm{d}t
\nonumber\\
&\|\partial_zv\|^2_{L^4([0,T],X^{m-1})}+\|\partial_zb\|^2_{L^4([0,T],X^{m-1})}
\end{align}
\end{lemma}

\begin{proof}
By $\nabla^\varphi\cdot v=0$
and $\nabla^\varphi\cdot b=0$, we have
\begin{align}
&\Delta^\varphi v=\nabla^\varphi(\nabla^\varphi\cdot v)
-\nabla^\varphi\times(\nabla^\varphi\times v)=-\nabla^\varphi\times\omega_v,
\\
&\Delta^\varphi b=\nabla^\varphi(\nabla^\varphi\cdot b)
-\nabla^\varphi\times(\nabla^\varphi\times b)=-\nabla^\varphi\times\omega_b.
\end{align}

Firstly, we deduce the $L^2$ estimate of $\omega_v$ and $\omega_b$. We have
\begin{align*}
&\int_{\mathbb{R}^3_-}(\partial_{t}^\varphi v-\epsilon\Delta^\varphi v+v\cdot\nabla^\varphi v+\nabla^\varphi q
-b\cdot\nabla^\varphi b)\cdot\nabla^\varphi\times\omega_v\mathrm{d}\mathcal{V}_t=0,
\\
&\int_{\mathbb{R}^3_-}(\partial_{t}^\varphi b-\epsilon\Delta^\varphi b+v\cdot\nabla^\varphi b
-b\cdot\nabla^\varphi v)\cdot\nabla^\varphi\times\omega_b\mathrm{d}\mathcal{V}_t=0,
\end{align*}
then
\begin{align*}
&\int_{\mathbb{R}^3_-}(\partial_{t}^\varphi \omega_v+v\cdot\nabla^\varphi \omega_v
+\nabla^\varphi\times\nabla^\varphi q
-b\cdot\nabla^\varphi \omega_b)\cdot\omega_v \mathrm{d}\mathcal{V}_t
+\epsilon\int_{\mathbb{R}^3_-}|\nabla^\varphi\times\omega_v|^2\mathrm{d}\mathcal{V}_t
\\
=&-\int_{z=0}(\partial_{t}^\varphi v+v_y\cdot\nabla_y v+V_z\partial_zv
-b_y\cdot\nabla_y b+\nabla^\varphi q)
\cdot N\times\omega_v\mathrm{d}\mathcal{V}_t
\\
&-\int_{\mathbb{R}^3_-}\sum_{i=1}^3\nabla^\varphi v^i\partial_i^\varphi v\cdot\omega_vd\mathcal{V}_t
+\int_{\mathbb{R}^3_-}\sum_{i=1}^3\nabla^\varphi b^i\partial_i^\varphi b\cdot\omega_vd\mathcal{V}_t,
\\
&\int_{\mathbb{R}^3_-}(\partial_{t}^\varphi \omega_b+v\cdot\nabla^\varphi \omega_b-b\cdot\nabla^\varphi \omega_v)\cdot\omega_b)\mathrm{d}\mathcal{V}_t
+\epsilon\int_{\mathbb{R}^3_-}|\nabla^\varphi\times\omega_b|^2\mathrm{d}\mathcal{V}_t
\\
=&-\int_{z=0}(\partial_{t}^\varphi b+v_y\cdot\nabla_y b+V_z\partial_zb
-b_y\cdot\nabla_y b)
\cdot N\times\omega_b\mathrm{d}\mathcal{V}_t
\\
&-\int_{\mathbb{R}^3_-}\sum_{i=1}^3\nabla^\varphi v^i\partial_i^\varphi b\cdot\omega_b\mathrm{d}\mathcal{V}_t
+\int_{\mathbb{R}^3_-}\sum_{i=1}^3\nabla^\varphi b^i\partial_i^\varphi v\cdot\omega_b\mathrm{d}\mathcal{V}_t.
\end{align*}
Hence
\begin{align*}
&\|\omega_v\|^2_{L^2}+\|\omega_b\|^2_{L^2}
+\epsilon\int_0^t\|\nabla\omega_v\|^2_{L^2}+\|\nabla\omega_b\|^2_{L^2}\mathrm{d}t
\\
\leq&\|\omega_v|_{t=0}\|^2_{L^2}+\|\omega_b|_{t=0}\|^2_{L^2}+{\text b.t.}.
\end{align*}
Note that $\int_{z=0}\nabla^\varphi q\cdot N\times\omega_v\mathrm{d}\mathcal{V}_t
\leq |\nabla^\varphi q|_{t=0}|_{-\frac{1}{2}}+|N\times\omega_v|_{t=0}|_{\frac{1}{2}}={\text b.t.}.$

When $1\leq l+|\alpha|\leq m-2$, applying $\partial_t^l Z^\alpha$ to the equations\eqref{MHDF}, we have
\begin{equation}\label{3.47}
\left\{\begin{aligned}
&\partial_{t}^\varphi \partial_t^l Z^\alpha v+\epsilon(\nabla^\varphi\times)^2 \partial_t^l Z^\alpha v
+v\cdot\nabla^\varphi \partial_t^l Z^\alpha v
-b\cdot\nabla^\varphi \partial_t^l Z^\alpha b
+\nabla^\varphi \partial_t^l Z^\alpha q
=\epsilon I^v_{1,1}+I^v_{1,2},
\\
&\partial_{t}^\varphi \partial_t^l Z^\alpha b+\epsilon(\nabla^\varphi\times)^2 \partial_t^l Z^\alpha b+v\cdot\nabla^\varphi \partial_t^l Z^\alpha b
-b\cdot\nabla^\varphi \partial_t^l Z^\alpha v=\epsilon I^b_{1,1}+I^b_{1,2},
\end{aligned}\right.
\end{equation}
where
\begin{equation}
\left\{\begin{aligned}
&I^v_{1,1}=-[\partial_t^l Z^\alpha,\nabla^\varphi\times]\omega_v
-\nabla^\varphi\times[\partial_t^l Z^\alpha,\nabla^\varphi\times]v,\\
&I^v_{1,2}=-[\partial_t^l Z^\alpha,v_y\cdot\nabla_y+V_z\partial_z]v
+[\partial_t^l Z^\alpha,b_y\cdot\nabla_y+B_z\partial_z]b
-[\partial_t^l Z^\alpha,N\partial_z^\varphi]q,
\\
&I^b_{1,1}=-[\partial_t^l Z^\alpha,\nabla^\varphi\times]\omega_b
-\nabla^\varphi\times[\partial_t^l Z^\alpha,\nabla^\varphi\times]b,\\
&I^b_{1,2}=-[\partial_t^l Z^\alpha,v_y\cdot\nabla_y+V_z\partial_z]b
+[\partial_t^l Z^\alpha,b_y\cdot\nabla_y+B_z\partial_z]v.
\end{aligned}\right.
\end{equation}
Multiplying \eqref{3.47} with $\nabla^\varphi\times(\nabla^\varphi\times\partial_t^l Z^\alpha v)$,
$\nabla^\varphi\times(\nabla^\varphi\times \partial_t^l Z^\alpha b)$, respectively,
integrating in $R^3_-$ and note that $[\partial_{t}^\varphi,\nabla^\varphi]=0$, we have
\begin{align*}
&\int_{\mathbb{R}^3_-}\partial_{t}^\varphi|\nabla^\varphi\times\partial_t^l Z^\alpha v|^2\mathrm{d}\mathcal{V}_t
+\int_{\mathbb{R}^3_-}v\cdot\nabla^\varphi|\nabla^\varphi\times\partial_t^l Z^\alpha v|^2\mathrm{d}\mathcal{V}_t
\\
&+\int_{\mathbb{R}^3_-}(\nabla^\varphi\times\nabla^\varphi \partial_t^l Z^\alpha q)\cdot(\nabla^\varphi\times\partial_t^l Z^\alpha v)\mathrm{d}\mathcal{V}_t
+\epsilon\int_{\mathbb{R}^3_-}|\nabla^\varphi\times(\nabla^\varphi\times\partial_t^l Z^\alpha v)|^2\mathrm{d}\mathcal{V}_t
\\
=&\int_{\mathbb{R}^3_-}b\cdot\nabla^\varphi\nabla^\varphi\times\partial_t^l Z^\alpha b\cdot\nabla^\varphi\times\partial_t^l Z^\alpha v\mathrm{d}\mathcal{V}_t
-\int_{\mathbb{R}^3_-}[(\mathop{\Sigma}\limits_{i=1}^3\nabla^\varphi v^i\cdot\partial_i^\varphi)\partial_t^l Z^\alpha v]\cdot(\nabla^\varphi\times\partial_t^l Z^\alpha v)\mathrm{d}\mathcal{V}_t
\\
&+\int_{\mathbb{R}^3_-}[(\mathop{\Sigma}\limits_{i=1}^3\nabla^\varphi b^i\cdot\partial_i^\varphi)\partial_t^l Z^\alpha b]\cdot(\nabla^\varphi\times\partial_t^l Z^\alpha v)\mathrm{d}\mathcal{V}_t
+\int_{z=0}I^v_{1,2}\cdot N\times(\nabla^\varphi\times\partial_t^l Z^\alpha v)\mathrm{d}y
\\
&-\int_{z=0}(\partial_{t}^\varphi \partial_t^l Z^\alpha v
+v\cdot\nabla^\varphi \partial_t^l Z^\alpha v
-b\cdot\nabla^\varphi \partial_t^l Z^\alpha b
+\nabla^\varphi \partial_t^l Z^\alpha q)\cdot N\times(\nabla^\varphi\times\partial_t^l Z^\alpha v)\mathrm{d}y
\\
&+\int_{\mathbb{R}^3_-}\epsilon I^v_{1,1}\nabla^\varphi\times(\nabla^\varphi\times\partial_t^l Z^\alpha v)\mathrm{d}\mathcal{V}_t
+\int_{\mathbb{R}^3_-}\nabla^\varphi\times I^v_{1,2}(\nabla^\varphi\times\partial_t^l Z^\alpha v)\mathrm{d}\mathcal{V}_t,
\end{align*}
and
\begin{align*}
&\int_{\mathbb{R}^3_-}\partial_{t}^\varphi|\nabla^\varphi\times\partial_t^l Z^\alpha b|^2\mathrm{d}\mathcal{V}_t
+\int_{\mathbb{R}^3_-}v\cdot\nabla^\varphi|\nabla^\varphi\times\partial_t^l Z^\alpha b|^2\mathrm{d}\mathcal{V}_t
+\epsilon\int_{\mathbb{R}^3_-}|\nabla^\varphi\times(\nabla^\varphi\times\partial_t^l Z^\alpha b)|^2\mathrm{d}\mathcal{V}_t
\\
=&\int_{\mathbb{R}^3_-}b\cdot\nabla^\varphi\nabla^\varphi\times\partial_t^l Z^\alpha v\cdot\nabla^\varphi\times\partial_t^l Z^\alpha b\mathrm{d}\mathcal{V}_t
-\int_{\mathbb{R}^3_-}[(\mathop{\Sigma}\limits_{i=1}^3\nabla^\varphi v^i\cdot\partial_i^\varphi)\partial_t^l Z^\alpha b]\cdot(\nabla^\varphi\times\partial_t^l Z^\alpha b)\mathrm{d}\mathcal{V}_t
\\
&+\int_{\mathbb{R}^3_-}[(\mathop{\Sigma}\limits_{i=1}^3\nabla^\varphi b^i\cdot\partial_i^\varphi)\partial_t^l Z^\alpha v]\cdot(\nabla^\varphi\times\partial_t^l Z^\alpha b)\mathrm{d}\mathcal{V}_t
+\int_{z=0}I^b_{1,2}\cdot N\times(\nabla^\varphi\times\partial_t^l Z^\alpha b)\mathrm{d}y
\\
&-\int_{z=0}(\partial_{t}^\varphi \partial_t^l Z^\alpha b
+v\cdot\nabla^\varphi \partial_t^l Z^\alpha b
-b\cdot\nabla^\varphi \partial_t^l Z^\alpha v)\cdot N\times(\nabla^\varphi\times\partial_t^l Z^\alpha b)\mathrm{d}y
\\
&+\int_{\mathbb{R}^3_-}\epsilon I^b_{1,1}\nabla^\varphi\times(\nabla^\varphi\times\partial_t^l Z^\alpha b)\mathrm{d}\mathcal{V}_t
+\int_{\mathbb{R}^3_-}\nabla^\varphi\times I^b_{1,2}(\nabla^\varphi\times\partial_t^l Z^\alpha b)\mathrm{d}\mathcal{V}_t.
\end{align*}
It follow from $\nabla^\varphi\times\nabla^\varphi \partial_t^l Z^\alpha q=0$ that
\begin{align*}
&\frac{d}{dt}\int_{\mathbb{R}^3_-}|\nabla^\varphi\times\partial_t^l Z^\alpha v|^2
+|\nabla^\varphi\times\partial_t^l Z^\alpha b|^2\mathrm{d}\mathcal{V}_t
+2\epsilon\int_{\mathbb{R}^3_-}|(\nabla^\varphi\times)^2\partial_t^l Z^\alpha v|^2
+|(\nabla^\varphi\times)^2\partial_t^l Z^\alpha b|^2\mathrm{d}\mathcal{V}_t
\\
&\leq \|\nabla^\varphi\times\partial_t^l Z^\alpha v\|^2_{L^2}
+\|\nabla^\varphi\times\partial_t^l Z^\alpha b\|^2_{L^2}
+\epsilon\|(\nabla^\varphi\times)^2\partial_t^l Z^\alpha v\|^2_{L^2}
+\epsilon\|(\nabla^\varphi\times)^2\partial_t^l Z^\alpha b\|^2_{L^2}
\\
&+|v|_{z=0}|^2_{X_{tan}^{m-1}}+|b|_{z=0}|^2_{X_{tan}^{m-1}}
+|\nabla^\varphi\times\partial_t^l Z^\alpha v|_{z=0}|^2_{\frac{1}{2}}
+|\nabla^\varphi\times\partial_t^l Z^\alpha b|_{z=0}|^2_{\frac{1}{2}}
+|\nabla^\varphi\partial_t^l Z^\alpha q|_{z=0}|^2_{-\frac{1}{2}}
\\
&+\|\nabla\partial_t^l Z^\alpha v\|^2_{L^2}+\|\nabla\partial_t^l Z^\alpha b\|^2_{L^2}
+\mathop{\Sigma}\limits_{l_1+|\alpha|_1\leq m-3}|\nabla^\varphi\partial_t^{l_1} Z^{\alpha_1} q|_{z=0}|^2_{-\frac{1}{2}}+\epsilon(\|I^v_{1,1}\|^2_{L^2}+\|I^b_{1,1}\|^2_{L^2})
\\
&+\|\nabla^\varphi\times I^v_{1,2}\|^2_{L^2}+\|\nabla^\varphi\times I^b_{1,2}\|^2_{L^2}.
\end{align*}

It is easy to prove that $\|I^v_{1,1}\|^2_{L^2}\leq \|\omega_v\|_{X^{m-2}}$
and $\|I^b_{1,1}\|^2_{L^2}\leq \|\omega_b\|_{X^{m-2}}$.
Next we estimate $\nabla^\varphi\times I^v_{1,2}$ and $\nabla^\varphi\times I^v_{1,2}$,
by normal co-norm estimate. We have
\begin{align}\label{Hady}
&\|\nabla^\varphi\times([\partial_t^l Z^\alpha,V_z\partial_z]v
+[\partial_t^l Z^\alpha,B_z\partial_z]b)\|_{L^2}
+\|\nabla^\varphi\times([\partial_t^l Z^\alpha,V_z\partial_z]b
+[\partial_t^l Z^\alpha,B_z\partial_z]v)\|_{L^2}
\nonumber\\
\leq&\mathop{\Sigma}\limits_{l^1+|\alpha^1|>0}(\|\frac{1-z}{z}\partial_t^{l^1}Z^{\alpha^1}V_z\cdot
\frac{z}{1-z}\nabla^\varphi\times\partial_t^{l^2}Z^{\alpha^2}\partial_zv\|_{L^2}
+\|\nabla^\varphi\partial_t^{l^1}Z^{\alpha^1}V_z\times\partial_t^{l^2}Z^{\alpha^2}\partial_zv\|_{L^2})
\nonumber\\
&+\mathop{\Sigma}\limits_{l^1+|\alpha^1|>0}(\|\frac{1-z}{z}\partial_t^{l^1}Z^{\alpha^1}B_z\cdot
\frac{z}{1-z}\nabla^\varphi\times\partial_t^{l^2}Z^{\alpha^2}\partial_zb\|_{L^2}
+\|\nabla^\varphi\partial_t^{l^1}Z^{\alpha^1}B_z\times\partial_t^{l^2}Z^{\alpha^2}\partial_zb\|_{L^2})
\nonumber\\
&\mathop{\Sigma}\limits_{l^1+|\alpha^1|>0}(\|\frac{1-z}{z}\partial_t^{l^1}Z^{\alpha^1}V_z\cdot
\frac{z}{1-z}\nabla^\varphi\times\partial_t^{l^2}Z^{\alpha^2}\partial_zb\|_{L^2}
+\|\nabla^\varphi\partial_t^{l^1}Z^{\alpha^1}V_z\times\partial_t^{l^2}Z^{\alpha^2}\partial_zb\|_{L^2})
\nonumber\\
&+\mathop{\Sigma}\limits_{l^1+|\alpha^1|>0}(\|\frac{1-z}{z}\partial_t^{l^1}Z^{\alpha^1}B_z\cdot
\frac{z}{1-z}\nabla^\varphi\times\partial_t^{l^2}Z^{\alpha^2}\partial_zv\|_{L^2}
+\|\nabla^\varphi\partial_t^{l^1}Z^{\alpha^1}B_z\times\partial_t^{l^2}Z^{\alpha^2}\partial_zv\|_{L^2})
\nonumber\\
\leq&\mathop{\Sigma}\limits_{l^1+|\alpha^1|>0}(\|\nabla^\varphi\partial_t^{l^1}Z^{\alpha^1}V_z\|_{L^2}
+\|\nabla^\varphi\partial_t^{l^1}Z^{\alpha^1}B_z\|_{L^2}
+\|\nabla^\varphi\times\partial_t^{l^2}Z^{\alpha^2}Z^3v\|_{L^2}
\nonumber\\
&+\|\nabla^\varphi\times\partial_t^{l^2}Z^{\alpha^2}Z^3b\|_{L^2}
+\|\partial_t^{l^2}Z^{\alpha^2}\partial_zv\|_{L^2}
+\|\partial_t^{l^2}Z^{\alpha^2}\partial_zb\|_{L^2})
\nonumber\\
\leq&\|\omega_v\|_{X^{m-2}}+\|\omega_b\|_{X^{m-2}}+\|v\|_{X^{m-1}}+\|b\|_{X^{m-1}}+|h|_{X^{m-1}},
\end{align}
and
\begin{align*}
\|\nabla^\varphi\times([\partial_t^l Z^\alpha,N\partial_z^\varphi]q)\|_{L^2}
\leq \|\omega_v\|_{X^{m-2}}+\|v\|_{X^{m-2}}+\|\nabla q\|_{X^{m-1}}.
\end{align*}

Applying the Gronwall's inequality, we have the estimate of the vorticity:
\begin{align*}
&\|\nabla^\varphi\times\partial_t^l Z^\alpha v\|^2_{L^2}
+\|\nabla^\varphi\times\partial_t^l Z^\alpha b\|^2_{L^2}
+\epsilon\int_0^t\|(\nabla^\varphi\times)^2\partial_t^l Z^\alpha v\|^2_{L^2}
+\|(\nabla^\varphi\times)^2\partial_t^l Z^\alpha b\|^2_{L^2}\mathrm{d}t
\\
\leq&\|\nabla^\varphi\times\partial_t^l Z^\alpha v|_{t=0}\|^2_{L^2}
+\|\nabla^\varphi\times\partial_t^l Z^\alpha b|_{t=0}\|^2_{L^2}
+\int_0^t\|\omega_v\|_{X^{m-2}}+\|\omega_b\|_{X^{m-2}}+\|v\|_{X^{m-1}}
\\
&+\|b\|_{X^{m-1}}+|h|_{X^{m-1}}
+\|\partial_zv\|^2_{X^{m-1}}+\|\partial_zb\|^2_{X^{m-1}}+\|\nabla q\|^2_{X^{m-1}}\mathrm{d}t+{\text b.t.}
\\
\leq&\|\nabla^\varphi\times\partial_t^l Z^\alpha v|_{t=0}\|^2_{L^2}
+\|\nabla^\varphi\times\partial_t^l Z^\alpha b|_{t=0}\|^2_{L^2}
+\int_0^t+\|v\|_{X^{m-1}}
\\
&+\|b\|_{X^{m-1}}+|h|_{X^{m-1}}
+\|\partial_zv\|^2_{X^{m-1}}+\|\partial_zb\|^2_{X^{m-1}}+\|\nabla q\|^2_{X^{m-1}}\mathrm{d}t+{\text b.t.}.
\end{align*}
$\partial_t^{\ell}\mathcal{Z}^{\alpha} \omega_v$, $\partial_t^{\ell}\mathcal{Z}^{\alpha} \omega_b$
are equivalent to $\nabla^{\varphi} \times \partial_t^{\ell}\mathcal{Z}^{\alpha} v$,
 $\nabla^{\varphi} \times \partial_t^{\ell}\mathcal{Z}^{\alpha} b$, due to
$\ell +|\alpha|\leq m-2$ and
\begin{equation}\label{Sect2_NormalDer_Estimate_2_Formula}
\begin{array}{ll}
\partial_t^{\ell}\mathcal{Z}^{\alpha} \omega_v - \partial_t^{\ell}\mathcal{Z}^{\alpha} (\nabla^{\varphi} \times v)
= \sum\limits_{\ell_1 +|\alpha_1| >0} \partial_t^{\ell_1}\mathcal{Z}^{\alpha_1} (\frac{\NN}{\partial_z\varphi})\partial_z
\times \partial_t^{\ell_2}\mathcal{Z}^{\alpha_2} v,
\\[14pt]

\partial_t^{\ell}\mathcal{Z}^{\alpha} \omega_b - \partial_t^{\ell}\mathcal{Z}^{\alpha} (\nabla^{\varphi} \times b)
= \sum\limits_{\ell_1 +|\alpha_1| >0} \partial_t^{\ell_1}\mathcal{Z}^{\alpha_1} (\frac{\NN}{\partial_z\varphi})\partial_z
\times \partial_t^{\ell_2}\mathcal{Z}^{\alpha_2} b,
\\[14pt]

\|\partial_t^{\ell}\mathcal{Z}^{\alpha} \omega_v - \partial_t^{\ell}\mathcal{Z}^{\alpha} (\nabla^{\varphi} \times v)\|_{L^2}
\lem \|\partial_z v\|_{X^{m-3}} + |h|_{X^{m-2,\frac{1}{2}}},
\\[14pt]

\|\partial_t^{\ell}\mathcal{Z}^{\alpha} \omega_b - \partial_t^{\ell}\mathcal{Z}^{\alpha} (\nabla^{\varphi} \times b)\|_{L^2}
\lem \|\partial_z b\|_{X^{m-3}} + |h|_{X^{m-2,\frac{1}{2}}}.
\end{array}
\end{equation}
Then we have the estimate of the vorticity:
\begin{equation}\label{Sect2_NormalDer_Estimate_7}
\begin{array}{ll}
\|\omega_v\|_{X^{m-2}}^2 +\|\omega_b\|_{X^{m-2}}^2
+ \e\int\limits_0^t\|\nabla\omega_v\|_{X^{m-2}}^2
+ \e\int\limits_0^t\|\nabla\omega_b\|_{X^{m-2}}^2
\mathrm{d}\mathcal{V}_t\mathrm{d}t \\[6pt]

\lem \|\omega_{v0} \|_{X^{m-2}}^2 +\|\omega_{b0} \|_{X^{m-2}}^2
+ \int\limits_0^t \|v\|_{X^{m-1,1}}^2 +\|b\|_{X^{m-1,1}}^2 + |h|_{X^{m-1}}^2
\\[6pt]\quad

+ \|\partial_z v\|_{X^{m-1}}^2+ \|\partial_z b\|_{X^{m-1}}^2
+ \|\nabla q\|_{X^{m-1}}^2\,\mathrm{d}t.
\end{array}
\end{equation}

Lemma \ref{Lemma2.7} is proved.
\end{proof}
By the estimates proved in Lemmas \ref{Lemma2.5}, \ref{Lemma2.6}, \ref{Lemma2.7}, it is standard to prove
Proposition \ref{Proposition1.1}.

\section{Strong Initial Vorticity Layer}
In this section, we establish the $L^{\infty}$ estimate of strong vorticity layer
under the assumption that the boundary value of ideal MHD equations satisfies
$\Pi S^\varphi v\nn|_{z=0}=0$ and
$\Pi S^\varphi b\nn|_{z=0}=0$ in $(0,T]$.
It shows that the strong initial vorticity layer is one of sufficient
conditions for the existence of strong vorticity layer
of the free boundary problems for MHD equations\eqref{MHDF}.

\subsection{Lagrangian Coordinates Maps}
In this subsection, we derive the evolution equations of
$\hat{\omega}_{vh}=\omega^\epsilon_{vh}-\omega_{vh}$ and
$\hat{\omega}_{bh}=\omega^\epsilon_{bh}-\omega_{bh}$,
and construct a variable which satisfies the heat equations with damping.

$\hat{\omega}_{vh}$ and $\hat{\omega}_{bh}$ satisfy the equations \eqref{vorticity equations}. Due to the symmetry of the coefficient,
one has
\begin{align}
&F_v^0[\nabla\varphi^\epsilon](\omega_{vh}^\epsilon,\omega_{bh}^\epsilon,\partial_jv^{\epsilon,i},\partial_jb^{\epsilon,i})
-F_v^0[\nabla\varphi](\omega_{vh},\omega_{bh},\partial_jv^i,\partial_jb^i)
\nonumber\\
=&f_v^7[\omega_{vh}^\epsilon,\omega_{bh}^\epsilon,\nabla\varphi^\epsilon,\partial_jv^{\epsilon,i},\partial_jb^{\epsilon,i},\nabla\varphi,
\partial_jv^i,\partial_jb^i,\omega_{vh},\omega_{bh}]\hat{\omega}_{vh}
\nonumber\\
&+f_b^7[\omega_{vh}^\epsilon,\omega_{bh}^\epsilon,\nabla\varphi^\epsilon,\partial_jv^{\epsilon,i},\partial_jb^{\epsilon,i},\nabla\varphi,
\partial_jv^i,\partial_jb^i,\omega_{vh},\omega_{bh}]\hat{\omega}_{bh}
\nonumber\\
&+f_v^8[\omega_{vh}^\epsilon,\omega_{bh}^\epsilon,\nabla\varphi^\epsilon,\partial_jv^{\epsilon,i},\partial_jb^{\epsilon,i},\nabla\varphi,
\partial_jv^i,\partial_jb^i,\omega_{vh},\omega_{bh}]\partial_j\hat{v}^i
\nonumber\\
&+f_b^8[\omega_{vh}^\epsilon,\omega_{bh}^\epsilon,\nabla\varphi^\epsilon,\partial_jv^{\epsilon,i},\partial_jb^{\epsilon,i},\nabla\varphi,
\partial_jv^i,\partial_jb^i,\omega_{vh},\omega_{bh}]\partial_j\hat{b}^i
\nonumber\\
&+f_v^9[\omega_{vh}^\epsilon,\omega_{bh}^\epsilon,\nabla\varphi^\epsilon,\partial_jv^{\epsilon,i},\partial_jb^{\epsilon,i},\nabla\varphi,
\partial_jv^i,\partial_jb^i,\omega_{vh},\omega_{bh}]\nabla\hat{\varphi},
\end{align}
and
\begin{align}
&F_b^0[\nabla\varphi^\epsilon](\omega_{vh}^\epsilon,\omega_{bh}^\epsilon,\partial_jv^{\epsilon,i},\partial_jb^{\epsilon,i})
-F_b^0[\nabla\varphi](\omega_{vh},\omega_{bh},\partial_jv^i,\partial_jb^i)
\nonumber\\
=&-f_b^7[\omega_{vh}^\epsilon,\omega_{bh}^\epsilon,\nabla\varphi^\epsilon,\partial_jv^{\epsilon,i},\partial_jb^{\epsilon,i},\nabla\varphi,
\partial_jv^i,\partial_jb^i,\omega_{vh},\omega_{bh}]\hat{\omega}_{vh}
\nonumber\\
&-f_v^7[\omega_{vh}^\epsilon,\omega_{bh}^\epsilon,\nabla\varphi^\epsilon,\partial_jv^{\epsilon,i},\partial_jb^{\epsilon,i},\nabla\varphi,
\partial_jv^i,\partial_jb^i,\omega_{vh},\omega_{bh}]\hat{\omega}_{bh}
\nonumber\\
&-f_b^8[\omega_{vh}^\epsilon,\omega_{bh}^\epsilon,\nabla\varphi^\epsilon,\partial_jv^{\epsilon,i},\partial_jb^{\epsilon,i},\nabla\varphi,
\partial_jv^i,\partial_jb^i,\omega_{vh},\omega_{bh}]\partial_j\hat{v}^i
\nonumber\\
&-f_v^8[\omega_{vh}^\epsilon,\omega_{bh}^\epsilon,\nabla\varphi^\epsilon,\partial_jv^{\epsilon,i},\partial_jb^{\epsilon,i},\nabla\varphi,
\partial_jv^i,\partial_jb^i,\omega_{vh},\omega_{bh}]\partial_j\hat{b}^i
\nonumber\\
&+f_v^9[\omega_{vh}^\epsilon,\omega_{bh}^\epsilon,\nabla\varphi^\epsilon,\partial_jv^{\epsilon,i},\partial_jb^{\epsilon,i},\nabla\varphi,
\partial_jv^i,\partial_jb^i,\omega_{vh},\omega_{bh}]\nabla\hat{\varphi},
\end{align}
where the coefficients $f_{v}^7[\cdots]$, $f_{v}^8[\cdots]$,
$f_{v}^9[\cdots]$,$f_{b}^7[\cdots]$, $f_{b}^8[\cdots]$ and
$f_{b}^9[\cdots]$ are uniformly bounded with respect
to $\epsilon$. For simplicity, we denote
$f_{v}^i[\omega_{vh}^\epsilon,\omega_{bh}^\epsilon,\nabla\varphi^\epsilon,\partial_jv^{\epsilon,i},\partial_jb^{\epsilon,i},\nabla\varphi,
\partial_jv^i,\partial_jb^i,\omega_{vh},\omega_{bh}]$
and $f_{b}^i[\omega_{vh}^\epsilon,\omega_{bh}^\epsilon,\nabla\varphi^\epsilon,\partial_jv^{\epsilon,i},\partial_jb^{\epsilon,i},\nabla\varphi,
\partial_jv^i,\partial_jb^i,\omega_{vh},\omega_{bh}]$
 as $f_{v}^i$ and $f_{b}^i$, $i=7,8,9$.
 Then we obtain the following system of $\hat{\omega}_{vh}$ and $\hat{\omega}_{bh}$:
\begin{equation}\label{Vorticity Layer1}
\left\{
\begin{aligned}
&\partial_t^{\varphi^\epsilon}\hat{\omega}_{vh}-\epsilon\Delta^{\varphi^\epsilon}\hat{\omega}_{vh}
+v^\epsilon\cdot\nabla^{\varphi^\epsilon}\hat{\omega}_{vh}
-b^\epsilon\cdot\nabla^{\varphi^\epsilon}\hat{\omega}_{bh}
-f_v^7\hat{\omega}_{vh}-f_b^7\hat{\omega}_{bh}
\\
&\quad=f_v^8\partial_j\hat{v}^i+f_b^8\partial_j\hat{b}^i
+f_v^9(\nabla\hat{\varphi})
+\epsilon\Delta^{\varphi^\epsilon}\omega_{vh}
+\partial_z^\varphi\omega_{vh}\partial_t^{\varphi^\epsilon}\hat{\eta}
\\
&\quad
+\partial_z^\varphi\omega_{vh}v\cdot\nabla^{\varphi^\epsilon}\hat{\eta}
-\hat{v}\cdot\nabla^\varphi\omega_{vh}
-\partial_z^\varphi\omega_{bh}b\cdot\nabla^{\varphi^\epsilon}\hat{\eta}
+\hat{b}\cdot\nabla^\varphi\omega_{bh},
\\
&\partial_t^{\varphi^\epsilon}\hat{\omega}_{bh}-\epsilon\Delta^{\varphi^\epsilon}\hat{\omega}_{bh}
+v^\epsilon\cdot\nabla^{\varphi^\epsilon}\hat{\omega}_{bh}
-b^\epsilon\cdot\nabla^{\varphi^\epsilon}\hat{\omega}_{vh}
+f_b^7\hat{\omega}_{vh}+f_v^7\hat{\omega}_{bh}
\\
&\quad=-f_b^8\partial_j\hat{v}^i-f_v^8\partial_j\hat{b}^i
+f_b^9(\nabla\hat{\varphi})
+\epsilon\Delta^{\varphi^\epsilon}\omega_{bh}
+\partial_z^\varphi\omega_{bh}\partial_t^{\varphi^\epsilon}\hat{\eta}
\\
&\quad
+\partial_z^\varphi\omega_{bh}v\cdot\nabla^{\varphi^\epsilon}\hat{\eta}
-\hat{v}\cdot\nabla^\varphi\omega_{bh}
-\partial_z^\varphi\omega_{vh}b\cdot\nabla^{\varphi^\epsilon}\hat{\eta}
+\hat{b}\cdot\nabla^\varphi\omega_{vh},
\\
&\hat{\omega}_{vh}|_{z=0}=F^{1,2}_v[\nabla\varphi^\epsilon](\partial_jv^{\epsilon,i})
-\omega_{vh}|_{z=0}:=\hat{\omega}_{vh}^b,
\\
&\hat{\omega}_{bh}|_{z=0}=F^{1,2}_b[\nabla\varphi](\partial_jb^i)
-\omega_{bh}|_{z=0}:=\hat{\omega}_{bh}^b,
\\
&(\hat{\omega}_{vh}|_{t=0},\hat{\omega}_{bh}|_{t=0})=(\hat{\omega}_{vh,0},\hat{\omega}_{bh,0}).
\end{aligned}
\right.
\end{equation}

Similarly to Lee \cite{Lee17}, we eliminate the convection term and loretz force by using
the following two Lagrangian maps
$Y_i(i=1,2)$:
\begin{align}
Y_1:\Omega\to R^3, \partial_tY_1(t,x)=(u-b)(t,Y_1(t,x)), Y_1(0,x)=x,\\
Y_2:\Omega\to R^3, \partial_tY_2(t,x)=(u+b)(t,Y_2(t,x)), Y_2(0,x)=x.
\end{align}
Note that images $Y_1(t,\Omega)$ and $Y_2(t,\Omega)$ are defined only by boundary values of vector
fields $(u\pm b)^b$.
If we write boundary graphs as $h_1$ and $h_2$, due to $b=0$ on the boundary, then
\begin{align}
h_{1,2}(t):=h_{1,2}(0)+\int_0^t(u\pm b)^b\cdot N=h_{1,2}(0)+\int_0^tu^b\cdot N=h(t).
\end{align}

Define the Jacobian of the change of variables $J^i(t,x)=|\mathrm{det}\nabla Y_i(t,x)|$,
$a_0=|J_0^1(x)|^{\frac{1}{2}}$, $b_0=|J_0^2(x)|^{\frac{1}{2}}$,
and the matrix $(a_{i,j})=|J^1_0|^{\frac{1}{2}}P^{-1}$,
where the matrix P satisfies $P_{i,j}=\partial_iY_1\cdot\partial_jY_1$.
Similarly, we can define $(b_{i,j})=|J_0^2|^{\frac{1}{2}}P^{-1}$,
where the matrix P satisfies $P_{i,j}=\partial_iY_2\cdot\partial_jY_2$.

Define $W_\pm =e^{-\gamma t}(\hat{\omega}_{vh}\pm\hat{\omega}_{bh})(t,\Phi^{-1}\circ Y_i)$,
then $W_\pm$ satisfy the equations:
\begin{align}
&a_0\partial_tW_+-\epsilon\partial_i(a_{i,j}\partial_jW_+)+\gamma a_0W_+-(f^7_v-f^7_b)W_-
=\epsilon e^{-\gamma t}((f_v^8-f_b^8)(\partial_j\hat{v}^i-\partial_j\hat{b}^i)
\nonumber\\[3pt]
&+(f_v^9+f_b^9)(\nabla\hat{\varphi})
+\epsilon\Delta^{\varphi^\epsilon}(\omega_{vh}+\omega_{bh})
+\partial_z^\varphi(\omega_{vh}+\omega_{bh})\partial_t^{\varphi^\epsilon}\hat{\eta}
+\partial_z^\varphi(\omega_{vh}+\omega_{bh})v\cdot\nabla^{\varphi^\epsilon}\hat{\eta}
\nonumber\\[3pt]
&-\hat{v}\cdot\nabla^\varphi(\omega_{vh}+\omega_{bh})
-\partial_z^\varphi(\omega_{vh}+\omega_{bh})b\cdot\nabla^{\varphi^\epsilon}\hat{\eta}
+\hat{b}\cdot\nabla^\varphi(\omega_{vh}+\omega_{bh})):=I_{v2},
\\[3pt]
&b_0\partial_tW_--\epsilon\partial_i(b_{i,j}\partial_jW_-)+\gamma b_0W_--(f^7_v+f^7_b))W_+
=\epsilon e^{-\gamma t}((f_v^8+f_b^8)(\partial_j\hat{v}^i+\partial_j\hat{b}^i)
\nonumber\\[3pt]
&+(f_v^9-f_b^9)(\nabla\hat{\varphi})
+\epsilon\Delta^{\varphi^\epsilon}(\omega_{vh}-\omega_{bh})
+\partial_z^\varphi(\omega_{vh}-\omega_{bh})\partial_t^{\varphi^\epsilon}\hat{\eta}
+\partial_z^\varphi(\omega_{vh}-\omega_{bh})v\cdot\nabla^{\varphi^\epsilon}\hat{\eta}
\nonumber\\[3pt]
&-\hat{v}\cdot\nabla^\varphi(\omega_{vh}-\omega_{bh})
-\partial_z^\varphi(\omega_{vh}-\omega_{bh})b\cdot\nabla^{\varphi^\epsilon}\hat{\eta}
+\hat{b}\cdot\nabla^\varphi(\omega_{vh}-\omega_{bh})):=I_{b2},
\end{align}
where we have used two different transforms $\Phi\circ Y_1$ and $\Phi\circ Y_2$, and we have
$\|I_{v2}\|_{L^\infty}\to 0$, $\|I_{b2}\|_{L^\infty}\to 0$ as $\epsilon\to 0$.

Since $a_0>0$, $b_0>0$ and $f^7_v\pm f^7_b$ are bounded, we choose suitably large $\gamma>0$ such that
\begin{equation}
\left\{
\begin{aligned}
&\gamma a_0-(f^7_v+f^7_b)>0,\  \gamma a_0-(f^7_v-f^7_b)>0,\\
&\gamma b_0-(f^7_v+f^7_b)>0,\  \gamma b_0-(f^7_v-f^7_b)>0,
\end{aligned}
\right.
\end{equation}
then $\gamma a_0W_+-(f^7_v-f^7_b)W_-$ and $\gamma b_0W_--(f^7_v+f^7_b))W_+$ are the damping term
of the coupled system.
Since the matrix $(a_{i,j})$ and $(b_{i,j})$ are definitely positive,
$-\epsilon\partial_i(a_{i,j}\partial_jW_+)$ and $-\epsilon\partial_i(b_{i,j}\partial_jW_-)$
are the diffusion term.

\subsection{$L^\infty$ Estimate of Strong Vorticity Layer}
In this subsection, if the initial vorticity layer
of velocity or magnetic field is strong,
we can prove that the vorticity layer for free boundary problems is strong.

The following theorem shows the existence of the strong vorticity layer for the free boundary MHD
equations \eqref{MHDF}, which arises from the strong initial vorticity layer of velocity or magnetic field.
\begin{theorem}\label{Theorem3.1}
Assume $\omega_v^\epsilon$ and $\omega_b^\epsilon$ are the vorticity of v ad b of MHD equations.
If the initial data of MHD equations satisfies $\lim\limits_{\epsilon\to0}(\nabla^{\varphi^\epsilon}\times v_0^\epsilon)-\nabla^\varphi\times\lim\limits_{\epsilon\to0}v_0^\epsilon\neq0$
or $\lim\limits_{\epsilon\to0}(\nabla^{\varphi^\epsilon}\times b_0^\epsilon)-\nabla^\varphi\times\lim\limits_{\epsilon\to0}b_0^\epsilon\neq0$
in the initial set $\mathcal{A}_0$, the ideal MHD boundary data satisfies
$\Pi S^\varphi v\nn|_{z=0}=0$ and $\Pi S^\varphi b\nn|_{z=0}=0$ in $(0,T]$,
then $\lim\limits_{\epsilon\to0}\|\omega^\epsilon_v-\omega_v\|_{L^\infty(X(\mathcal{A}_0))\times(0,T])}\neq 0$
or $\lim\limits_{\epsilon\to0}\|\omega^\epsilon_b-\omega_b\|_{L^\infty(X(\mathcal{A}_0))\times(0,T])}\neq 0$.
\end{theorem}
\begin{proof}
Since $\Pi S^\varphi v\nn|_{z=0}=0,\Pi S^\varphi b\nn|_{z=0}=0$ in $(0,T]$,
$|\omega^\epsilon_{v0}-\omega_{v0}|_{\infty}\to0$ and $|\omega^\epsilon_{b0}-\omega_{b0}|_{\infty}\to0$
as $\epsilon\to0$. Then there exist a set $\mathcal{A}_0\cap \{x|z<0\}\neq0$
such that $\lim\limits_{\epsilon\to0}(\nabla^{\varphi^\epsilon}\times v^\epsilon)-\nabla^\varphi\times\lim\limits_{\epsilon\to0}v^\epsilon\neq0$
or
$\lim\limits_{\epsilon\to0}(\nabla^{\varphi^\epsilon}\times B^\epsilon)-\nabla^\varphi\times\lim\limits_{\epsilon\to0}b^\epsilon\neq0$
in the initial set
$\mathcal{A}_0$.

We investigate the following equations:
\begin{equation}
\left\{
\begin{aligned}
&a_0\partial_tW_+-\epsilon\partial_i(a_{ij}\partial_jW_+)+\gamma a_0W_+-(f^7_v-f^7_b)W_-
=I_{v2},
\\
&b_0\partial_tW_--\epsilon\partial_i(b_{ij}\partial_jW_-)+\gamma b_0W_--(f^7_v+f^7_b)W_+
=I_{b2},
\\
&W_+|_{z=0}=\hat{\omega}^b_{vh}+\hat{\omega}^b_{bh}
=\omega^\epsilon_{v0}-\omega_{v0}+\omega^\epsilon_{b0}-\omega_{b0}\to0,
\\
& W_-|_{z=0}=\hat{\omega}^b_{vh}-\hat{\omega}^b_{bh}
=\omega^\epsilon_{v0}-\omega_{v0}-\omega^\epsilon_{b0}+\omega_{b0}\to0,\\
&W_+|_{t=0}=\hat{\omega}_{vh,0}+\hat{\omega}_{bh,0}
=\omega^\epsilon_{vh0}-\omega_{vh0}+\omega^\epsilon_{bh0}-\omega_{bh0},\\
& W_-|_{t=0}=\hat{\omega}_{vh,0}-\hat{\omega}_{bh,0}
=\omega^\epsilon_{vh0}-\omega_{vh0}-\omega^\epsilon_{bh0}+\omega_{bh0}.
\end{aligned}
\right.
\end{equation}

We decompose $W_{\pm}=W_{\pm}^{ini}+W_{\pm}^{fo}+W_{\pm}^{bdy}$,
such that $W_{\pm}^{ini}$ satisfy the nonhomogeneous equations with force term:
\begin{equation}\label{ini}
\left\{
\begin{aligned}
&a_0\partial_tW^{ini}_+-\epsilon\partial_i(a_{ij}\partial_jW^{ini}_+)+\gamma a_0W^{ini}_+-(f^7_v-f^7_b)W^{ini}_-
=0,
\\
&b_0\partial_tW^{ini}_--\epsilon\partial_i(b_{ij}\partial_jW^{ini}_-)+\gamma b_0W^{ini}_--(f^7_v+f^7_b)W^{ini}_+
=0,
\\
&W_+|_{z=0}=0, W_-|_{z=0}=0,\\
&W_+|_{t=0}=\hat{\omega}_{vh,0}+\hat{\omega}_{bh,0}, W_-|_{t=0}=\hat{\omega}_{vh,0}-\hat{\omega}_{bh,0}.
\end{aligned}
\right.
\end{equation}
$W_{\pm}^{fo}$ satisfy the following equations:
\begin{equation}\label{fo}
\left\{
\begin{aligned}
&a_0\partial_tW^{fo}_+-\epsilon\partial_i(a_{ij}\partial_jW^{fo}_+)
+\gamma a_0W^{fo}_+-(f^7_v-f^7_b)W^{fo}_-
=I_{v2},
\\
&b_0\partial_tW^{fo}_--\epsilon\partial_i(b_{ij}\partial_jW^{fo}_-)
+\gamma b_0W^{fo}_--(f^7_v+f^7_b)W^{fo}_+
=I_{b2},
\\
&W^{fo}_+|_{z=0}=0, W^{fo}_-|_{z=0}=0,\\
&W^{fo}_+|_{t=0}=0, W^{fo}_-|_{t=0}=0.
\end{aligned}
\right.
\end{equation}
and $W_{\pm}^{bdy}$ satisfy the homogeneous equations:
\begin{equation}\label{bdy}
\left\{
\begin{aligned}
&a_0\partial_tW^{bdy}_+-\epsilon\partial_i(a_{ij}\partial_jW^{bdy}_+)+\gamma a_0W^{bdy}_+-(f^7_v-f^7_b)W^{bdy}_-
=0,
\\
&b_0\partial_tW^{bdy}_--\epsilon\partial_i(b_{ij}\partial_jW^{bdy}_-)+\gamma b_0W^{bdy}_--(f^7_v+f^7_b)W^{bdy}_+
=0,
\\
&W^{bdy}_+|_{z=0}=\hat{\omega}^b_{vh}+\hat{\omega}^b_{bh},
W^{bdy}_-|_{z=0}=\hat{\omega}^b_{vh}-\hat{\omega}^b_{bh},
\\
&W^{bdy}_+|_{t=0}=0, W^{bdy}_-|_{t=0}=0.
\end{aligned}
\right.
\end{equation}
For \eqref{ini},
we denote the variable by $\tilde{W}_+^{ini}=W_+^{ini}e^{\gamma t}$,
$\tilde{W}_-^{ini}=W_-^{ini} e^{\gamma t}$ and consider the following equations:
\begin{equation}\label{ini1}
\left\{\begin{aligned}
&\partial_t \tilde{W}^{ini}_+ - \frac{\e}{a_0}
 \cdot a_{ij} \partial_{ij} \tilde{W}^{ini}_+
 = \sqrt{\e}\,\mathcal{I}_{v3}
 +\frac{1}{a_0}(f^7_v-f^7_b)\tilde{W}_-^{ini},
 \\
&\partial_t \tilde{W}^{ini}_- - \frac{\e}{b_0}
 \cdot a_{ij} \partial_{ij} \tilde{W}^{ini}_-
 = \sqrt{\e}\,\mathcal{I}_{b3}
 +\frac{1}{b_0}(f^7_v+f^7_b)\tilde{W}_+^{ini},
 \\
&\tilde{W}_+|_{z=0}=0,\tilde{W}_-|_{z=0} = 0,
 \\
&W_+|_{t=0}=\hat{\omega}_{vh,0}+\hat{\omega}_{bh,0}, W_-|_{t=0}=\hat{\omega}_{vh,0}-\hat{\omega}_{bh,0},
\end{aligned}\right.
\end{equation}
where
\begin{equation}
\begin{aligned}
&\mathcal{I}_{v3}
= \frac{\sqrt{\e}}{a_0}\sum\limits_{i,j =1}^3\partial_i a_{ij}
\cdot \partial_j \tilde{W}^{ini}_+,
\\
&\mathcal{I}_{b3}
= \frac{\sqrt{\e}}{b_0}\sum\limits_{i,j =1}^3\partial_i b_{ij} \cdot \partial_j
\tilde{W}^{ini}_-.
\end{aligned}
\end{equation}
Note that $\|\mathcal{I}_{v3}\|_{L^{\infty}} <+\infty$ and $\|\mathcal{I}_{b3}\|_{L^{\infty}} <+\infty$,
since $\mathcal{I}_{v3}$ and $\mathcal{I}_{v3}$ contains normal differential operator $\partial_z$ of order at most one.

By definition, \eqref{ini1} is uniformly parabolic,
which has fundamental solution satisfying the parabolic scaling.
Let $\textsf{H}_v(\frac{x}{\sqrt{\e t}})$,
 $\textsf{H}_b(\frac{x}{\sqrt{\e t}})$ be the fundamental solution
 of the following homogeneous parabolic equation in $\mathbb{R}^3$,
 respectively,
\begin{equation}
\begin{aligned}
\partial_t f - \frac{\e}{a_0} a_{ij} e^{\gamma t}  \cdot \partial_{ij} f =0,
\\
\partial_t f - \frac{\e}{b_0} b_{ij} e^{\gamma t}  \cdot \partial_{ij} f =0.
\end{aligned}
\end{equation}
Then the equations \eqref{ini1} has the explicit formula
by using Duhamel's principle
\begin{equation}\label{D3}
\begin{aligned}
&\tilde{W}_+^{ini}(t,x)=\int_{\mathbb{R}^3_{-}}(\hat{\omega}_{v,0}
+\hat{\omega}_{b,0})(y) (\textsf{H}_v(\frac{x-y}{\sqrt{\e t}})
- \textsf{H}_v(\frac{x+y}{\sqrt{\e t}})) \,\mathrm{d}y
\\
&\quad + \int_0^t\int_{\mathbb{R}^3_{-}}
(\sqrt{\e}\,\mathcal{I}_{v3}+\frac{1}{a_0}(f^7_v-f^7_b)\tilde{W}_-^{ini})(t-s,y)
(\textsf{H}_v(\frac{x-y}{\sqrt{\e s}}) - \textsf{H}_v(\frac{x+y}{\sqrt{\e s}})) \,\mathrm{d}y
\mathrm{d}s,
\\
&\tilde{W}_-^{ini}(t,x)=\int_{\mathbb{R}^3_{-}}(\hat{\omega}_{vh,0}-\hat{\omega}_{bh,0})(y)
(\textsf{H}_b(\frac{x-y}{\sqrt{\e t}})- \textsf{H}_b(\frac{x+y}{\sqrt{\e t}})) \,\mathrm{d}y
\\
&\quad+ \int_0^t\int_{\mathbb{R}^3_{-}}
(\sqrt{\e}\,\mathcal{I}_{b3}+\frac{1}{b_0}(f^7_v+f^7_b)\tilde{W}_+^{ini})(t- s,y)
(\textsf{H}_b(\frac{x-y}{\sqrt{\e s}}) - \textsf{H}_b(\frac{x+y}{\sqrt{\e s}})) \,\mathrm{d}y
\mathrm{d}s.
\end{aligned}
\end{equation}
As $\e\rto 0$,
we have
\begin{align*}
&\int_{\mathbb{R}^3_{-}}(\hat{\omega}_{v,0}
+\hat{\omega}_{b,0})(y) (\textsf{H}_v(\frac{x-y}{\sqrt{\e t}})
- \textsf{H}_v(\frac{x+y}{\sqrt{\e t}})) \,\mathrm{d}y
\rto \hat{\omega}_{v,0}+\hat{\omega}_{b,0}+\sqrt{\e} O(1)
\rto \hat{\omega}_{v,0}+\hat{\omega}_{b,0},
\\
&\int_{\mathbb{R}^3_{-}}(\hat{\omega}_{v,0}-\hat{\omega}_{b,0})(y)
(\textsf{H}_b(\frac{x-y}{\sqrt{\e t}})
- \textsf{H}_b(\frac{x+y}{\sqrt{\e t}})) \,\mathrm{d}y
\rto \hat{\omega}_{vh,0}-\hat{\omega}_{b,0}
+\sqrt{\e} O(1) \rto \hat{\omega}_{v,0}-\hat{\omega}_{b,0},
\end{align*}
and
\begin{align*}
&\int_0^t\int_{\mathbb{R}^3_{-}}
\sqrt{\e}\,\mathcal{I}_{v3}(t-s,y)
(\textsf{H}_v(\frac{x-y}{\sqrt{\e s}}) - \textsf{H}_v(\frac{x+y}{\sqrt{\e s}})) \,\mathrm{d}y
\mathrm{d}s\rto0,
\\
&\int_0^t\int_{\mathbb{R}^3_{-}}
\sqrt{\e}\,\mathcal{I}_{b3}(t- s,y)
(\textsf{H}_b(\frac{x-y}{\sqrt{\e s}}) - \textsf{H}_b(\frac{x+y}{\sqrt{\e s}})) \,\mathrm{d}y
\mathrm{d}s\rto0.
\end{align*}
In addition
\begin{equation}
\begin{array}{ll}
\lim\limits_{\e\rto 0}\|\int_0^t\int_{\mathbb{R}^3_{-}}
\frac{1}{a_0}(f^7_v-f^7_b)\tilde{W}_-^{ini}(t-s,y)
(\textsf{H}_v(\frac{x-y}{\sqrt{\e s}})
- \textsf{H}_v(\frac{x+y}{\sqrt{\e s}})) \,\mathrm{d}y
\mathrm{d}s\|_{L^\infty},
\\[8pt]

\leq \lim\limits_{\e\rto 0}\|\int_0^t\int_{\mathbb{R}^3_{-}}
\frac{1}{a_0}(f^7_v-f^7_b)(t-s,y)
(\textsf{H}_v(\frac{x-y}{\sqrt{\e s}})
- \textsf{H}_v(\frac{x+y}{\sqrt{\e s}})) \,\mathrm{d}y
\mathrm{d}s\|_{L^\infty}\|\tilde{W}_-^{ini}\|_{L^\infty},
\\[8pt]

\leq \|\int_0^t\frac{1}{a_0}(f^7_v-f^7_b)(s,x)
\, \mathrm{d}s\|_{L^\infty}\|\tilde{W}_-^{ini}\|_{L^\infty},
\\[8pt]

\lim\limits_{\e\rto 0}\|\int_0^t\int_{\mathbb{R}^3_{-}}
\frac{1}{b_0}(f^7_v+f^7_b)\tilde{W}_+^{ini}(t- s,y)
(\textsf{H}_b(\frac{x-y}{\sqrt{\e s}})
- \textsf{H}_b(\frac{x+y}{\sqrt{\e s}}))
\,\mathrm{d}y\mathrm{d}s\|_{L^\infty}
\\[8pt]

\leq \lim\limits_{\e\rto 0}\|\int_0^t\int_{\mathbb{R}^3_{-}}
(\frac{1}{b_0}(f^7_v+f^7_b)(t- s,y)
(\textsf{H}_b(\frac{x-y}{\sqrt{\e s}})
- \textsf{H}_b(\frac{x+y}{\sqrt{\e s}}))
\,\mathrm{d}y\mathrm{d}s\|_{L^\infty}\|\tilde{W}_+^{ini}\|_{L^\infty}
\\[8pt]

\leq \|\int_0^t(\frac{1}{b_0}(f^7_v+f^7_b)(s,x)
\, \mathrm{d}s\|_{L^\infty}\|\tilde{W}_+^{ini}\|_{L^\infty},
\end{array}
\end{equation}
where $\|\int_0^t\frac{1}{a_0}(f^7_v-f^7_b)(s,x)\, \mathrm{d}s\|_{L^\infty}$,
$\|\int_0^t(\frac{1}{b_0}(f^7_v+f^7_b)(s,x)\, \mathrm{d}s\|_{L^\infty}$
are bounded.
Applying the $L^\infty$ to \eqref{D3},
we have
\begin{equation}
\begin{array}{ll}
\lim\limits_{\e\rto 0}\|\tilde{W}_+^{ini}(t,x)\|_{L^\infty}
=\|\int_{\mathbb{R}^3_{-}}(\hat{\omega}_{v,0}
+\hat{\omega}_{b,0})(y) (\textsf{H}_v(\frac{x-y}{\sqrt{\e t}})
- \textsf{H}_v(\frac{x+y}{\sqrt{\e t}})) \,\mathrm{d}y
\\[8pt]

 + \int_0^t\int_{\mathbb{R}^3_{-}}
(\sqrt{\e}\,\mathcal{I}_{v3}+\frac{1}{a_0}(f^7_v-f^7_b)\tilde{W}_-^{ini})(t-s,y)
(\textsf{H}_v(\frac{x-y}{\sqrt{\e s}}) - \textsf{H}_v(\frac{x+y}{\sqrt{\e s}})) \,\mathrm{d}y
\mathrm{d}s\|_{L^\infty}
\\[8pt]

\geq \|\hat{\omega}_{v,0}+\hat{\omega}_{b,0}\|_{L^\infty}
-\|\int_0^t\frac{1}{a_0}(f^7_v-f^7_b)(s,x)\, \mathrm{d}s\|_{L^\infty}
\|\tilde{W}_-^{ini}(t,x)\|_{L^\infty},
\\[8pt]

\lim\limits_{\e\rto 0}\|\tilde{W}_-^{ini}(t,x)\|_{L^\infty}
=\|\int_{\mathbb{R}^3_{-}}(\hat{\omega}_{vh,0}-\hat{\omega}_{bh,0})(y)
(\textsf{H}_b(\frac{x-y}{\sqrt{\e t}})- \textsf{H}_b(\frac{x+y}{\sqrt{\e t}})) \,\mathrm{d}y
\\[8pt]

+\int_0^t\int_{\mathbb{R}^3_{-}}
(\sqrt{\e}\,\mathcal{I}_{b3}+\frac{1}{b_0}(f^7_v+f^7_b)\tilde{W}_+^{ini})(t- s,y)
(\textsf{H}_b(\frac{x-y}{\sqrt{\e s}}) - \textsf{H}_b(\frac{x+y}{\sqrt{\e s}})) \,\mathrm{d}y
\mathrm{d}s\|_{L^\infty}
\\[8pt]

\geq \|\hat{\omega}_{v,0}-\hat{\omega}_{b,0}\|_{L^\infty}
-\|\int_0^t(\frac{1}{b_0}(f^7_v+f^7_b)(s,x)\, \mathrm{d}s\|_{L^\infty}
\|\tilde{W}_+^{ini}(t,x)\|_{L^\infty}.
\end{array}
\end{equation}
Hence, we have
\begin{align}
\lim\limits_{\e\rto 0}(C_v\|\tilde{W}_+^{ini}(t,x)\|_{L^\infty}+C_b\|\tilde{W}_-^{ini}(t,x)\|_{L^\infty})
\geq \|\hat{\omega}_{v,0}+\hat{\omega}_{b,0}\|_{L^\infty}
+\|\hat{\omega}_{v,0}-\hat{\omega}_{b,0}\|_{L^\infty},
\end{align}
where $C_v=1+\|\int_0^t(\frac{1}{b_0}(f^7_v+f^7_b)(s,x)\, \mathrm{d}s\|_{L^\infty}>0$
and $C_b=1+\|\int_0^t(\frac{1}{a_0}(f^7_v-f^7_b)(s,x)\, \mathrm{d}s\|_{L^\infty}>0$.
By the definition of $\tilde{W}_\pm^{ini}$, we know that
$\lim\limits_{\e\rto 0} \tilde{W}_\pm^{ini}(t,x)$ and
$\lim\limits_{\e\rto 0} \hat{\omega}_{vh,0}\pm\hat{\omega}_{bh,0}(x)$ have the same support.
It follows from
$\lim\limits_{\e\rto 0}\|\hat{\omega}_{vh}+\hat{\omega}_{bh}\|_{L^{\infty}(\mathcal{X}(\mathcal{A}_0)\times (0,T])} \neq 0$
or
$\lim\limits_{\e\rto 0}\|\hat{\omega}_{vh}-\hat{\omega}_{bh}\|_{L^{\infty}
(\mathcal{X}(\mathcal{A}_0)\times (0,T])} \neq 0$
that
\begin{align}
\lim\limits_{\e\rto 0}(C_v\|\tilde{W}_+^{ini}(t,x)\|_{L^\infty}+C_b\|\tilde{W}_-^{ini}(t,x)\|_{L^\infty})>0
\end{align}
hold in $\mathcal{X}(\mathcal{A}_0)\times (0,T]$.

We can deal with \eqref{fo} in the similar manner to get
$$\lim\limits_{\e\rto 0}(\|\tilde{W}_+^{ini}(t,x)\|_{L^\infty}+\|\tilde{W}_-^{ini}(t,x)\|_{L^\infty})
\geq\|\hat{\omega}_{v,0}+\hat{\omega}_{b,0}\|_{L^\infty}
+\|\hat{\omega}_{v,0}-\hat{\omega}_{b,0}\|_{L^\infty},$$
since the force term $\|I_{v2}\|_{L^\infty}\rto0$, $\|I_{b2}\|_{L^\infty}\rto0$,
when $\epsilon \to0$.

Therefore, since $\|\hat{\omega}_{v,0}+\hat{\omega}_{b,0}\|_{L^\infty}\to 0$ and
$\|\hat{\omega}_{v,0}-\hat{\omega}_{b,0}\|_{L^\infty}\to 0$, one has
$$\lim\limits_{\e\rto 0}(\|\tilde{W}_+^{ini}(t,x)\|_{L^\infty}+\|\tilde{W}_-^{ini}(t,x)\|_{L^\infty})
\geq0$$
in $\mathcal{X}(\mathcal{A}_0)\times (0,T]$.

For \eqref{bdy}, we define
\begin{align*}
&\phi_+=W^{bdy}_+-(F^{1,2}_v[\nabla\varphi^\epsilon](\partial_jv^{\epsilon,i})
-F^{1,2}_v[\nabla\varphi](\partial_jv^i)\\
&\qquad+F^{1,2}_b[\nabla\varphi^\epsilon](\partial_jb^{\epsilon,i})
-F^{1,2}_b[\nabla\varphi](\partial_jb^i)):=W^{bdy}_+-\Psi_+,\\
&\phi_-=W^{bdy}_--(F^{1,2}_v[\nabla\varphi^\epsilon](\partial_jv^{\epsilon,i})
-F^{1,2}_v[\nabla\varphi](\partial_jv^i)\\
&\qquad-F^{1,2}_b[\nabla\varphi^\epsilon](\partial_jb^{\epsilon,i})
+F^{1,2}_b[\nabla\varphi](\partial_jb^i)):=W^{bdy}_--\Psi_-,
\end{align*}
then $\phi_{\pm}$ satisfy the following equations:
\begin{equation} \label{bdys3}
\left\{
\begin{aligned}
&a_0\partial_t\phi_+-\epsilon\partial_i(a_{ij}\partial_j\phi_+)+\gamma a_0\phi_+-(f^7_v-f^7_b)\phi_-
\\
&\quad=-a_0\partial_t\Psi_+
+\epsilon\partial_i[a_{ij}\partial_j\Psi_+]
-\gamma a_0\Psi_++(f^7_v-f^7_b))\Psi_-,
\\
&b_0\partial_t\phi_--\epsilon\partial_i(b_{ij}\partial_j\phi_-)+\gamma b_0\phi_--(f^7_v+f^7_b)\phi_+
\\
&\quad=-b_0\partial_t\Psi_-+\epsilon\partial_i[b_{i,j}\partial_j\Psi_-]-\gamma b_0\Psi_-+(f^7_v+f^7_b))\Psi_+,
\\
&\phi_+|_{z=0}=0,\phi_-|_{z=0}=0,
\\
&\phi_+|_{t=0}=-\Psi_+|_{t=0},\phi_-|_{t=0}=-\Psi_-|_{t=0}.
\end{aligned}
\right.
\end{equation}
In fact, $\Psi_+|_{t=0}\rto0$ and $\Psi_-|_{t=0}\rto0$ when $\epsilon \to0$.
Dealing with \eqref{bdys3} as above, we have
$$\lim\limits_{\e\rto 0}(\|\tilde{W}_+^{ini}(t,x)\|_{L^\infty}+\|\tilde{W}_-^{ini}(t,x)\|_{L^\infty})
\geq0.$$

By solving equations \eqref{ini}, \eqref{fo} and \eqref{bdy},
we have $\lim\limits_{\e\rto 0}(\|\tilde{W}_+^{ini}(t,x)\|_{L^\infty}+\|\tilde{W}_-^{ini}(t,x)\|_{L^\infty})>0$.
By the definition of $\tilde{W}_+^{ini}$ and $\tilde{W}_-^{ini}$,
we have
\begin{equation}
\lim\limits_{\e\rto 0}(\|W_+^{ini}(t,x)\|_{L^\infty}+\|W_-^{ini}(t,x)\|_{L^\infty})>0.
\end{equation}
Thus, $\|W_+^{ini}(t,x)\|_{L^\infty}>0$
or $\|W_-^{ini}(t,x)\|_{L^\infty}>0$.
Theorem \ref{Theorem3.1} is proved.
\end{proof}

\begin{remark}
It is easy to show that if the initial MHD data satisfies $\lim\limits_{\epsilon\to0}(\nabla^{\varphi^\epsilon}\times v_0^\epsilon)-\nabla^\varphi\times\lim\limits_{\epsilon\to0}v_0^\epsilon\neq0$
and $\lim\limits_{\epsilon\to0}(\nabla^{\varphi^\epsilon}\times b_0^\epsilon)-\nabla^\varphi\times\lim\limits_{\epsilon\to0}b_0^\epsilon\neq0$.
Then the strong initial layer implies the following strong boundary layers:
\begin{equation}
\begin{aligned}
&\lim\limits_{\e\rto 0}\|\partial_z v^{\e} -\partial_z v\|_{L^{\infty}(\mathcal{X}(\mathcal{A}_0)\times (0,T])} \neq 0,\\
&\lim\limits_{\e\rto 0}\|\partial_z b^{\e} -\partial_z b\|_{L^{\infty}(\mathcal{X}(\mathcal{A}_0)\times (0,T])} \neq 0, \\
&\lim\limits_{\e\rto 0}\|\mathcal{S}v^{\e} -\mathcal{S}v\|_{L^{\infty}(\mathcal{X}(\mathcal{A}_0)\times (0,T])} \neq 0,\\
&\lim\limits_{\e\rto 0}\|\mathcal{S}b^{\e} -\mathcal{S}b\|_{L^{\infty}(\mathcal{X}(\mathcal{A}_0)\times (0,T])} \neq 0,\\
&\lim\limits_{\e\rto 0}\|\nabla q^{\e} - \nabla q\|_{L^{\infty}(\mathcal{X}(\mathcal{A}_0)\times (0,T])} \neq 0.
\end{aligned}
\end{equation}
\end{remark}

\section{The Discrepancy from the Boundary Values of Vorticities}
In this section, we discuss the existence of strong vorticity layer
under the assumption that the boundary value of ideal MHD equations satisfies
$\Pi\mathcal{S}^{\varphi} v\nn|_{z=0} \neq 0$ or
$\Pi\mathcal{S}^{\varphi} b\nn|_{z=0} \neq 0$ in $(0,T]$.
It shows that the discrepancy between boundary value of MHD vorticity
and boundary value of ideal MHD vorticity is also one of sufficient
conditions for the existence strong vorticity layer
for the free boundary MHD equations \eqref{MHDF}.

\subsection{Analysis the Discrepancy of Vorticities}
In this subsection,
we consider the case if the cross product of the tangential projection on the free boundary of
the ideal MHD strain tensor for the vorlocity or the magnetic field
with the normal vector does not vanishes on the free surface,
then there is a discrepancy between MHD vorticity and ideal MHD vorticity.

The following lemma shows the existence of the strong vorticity layer for
the free boundary problems for MHD equations \eqref{MHDF} arise from the discrepancy.

\begin{lemma}\label{lemma4.1}
Assume $\omega_v^{\e}, \omega_v^{\e}$ are the vorticity of $v^{\e},b^{\e}$, respectively,
$\NN^{\e}$ is the normal vector of MHD equations\eqref{MHDF};
and $\omega_v,\omega_b$ are the vorticity of v, b, respectively,
$\NN$ is the normal vector of ideal MHD equations\eqref{IMHDF},
$\omega_{v}^{\e,b}$, $\omega_{b}^{\e,b}$, $\omega^b_{v}$ and $\omega^b_{b}$ are boundary values of $\omega_v^{\e}$, $\omega_b^{\e}$, $\omega_v$ and $\omega_b$ respectively.
If $\Pi\mathcal{S}^{\varphi} v\nn|_{z=0} \neq 0$
or $\Pi\mathcal{S}^{\varphi} b\nn|_{z=0} \neq 0$ in $(0,T]$,
then $\lim\limits_{\e\rto 0}|\omega_v^{\e,b}-\omega_v^b|
_{L^{\infty}(\mathbb{R}^2\times (0,T])} \neq 0$
or $\lim\limits_{\e\rto 0}|\omega_b^{\e,b}-\omega_b^b|
_{L^{\infty}(\mathbb{R}^2\times (0,T])} \neq 0$.
\end{lemma}

\begin{proof}
We denote $\textsf{S}_{vn} =\Pi \mathcal{S}^{\varphi} v \nn$
and $\textsf{S}_{bn} =\Pi \mathcal{S}^{\varphi} b \nn$.
Since $\textsf{S}_{vn}|_{z=0} \neq 0$ and $\textsf{S}_{bn}|_{z=0} \neq 0$,
due to the fact that
$\mathcal{S}^{\varphi}v\nn = (\mathcal{S}^{\varphi}v \nn\cdot\nn)\nn+\Pi\mathcal{S}^{\varphi} v\nn$
and $\mathcal{S}^{\varphi}b\nn = (\mathcal{S}^{\varphi}b \nn\cdot\nn)\nn + \Pi\mathcal{S}^{\varphi} b\nn$
on the free boundary,
then $\mathcal{S}^{\varphi}v \nn$ and $\mathcal{S}^{\varphi}b \nn$
are not parallel to $\nn$ on the free boundary, namely,
\begin{equation}
\begin{array}{ll}
&\mathcal{S}^{\varphi}v \nn\times\nn = (\mathcal{S}^{\varphi}v \nn\cdot\nn)\nn \times\nn + \Pi\mathcal{S}^{\varphi} v\nn \times\nn
= \Pi\mathcal{S}^{\varphi} v\nn \times\nn \neq 0,
\\
&\mathcal{S}^{\varphi}b \nn\times\nn = (\mathcal{S}^{\varphi}b \nn\cdot\nn)\nn \times\nn + \Pi\mathcal{S}^{\varphi} b\nn \times\nn
= \Pi\mathcal{S}^{\varphi} b\nn \times\nn \neq 0.
\end{array}
\end{equation}
There exist $\Theta_v^1,\Theta_v^2,\Theta_v^3, \Theta_b^1,\Theta_b^2,\Theta_b^3$,
such that $\Pi\mathcal{S}^{\varphi} v\nn \times\nn := (\Theta_v^1,\Theta_v^2,\Theta_v^3)^{\top}$
and $\Pi\mathcal{S}^{\varphi} v\nn \times\nn := (\Theta_b^1,\Theta_b^2,\Theta_b^3)^{\top}$,
which are nonzero vectors.

By $\mathcal{S}^{\varphi}v \nn\times\nn = (\Theta_v^1,\Theta_v^2,\Theta_v^3)^{\top}$,
one has
\begin{equation}\label{Sect4_Vorticity_Discrepancy_2}
\left\{\begin{array}{ll}
n^3[n^1\partial_1^{\varphi} v^1 + \frac{n^2}{2}(\partial_1^{\varphi} v^2 + \partial_2^{\varphi} v^1)
+ \frac{n^3}{2}(\partial_1^{\varphi} v^3 + \partial_z^{\varphi} v^1)] \\[5pt]\quad

= n^1[\frac{n^1}{2}(\partial_1^{\varphi} v^3 + \partial_z^{\varphi} v^1) + \frac{n^2}{2}(\partial_2^{\varphi} v^3 + \partial_z^{\varphi} v^2)
- n^3\partial_1^{\varphi} v^1 - n^3\partial_2^{\varphi} v^2] -\Theta_v^2,

\\[10pt]

n^3[\frac{n^1}{2}(\partial_1^{\varphi} v^2 + \partial_2^{\varphi} v^1) + n^2\partial_2^{\varphi} v^2
+ \frac{n^3}{2}(\partial_2^{\varphi} v^3 + \partial_z^{\varphi} v^2)] \\[5pt]\quad

= n^2[\frac{n^1}{2}(\partial_1^{\varphi} v^3 + \partial_z^{\varphi} v^1) + \frac{n^2}{2}(\partial_2^{\varphi} v^3 + \partial_z^{\varphi} v^2)
- n^3\partial_1^{\varphi} v^1 - n^3\partial_2^{\varphi} v^2] +\Theta_v^1,

\\[10pt]

n^2[n^1\partial_1^{\varphi} v^1 + \frac{n^2}{2}(\partial_1^{\varphi} v^2 + \partial_2^{\varphi} v^1)
+ \frac{n^3}{2}(\partial_1^{\varphi} v^3 + \partial_z^{\varphi} v^1)] \\[5pt]\quad

= n^1[\frac{n^1}{2}(\partial_1^{\varphi} v^2 + \partial_2^{\varphi} v^1) + n^2\partial_2^{\varphi} v^2
+ \frac{n^3}{2}(\partial_2^{\varphi} v^3 + \partial_z^{\varphi} v^2)] + \Theta_v^3.
\end{array}\right.
\end{equation}

Then $\partial_z v^1$ and $\partial_z v^2$ satisfy
\begin{equation}
\begin{array}{ll}
\big[(n^1)^2 + \frac{(n^3)^2}{2} + \frac{1}{2}\frac{(n^1)^4}{(n^3)^2}
-\frac{1}{2}\frac{(n^2)^4}{(n^3)^2}\big]\partial_z v^1
+\big[n^1n^2 + \frac{(n^1)^3n^2}{(n^3)^2}
+ \frac{n^1(n^2)^3}{(n^3)^2} \big]\partial_z v^2
\\[10pt]

= -[\frac{(n^3)^2}{2} - \frac{(n^1)^2}{2}-\frac{(n^2)^2}{2}]\big[\partial_z\varphi\partial_1 v^3
- \partial_1\varphi[- \partial_z\varphi(\partial_1 v^1 + \partial_2 v^2)] \big]
\\[5pt]\quad

+ (\frac{n^1(n^2)^2}{n^3} - 2n^1n^3)(\partial_z\varphi\partial_1 v^1)
- (n^1n^3 + \frac{n^1(n^2)^2}{n^3})(\partial_z\varphi\partial_2 v^2)
\\[5pt]\quad

+ [\frac{(n^2)^2}{n^3}\frac{n^2}{2}- \frac{n^1n^2}{n^3}\frac{n^1}{2}
-\frac{n^2n^3}{2}](\partial_z\varphi\partial_1 v^2 + \partial_z\varphi\partial_2 v^1)
-\partial_z\varphi\Theta_v^2
- \frac{n^2}{n^3}\partial_z\varphi\Theta_v^3,
\end{array}
\end{equation}

\begin{equation}
\begin{array}{ll}
\big[n^1n^2 + \frac{(n^1)^3n^2}{(n^3)^2}
+ \frac{n^1(n^2)^3}{(n^3)^2} \big]\partial_z v^1
+ \big[(n^2)^2 + \frac{1}{2}(n^3)^2
+ \frac{(n^2)^4}{2(n^3)^2}
- \frac{(n^1)^4}{2(n^3)^2}\big]\partial_z v^2
\\[10pt]

= -[\frac{(n^3)^2}{2} - \frac{(n^2)^2}{2}
- \frac{(n^1)^2}{2}]\big[\partial_z\varphi\partial_2 v^3
- {\partial_z\varphi}[-\partial_z\varphi(\partial_1 v^1 + \partial_2 v^2)]\big]
\\[5pt]\quad

- (n^2n^3 + \frac{(n^1)^2n^2}{n^3})(\partial_z\varphi\partial_1 v^1)
+ (\frac{(n^1)^2n^2}{n^3} -2n^2n^3)(\partial_z\varphi\partial_2 v^2)
\\[5pt]\quad

+ (\frac{(n^1)^2}{n^3}\frac{n^1}{2} -\frac{n^1n^3}{2} - \frac{n^1n^2}{n^3}\frac{n^2}{2})
(\partial_z\varphi\partial_1 v^2 + \partial_z\varphi\partial_2 v^1)
+\partial_z\varphi\Theta_v^1 + \frac{n^1}{n^3}\partial_z\varphi\Theta_v^3.
\end{array}
\end{equation}

Similarly, for $b$, one has
\begin{equation*}
\begin{array}{ll}
\big[(n^1)^2 + \frac{(n^3)^2}{2} + \frac{1}{2}\frac{(n^1)^4}{(n^3)^2} - \frac{1}{2}\frac{(n^2)^4}{(n^3)^2}
\big]\partial_z b^1
+ \big[n^1n^2 + \frac{(n^1)^3n^2}{(n^3)^2} + \frac{n^1(n^2)^3}{(n^3)^2} \big]\partial_z b^2
\\[10pt]

= -[\frac{(n^3)^2}{2} - \frac{(n^1)^2}{2} - \frac{(n^2)^2}{2}]\big[\partial_z\varphi\partial_1 b^3
- \partial_1\varphi[- \partial_z\varphi(\partial_1 b^1 + \partial_2 b^2)] \big]
\\[5pt]\quad

+ (\frac{n^1(n^2)^2}{n^3} - 2n^1n^3)(\partial_z\varphi\partial_1 b^1)
- (n^1n^3 + \frac{n^1(n^2)^2}{n^3})(\partial_z\varphi\partial_2 b^2)
\\[5pt]\quad

+ [\frac{(n^2)^2}{n^3}\frac{n^2}{2}- \frac{n^1n^2}{n^3}\frac{n^1}{2} -\frac{n^2n^3}{2}]
(\partial_z\varphi\partial_1 b^2 + \partial_z\varphi\partial_2 b^1)-\partial_z\varphi\Theta_b^2 - \frac{n^2}{n^3}\partial_z\varphi\Theta_b^3,
\end{array}
\end{equation*}

\begin{equation*}
\begin{array}{ll}
\big[n^1n^2 + \frac{(n^1)^3n^2}{(n^3)^2} + \frac{n^1(n^2)^3}{(n^3)^2} \big]\partial_z b^1
+ \big[(n^2)^2 + \frac{1}{2}(n^3)^2 + \frac{(n^2)^4}{2(n^3)^2}
- \frac{(n^1)^4}{2(n^3)^2}\big]\partial_z b^2
\\[10pt]

= -[\frac{(n^3)^2}{2} - \frac{(n^2)^2}{2}
- \frac{(n^1)^2}{2}]\big[\partial_z\varphi\partial_2 b^3
- {\partial_z\varphi}[- \partial_z\varphi(\partial_1 b^1 + \partial_2 b^2)]\big]
\\[5pt]\quad

- (n^2n^3 + \frac{(n^1)^2n^2}{n^3})(\partial_z\varphi\partial_1 b^1)
+ (\frac{(n^1)^2n^2}{n^3} -2n^2n^3)(\partial_z\varphi\partial_2 b^2)
\\[5pt]\quad

+ (\frac{(n^1)^2}{n^3}\frac{n^1}{2} -\frac{n^1n^3}{2}
- \frac{n^1n^2}{n^3}\frac{n^2}{2})(\partial_z\varphi\partial_1 b^2
+ \partial_z\varphi\partial_2 b^1)+\partial_z\varphi\Theta_b^1
+ \frac{n^1}{n^3}\partial_z\varphi\Theta_b^3.
\end{array}
\end{equation*}

We assume $|\nabla h|_{\infty}$ is suitably small such that
the coefficient matrices of
$(\partial_z v^1, \partial_z v^2)^{\top}$
and $(\partial_z b^1, \partial_z b^2)^{\top}$
are non-degenerate to solve
\begin{equation}
\left\{\begin{array}{ll}
\partial_z v^1 = f^5[\nabla\varphi](\partial_j v^i)
-\textsf{M}^{11}(\partial_z\varphi\Theta_v^2 + \frac{n^2}{n^3}\partial_z\varphi\Theta_v^3)
+\textsf{M}^{12}(\partial_z\varphi\Theta_v^1 + \frac{n^1}{n^3}\partial_z\varphi\Theta_v^3),
\\[8pt]

\partial_z v^2 = f^6[\nabla\varphi](\partial_j v^i)
-\textsf{M}^{21}(\partial_z\varphi\Theta_v^2 + \frac{n^2}{n^3}\partial_z\varphi\Theta_v^3)
+\textsf{M}^{22}(\partial_z\varphi\Theta_v^1 + \frac{n^1}{n^3}\partial_z\varphi\Theta_v^3),
\end{array}\right.
\end{equation}
and
\begin{equation}
\left\{\begin{array}{ll}
\partial_z b^1 = f^5[\nabla\varphi](\partial_j b^i)
-\textsf{M}^{11}(\partial_z\varphi\Theta_b^2 + \frac{n^2}{n^3}\partial_z\varphi\Theta_b^3)
+\textsf{M}^{12}(\partial_z\varphi\Theta_b^1 + \frac{n^1}{n^3}\partial_z\varphi\Theta_b^3),
\\[8pt]

\partial_z b^2 = f^6[\nabla\varphi](\partial_j b^i)
-\textsf{M}^{21}(\partial_z\varphi\Theta_b^2 + \frac{n^2}{n^3}\partial_z\varphi\Theta_b^3)
+\textsf{M}^{22}(\partial_z\varphi\Theta_b^1 + \frac{n^1}{n^3}\partial_z\varphi\Theta_b^3),
\\
\end{array}\right.
\end{equation}
where $j=1,2,\ i=1,2,3$ and the matrix $\textsf{M} = (\textsf{M}_{ij})$ is defined in \eqref{M},
$(\textsf{M}^{ij}) = (\textsf{M}_{ij})^{-1}$.

It follows that the boundary values of $\omega_{vh} = (\omega_v^1,\omega_v^2)$ and $\omega_{bh} = (\omega_b^1,\omega_b^2)$ are as follows
\begin{equation}
\begin{array}{ll}
\omega_v^1 = - \frac{\partial_1\varphi\partial_2\varphi}{\partial_z\varphi}\partial_z v^1
- \frac{1 + (\partial_2\varphi)^2}{\partial_z\varphi}\partial_z v^2
+ \partial_2 v^3 + \partial_2\varphi(\partial_1 v^1 + \partial_2 v^2)
\\[8pt]\hspace{0.47cm}

:= \textsf{F}^1 [\nabla\varphi](\partial_j v^i) + \varsigma_1\Theta_v^1 + \varsigma_2\Theta_v^2 + \varsigma_3\Theta_v^3, \\[9pt]

\omega_v^2 = \frac{1+(\partial_1\varphi)^2}{\partial_z\varphi}\partial_z v^1
+ \frac{\partial_1\varphi\partial_2\varphi}{\partial_z\varphi}\partial_z v^2
- \partial_1 v^3 - \partial_1\varphi(\partial_1 v^1 + \partial_2 v^2) \\[8pt]\hspace{0.47cm}

:= \textsf{F}^2 [\nabla\varphi](\partial_j v^i) + \varsigma_4\Theta^1 + \varsigma_5\Theta^2 + \varsigma_6\Theta^3,
\end{array}
\end{equation}
and
\begin{equation}
\begin{array}{ll}
\omega_b^1 = - \frac{\partial_1\varphi\partial_2\varphi}{\partial_z\varphi}\partial_z b^1
- \frac{1 + (\partial_2\varphi)^2}{\partial_z\varphi}\partial_z b^2
+ \partial_2 b^3 + \partial_2\varphi(\partial_1 b^1 + \partial_2 b^2) \\[8pt]\hspace{0.47cm}

:= \textsf{F}^1 [\nabla\varphi](\partial_j b^i) + \varsigma_1\Theta_b^1 + \varsigma_2\Theta_b^2 + \varsigma_3\Theta_b^3, \\[8pt]\hspace{0.47cm}

:= \varsigma_1\Theta_b^1 + \varsigma_2\Theta_b^2 + \varsigma_3\Theta_b^3, \\[9pt]

\omega_b^2 = \frac{1+(\partial_1\varphi)^2}{\partial_z\varphi}\partial_z b^1
+ \frac{\partial_1\varphi\partial_2\varphi}{\partial_z\varphi}\partial_z b^2
- \partial_1 b^3 - \partial_1\varphi(\partial_1 b^1 + \partial_2 b^2) \\[8pt]\hspace{0.47cm}

:= \textsf{F}^2 [\nabla\varphi](\partial_j b^i) + \varsigma_4\Theta_b^1 + \varsigma_5\Theta_b^2 + \varsigma_6\Theta_b^3\\[8pt]\hspace{0.47cm}

:= \varsigma_4\Theta_b^1 + \varsigma_5\Theta_b^2 + \varsigma_6\Theta_b^3,
\end{array}
\end{equation}
where the coefficients $\varsigma_i$ satisfy
\begin{equation}\label{Sect4_Vorticity_Discrepancy_6}
\begin{array}{ll}
\varsigma_1 =\partial_z\varphi[\partial_1\varphi\partial_2\varphi\textsf{M}^{12}
+ (1+(\partial_2\varphi)^2)\textsf{M}^{22}], \\[8pt]
\varsigma_2 =\partial_1\varphi\partial_2\varphi\textsf{M}^{11} + (1 + (\partial_2\varphi)^2)\textsf{M}^{21}, \\[8pt]

\varsigma_3 =\big[(1+(\partial_1\varphi)^2)
\big(-\textsf{M}^{11}\frac{n^2}{n^3} +\textsf{M}^{12}\frac{n^1}{n^3}\partial_z\varphi\big) \\[7pt]\hspace{0.85cm}
+ \partial_1\varphi\partial_2\varphi
\big(-\textsf{M}^{21}\frac{n^2}{n^3}
+\textsf{M}^{22}\frac{n^1}{n^3}\partial_z\varphi\big)\big], \\[9pt]

\varsigma_4 = \partial_z\varphi[(1+(\partial_1\varphi)^2)\textsf{M}^{12} + \partial_1\varphi\partial_2\varphi\textsf{M}^{22}], \\[8pt]
\varsigma_5 = -(1+(\partial_1\varphi)^2)\textsf{M}^{11} - \partial_1\varphi\partial_2\varphi \textsf{M}^{21}, \\[8pt]

\varsigma_6 = (1+(\partial_1\varphi)^2)
\big(-\textsf{M}^{11}\frac{n^2}{n^3}
+\textsf{M}^{12}\frac{n^1}{n^3}\partial_z\varphi\big) \\[7pt]\hspace{0.85cm}
+ \partial_1\varphi\partial_2\varphi
\big(-\textsf{M}^{21}\frac{n^2}{n^3}
+\textsf{M}^{22}\frac{n^1}{n^3}\partial_z\varphi\big).
\end{array}
\end{equation}

If $|\varsigma_1\Theta_v^1 + \varsigma_2\Theta_v^2 + \varsigma_3\Theta_v^3|_{\infty}
= |\varsigma_4\Theta_v^1 + \varsigma_5\Theta_v^2 + \varsigma_6\Theta_v^3|_{\infty} =0$ and $|\varsigma_1\Theta_v^1 + \varsigma_2\Theta_v^2 + \varsigma_3\Theta_v^3|_{\infty}
= |\varsigma_4\Theta_v^1 + \varsigma_5\Theta_v^2 + \varsigma_6\Theta_v^3|_{\infty}=0$,
then
\begin{equation}
\left\{\begin{array}{ll}
\partial_z v^1 = f^5[\nabla\varphi](\partial_j v^i),\ j=1,2,\ i=1,2,3,\\[8pt]

\partial_z v^2 = f^6[\nabla\varphi](\partial_j v^i),\ j=1,2,\ i=1,2,3,
\end{array}\right.
\end{equation}
and
\begin{equation}
\left\{\begin{array}{ll}
\partial_z b^1 = f^5[\nabla\varphi](\partial_j b^i)=0,\ j=1,2,\ i=1,2,3,\\[8pt]

\partial_z b^2 = f^6[\nabla\varphi](\partial_j b^i)=0,\ j=1,2,\ i=1,2,3.
\end{array}\right.
\end{equation}
These imply $\mathcal{S}^{\varphi}v \nn\times\nn =0$
and $\mathcal{S}^{\varphi}b \nn\times\nn =0$, which contradicts our assumption.
Therefore, one of $|\varsigma_1\Theta_v^1 + \varsigma_2\Theta_v^2 + \varsigma_3\Theta_v^3|_{\infty}$,
$|\varsigma_4\Theta_v^1 + \varsigma_5\Theta_v^2 + \varsigma_6\Theta_v^3|_{\infty}$,
$|\varsigma_1\Theta_b^1 + \varsigma_2\Theta_b^2 + \varsigma_3\Theta_b^3|_{\infty}$
and $|\varsigma_4\Theta_v^1 + \varsigma_5\Theta_v^2 + \varsigma_6\Theta_v^3|_{\infty}$
must be nonzero.
Without lose of generality,
we assume $|\varsigma_1\Theta_v^1+\varsigma_2\Theta_v^2+\varsigma_3\Theta_v^3|_{\infty}\neq 0$
or $|\varsigma_1\Theta_b^1 + \varsigma_2\Theta_b^2+\varsigma_3\Theta_b^3|_{\infty}\neq 0$.

As $\e\to 0$, we have the following strong convergence
\begin{equation}
\begin{array}{ll}
|\textsf{F}^1 [\nabla\varphi^{\e}](\partial_j v^{\e,i})-\textsf{F}^1
[\nabla\varphi](\partial_j v^i)|_{L^{\infty}}\to 0,
\\[8pt]

|\textsf{F}^1 [\nabla\varphi^{\e}](\partial_j b^{\e,i})-\textsf{F}^1
[\nabla\varphi](\partial_j b^i)|_{L^{\infty}}\to 0,
\end{array}
\end{equation}
due to enough uniform regularities in co-normal Sobolev space of MHD solutions and its tangential derivatives by \cite{Lee17}.
Thus, if $\e$ is sufficiently small,
$|\textsf{F}^1 [\nabla\varphi^{\e}](\partial_j v^{\e,i}) - \textsf{F}^1 [\nabla\varphi](\partial_j v^i)|_{\infty}
\leq \frac{1}{2}|\varsigma_1\Theta_v^1 + \varsigma_2\Theta_v^2 + \varsigma_3\Theta_v^3|_{\infty}$
and
$|\textsf{F}^1 [\nabla\varphi^{\e}](\partial_j b^{\e,i}) - \textsf{F}^1 [\nabla\varphi](\partial_j b^i)|_{\infty}=0$.
One has
\begin{equation}
\begin{array}{ll}
|\omega_v^{\e,1} -\omega_v^1|_{\infty} \geq |\varsigma_1\Theta_v^1 + \varsigma_2\Theta_v^2 + \varsigma_3\Theta_v^3|_{\infty}
- \big|\textsf{F}^1 [\nabla\varphi^{\e}](\partial_j v^{\e,i}) - \textsf{F}^1 [\nabla\varphi](\partial_j v^i)\big|_{\infty} \\[8pt]\hspace{2cm}

\geq \frac{1}{2}|\varsigma_1\Theta_v^1 + \varsigma_2\Theta_v^2 + \varsigma_3\Theta_v^3|_{\infty},
\\[8pt]

|\omega_b^{\e,1} -\omega_b^1|_{\infty} \geq |\varsigma_1\Theta_b^1 + \varsigma_2\Theta_b^2 + \varsigma_3\Theta_b^3|_{\infty}
- \big|\textsf{F}^1 [\nabla\varphi^{\e}](\partial_j b^{\e,i}) - \textsf{F}^1 [\nabla\varphi](\partial_j b^i)\big|_{\infty} \\[8pt]\hspace{2cm}

\geq \frac{1}{2}|\varsigma_1\Theta_b^1 + \varsigma_2\Theta_b^2 + \varsigma_3\Theta_b^3|_{\infty}.
\end{array}
\end{equation}
Then
\begin{equation}
\begin{array}{ll}
|\omega_{vh}^{\e,b} -\omega_{vh}^b|_{L^{\infty}(\mathbb{R}^2\times (0,T])}
\geq \max\{ |\omega_v^{\e,1} -\omega_v^1|_{\infty} ,\, |\omega_v^{\e,2} -\omega_v^2|_{\infty}\}
\\[8pt]

\geq \frac{1}{2}\max\{|\varsigma_1\Theta_v^1 + \varsigma_2\Theta_v^2 + \varsigma_3\Theta_v^3|_{\infty}, \,
|\varsigma_4\Theta_v^1 + \varsigma_5\Theta_v^2 + \varsigma_6\Theta_v^3|_{\infty}\} >0,
\\[8pt]

|\omega_{bh}^{\e,b} -\omega_{bh}^b|_{L^{\infty}(\mathbb{R}^2\times (0,T])}
\geq \max\{ |\omega_b^{\e,1} -\omega_b^1|_{\infty} ,\, |\omega_b^{\e,2} -\omega_b^2|_{\infty}\}
\\[8pt]

\geq \max\{|\varsigma_1\Theta_b^1 + \varsigma_2\Theta_b^2 + \varsigma_3\Theta_b^3|_{\infty}, \,
|\varsigma_4\Theta_b^1 + \varsigma_5\Theta_b^2 + \varsigma_6\Theta_b^3|_{\infty}\} >0.
\end{array}
\end{equation}
Lemma \ref{lemma4.1} is proved.
\end{proof}

\begin{remark}
In the process of establishing the convergence rate estimate,
we only give the case either $\Pi\mathcal{S}^{\varphi} v\nn|_{z=0} \neq 0$
and $\Pi\mathcal{S}^{\varphi} b\nn|_{z=0} \neq 0$
or the case
$\Pi\mathcal{S}^{\varphi} v\nn|_{z=0} = 0$
and $\Pi\mathcal{S}^{\varphi} b\nn|_{z=0}= 0$ in $(0,T]$.
Similarly, one can obtain the convergence estimate for other cases,
one is $\Pi\mathcal{S}^{\varphi} v\nn|_{z=0} \neq 0$, $\Pi\mathcal{S}^{\varphi} b\nn|_{z=0} = 0$,
the other is $\Pi\mathcal{S}^{\varphi} v\nn|_{z=0} = 0$, $\Pi\mathcal{S}^{\varphi} b\nn|_{z=0} \neq 0$.
\end{remark}

\subsection{Existence of Strong Vorticity Layer}
In this subsection, we use the $L^\infty$ estimate to prove the
existence of strong vorticity layer
when the ideal MHD boundary value satisfies
$\Pi\mathcal{S}^{\varphi} v\nn|_{z=0} \neq 0$
or $\Pi\mathcal{S}^{\varphi} b\nn|_{z=0} \neq 0$ in $(0,T]$.

As the initial vorticity layer is weak,
the following theorem shows that it exists strong vorticity layer
due to the discrepancy between boundary value of MHD vorticity and boundary value of
ideal MHD vorticity.
\begin{theorem}\label{Theorem4.1}
If the ideal MHD solution satisfies $\Pi\mathcal{S}^{\varphi} v\nn|_{z=0} \neq 0$
or $\Pi\mathcal{S}^{\varphi} b\nn|_{z=0} \neq 0$ in $(0,T]$,
and $\Pi\mathcal{S}^{\varphi} v\nn|_{z=0},\Pi\mathcal{S}^{\varphi} b\nn|_{z=0} \in H^4(\mathbb{R}^2\times[0,T])$,
the initial data satisfies $\lim\limits_{\e\rto 0}(\nabla^{\varphi^{\e}}\times v_0^{\e})
- \nabla^{\varphi}\times\lim\limits_{\e\rto 0} v_0^{\e} = 0$
and $\lim\limits_{\e\rto 0}(\nabla^{\varphi^{\e}}\times b_0^{\e})
- \nabla^{\varphi}\times\lim\limits_{\e\rto 0} b_0^{\e} = 0$,
then $\lim\limits_{\e\rto 0}\|\omega_v^{\e} -\omega_v\|_{L^{\infty}(\mathbb{R}^2\times [0, O(\sqrt{\e}))\times (0,T])} \neq 0$ or $\lim\limits_{\e\rto 0}\|\omega_b^{\e} -\omega_b\|_{L^{\infty}(\mathbb{R}^2\times [0, O(\sqrt{\e}))\times (0,T])} \neq 0$.
\end{theorem}

\begin{proof}
First, we consider the equations $(\ref{Vorticity Layer1})$ with small initial data:
\begin{equation}\label{4.16}
\left\{
\begin{aligned}
&a_0\partial_tW_+-\epsilon\partial_i(a_{i,j}\partial_jW_+)+\gamma a_0W_+-(f^7_v-f^7_b)W_-
=I_{v2},
\\
&b_0\partial_tW_--\epsilon\partial_i(b_{i,j}\partial_jW_-)+\gamma b_0W_--(f^7_v+f^7_b)W_+
=I_{b2},
\\
&W_+|_{z=0}=e^{-\gamma t}(\hat{\omega}^b_{vh}+\hat{\omega}^b_{bh})
=e^{-\gamma t}(\omega^\epsilon_{v0}-\omega_{v0}+\omega^\epsilon_{b0}-\omega_{b0})\nrightarrow0,
\\
& W_-|_{z=0}=e^{-\gamma t}(\hat{\omega}^b_{vh}-\hat{\omega}^b_{bh})
=e^{-\gamma t}(\omega^\epsilon_{v0}-\omega_{v0}-\omega^\epsilon_{b0}+\omega_{b0})\nrightarrow0,\\
&W_+|_{t=0}=\hat{\omega}_{vh,0}+\hat{\omega}_{bh,0}
=\omega^\epsilon_{vh0}-\omega_{vh0}+\omega^\epsilon_{bh0}-\omega_{bh0}\to0,\\
& W_-|_{t=0}=\hat{\omega}_{vh,0}-\hat{\omega}_{bh,0}
=\omega^\epsilon_{vh0}-\omega_{vh0}-\omega^\epsilon_{bh0}+\omega_{bh0}\to0.
\end{aligned}
\right.
\end{equation}

We decompose $W_{\pm} = W_{\pm}^{bdy} + W_{\pm}^{fo}$,
such that $W_{\pm}^{fo}$ satisfy the nonhomogeneous equations:
\begin{equation}\label{fo1}
\left\{
\begin{aligned}
&a_0\partial_tW^{fo}_+-\epsilon\partial_i(a_{i,j}\partial_jW^{fo}_+)
+\gamma a_0W^{fo}_+-(f^7_v-f^7_b)W^{fo}_-
=I_{v2},
\\
&b_0\partial_tW^{fo}_--\epsilon\partial_i(b_{i,j}\partial_jW^{fo}_-)
+\gamma b_0W^{fo}_--(f^7_v+f^7_b)W^{fo}_+
=I_{b2},
\\
&W^{fo}_+|_{z=0}=0, W^{f0}_-|_{z=0}=0,\\
&W^{fo}_+|_{t=0}=\hat{\omega}_{vh,0}+\hat{\omega}_{bh,0}\to0,\\
&W^{fo}_-|_{t=0}=\hat{\omega}_{vh,0}-\hat{\omega}_{bh,0}\to0,
\end{aligned}
\right.
\end{equation}
and $W_{\pm}^{bdy}$ satisfy the following equations:
\begin{equation}\label{bdy1}
\left\{
\begin{aligned}
&a_0\partial_tW^{bdy}_+-\epsilon\partial_i(a_{i,j}\partial_jW^{bdy}_+)
+\gamma a_0W^{bdy}_+-(f^7_v-f^7_b)W^{bdy}_-
=0,
\\
&b_0\partial_tW^{bdy}_--\epsilon\partial_i(b_{i,j}\partial_jW^{bdy}_-)
+\gamma b_0W^{bdy}_--(f^7_v+f^7_b)W^{bdy}_+
=0,
\\
&W^{bdy}_+|_{z=0}=e^{-\gamma t}(\hat{\omega}^b_{vh}+\hat{\omega}^b_{bh})\nrightarrow0,\\
&W^{bdy}_-|_{z=0}=e^{-\gamma t}(\hat{\omega}^b_{vh}-\hat{\omega}^b_{bh})\nrightarrow0,
\\
&W^{bdy}_+|_{t=0}=0, W^{bdy}_-|_{t=0}=0.
\end{aligned}
\right.
\end{equation}

For the diffusion equation \eqref{fo1} with coupled damping term
and homogeneous Dirichlet boundary condition,
the estimate of
$\|W_{\pm}^{f\!o}\|_{L^{\infty}}$
is similar to that of \eqref{ini}.
It follows that
$$\|W_{+}^{f\!o}\|_{L^{\infty}}+\|W_{-}^{f\!o}\|_{L^{\infty}}\geq 0.$$

The homogeneous equations $(\ref{bdy1})$ is a heat equations with coupled damping and nonzero boundary value.
It is different from the case of \eqref{bdy},
since $W^{bdy}_+|_{z=0}\nrightarrow0$ or $W^{bdy}_-|_{z=0}\nrightarrow0$.
By using symbolic analysis,
see \cite{Masmoudi17} for more details, we rewrite $(\ref{bdy1})$ in the following form:
\begin{equation}\label{bdy2}
\left\{\begin{aligned}
&\e \partial_{zz} W_+^{bdy} + \e(\frac{\partial_z a_{33}}{a_{33}}+ \sum\limits_{j=1,2}\frac{\partial_j a_{j3}}{a_{33}})\partial_z W_+^{bdy}
+ 2\e \sum\limits_{j=1,2}\frac{a_{j3}}{a_{33}}\partial_{jz} W_+^{bdy}
\\
&\quad+ \e \sum\limits_{j=1,2}\frac{\partial_z a_{j3}}{a_{33}}\partial_j W_+^{bdy} + \e \sum\limits_{j=1,2}\frac{\partial_i a_{ij}}{a_{33}}\partial_j W_+^{bdy}
+ \e \sum\limits_{j=1,2}\frac{a_{ij}}{a_{33}}\partial_{ij} W_+^{bdy}
\\
&\quad- \frac{a_0}{a_{33}}\partial_t W_+^{bdy}
- \frac{1}{a_{33}}\big(\gamma a_0 W_+^{bdy} - (f^7_v-f^7_b) W_-^{bdy}\big) =0,
\\
&\e \partial_{zz} W_-^{bdy} + \e(\frac{\partial_z b_{33}}{b_{33}}
+ \sum\limits_{j=1,2}\frac{\partial_j b_{j3}}{b_{33}})\partial_z W_-^{bdy}
+ 2\e \sum\limits_{j=1,2}\frac{a_{j3}}{b_{33}}\partial_{jz} W_-^{bdy}
\\
&\quad+ \e \sum\limits_{j=1,2}\frac{\partial_z b_{j3}}{b_{33}}\partial_j W_-^{bdy}
+ \e \sum\limits_{j=1,2}\frac{\partial_i b_{ij}}{b_{33}}\partial_j W_-^{bdy}
+ \e \sum\limits_{j=1,2}\frac{b_{ij}}{b_{33}}\partial_{ij} W_-^{bdy}
\\
&\quad- \frac{b_0}{b_{33}}\partial_t W_-^{bdy}
- \frac{1}{a_{33}}\big(\gamma b_0W_-^{bdy} - (f^7_v+f^7_b) W_+^{bdy}\big) =0,
\\
&W^{bdy}_+|_{z=0}=e^{-\gamma t}(\hat{\omega}^b_{vh}+\hat{\omega}^b_{bh})\nrightarrow0,
\\
&W^{bdy}_-|_{z=0}=e^{-\gamma t}(\hat{\omega}^b_{vh}-\hat{\omega}^b_{bh})\nrightarrow0,
\\
&W^{bdy}_+|_{t=0}=0, W^{bdy}_-|_{t=0}=0.
\end{aligned}\right.
\end{equation}

We define the following Fourier transformation
with respect to $(t,y)\in\mathbb{R}_{+}\times\mathbb{R}^2$,
\begin{equation*}
\begin{array}{ll}
\mathcal{F}[W](\tau,\xi,z) = \int_0^{+\infty}\int_{\mathbb{R}^3_{-}}
e^{-i \tau t - i\xi\cdot y} W(t,y,z) \,\mathrm{d}t\mathrm{d}y.
\end{array}
\end{equation*}

Taking $z$ as a parameter,
denote that $\tilde{W}_\pm^{bdy}=\mathcal{F}[e^{-\gamma t}(\hat{\omega}^{bdy}_{vh}\pm\hat{\omega}^{bdy}_{bh})]$,
then the symbolic version of $(\ref{bdy2})$ becomes
\begin{equation}\label{4.20}
\left\{\begin{aligned}
&\e \partial_{zz} \tilde{W}_+^{bdy} + A_{v1} \sqrt{\e} \partial_z \tilde{W}_+^{bdy}
+ A_{v3} \tilde{W}_+^{bdy}= A_{v2}\tilde{W}_-^{bdy},
\\
&\e \partial_{zz} \tilde{W}_-^{bdy} + A_{b1} \sqrt{\e} \partial_z \tilde{W}_-^{bdy}
+ A_{b3} \tilde{W}_-^{bdy}= A_{b2} \tilde{W}_+^{bdy},
\\
&\tilde{W}_+^{bdy}|_{z=0} = \mathcal{F}[e^{-\gamma t}(\hat{\omega}^b_{vh}+\hat{\omega}^b_{bh})]
\nrightarrow 0,
\\
&\tilde{W}_-^{bdy}|_{z=0} = \mathcal{F}[e^{-\gamma t}(\hat{\omega}^b_{vh}-\hat{\omega}^b_{bh})]
\nrightarrow 0,
\\
&\tilde{W}_+^{bdy}|_{t=0} = 0, \tilde{W}_-^{bdy}|_{t=0} = 0.
\end{aligned}\right.
\end{equation}
where the Fourier multipliers are as follows:
\begin{equation}
\begin{aligned}
&A_{v1} = \sqrt{\e}(\frac{\partial_z a_{33}}{a_{33}}+ \sum\limits_{j=1,2}\frac{\partial_j a_{j3}}{a_{33}}
+ 2i \sum\limits_{j=1,2}\frac{a_{j3}}{a_{33}} \xi^j),
\\
&A_{v0} = i\e \sum\limits_{j=1,2}\frac{\partial_z a_{j3}}{a_{33}}\xi^j + i\e \sum\limits_{j=1,2}\frac{\partial_i a_{ij}}{a_{33}}\xi^j
- \e \sum\limits_{j=1,2}\frac{a_{ij}}{a_{33}}\xi^i\xi^j
- i\tau \frac{a_0}{a_{33}}
- \frac{1}{a_{33}}\gamma a_0,
\\
&A_{v2} =\frac{1}{a_{33}}(f^7_v-f^7_b),
\end{aligned}
\end{equation}
\begin{equation}
\begin{aligned}
&A_{b1} = \sqrt{\e}(\frac{\partial_z b_{33}}{b_{33}}+ \sum\limits_{j=1,2}\frac{\partial_j b_{j3}}{b_{33}}
+ 2i \sum\limits_{j=1,2}\frac{b_{j3}}{b_{33}} \xi^j),
\\
&A_{b0} = i\e \sum\limits_{j=1,2}\frac{\partial_z b_{j3}}{b_{33}}\xi^j + i\e \sum\limits_{j=1,2}\frac{\partial_i b_{ij}}{b_{33}}\xi^j
- \e \sum\limits_{j=1,2}\frac{b_{ij}}{b_{33}}\xi^i\xi^j
- i\tau \frac{b_0}{b_{33}}
- \frac{1}{b_{33}}\gamma b_0,
\\
&A_{b2} =\frac{1}{b_{33}}(f^7_v+f^7_b),
\end{aligned}
\end{equation}
Due to $|a_0| + |a_{ij}| + \sqrt{\e}|\partial_z a_{ij}| \leq C$,
$|b_0| + |b_{ij}| + \sqrt{\e}|\partial_z b_{ij}| \leq C$, for some $C>0$,
we can neglect the dependence in $t, y$ in the coefficient of \eqref{bdy2}.
As $\e\rto 0$, one has
\begin{equation*}
\begin{aligned}
&A_{v1} \rto \sqrt{\e}\frac{\partial_z a_{33}}{a_{33}},
A_{b1} \rto \sqrt{\e}\frac{\partial_z b_{33}}{b_{33}},
\\
&- A_{v0} \rto \frac{1}{a_{33}}\gamma a_0
+ i\tau \frac{a_0}{a_{33}} - i\e \sum\limits_{j=1,2}\frac{\partial_z a_{j3}}{a_{33}}\xi^j,
\\
&- A_{b0} \rto \frac{1}{a_{33}}\gamma b_0
+ i\tau \frac{b_0}{b_{33}} - i\e \sum\limits_{j=1,2}\frac{\partial_z b_{j3}}{a_{33}}\xi^j,
\\
&A_{v2} \rto \frac{1}{a_{33}}(f^7_v-f^7_b),
A_{b2} \rto \frac{1}{b_{33}}(f^7_v+f^7_b).
\end{aligned}
\end{equation*}
When $\e>0$ is sufficiently small, we can treat the values of
$A_{v1}$, $A_{b1}$, $A_{v3}$ and $A_{b3}$ as their limits.

It follows from \eqref{4.20} that
\begin{equation}\label{bdyz}
\begin{aligned}
&\e \partial_{zz} \tilde{W}_+^{bdy} + A_{v1} \sqrt{\e} \partial_z \tilde{W}_+^{bdy}
+ A_{v0} \tilde{W}_+^{bdy}= A_{v2} \tilde{W}_-^{bdy},
\\[6pt]
&\e \partial_{zz} \tilde{W}_-^{bdy} + A_{b1} \sqrt{\e} \partial_z \tilde{W}_-^{bdy}
+ A_{b0} \tilde{W}_-^{bdy}= A_{b2} \tilde{W}_+^{bdy},
\end{aligned}
\end{equation}
which is a coupled ODE system.
The solution of \eqref{bdyz} is that
\begin{equation}\label{bdys}
\left\{\begin{array}{ll}
\tilde{W}_+^{bdy} = \exp\{\frac{- A_{v1} + \sqrt{A_{v1}^2
-4 A_{v0}}}{2} \frac{z}{\sqrt{\e}}\}\tilde{W}_+^{bdy}|_{z=0}
+\int_0^z\Psi_v(z,s)A_{v2} \tilde{W}_-^{bdy}\mathrm{d}s,
\\[12pt]

\tilde{W}_-^{bdy} = \exp\{\frac{- A_{b1} + \sqrt{A_{b1}^2 -4 A_{b0}}}{2}
\frac{z}{\sqrt{\e}}\}\tilde{W}_-^{bdy}|_{z=0}
+\int_0^z\Psi_b(z,s)A_{b2} \tilde{W}_+^{bdy}\mathrm{d}s,
\end{array}\right.
\end{equation}
where $\Psi_v(z,s)$, $\Psi_b(z,s)$ depends on
the solutions of homogeneous equations of \eqref{bdyz}.

When $\e$ is sufficiently small, one of the complex numbers, at least,
$\sqrt{A_{v1}^2 -4 A_{v0}}$ and $\sqrt{A_{b1}^2 -4 A_{b0}}$
always has positive real part since $\Re (A_{v1}^2 -4 A_{v0})>0$,
$\Re (A_{b1}^2 -4 A_{b0})>0$
where $\Re$ represents the real part. Then we choose this one.
It follows that $|\Re \sqrt{A_{v1}^2 -4 A_{v0}}| >| - A_{v1}|$,
$|\Re \sqrt{A_{b1}^2 -4 A_{b0}}| >| - A_{b1}|$
and $\big\|\frac{- A_{v1} + \sqrt{A_{v1}^2 -4 A_{v0}}}{2}\big\|_{L^{\infty}}<+\infty$,
$\big\|\frac{- A_{b1} + \sqrt{A_{b1}^2 -4 A_{b0}}}{2}\big\|_{L^{\infty}}<+\infty$.
Then, one has $\|\int_0^z\Psi_v(z,s)A_{v2}\mathrm{d}s\|_{L^\infty}\leq+\infty$,
and $\|\int_0^z\Psi_b(z,s)A_{b2}\mathrm{d}s\|_{L^\infty}\leq+\infty$.

If $z =O(\e^{\frac{1}{2} +\delta_z})$, where $\delta_z \geq0$,
for example, we simply assume
$z = - \e^{\frac{1}{2} +\delta_z}$,
as $\e\rto 0$,
$|\exp\{\frac{- A_{v1} + \sqrt{A_{v1}^2-4 A_{v0}}}{2} \frac{z}{\sqrt{\e}}\}|\rto c_1$,
$|\exp\{\frac{- A_{b1} + \sqrt{A_{b1}^2-4 A_{b0}}}{2} \frac{z}{\sqrt{\e}}\}|\rto c_2$,
 where $c_i>0, i=1, 2$.
Hence,
$\|\exp\{\frac{- A_{v1} + \sqrt{A_{v1}^2
-4 A_{v0}}}{2} \frac{z}{\sqrt{\e}}\}\tilde{W}_+^{bdy}|_{z=0}\|_{L^\infty} \nrightarrow 0$,
$\|\exp\{\frac{- A_{b1} + \sqrt{A_{b1}^2
-4 A_{b0}}}{2} \frac{z}{\sqrt{\e}}\}\tilde{W}_+^{bdy}|_{z=0}\|_{L^\infty} \nrightarrow 0$.

If $z =O(\e^{\frac{1}{2} -\delta_z})$, where $\delta_z <0$,
for example, we simply assume
$z = - \e^{\frac{1}{2} -\delta_z}$,
as $\e\rto 0$,
$|\exp\{\frac{- A_{v1} + \sqrt{A_{v1}^2-4 A_{v0}}}{2} \frac{z}{\sqrt{\e}}\}|\rto 0$,
$|\exp\{\frac{- A_{b1} + \sqrt{A_{b1}^2-4 A_{b0}}}{2} \frac{z}{\sqrt{\e}}\}|\rto 0$.
Therefore, we obtain the results that
$\|\exp\{\frac{- A_{v1} + \sqrt{A_{v1}^2
-4 A_{v0}}}{2} \frac{z}{\sqrt{\e}}\}\tilde{W}_+^{bdy}|_{z=0}\|_{L^\infty} \rto 0$,
$\|\exp\{\frac{- A_{b1} + \sqrt{A_{b1}^2
-4 A_{b0}}}{2} \frac{z}{\sqrt{\e}}\}\tilde{W}_+^{bdy}|_{z=0}\|_{L^\infty} \rto 0$.

Applying $L^\infty$ to \eqref{bdys}, then
\begin{equation}
\begin{array}{ll}
\|\tilde{W}_+^{bdy}\|_{L^\infty} = \|\exp\{\frac{- A_{v1} + \sqrt{A_{v1}^2
-4 A_{v0}}}{2} \frac{z}{\sqrt{\e}}\}\tilde{W}_+^{bdy}|_{z=0}
+\int_0^z\Psi_v(z,s)A_{v2} \tilde{W}_-^{bdy}\mathrm{d}s\|_{L^\infty}
\\[12pt]

\geq \|\exp\{\frac{- A_{v1} + \sqrt{A_{v1}^2
-4 A_{v0}}}{2} \frac{z}{\sqrt{\e}}\}\tilde{W}_+^{bdy}|_{z=0}\|_{L^\infty}
-\|\int_0^z\Psi_v(z,s)A_{v2} \mathrm{d}s\|_{L^\infty}
\|\tilde{W}_-^{bdy}\|_{L^\infty},
\\[12pt]

\|\tilde{W}_-^{bdy}\|_{L^\infty} = \|\exp\{\frac{- A_{b1} + \sqrt{A_{b1}^2 -4 A_{b0}}}{2}
\frac{z}{\sqrt{\e}}\}\tilde{W}_-^{bdy}|_{z=0}
+\int_0^z\Psi_b(z,s)A_{b2} \tilde{W}_+^{bdy}\mathrm{d}s\|_{L^\infty}
\\[12pt]

\geq \|\exp\{\frac{- A_{b1} + \sqrt{A_{b1}^2 -4 A_{b0}}}{2}
\frac{z}{\sqrt{\e}}\}\tilde{W}_-^{bdy}|_{z=0}\|_{L^\infty}
-\|\int_0^z\Psi_b(z,s)A_{b2} \mathrm{d}s\|_{L^\infty}
\|\tilde{W}_+^{bdy}\|_{L^\infty}.
\end{array}
\end{equation}
Therefore,
\begin{equation}
\begin{aligned}
A_{v4}\|\tilde{W}_+^{bdy}\|_{L^\infty}
+A_{b4}\|\tilde{W}_-^{bdy}\|_{L^\infty}
\geq& \|\exp\big\{\frac{- A_{v1} + \sqrt{A_{v1}^2
-4 A_{v0}}}{2} \frac{z}{\sqrt{\e}}\big\}\tilde{W}_+^{bdy}|_{z=0}\|_{L^\infty}
\\
&+ \|\exp\big\{\frac{- A_{b1} + \sqrt{A_{b1}^2 -4 A_{b0}}}{2}
\frac{z}{\sqrt{\e}}\big\}\tilde{W}_-^{bdy}|_{z=0}\|_{L^\infty},
\end{aligned}
\end{equation}
where $A_{v4}=1+\|\int_0^z\Psi_b(z,s)A_{b2} \mathrm{d}s\|_{L^\infty}$,
$A_{v4}=1+\|\int_0^z\Psi_v(z,s)A_{v2} \mathrm{d}s\|_{L^\infty}$.

Therefore, we have $\|W_+^{bdy}\|_{L^\infty}+\|W_-^{bdy}\|_{L^\infty}>0$,
which implies that $\lim\limits_{\e\rto 0}\|W_+^{bdy}\|_{L^{\infty}} \neq 0$
or $\lim\limits_{\e\rto 0}\|W_-^{bdy}\|_{L^{\infty}} \neq 0$ holds
in some set located in the interior.
This completes the proof of Theorem \ref{Theorem4.1}.
\end{proof}

\section{The Convergence Rates Estimates for $\sigma=0$}
In this section,
we investigate the asymptotic behavior of the solutions to the
free boundary problems for MHD systems \eqref{MHDF} with $\sigma=0$ as $\e\to0$.
Denote by $\hat{v} =v^{\e} -v,\hat{b} =b^{\e} -b,\hat{q} =q^{\e} -q, \hat{h} =h^{\e} -h$,
and denote the $i-$th components of $f^{\e}$ and $f$ by $f^{\e,i}$ and $f^i$ respectively.

\subsection{Estimates for the Tangential Derivatives}
In order to estimate the tangential derivatives of $\hat{v}$ and $\hat{b}$,
that is, to bound $\|\partial_t^{\ell} \hat{v}\|_{L^2}$, $\|\partial_t^{\ell} \hat{b}\|_{L^2}$,
$\sqrt{\e}\|\nabla \partial_t^{\ell}\mathcal{Z}^{\alpha}\hat{v}\|_{L^2}$,
$\sqrt{\e}\|\nabla \partial_t^{\ell}\mathcal{Z}^{\alpha}\hat{b}\|_{L^2}$,
$\sqrt{\e}\|\mathcal{S}^{\varphi}\partial_t^{\ell}\mathcal{Z}^{\alpha} \hat{v}\|_{L^2}$
and
$\sqrt{\e}\|\mathcal{S}^{\varphi}\partial_t^{\ell}\mathcal{Z}^{\alpha} \hat{b}\|_{L^2}$, we need the following two preliminary lemmas of
$\hat{h}$ by using the kinetical boundary condition.

\begin{lemma}[\cite{Wu16}]\label{Lemma5.1}
Assume $0\leq k\leq m-2$, $0\leq\ell\leq k-1$. We have
\begin{equation}\label{Sect5_Height_Estimates_Lemma_Eq}
\begin{array}{ll}
|\partial_t^{\ell}\hat{h}|_{L^2}^2
\lem |\hat{h}_0|_{X^{k-1}}^2
+ \int_0^t |\hat{h}|_{X^{k-1,1}}^2 + \|\hat{v}\|_{X^{k-1,1}}^2 \,\mathrm{d}t
+ \|\partial_z\hat{v}\|_{L^4([0,T],X^{k-1})}^2.
\end{array}
\end{equation}
\end{lemma}

\begin{lemma}[\cite{Wu16}]\label{Lemma5.2}
Assume $0\leq k\leq m-2$, $0\leq\ell\leq k-1$, $\ell+|\alpha| \leq k$. We have
\begin{equation}\label{Sect5_Height_Viscous_Estimates_Lemma_Eq}
\begin{array}{ll}
\e|\hat{h}|_{X^{k-1,\frac{3}{2}}}^2 \leq \e|\hat{h}_0|_{X^{k-1,\frac{3}{2}}}^2 + \int_0^t|\hat{h}|_{X^{k-1,1}}^2
+ \e\|\nabla \hat{v}\|_{X^{k-1,1}}^2 \,\mathrm{d}t.
\end{array}
\end{equation}
\end{lemma}

We show that $\partial_z \hat{v}^3$, $\partial_z \hat{b}^3$
 can be estimated by $\partial_z \hat{v}_h$,$\partial_z \hat{b}_h$, that is
\begin{align*}
&\|\partial_z \hat{v}^3\|_{X^s} \lem \|\hat{v}_h\|_{X^{s,1}} + \|\partial_z \hat{v}_h\|_{X^s}
+ |\hat{h}|_{X^{s,\frac{1}{2}}}, \\
&\|\partial_z \hat{b}^3\|_{X^s} \lem \|\hat{b}_h\|_{X^{s,1}} + \|\partial_z \hat{b}_h\|_{X^s}
+ |\hat{b}|_{X^{s,\frac{1}{2}}}.
\end{align*}
In fact, we have from the divergence free condition that
\begin{equation}
\left\{\begin{aligned}
\partial_z \hat{v}^3 =& -\partial_z\varphi^{\e}(\partial_1 \hat{v}^1 + \partial_2 \hat{v}^2)
- \partial_z\hat{\varphi}(\partial_1 v^1 + \partial_2 v^2)
\\
&+\partial_1\varphi^{\e}\partial_z \hat{v}^1 +\partial_1\hat{\varphi}\partial_z v^1
+ \partial_2\varphi^{\e}\partial_z \hat{v}^2 +\partial_2\hat{\varphi}\partial_z v^2,
\\
\partial_z \hat{b}^3 =& -\partial_z\varphi^{\e}(\partial_1 \hat{b}^1 + \partial_2 \hat{b}^2)
- \partial_z\hat{\varphi}(\partial_1 b^1 + \partial_2 b^2)
\\
&+\partial_1\varphi^{\e}\partial_z \hat{b}^1 +\partial_1\hat{\varphi}\partial_z b^1
+ \partial_2\varphi^{\e}\partial_z \hat{b}^2 +\partial_2\hat{\varphi}\partial_z b^2.
\end{aligned}\right.
\end{equation}

Next we show the estimates for the tangential derivatives, $\partial_t^{\ell}\mathcal{Z}^{\alpha}\hat{v}$, $\partial_t^{\ell}\mathcal{Z}^{\alpha}\hat{b}$
 and $\partial_t^{\ell}\mathcal{Z}^{\alpha}\hat{h}$.
\begin{lemma}\label{Lemma5.3}
Assume $0\leq k\leq m-2$,
$\hat{v},\hat{b},\hat{h}$ satisfy system \eqref{difference1}.
Then we can estimate  $\partial_t^{\ell}\mathcal{Z}^{\alpha}\hat{v}$,
$\partial_t^{\ell}\mathcal{Z}^{\alpha}\hat{b}$
 and $\partial_t^{\ell}\mathcal{Z}^{\alpha}\hat{h}$ as follows
\begin{equation}\label{5.4}
\begin{aligned}
&\|\hat{v}\|_{X^{k-1,1}}^2 +\|\hat{b}\|_{X^{k-1,1}}^2 + |\hat{h}|_{X^{k-1,1}}^2
+ \e|\hat{h}|_{X^{k-1,\frac{3}{2}}}^2\\
&+ \e\int_0^t\|\nabla\hat{v}\|_{X^{k-1,1}}^2
+ \e\int_0^t\|\nabla\hat{b}\|_{X^{k-1,1}}^2 \,\mathrm{d}t
\\
\lem &\|\hat{v}_0\|_{X^{k-1,1}}^2 +\|\hat{b}_0\|_{X^{k-1,1}}^2 + |\hat{h}_0|_{X^{k-1,1}}^2 + \e|\hat{h}_0|_{X^{k-1,\frac{3}{2}}}^2
+\|\partial_z \hat{v}\|_{L^4([0,T],X^{k-1})}^2
\\
&+\|\partial_z \hat{b}\|_{L^4([0,T],X^{k-1})}^2
+ |\partial_t^k\hat{h}|_{L^4([0,T],L^2)}^2
+ \|\nabla\hat{q}\|_{L^4([0,T],X^{k-1})}^2 + O(\e).
\end{aligned}
\end{equation}
\end{lemma}

\begin{proof}
Applying the operator $\partial_t^{\ell}\mathcal{Z}^{\alpha}$ to the equations \eqref{difference1}
and using the following formulas
\begin{equation*}
\begin{array}{ll}
\partial_i^{\varphi^{\e}}v^{\e} - \partial_i^{\varphi} f
= \partial_i^{\varphi^{\e}}\hat{f}
- (\frac{\partial_i \varphi^{\e}}{\partial_z \varphi^{\e}} - \frac{\partial_i \varphi}{\partial_z \varphi})\partial_z f

= \partial_i^{\varphi^{\e}}\hat{f}
+ (\partial_i \varphi - \frac{\partial_i \varphi^{\e}}{\partial_z \varphi^{\e}}\partial_z \varphi)\frac{1}{\partial_z \varphi}\partial_z f
\\[8pt]\hspace{1.95cm}

= \partial_i^{\varphi^{\e}}\hat{f}
+ \partial_z^{\varphi} f \partial_i^{\varphi^{\e}}\varphi
= \partial_i^{\varphi^{\e}}\hat{f}
+ \partial_z^{\varphi} f \partial_i^{\varphi^{\e}}\varphi -\partial_z^{\varphi} f \partial_i^{\varphi^{\e}}\varphi^{\e} \\[8pt]\hspace{1.95cm}

= \partial_i^{\varphi^{\e}}\hat{f}
-\partial_z^{\varphi} f \partial_i^{\varphi^{\e}}\hat{\varphi}

= \partial_i^{\varphi^{\e}}\hat{f}
- \partial_i^{\varphi^{\e}}\hat{\eta} \, \partial_z^{\varphi} f,
\hspace{0.5cm} i=t,1,2,\\[11pt]

\partial_z^{\varphi^{\e}}f^{\e} - \partial_z^{\varphi}f
= \partial_z^{\varphi^{\e}} \hat{f}
+ (\frac{1}{\partial_z \varphi^{\e}} - \frac{1}{\partial_z \varphi}) \partial_z f \\[8pt]\hspace{1.95cm}

= \partial_z^{\varphi^{\e}} \hat{f}
+ (\frac{1}{\partial_z \varphi^{\e}}\partial_z \varphi - 1) \frac{1}{\partial_z \varphi}\partial_z f
= \partial_z^{\varphi^{\e}} \hat{f}
+ (\partial_z^{\varphi^{\e}} \varphi - 1) \frac{1}{\partial_z \varphi}\partial_z f \\[8pt]\hspace{1.95cm}

= \partial_z^{\varphi^{\e}} \hat{f}
+ (\partial_z^{\varphi^{\e}} \varphi - \partial_z^{\varphi^{\e}} \varphi^{\e}) \partial_z^{\varphi} f \\[8pt]\hspace{1.95cm}

= \partial_z^{\varphi^{\e}} \hat{f}
- \partial_z^{\varphi} v\partial_z^{\varphi^{\e}}\hat{\varphi}

= \partial_z^{\varphi^{\e}} \hat{f}
- \partial_z^{\varphi^{\e}}\hat{\eta} \, \partial_z^{\varphi} f,
\end{array}
\end{equation*}
we have the following equations for $(\hat{V}^{\ell,\alpha},\hat{B}^{\ell,\alpha}, \hat{Q}^{\ell,\alpha})$
\begin{equation}\label{5.5}
\left\{\begin{array}{ll}
\partial_t^{\varphi^{\e}} \hat{V}^{\ell,\alpha}
+ v^{\e} \cdot\nabla^{\varphi^{\e}} \hat{V}^{\ell,\alpha}
-b^{\e} \cdot\nabla^{\varphi^{\e}} \hat{B}^{\ell,\alpha}
+ \nabla^{\varphi^{\e}} \hat{Q}^{\ell,\alpha}
- 2\e\nabla^{\varphi^{\e}}\cdot\mathcal{S}^{\varphi^{\e}} \partial_t^{\ell}\mathcal{Z}^{\alpha}\hat{v}
\\[8pt]\quad

= 2\e[\partial_t^{\ell}\mathcal{Z}^{\alpha},\nabla^{\varphi^{\e}}\cdot]\mathcal{S}^{\varphi^{\e}} \hat{v}
+ 2\e\nabla^{\varphi^{\e}}\cdot[\partial_t^{\ell}\mathcal{Z}^{\alpha},\mathcal{S}^{\varphi^{\e}}] \hat{v}
+ \e\partial_t^{\ell}\mathcal{Z}^{\alpha}\Delta^{\varphi^\e} v
\\[8pt]\quad

- \partial_t^{\ell}\mathcal{Z}^{\alpha}\hat{\varphi}\partial_t^{\varphi^{\e}}\partial_z^{\varphi} v
-  \partial_t^{\ell}\mathcal{Z}^{\alpha}\hat{\varphi}\, v^{\e}\cdot \nabla^{\varphi^{\e}}\partial_z^{\varphi} v
- \partial_t^{\ell}\mathcal{Z}^{\alpha}\hat{v}\cdot\nabla^{\varphi} v
- \partial_t^{\ell}\mathcal{Z}^{\alpha}\hat{\varphi} \nabla^{\varphi^{\e}}\partial_z^{\varphi} q
\\[8pt]\quad

- [\partial_t^{\ell}\mathcal{Z}^{\alpha},\partial_t^{\varphi^{\e}}]\hat{v}
+ [\partial_t^{\ell}\mathcal{Z}^{\alpha}, \partial_z^{\varphi} v \partial_t^{\varphi^{\e}}]\hat{\varphi}
- [\partial_t^{\ell}\mathcal{Z}^{\alpha}, v^{\e} \cdot\nabla^{\varphi^{\e}}] \hat{v}
- [\partial_t^{\ell}\mathcal{Z}^{\alpha}, \nabla^{\varphi} v\cdot]\hat{v}
\\[8pt]\quad

+ [\partial_t^{\ell}\mathcal{Z}^{\alpha}, \partial_z^{\varphi} v \, v^{\e}\cdot \nabla^{\varphi^{\e}}]\hat{\varphi}
- [\partial_t^{\ell}\mathcal{Z}^{\alpha},\nabla^{\varphi^{\e}}] \hat{q}
+ [\partial_t^{\ell}\mathcal{Z}^{\alpha},\partial_z^{\varphi} q\nabla^{\varphi^{\e}}]\hat{\varphi}
\\
\end{array}\right.
\end{equation}
with the following initial boundary conditions
\begin{equation}\label{IBVD}
\left\{\begin{array}{ll}
\partial_t \partial_t^{\ell}\mathcal{Z}^{\alpha}\hat{h} + v_y^{\e}\cdot \nabla_y \partial_t^{\ell}\mathcal{Z}^{\alpha}\hat{h}
- \NN^{\e}\cdot \hat{V}^{\ell,\alpha}
= \NN^{\e}\cdot \partial_z^{\varphi} v \partial_t^{\ell}\mathcal{Z}^{\alpha}\hat{\eta}
\\[8pt]\quad

-\hat{v}_y \cdot \nabla_y \partial_t^{\ell}\mathcal{Z}^{\alpha} h
-\partial_y \hat{h}\cdot \partial_t^{\ell}\mathcal{Z}^{\alpha}v_y
+[\partial_t^{\ell}\mathcal{Z}^{\alpha}, \hat{v},\NN^{\e}]
-[\partial_t^{\ell}\mathcal{Z}^{\alpha}, v_y, \partial_y \hat{h}],
\\[12pt]

\hat{Q}^{\ell,\alpha}\NN^{\e}
-2\e \mathcal{S}^{\varphi^{\e}} \partial_t^{\ell}\mathcal{Z}^{\alpha}\hat{v}\,\NN^{\e}
- (g-\partial_z^{\varphi}q)\partial_t^{\ell}\mathcal{Z}^{\alpha}\hat{h} \,\NN^{\e}
\\[8pt]\quad

= 2\e [\partial_t^{\ell}\mathcal{Z}^{\alpha},\mathcal{S}^{\varphi^{\e}}] v^{\e}\,\NN^{\e}
+ (2\e \mathcal{S}^{\varphi^{\e}}v^{\e} - 2\e \mathcal{S}^{\varphi^{\e}}v^{\e}\nn^{\e}\cdot\nn^{\e})
\,\partial_t^{\ell}\mathcal{Z}^{\alpha}\NN^{\e}
\\[8pt]\quad

- [\partial_t^{\ell}\mathcal{Z}^{\alpha},2\e \mathcal{S}^{\varphi^{\e}}v^{\e}\nn^{\e}\cdot\nn^{\e},\NN^{\e}]
+2\e [\partial_t^{\ell}\mathcal{Z}^{\alpha},\mathcal{S}^{\varphi^{\e}}v^{\e}, \NN^{\e}]
+ 2\e \mathcal{S}^{\varphi^{\e}} \partial_t^{\ell}\mathcal{Z}^{\alpha}v\,\NN^{\e},
\\[12pt]

\hat{B}^{\ell,\alpha}|_{z=0}=\partial_t^{\ell}\mathcal{Z}^{\alpha}\hat{b}
-\partial_z^\varphi\hat{b}\partial_t^{\ell}\mathcal{Z}^{\alpha}\eta
=\partial_z^\varphi b\partial_t^{\ell}\mathcal{Z}^{\alpha}\eta,
\\[12pt]

(\partial_t^{\ell}\mathcal{Z}^{\alpha}\hat{v},\partial_t^{\ell}\mathcal{Z}^{\alpha}\hat{b},
\partial_t^{\ell}\mathcal{Z}^{\alpha}\hat{h})|_{t=0}
\\[8pt]\quad

= (\partial_t^{\ell}\mathcal{Z}^{\alpha}v_0^\e -\partial_t^{\ell}\mathcal{Z}^{\alpha}v_0,
\partial_t^{\ell}\mathcal{Z}^{\alpha}b_0^\e -\partial_t^{\ell}\mathcal{Z}^{\alpha}b_0,
\partial_t^{\ell}\mathcal{Z}^{\alpha}h_0^\e -\partial_t^{\ell}\mathcal{Z}^{\alpha}h_0),
\end{array}\right.
\end{equation}

Step one:
We establish the $L^2$ estimate of $\hat{V}^{\ell,\alpha}$ and $\hat{B}^{\ell,\alpha}$
as $|\alpha|\geq 1, \ell\leq k-1, 1\leq \ell+|\alpha|\leq k$.

Since $(b^{\e} \cdot\nabla^{\varphi^{\e}} \hat{B}^{\ell,\alpha},\hat{V}^{\ell,\alpha})+(b^{\e} \cdot\nabla^{\varphi^{\e}} \hat{V}^{\ell,\alpha},\hat{B}^{\ell,\alpha})=0$,
we have
\begin{equation}\label{5.7}
\begin{array}{ll}
\frac{1}{2}\frac{\mathrm{d}}{\mathrm{d}t} \int_{\mathbb{R}^3_{-}} (|\hat{V}^{\ell,\alpha}|^2 +|\hat{B}^{\ell,\alpha}|^2) \,\mathrm{d}\mathcal{V}_t
- \int_{\mathbb{R}^3_{-}} \hat{Q}^{\ell,\alpha} \nabla^{\varphi^{\e}}\cdot \hat{V}^{\ell,\alpha} \,\mathrm{d}\mathcal{V}_t
\\[10pt]\quad
+ 2\e \int_{\mathbb{R}^3_{-}} (|\mathcal{S}^{\varphi^{\e}}\partial_t^{\ell}\mathcal{Z}^{\alpha} \hat{v}|^2+|\mathcal{S}^{\varphi^{\e}}\partial_t^{\ell}\mathcal{Z}^{\alpha} \hat{b}|^2)
\,\mathrm{d}\mathcal{V}_t
\\[10pt]

= -\int_{z=0} \big(\hat{Q}^{\ell,\alpha} \NN^{\e} - 2\e \mathcal{S}^{\varphi^{\e}}\partial_t^{\ell}\mathcal{Z}^{\alpha} \hat{v} \NN^{\e} \big)
\cdot \hat{V}^{\ell,\alpha} \,\mathrm{d}y
 +\int_{z=0}  2\e \mathcal{S}^{\varphi^{\e}}\partial_t^{\ell}\mathcal{Z}^{\alpha} \hat{b} \NN^{\e} \cdot \hat{B}^{\ell,\alpha} \,\mathrm{d}y
 \\[10pt]\quad

+ 2\e \int_{\mathbb{R}^3_{-}} \mathcal{S}^{\varphi^{\e}}\partial_t^{\ell}\mathcal{Z}^{\alpha} \hat{v}
\cdot \mathcal{S}^{\varphi^{\e}} (\partial_z^{\varphi}v \partial_t^{\ell}\mathcal{Z}^{\alpha}\hat{\eta})\,\mathrm{d}\mathcal{V}_t
+ 2\e \int_{\mathbb{R}^3_{-}} \mathcal{S}^{\varphi^{\e}}\partial_t^{\ell}\mathcal{Z}^{\alpha} \hat{b}
\cdot \mathcal{S}^{\varphi^{\e}} (\partial_z^{\varphi}b \partial_t^{\ell}\mathcal{Z}^{\alpha}\hat{\eta})\,\mathrm{d}\mathcal{V}_t
 \\[10pt]\quad

 + \int_{\mathbb{R}^3_{-}} \mathcal{I}_{v4} \cdot V^{\ell,\alpha}\,\mathrm{d}\mathcal{V}_t
+ \int_{\mathbb{R}^3_{-}} \mathcal{I}_{b4} \cdot B^{\ell,\alpha}
\,\mathrm{d}\mathcal{V}_t
\\[10pt]

\lem -\int_{z=0} \big(\hat{Q}^{\ell,\alpha} \NN^{\e}
- 2\e \mathcal{S}^{\varphi^{\e}}\partial_t^{\ell}\mathcal{Z}^{\alpha} \hat{v} \NN^{\e} \big)
\cdot \hat{V}^{\ell,\alpha} \,\mathrm{d}y
 +\int_{z=0}  2\e \mathcal{S}^{\varphi^{\e}}\partial_t^{\ell}\mathcal{Z}^{\alpha} \hat{b} \NN^{\e} \cdot \hat{B}^{\ell,\alpha} \,\mathrm{d}y
 \\
\end{array}
\end{equation}

Next we estimate the boundary estimates in \eqref{5.7}. Similarly to \cite{Wu16},
one has,
\begin{equation}\label{5.8}
\begin{array}{ll}
-\int_{z=0} \big(\hat{Q}^{\ell,\alpha} \NN^{\e} - 2\e \mathcal{S}^{\varphi^{\e}}\partial_t^{\ell}\mathcal{Z}^{\alpha} \hat{v} \NN^{\e} \big)
\cdot \hat{V}^{\ell,\alpha} \,\mathrm{d}y
\\[10pt]

= \int_{z=0}
-(g -\partial_z^{\varphi}q) \partial_t^{\ell}\mathcal{Z}^{\alpha}\hat{h}\NN^{\e}\cdot \hat{V}^{\ell,\alpha}
- \big(2\e [\partial_t^{\ell}\mathcal{Z}^{\alpha},\mathcal{S}^{\varphi^{\e}}] v^{\e}\,\NN^{\e}
\\[9pt]\qquad

+ (2\e \mathcal{S}^{\varphi^{\e}}v^{\e}-2\e \mathcal{S}^{\varphi^{\e}}v^{\e}\nn^{\e}\cdot\nn^{\e})\,\partial_t^{\ell}\mathcal{Z}^{\alpha}\NN^{\e}
-[\partial_t^{\ell}\mathcal{Z}^{\alpha},2\e \mathcal{S}^{\varphi^{\e}}v^{\e}\nn^{\e}\cdot\nn^{\e},\NN^{\e}]
\\[7pt]\qquad

+ 2\e [\partial_t^{\ell}\mathcal{Z}^{\alpha},\mathcal{S}^{\varphi^{\e}}v^{\e}, \NN^{\e}]
+ 2\e \mathcal{S}^{\varphi^{\e}} \partial_t^{\ell}\mathcal{Z}^{\alpha}v\,\NN^{\e}\big)\cdot \hat{V}^{\ell,\alpha}
\,\mathrm{d}y
\\[8pt]

\lem - \frac{1}{2}\frac{\mathrm{d}}{\mathrm{d}t}\int_{z=0}
(g -\partial_z^{\varphi}q) |\partial_t^{\ell}\mathcal{Z}^{\alpha}\hat{h}|^2
\,\mathrm{d}y
+ \|\hat{v}\|_{X^{k-1}} + \|\partial_z\hat{v}\|_{X^{k-1}} + |\hat{h}|_{X^{k-1,1}}
+ O(\e).
\end{array}
\end{equation}
and
\begin{equation}\label{bdyb}
\begin{array}{ll}
\int_{z=0}  2\e \mathcal{S}^{\varphi^{\e}}\partial_t^{\ell}\mathcal{Z}^{\alpha} \hat{b} \NN^{\e} \big)\cdot \hat{B}^{\ell,\alpha} \,\mathrm{d}y
=\int_{z=0}  2\e \mathcal{S}^{\varphi^{\e}}\partial_t^{\ell}\mathcal{Z}^{\alpha} b \NN^{\e} \cdot (\partial_z^\varphi\hat{b}\partial_t^{\ell}\mathcal{Z}^{\alpha}\eta) \,\mathrm{d}y\\[11pt]

\leq\|\mathcal{S}^{\varphi^{\e}}\partial_t^{\ell}\mathcal{Z}^{\alpha} \hat{b} \NN^{\e}|_{z=0}\|_{-\frac{1}{2}}
\|\partial_z^\varphi b\partial_t^{\ell}\mathcal{Z}^{\alpha}\eta|_{z=0}\|_{\frac{1}{2}}\\[11pt]

\leq \delta\epsilon\|\mathcal{S}^{\varphi^{\e}}\partial_t^{\ell}\mathcal{Z}^{\alpha} \hat{b}\|^2
+\|\hat{b}\|_{X^{k-1,1}}^2
+ \|\hat{h}\|_{X^{k-1,1}}^2+ O(\e),
\end{array}
\end{equation}

It follows from \eqref{5.7}, \eqref{5.8} and \eqref{bdyb} that
\begin{equation}\label{5.10}
\begin{array}{ll}
\frac{\mathrm{d}}{\mathrm{d}t} \int_{\mathbb{R}^3_{-}} (|\hat{V}^{\ell,\alpha}|^2 +
 |\hat{B}^{\ell,\alpha}|^2 )\,\mathrm{d}\mathcal{V}_t
+ \frac{\mathrm{d}}{\mathrm{d}t}\int_{z=0}
(g -\partial_z^{\varphi}q) |\partial_t^{\ell}\mathcal{Z}^{\alpha}\hat{h}|^2
\,\mathrm{d}y
\\[13pt]\quad
+ \e \int_{\mathbb{R}^3_{-}} |\partial_t^{\ell}\mathcal{Z}^{\alpha}\mathcal{S}^{\varphi^{\e}} \hat{v}|^2
\,\mathrm{d}\mathcal{V}_t
+ \e \int_{\mathbb{R}^3_{-}}|\partial_t^{\ell}\mathcal{Z}^{\alpha}\mathcal{S}^{\varphi^{\e}} \hat{b}|^2  \,\mathrm{d}\mathcal{V}_t
\\[13pt]

\lem \|\hat{V}^{\ell,\alpha}\|_{L^2}^2+\|\hat{B}^{\ell,\alpha}\|_{L^2}^2

+ \|\partial_z \hat{v}\|_{X^{k-1}}^2+ \|\partial_z \hat{b}\|_{X^{k-1}}^2
  + \|\hat{v}\|_{X^{k-1,1}}^2 + \|\hat{b}\|_{X^{k-1,1}}^2
\\[5pt]\quad
 + |\hat{h}|_{X^{k-1,1}}^2 + |\partial_t^k\hat{h}|_{L^2}^2
 + \e|\hat{h}|_{X^{k-1,\frac{3}{2}}}^2 + \|\nabla\hat{q}\|_{X^{k-1}}^2
+ O(\e).
\end{array}
\end{equation}

Since $g -\partial_z^{\varphi}q \geq c_0>0$, one has, by applying Gronwall's inequality to $(\ref{5.10})$, that
\begin{equation}\label{5.11}
\begin{array}{ll}
\|\partial_t^{\ell}\mathcal{Z}^{\alpha}\hat{v}\|_{L^2}^2
+\|\partial_t^{\ell}\mathcal{Z}^{\alpha}\hat{b}\|_{L^2}^2
+ |\partial_t^{\ell}\mathcal{Z}^{\alpha}\hat{h}|^2
+\e|\hat{h}|_{X^{k-1,\frac{3}{2}}}^2
\\[8pt]\quad

+ \e \int_0^t\|\nabla\partial_t^{\ell}\mathcal{Z}^{\alpha}\hat{v}\|_{L^2}^2 \,\mathrm{d}t
+ \e \int_0^t\|\nabla\partial_t^{\ell}\mathcal{Z}^{\alpha}\hat{b}\|_{L^2}^2 \,\mathrm{d}t
\\[8pt]

\lem \|\hat{v}_0\|_{X^{k-1,1}}^2 +\|\hat{b}_0\|_{X^{k-1,1}}^2+ |\hat{h}_0|_{X^{k-1,1}}^2
+ \|\partial_z \hat{v}\|_{L^4([0,T],X^{k-1})}^2
+ \|\partial_z \hat{b}\|_{L^4([0,T],X^{k-1})}^2
\\[8pt]\quad

+ \int\limits_0^t \|\hat{v}\|_{X^{k-1,1}}^2 + \|\hat{b}\|_{X^{k-1,1}}^2 + |\hat{h}|_{X^{k-1,1}}^2 +\e|\hat{h}|_{X^{k-1,\frac{3}{2}}}^2 \,\mathrm{d}t
+ |\partial_t^k\hat{h}|_{L^4([0,T],L^2)}^2
\\[8pt]\quad

+ \|\nabla\hat{q}\|_{L^4([0,T],X^{k-1})}^2 + O(\e).
\end{array}
\end{equation}

Step two: We establish the $L^2$ estimate of $\hat{V}^{\ell,\alpha}$ and $\hat{B}^{\ell,\alpha}$
as $|\alpha|=0, \, 0\leq\ell\leq k-1$.

Now, we have no bounds of $\hat{q}$ and $\partial_t^{\ell} \hat{q}$
as $|\alpha|=0, \, 0\leq\ell\leq k-1$. Hence, we can not use the variable $\hat{Q}^{\ell,\alpha}$ and the dynamical boundary condition.
The  equations of $\hat{V}^{\ell,0}$, $\hat{B}^{\ell,0}$ are
\begin{equation}\label{5.12}
\left\{\begin{array}{ll}
\partial_t^{\varphi^{\e}} \hat{V}^{\ell,0}
+ v^{\e} \cdot\nabla^{\varphi^{\e}}\hat{V}^{\ell,0}
-b^{\e} \cdot\nabla^{\varphi^{\e}}\hat{B}^{\ell,0}
- 2\e\nabla^{\varphi^{\e}}\cdot\mathcal{S}^{\varphi^{\e}} \partial_t^{\ell}\hat{v} \\[8pt]\quad

= \e\partial_t^{\ell}\Delta^{\varphi^\e} v^\e
+ 2\e [\partial_t^{\ell},\nabla^{\varphi^{\e}}\cdot]\mathcal{S}^{\varphi^{\e}} \hat{v}
+ 2\e \nabla^{\varphi^{\e}}\cdot[\partial_t^{\ell},\mathcal{S}^{\varphi^{\e}}] \hat{v}
- \partial_t^{\ell}\nabla^{\varphi^{\e}}\hat{q}
+ \partial_z^{\varphi} q\nabla^{\varphi^{\e}} \partial_t^{\ell}\hat{\varphi}
\\[8pt]\quad

- \partial_t^{\ell}\hat{\varphi}\partial_t^{\varphi^{\e}}\partial_z^{\varphi} v
- \partial_t^{\ell}\hat{v}\cdot\nabla^{\varphi} v
- \partial_t^{\ell}\hat{\varphi}\, v^{\e}\cdot \nabla^{\varphi^{\e}}\partial_z^{\varphi} v
- [\partial_t^{\ell},\partial_t^{\varphi^{\e}}]\hat{v}
+ [\partial_t^{\ell}, \partial_z^{\varphi} v \partial_t^{\varphi^{\e}}]\hat{\varphi}
\\[8pt]\quad

- [\partial_t^{\ell}, v^{\e} \cdot\nabla^{\varphi^{\e}}] \hat{v}
+ [\partial_t^{\ell}, \partial_z^{\varphi} v \, v^{\e}\cdot \nabla^{\varphi^{\e}}]\hat{\varphi}
- [\partial_t^{\ell}, \nabla^{\varphi} v\cdot]\hat{v}
+ [\partial_t^{\ell},\partial_z^{\varphi} q\nabla^{\varphi^{\e}}]\hat{\varphi}
\\[8pt]\quad

+ \partial_t^{\ell}\hat{b}\cdot\nabla^{\varphi} b
+ \partial_t^{\ell}\hat{\varphi}\, b^{\e}\cdot \nabla^{\varphi^{\e}}\partial_z^{\varphi} b
- [\partial_t^{\ell}, \partial_z^{\varphi} b \partial_t^{\varphi^{\e}}]\hat{\varphi}
+ [\partial_t^{\ell}, \nabla^{\varphi} b\cdot]\hat{b}
:=\mathcal{I}_{v5},  \\[13pt]

\partial_t^{\varphi^{\e}} \hat{B}^{\ell,0}
+ v^{\e} \cdot\nabla^{\varphi^{\e}}\hat{B}^{\ell,0}
-b^{\e} \cdot\nabla^{\varphi^{\e}}\hat{V}^{\ell,0}
- 2\e\nabla^{\varphi^{\e}}\cdot\mathcal{S}^{\varphi^{\e}} \partial_t^{\ell}\hat{b}
\\
\end{array}\right.
\end{equation}
with the following initial boundary conditions
\begin{equation}
\left\{\begin{array}{ll}
\partial_t \partial_t^{\ell}\hat{h} + v_y^{\e}\cdot \nabla_y \partial_t^{\ell}\hat{h}
- \NN^{\e}\cdot \hat{V}^{\ell,0}
= \NN^{\e}\cdot \partial_z^{\varphi} v \partial_t^{\ell}\hat{\eta} \\[7pt]\quad
 - \hat{v}_y \cdot \nabla_y \partial_t^{\ell} h
- \partial_y \hat{h}\cdot \partial_t^{\ell}v_y
 + [\partial_t^{\ell}, \hat{v},\NN^{\e}]
 - [\partial_t^{\ell}, v_y, \partial_y \hat{h}], \\[6pt]

\partial_t^{\ell}\hat{B}|_{z=0}=\partial_t^{\ell}\hat{b}-\partial_z^\varphi\hat{b}\partial_t^{\ell}\eta
=\partial_z^\varphi b\partial_t^{\ell}\eta, \\[6pt]

(\partial_t^{\ell}\hat{v},\partial_t^{\ell}\hat{b},\partial_t^{\ell}\hat{h})|_{t=0}
= (\partial_t^{\ell}v_0^\e -\partial_t^{\ell}v_0,\partial_t^{\ell}b_0^\e -\partial_t^{\ell}b_0,
\partial_t^{\ell}h_0^\e -\partial_t^{\ell}h_0).
\end{array}\right.
\end{equation}
For the $L^2$ estimates of $\hat{V}^{\ell,0}$ and $\hat{B}^{\ell,0}$, we have:
\begin{equation}\label{5.14}
\begin{array}{ll}
\frac{1}{2}\frac{\mathrm{d}}{\mathrm{d}t} \int_{\mathbb{R}^3_{-}} (|\hat{V}^{\ell,0}|^2 +|\hat{B}^{\ell,0}|^2) \,\mathrm{d}\mathcal{V}_t
+ 2\e \int_{\mathbb{R}^3_{-}} |\mathcal{S}^{\varphi^{\e}}\partial_t^{\ell} \hat{v}|^2
\,\mathrm{d}\mathcal{V}_t
+ 2\e \int_{\mathbb{R}^3_{-}} |\mathcal{S}^{\varphi^{\e}}\partial_t^{\ell} \hat{b}|^2
\,\mathrm{d}\mathcal{V}_t
\\[10pt]

= 2\e \int_{\{z=0\}} \mathcal{S}^{\varphi^{\e}}\partial_t^{\ell} \hat{v} \NN^{\e} \cdot \hat{V}^{\ell,0} \,\mathrm{d}y
+2\e \int_{\mathbb{R}^3_{-}} \mathcal{S}^{\varphi^{\e}}\partial_t^{\ell} \hat{v}
\cdot \mathcal{S}^{\varphi^{\e}} (\partial_z^{\varphi}v\partial_t^{\ell} \hat{\eta}) \,\mathrm{d}\mathcal{V}_t
\\[10pt]\quad

+2\e \int_{\{z=0\}} \mathcal{S}^{\varphi^{\e}}\partial_t^{\ell} \hat{b} \NN^{\e} \cdot \hat{B}^{\ell,0} \,\mathrm{d}y
+2\e \int_{\mathbb{R}^3_{-}} \mathcal{S}^{\varphi^{\e}}\partial_t^{\ell} \hat{b}
\cdot \mathcal{S}^{\varphi^{\e}} (\partial_z^{\varphi}b\partial_t^{\ell} \hat{\eta}) \,\mathrm{d}\mathcal{V}_t
\\[10pt]\quad

 + \int_{\mathbb{R}^3_{-}}\mathcal{I}_{v5} \cdot \hat{V}^{\ell,0} \,\mathrm{d}\mathcal{V}_t
+ \int_{\mathbb{R}^3_{-}}\mathcal{I}_{b5} \cdot \hat{V}^{\ell,0} \,\mathrm{d}\mathcal{V}_t.
\end{array}
\end{equation}

The boundary terms of \eqref{5.14} can be estimated as follows
\begin{equation}\label{5.15}
\begin{array}{ll}
2\e\int_{z=0} \mathcal{S}^{\varphi^\e} \partial_t^{\ell} v^\e \NN^{\e} \cdot \hat{V}^{\ell,0} \,\mathrm{d}y
= 2\e\int_{z=0} \mathcal{S}^{\varphi^\e} \partial_t^{\ell} v^\e \NN^{\e} \cdot
(\partial_t^{\ell}\hat{v} - \partial_z^{\varphi} v \partial_t^{\ell}\hat{\eta}) \,\mathrm{d}y
\\[12pt]\quad

\lem \big|\partial_t^{\ell}\hat{v}|_{z=0}\big|_{L^2}^2 + |\partial_t^{\ell}\hat{h}|_{L^2}^2 + O(\e)
\lem \|\partial_t^{\ell}\hat{v}\|_{L^2}^2 + \|\partial_t^{\ell}\partial_z \hat{v}\|_{L^2}^2 + |\partial_t^{\ell}\hat{h}|_{L^2}^2 + O(\e),
\\[12pt]

2\e \int_{z=0} \mathcal{S}^{\varphi^{\e}}\partial_t^{\ell} \hat{b} \NN^{\e} \cdot \hat{B}^{\ell,0} \,\mathrm{d}y
= 2\e\int_{z=0} \mathcal{S}^{\varphi^\e} \partial_t^{\ell} b^\e \NN^{\e} \cdot
( \partial_z^{\varphi} b \partial_t^{\ell}\hat{\eta}) \,\mathrm{d}y
\\[12pt]\quad

\lem \big|\partial_t^{\ell}\hat{b}|_{z=0}\big|_{L^2}^2 + |\partial_t^{\ell}\hat{h}|_{L^2}^2 + O(\e)
\lem \|\partial_t^{\ell}\hat{b}\|_{L^2}^2 + \|\partial_t^{\ell}\partial_z \hat{b}\|_{L^2}^2 + |\partial_t^{\ell}\hat{h}|_{L^2}^2 + O(\e).
\end{array}
\end{equation}
And, the other terms of \eqref{5.14} are estimated as
\begin{equation}\label{5.16}
\begin{array}{ll}
\int_{\mathbb{R}^3_{-}}\mathcal{I}_{v5} \cdot \hat{V}^{\ell,0} \,\mathrm{d}\mathcal{V}_t
+\int_{\mathbb{R}^3_{-}}\mathcal{I}_{b5} \cdot \hat{V}^{\ell,0} \,\mathrm{d}\mathcal{V}_t
\lem \|\hat{V}^{\ell,0}\|_{L^2}^2 +\|\hat{B}^{\ell,0}\|_{L^2}^2
  \\[8pt]\quad

+ \|\partial_t^{\ell-1}\partial_z \hat{v}\|_{L^2}^2 + \|\partial_t^{\ell-1}\partial_z \hat{b}\|_{L^2}^2
+ \|\partial_t^{\ell-1}\partial_y \hat{v}\|_{L^2}^2+ \|\partial_t^{\ell-1}\partial_y \hat{b}\|_{L^2}^2
+ \|\partial_t^{\ell-1} \hat{v}\|_{L^2}^2
 \\[8pt]\quad

+ \|\partial_t^{\ell-1} \hat{b}\|_{L^2}^2
+ \|\partial_t^{\ell-1}\nabla\hat{\eta}\|_{L^2}^2 + \|\partial_t^{\ell}\hat{\eta}\|_{L^2}^2 + \|\partial_t^{\ell}\nabla\hat{\eta}\|_{L^2}^2
+ \|\partial_t^{\ell}\nabla\hat{q}\|_{L^2}^2 +O(\e).
\end{array}
\end{equation}

By \eqref{5.14}, \eqref{5.15} and \eqref{5.16}, we have
\begin{equation}\label{5.17}
\begin{array}{ll}
\frac{\mathrm{d}}{\mathrm{d}t} \int_{\mathbb{R}^3_{-}} (|\hat{V}^{\ell,0}|^2+|\hat{V}^{\ell,0}|^2) \,\mathrm{d}\mathcal{V}_t
+ \e \int_{\mathbb{R}^3_{-}} |\mathcal{S}^{\varphi^{\e}}\partial_t^{\ell} \hat{v}|^2
\,\mathrm{d}\mathcal{V}_t
+ \e \int_{\mathbb{R}^3_{-}} |\mathcal{S}^{\varphi^{\e}}\partial_t^{\ell} \hat{b}|^2
\,\mathrm{d}\mathcal{V}_t
\\[8pt]\quad

\lem \|\hat{V}^{\ell,0}\|_{L^2}^2+\|\hat{B}^{\ell,0}\|_{L^2}^2
+\|\partial_z \hat{v}\|_{X^{k-1}}^2+\|\partial_z \hat{b}\|_{X^{k-1}}^2
+\|\hat{v}\|_{X^{k-1,1}}^2+\|\hat{b}\|_{X^{k-1,1}}^2
\\[8pt]\quad

+ |\hat{h}|_{X^{k-1,1}}^2 + \e|\hat{h}|_{X^{k-1,\frac{3}{2}}}^2
+ \|\nabla\hat{q}\|_{X^{k-1}}^2 +O(\e).
\end{array}
\end{equation}

Applying the Gronwall's inequality to \eqref{5.17}, one has
\begin{equation}\label{5.18}
\begin{array}{ll}
\|\hat{V}^{\ell,0}\|_{L^2}^2 +\|\hat{B}^{\ell,0}\|_{L^2}^2
+ \e \int_0^t\|\nabla\partial_t^{\ell}\hat{v}\|_{L^2}^2
+ \|\nabla\partial_t^{\ell}\hat{b}\|_{L^2}^2 \,\mathrm{d}t
\\[8pt]

\lem \|\hat{v}_0\|_{X^{k-1}}^2 +\|\hat{b}_0\|_{X^{k-1}}^2
+|\hat{h}_0|_{X^{k-1}}^2
+ \int_0^t \|\partial_z \hat{v}\|_{X^{k-1}}^2
+\|\partial_z \hat{b}\|_{X^{k-1}}^2
+ \|\hat{v}\|_{X^{k-1,1}}^2
\\[8pt]\quad

+ \|\hat{b}\|_{X^{k-1,1}}^2+ |\hat{h}|_{X^{k-1,1}}^2
+ \|\nabla\hat{q}\|_{X^{k-1}}^2 + \e|\hat{h}|_{X^{k-1,\frac{3}{2}}}^2\,\mathrm{d}t +O(\e)
\\[8pt]

\lem \|\hat{v}_0\|_{X^{k-1}}^2 +\|\hat{b}_0\|_{X^{k-1}}^2 +|\hat{h}_0|_{X^{k-1}}^2
+ \int_0^t \|\hat{v}\|_{X^{k-1,1}}^2
+\|\hat{b}\|_{X^{k-1,1}}^2 + |\hat{h}|_{X^{k-1,1}}^2
+ \e|\hat{h}|_{X^{k-1,\frac{3}{2}}}^2\,\mathrm{d}t
 \\[8pt]\quad

 + \|\partial_z \hat{v}\|_{L^4([0,T],X^{k-1})}^2
 + \|\partial_z \hat{b}\|_{L^4([0,T],X^{k-1})}^2
 + \|\nabla\hat{q}\|_{L^4([0,T],X^{k-1})}^2 +O(\e).
\end{array}
\end{equation}

It follows from \eqref{5.18} and Lemma \eqref{Lemma5.1}, \eqref{Lemma5.2} that
\begin{equation}\label{5.19}
\begin{array}{ll}
\|\partial_t^{\ell}\hat{v}\|_{L^2}^2 +\|\partial_t^{\ell}\hat{b}\|_{L^2}^2
+ |\partial_t^{\ell}\hat{h}|_{L^2}^2 + \e|\hat{h}|_{X^{k-1,\frac{3}{2}}}^2
+ \e \int_0^t\|\nabla\partial_t^{\ell}\hat{v}\|_{L^2}^2
+ \|\nabla\partial_t^{\ell}\hat{b}\|_{L^2}^2 \,\mathrm{d}t
\\[8pt]

\lem \|\hat{v}_0\|_{X^{k-1}}^2 +\|\hat{b}_0\|_{X^{k-1}}^2 +
|\hat{h}_0|_{X^{k-1}}^2 + \int_0^t \|\hat{v}\|_{X^{k-1,1}}^2
+\|\hat{b}\|_{X^{k-1,1}}^2 + |\hat{h}\|_{X^{k-1,1}}^2
+ \e|\hat{h}|_{X^{k-1,\frac{3}{2}}}^2\,\mathrm{d}t
\\[8pt]\quad

+ \|\partial_z \hat{v}\|_{L^4([0,T],X^{k-1})}^2
+\|\partial_z \hat{b}\|_{L^4([0,T],X^{k-1})}^2
+\|\nabla\hat{q}\|_{L^4([0,T],X^{k-1})}^2 +O(\e).
\end{array}
\end{equation}

Combining \eqref{5.11} and \eqref{5.19},
applying the Gronwall's inequality,
one has \eqref{5.4}.

 Lemma \ref{Lemma5.3} is proved.
\end{proof}

\subsection{Estimates for the Gradient of Pressure }
In this subsection,
we estimate the pressure $\hat{q}$ on the right hand side of \eqref{5.4}.

\begin{lemma}\label{Lemma5.4}
Assume $0\leq s\leq k-1,\, k \leq m-2$,
we have
\begin{equation}\label{5.20}
\begin{array}{ll}
\|\nabla \hat{q}\|_{X^s} \lem \|\hat{v}\|_{X^{s,1}} +\|\hat{b}\|_{X^{s,1}} + \|\partial_z\hat{b}\|_{X^s} + \|\partial_z\hat{v}\|_{X^s} + |\hat{h}|_{X^{s,\frac{1}{2}}}
+ O(\e).
\end{array}
\end{equation}
\end{lemma}

\begin{proof}
The following matrices $\textsf{E}$ and $\textsf{P}$
was introduced in \cite{Masmoudi12}
\begin{equation*}
%
\end{equation*}
where $E$ satisfies $\textsf{E}=\frac{1}{\partial_z\varphi}\textsf{P}\textsf{P}^{\top}$.

It follows from $(\ref{MHDF})_1$, $(\ref{IMHDF})_1$ that $q^{\e}$
and $q$ satisfy
\begin{equation}\label{5.21}
\left\{\begin{array}{ll}
\nabla\cdot(\textsf{E}^{\e}\nabla q^{\e})
=\partial_z \varphi^{\e}\Delta^{\varphi^{\e}} q^{\e}
=-\partial_z \varphi^{\e}\nabla^{\varphi^{\e}}\cdot(v^{\e}
\cdot\nabla^{\varphi^{\e}} v^{\e}-b^{\e} \cdot\nabla^{\varphi^{\e}} b^{\e}),
\\[6pt]

\nabla\cdot(\textsf{E}\nabla q) = \partial_z \varphi\Delta^{\varphi} q
=-\partial_z \varphi\nabla^{\varphi}\cdot (v \cdot\nabla^{\varphi} v-b \cdot\nabla^{\varphi} b).
\end{array}\right.
\end{equation}

It follows from $(\ref{5.21})$ that
\begin{equation}
\begin{aligned}
&\nabla\cdot(\textsf{E}^{\e}\nabla q^{\e}) - \nabla\cdot(\textsf{E}\nabla q)
= \nabla\cdot(\textsf{E}^{\e}\nabla \hat{q}) + \nabla\cdot((\textsf{E}^{\e} -\textsf{E}) \nabla q)
\\
=& - \partial_z \varphi^{\e}\nabla^{\varphi^{\e}}\cdot (v^{\e} \cdot\nabla^{\varphi^{\e}} v^{\e}
-b^{\e} \cdot\nabla^{\varphi^{\e}} b^{\e})
+ \partial_z \varphi \nabla^{\varphi}\cdot (v \cdot\nabla^{\varphi} v-b \cdot\nabla^{\varphi} b)
\\
= &-\nabla\cdot \big[\textsf{P}^{\e} (v^{\e} \cdot\nabla^{\varphi^{\e}} v^{\e}
-b^{\e} \cdot\nabla^{\varphi^{\e}} b^{\e})\big]
+ \nabla\cdot \big[\textsf{P} (v \cdot\nabla^{\varphi} v-b \cdot\nabla^{\varphi} b)\big]
 \\
=& -\nabla\cdot \big[\textsf{P}^{\e} (v^{\e} \cdot\nabla^{\varphi^{\e}} v^{\e} - v \cdot\nabla^{\varphi} v
-b^{\e} \cdot\nabla^{\varphi^{\e}} b^{\e} + b \cdot\nabla^{\varphi} b)\big]
\\
&- \nabla\cdot \big[(\textsf{P}^{\e} - \textsf{P}) (v \cdot\nabla^{\varphi} v
-b \cdot\nabla^{\varphi} b)\big]
\\
=& -\nabla\cdot \big[\textsf{P}^{\e} (v^{\e} \cdot\nabla^{\varphi^{\e}} \hat{v} - v^{\e}\cdot\nabla^{\varphi^{\e}}\hat{\varphi}\partial_z^{\varphi} v
+ \hat{v}\cdot\nabla^{\varphi} v-b^{\e} \cdot\nabla^{\varphi^{\e}} \hat{b}
\\
&
+ b^{\e}\cdot\nabla^{\varphi^{\e}}\hat{\varphi}\partial_z^{\varphi} b
-\hat{b}\cdot\nabla^{\varphi} b)\big]
- \nabla\cdot\big[(\textsf{P}^{\e}-\textsf{P}) (v\cdot\nabla^{\varphi}v-b\cdot\nabla^{\varphi} b)\big].
\end{aligned}
\end{equation}

Hence,
$\hat{q}$ satisfies the following elliptic equation:
\begin{equation}
\left\{\begin{aligned}
&\nabla\cdot(\textsf{E}^{\e}\nabla \hat{q})
= -\nabla\cdot((\textsf{E}^{\e} -\textsf{E}) \nabla q) - \nabla\cdot [(\textsf{P}^{\e} - \textsf{P}) (v \cdot\nabla^{\varphi} v-b \cdot\nabla^{\varphi} b)]
\\
&\hspace{2.2cm}
-\nabla\cdot [\textsf{P}^{\e} (v^{\e} \cdot\nabla^{\varphi^{\e}} \hat{v} - v^{\e}\cdot\nabla^{\varphi^{\e}}\hat{\varphi}\partial_z^{\varphi} v
+ \hat{v}\cdot\nabla^{\varphi} v
\\
&\hspace{2.2cm}
-b^{\e} \cdot\nabla^{\varphi^{\e}} \hat{b} + b^{\e}\cdot\nabla^{\varphi^{\e}}\hat{\varphi}\partial_z^{\varphi} b
- \hat{b}\cdot\nabla^{\varphi} b)],
\\
&
q|_{z=0} = g\hat{h} + 2\e\mathcal{S}^{\varphi^{\e}} v^{\e} \nn^{\e}\cdot\nn^{\e}.
\end{aligned}\right.
\end{equation}

Since the matrix $\textsf{E}^{\e}$ is definitely positive,
we know that $\hat{q}$ satisfies the following estimate:
\begin{equation}\label{5.24}
\begin{aligned}
\|\nabla \hat{q}\|_{X^s}
\lem& \|(\textsf{E}^{\e} -\textsf{E}) \nabla q \|_{X^s}
+ \|(\textsf{P}^{\e} - \textsf{P}) (v \cdot\nabla^{\varphi} v)\|_{X^s}
\\
&
+ \|\textsf{P}^{\e} (v^{\e} \cdot\nabla^{\varphi^{\e}} \hat{v} - v^{\e}\cdot\nabla^{\varphi^{\e}}\hat{\varphi}\partial_z^{\varphi} v
+ \hat{v}\cdot\nabla^{\varphi} v-b^{\e} \cdot\nabla^{\varphi^{\e}} \hat{b}
\\
&
+ b^{\e}\cdot\nabla^{\varphi^{\e}}\hat{\varphi}\partial_z^{\varphi} b
- \hat{b}\cdot\nabla^{\varphi} b)\|_{X^s}+ |g\hat{h} + 2\e\mathcal{S}^{\varphi^{\e}} v^{\e} \nn^{\e}\cdot\nn^{\e}|_{X^{s,\frac{1}{2}}}
\\
\lem& \|\hat{v}\|_{X^{s,1}} +\|\hat{b}\|_{X^{s,1}} + \|\partial_z\hat{v}\|_{X^s}+ \|\partial_z\hat{b}\|_{X^s} + |\hat{h}|_{X^{s,\frac{1}{2}}}
+ O(\e),
\end{aligned}
\end{equation}
where $\big|\mathcal{S}^{\varphi^{\e}} v^{\e}|_{z=0}\big|_{X^{s,\frac{1}{2}}}
\lem \|\partial_z\partial_j v^{\e}\|_{X^s}^{\frac{1}{2}}\|\partial_j v^{\e}\|_{X^{s+1}}^{\frac{1}{2}} <+\infty$.

Lemma $\ref{5.4}$ is proved.
\end{proof}

\subsection{Estimates for the Normal Derivatives when $\Pi\mathcal{S}^{\varphi} v\nn|_{z=0} \neq 0$
and $\Pi\mathcal{S}^{\varphi} b\nn|_{z=0} \neq 0$}

In this subsection,
we give the estimates for $\partial_z \hat{v}$
and $\partial_z \hat{b}$ on the right hand side of \eqref{5.4}.
The following lemma concerns the estimates of $\|\partial_z\hat{v}\|_{L^4([0,T],X^{k-1})}^2$
and $\|\partial_z\hat{b}\|_{L^4([0,T],X^{k-1})}^2$ by studying the equations of
 $\hat{\omega}_{vh}$, $\hat{\omega}_{bh}$.
\begin{lemma}\label{Lemma5.5}
Let $k\leq m-2$. If the boundary data of the ideal MHD satisfy $\Pi\mathcal{S}^{\varphi} v \nn|_{z=0} \neq 0$
and $\Pi\mathcal{S}^{\varphi} b \nn|_{z=0} \neq 0$,
then we have
\begin{equation}\label{5.25}
\begin{array}{ll}
\|\partial_z\hat{v}_h\|_{L^4([0,T],X^{k-1})}^2
+\|\partial_z\hat{b}_h\|_{L^4([0,T],X^{k-1})}^2
+ \|\hat{\omega}_{vh}\|_{L^4([0,T],X^{k-1})}^2
+ \|\hat{\omega}_{bh}\|_{L^4([0,T],X^{k-1})}^2
\\[6pt]
\lem \big\|\hat{\omega}_{v0}\big\|_{X^{k-1}}^2+\big\|\hat{\omega}_{b0}\big\|_{X^{k-1}}^2
+ \int_0^T\|\hat{v}\|_{X^{k-1,1}}^2+\|\hat{b}\|_{X^{k-1,1}}^2 \,\mathrm{d}t
\\[6pt]\quad

+ \int_0^T|\hat{h}|_{X^{k-1,1}}^2 \,\mathrm{d}t
+ \|\partial_t^k\hat{h}\|_{L^4([0,T],L^2)}^2
+ O(\sqrt{\e}).
\end{array}
\end{equation}
\end{lemma}

\begin{proof}
Assume $\ell+|\alpha|\leq k-1$. We study the equations $(\ref{1.25})$ by decomposing $\hat{\omega}_{vh} = \hat{\omega}_{vh}^{nh} + \hat{\omega}_{vh}^{h}$,
$\hat{\omega}_{bh} = \hat{\omega}_{bh}^{nh} + \hat{\omega}_{bh}^{h}$ such that
$\hat{\omega}_{vh}^{nh}$, $\hat{\omega}_{bh}^{nh}$ satisfy the following nonhomogeneous equations:
\begin{equation}\label{5.26}
\left\{\begin{array}{ll}
\partial_t^{\varphi^{\e}} \hat{\omega}_{vh}^{nh}
- \e\Delta^{\varphi^{\e}}\hat{\omega}_{vh}^{nh}
+ v^{\e} \cdot\nabla^{\varphi^{\e}} \hat{\omega}_{vh}^{nh}
- b^{\e} \cdot\nabla^{\varphi^{\e}} \hat{\omega}_{bh}^{nh}
 \\[6pt]\quad
= \vec{\textsf{F}}_v^0[\nabla\varphi^{\e}](\omega_{vh},\omega_{bh},\partial_jv^i,\partial_jb^i)
- \vec{\textsf{F}}_v^0[\nabla\varphi](\omega_{vh},\omega_{bh},\partial_jv^i,\partial_jb^i)
+ \e\Delta^{\varphi^{\e}}\omega_{vh}
\\[6pt]\qquad

+ \partial_z^{\varphi}\omega_{vh} \partial_t^{\varphi^{\e}} \hat{\eta}
+ \partial_z^{\varphi} \omega_{vh}\, v^{\e}\cdot \nabla^{\varphi^{\e}}\hat{\eta}
- \hat{v}\cdot\nabla^{\varphi} \omega_{vh}
- \partial_z^{\varphi} \omega_{bh}\, b^{\e}\cdot \nabla^{\varphi^{\e}}\hat{\eta}
+ \hat{b}\cdot\nabla^{\varphi} \omega_{bh} ,
\\[7pt]

\partial_t^{\varphi^{\e}} \hat{\omega}_{bh}^{nh}
- \e\Delta^{\varphi^{\e}}\hat{\omega}_{bh}^{nh}
+ v^{\e} \cdot\nabla^{\varphi^{\e}} \hat{\omega}_{bh}^{nh}
- b^{\e} \cdot\nabla^{\varphi^{\e}} \hat{\omega}_{vh}^{nh}
 \\[6pt]\quad
= \vec{\textsf{F}}_b^0[\nabla\varphi^{\e}](\omega_{vh},\omega_{bh},\partial_jv^i,\partial_jb^i)
- \vec{\textsf{F}}_b^0[\nabla\varphi](\omega_{vh},\omega_{bh},\partial_jv^i,\partial_jb^i)
+ \e\Delta^{\varphi^{\e}}\omega_{bh}
\\[6pt]\qquad

+ \partial_z^{\varphi}\omega_{bh} \partial_t^{\varphi^{\e}} \hat{\eta}
+ \partial_z^{\varphi} \omega_{bh}\, v^{\e}\cdot \nabla^{\varphi^{\e}}\hat{\eta}
- \hat{v}\cdot\nabla^{\varphi} \omega_{bh}
- \partial_z^{\varphi} \omega_{vh}\, b^{\e}\cdot \nabla^{\varphi^{\e}}\hat{\eta}
+ \hat{b}\cdot\nabla^{\varphi} \omega_{vh} ,
\\[7pt]

\hat{\omega}_{vh}^{nh}|_{z=0} =0, \hat{\omega}_{bh}^{nh}|_{z=0} =0,
\\[7pt]

\hat{\omega}_{vh}^{nh}|_{t=0} = (\hat{\omega}_{v0}^1, \hat{\omega}_{v0}^2)^{\top},
\hat{\omega}_{bh}^{nh}|_{t=0} = (\hat{\omega}_{b0}^1, \hat{\omega}_{b0}^2)^{\top},
\end{array}\right.
\end{equation}

and $\hat{\omega}_v^{h}$, $\hat{\omega}_b^{h}$ satisfy the following homogeneous equations:
\begin{equation}\label{5.27}
\left\{\begin{array}{ll}
\partial_t^{\varphi^{\e}} \hat{\omega}_{vh}^{h}
- \e\Delta^{\varphi^{\e}}\hat{\omega}_{vh}^{h}
+ v^{\e} \cdot\nabla^{\varphi^{\e}} \hat{\omega}_{vh}^{h}
- b^{\e} \cdot\nabla^{\varphi^{\e}} \hat{\omega}_{bh}^{h}
= 0,
\\[7pt]

\partial_t^{\varphi^{\e}} \hat{\omega}_{bh}^{h}
- \e\Delta^{\varphi^{\e}}\hat{\omega}_{bh}^{h}
+ v^{\e} \cdot\nabla^{\varphi^{\e}} \hat{\omega}_{bh}^{h}
- b^{\e} \cdot\nabla^{\varphi^{\e}} \hat{\omega}_{vh}^{h}
= 0,
\\[7pt]

\hat{\omega}^h_{vh}|_{z=0}=F^{1,2}[\nabla\varphi^{\epsilon,i}](\partial_jv^{\epsilon,i})
-\omega^h_{vh}|_{z=0},
\\[7pt]

\hat{\omega}^h_{bh}|_{z=0}=F^{1,2}[\nabla\varphi^{\epsilon,i}](\partial_jb^{\epsilon,i})
-\omega^h_{bh}|_{z=0},
\\[7pt]

(\hat{\omega}^h_{vh}|_{t=0},\hat{\omega}^h_{bh}|_{t=0})=0.
\end{array}\right.
\end{equation}
Denote by $\partial_t^{\varphi^{\e}} + v^{\e} \cdot\nabla^{\varphi^{\e}}
= \partial_t + v_y^{\e}\cdot\nabla_y + V_z^{\e}\partial_z$,
$b^{\e} \cdot\nabla^{\varphi^{\e}}=b_y^{\e}\cdot\nabla_y + B_z^{\e}\partial_z$,
where $V_z^{\e}=\frac{1}{\partial_z\varphi^{\e}}(V^{\e}\cdot\NN^{\e}-\partial_t\eta^{\e})$
and $B_z^{\e}=\frac{1}{\partial_z\varphi^{\e}}(B^{\e}\cdot\NN^{\e})$,
then \eqref{5.26} is equivalent to the following equations:
\begin{equation}\label{5.28}
\left\{\begin{array}{ll}
\partial_t \hat{\omega}_{vh}^{nh}
- \e\Delta^{\varphi^{\e}}\hat{\omega}_{vh}^{nh}
 + v_y^{\e}\cdot\nabla_y\hat{\omega}_{vh}^{nh}
+ V_z^{\e}\partial_z \hat{\omega}_{vh}^{nh}-b_y^{\e}\cdot\nabla_y\hat{\omega}_{bh}^{nh}
- B_z^{\e}\partial_z \hat{\omega}_{bh}^{nh}
 \\[8pt]\quad

= f_v^7\hat{\omega}_{vh}+f_b^7\hat{\omega}_{bh}
+ f_v^8\partial_j\hat{v}^i+ f_b^8\partial_j\hat{b}^i
+ f_v^9\nabla\hat{\varphi}
+ \e\Delta^{\varphi^{\e}}\omega_h
+ \partial_z^{\varphi}\omega_h \partial_t^{\varphi^{\e}} \hat{\eta}
\\[7pt]\qquad

+ \partial_z^{\varphi} \omega_h\, v^{\e}\cdot \nabla^{\varphi^{\e}}\hat{\eta}
- \hat{v}\cdot\nabla^{\varphi} \omega_h
- \partial_z^{\varphi} \omega_{bh}\, b^{\e}\cdot \nabla^{\varphi^{\e}}\hat{\eta}
+ \hat{b}\cdot\nabla^{\varphi} \omega_{bh} := \mathcal{I}_{v6},
\\[9pt]

\partial_t^{\varphi^{\e}} \hat{\omega}_{bh}^{nh}
- \e\Delta^{\varphi^{\e}}\hat{\omega}_{bh}^{nh}
+ v_y^{\e}\cdot\nabla_y\hat{\omega}_{bh}^{nh}
+ V_z^{\e}\partial_z \hat{\omega}_{bh}^{nh}
-b_y^{\e}\cdot\nabla_y\hat{\omega}_{vh}^{nh}
- B_z^{\e}\partial_z \hat{\omega}_{vh}^{nh}
 \\[6pt]\quad

= -f_b^7\hat{\omega}_{vh} -f_v^7\hat{\omega}_{bh}
- f_b^8\partial_j\hat{v}^i-f_v^8\partial_j\hat{b}^i
+ f_b^9\nabla\hat{\varphi}
+ \e\Delta^{\varphi^{\e}}\omega_{bh}
+ \partial_z^{\varphi}\omega_{bh}\partial_t^{\varphi^{\e}} \hat{\eta}
\\[6pt]\qquad

+ \partial_z^{\varphi} \omega_{bh}\, v^{\e}\cdot \nabla^{\varphi^{\e}}\hat{\eta}
- \hat{v}\cdot\nabla^{\varphi} \omega_{bh}
- \partial_z^{\varphi} \omega_{vh}\, b^{\e}\cdot \nabla^{\varphi^{\e}}\hat{\eta}
+ \hat{b}\cdot\nabla^{\varphi} \omega_{vh}:= \mathcal{I}_{b6},\\[7pt]

\hat{\omega}_{vh}^{nh}|_{z=0} =0, \hat{\omega}_{bh}^{nh}|_{z=0} =0,\\[7pt]

\hat{\omega}_{vh}^{nh}|_{t=0} = (\hat{\omega}_{v0}^1, \hat{\omega}_{v0}^2)^{\top},
\hat{\omega}_{bh}^{nh}|_{t=0} = (\hat{\omega}_{b0}^1, \hat{\omega}_{b0}^2)^{\top},
\end{array}\right.
\end{equation}

Applying $\partial_t^{\ell}\mathcal{Z}^{\alpha}$ to \eqref{5.28}, we get
\begin{equation}\label{5.29}
\left\{\begin{array}{ll}
\partial_t \partial_t^{\ell}\mathcal{Z}^{\alpha}\hat{\omega}_{vh}^{nh}
 - \e\Delta^{\varphi^{\e}}\partial_t^{\ell}\mathcal{Z}^{\alpha}\hat{\omega}_{vh}^{nh}
 + v_y^{\e}\cdot\nabla_y \partial_t^{\ell}\mathcal{Z}^{\alpha}\hat{\omega}_{vh}^{nh}
\\[8pt]\qquad

+ V_z^{\e}\partial_z \partial_t^{\ell}\mathcal{Z}^{\alpha}\hat{\omega}_{vh}^{nh}
- b_y^{\e}\cdot\nabla_y \partial_t^{\ell}\mathcal{Z}^{\alpha}\hat{\omega}_{bh}^{nh}
- B_z^{\e}\partial_z \partial_t^{\ell}\mathcal{Z}^{\alpha}\hat{\omega}_{bh}^{nh}
\\[8pt]\quad

= \partial_t^{\ell}\mathcal{Z}^{\alpha} \mathcal{I}_{v6}
-[\partial_t^{\ell}\mathcal{Z}^{\alpha}, v_y^{\e}\cdot\nabla_y]\hat{\omega}_{vh}^{nh}
-[\partial_t^{\ell}\mathcal{Z}^{\alpha}, V_z\partial_z]\hat{\omega}_{vh}^{nh}
+[\partial_t^{\ell}\mathcal{Z}^{\alpha}, b_y^{\e}\cdot\nabla_y]\hat{\omega}_{bh}^{nh}
\\[8pt]\qquad

+[\partial_t^{\ell}\mathcal{Z}^{\alpha}, B_z\partial_z]\hat{\omega}_{bh}^{nh}
+ \e\nabla^{\varphi^{\e}}\cdot [\partial_t^{\ell}\mathcal{Z}^{\alpha}, \nabla^{\varphi}]\hat{\omega}_{vh}^{nhom}
+ \e[\partial_t^{\ell}\mathcal{Z}^{\alpha}, \nabla^{\varphi}\cdot]\nabla^{\varphi^{\e}}\hat{\omega}_{vh}^{nhom},
\\[9pt]

\partial_t \partial_t^{\ell}\mathcal{Z}^{\alpha}\hat{\omega}_{bh}^{nh}
 - \e\Delta^{\varphi^{\e}}\partial_t^{\ell}\mathcal{Z}^{\alpha}\hat{\omega}_{bh}^{nh}
 + v_y^{\e}\cdot\nabla_y \partial_t^{\ell}\mathcal{Z}^{\alpha}\hat{\omega}_{bh}^{nh}
 \\[8pt]\qquad

+ V_z^{\e}\partial_z \partial_t^{\ell}\mathcal{Z}^{\alpha}\hat{\omega}_{bh}^{nh}
- b_y^{\e}\cdot\nabla_y \partial_t^{\ell}\mathcal{Z}^{\alpha}\hat{\omega}_{vh}^{nh}
- B_z^{\e}\partial_z \partial_t^{\ell}\mathcal{Z}^{\alpha}\hat{\omega}_{vh}^{nh}
 \\[8pt]\quad

= \partial_t^{\ell}\mathcal{Z}^{\alpha} \mathcal{I}_{b6}
-[\partial_t^{\ell}\mathcal{Z}^{\alpha}, v_y^{\e}\cdot\nabla_y]\hat{\omega}_{bh}^{nh}
-[\partial_t^{\ell}\mathcal{Z}^{\alpha}, V_z\partial_z]\hat{\omega}_{bh}^{nh}
+[\partial_t^{\ell}\mathcal{Z}^{\alpha}, b_y^{\e}\cdot\nabla_y]\hat{\omega}_{vh}^{nh}
 \\[8pt]\qquad

+[\partial_t^{\ell}\mathcal{Z}^{\alpha}, B_z\partial_z]\hat{\omega}_{vh}^{nh}
+ \e\nabla^{\varphi^{\e}}\cdot [\partial_t^{\ell}\mathcal{Z}^{\alpha}, \nabla^{\varphi}]\hat{\omega}_{bh}^{nhom}
+ \e[\partial_t^{\ell}\mathcal{Z}^{\alpha}, \nabla^{\varphi}\cdot]\nabla^{\varphi^{\e}}\hat{\omega}_{bh}^{nhom},
\\[9pt]

\partial_t^{\ell}\mathcal{Z}^{\alpha}\hat{\omega}_{vh}^{nh}|_{z=0} =0,
\partial_t^{\ell}\mathcal{Z}^{\alpha}\hat{\omega}_{bh}^{nh}|_{z=0} =0,
\\[7pt]

\partial_t^{\ell}\mathcal{Z}^{\alpha}\hat{\omega}_{vh}^{nh}|_{t=0}
= (\partial_t^{\ell}\mathcal{Z}^{\alpha}\hat{\omega}_{v0}^1, \partial_t^{\ell}\mathcal{Z}^{\alpha}\hat{\omega}_{v0}^2)^{\top},
\partial_t^{\ell}\mathcal{Z}^{\alpha}\hat{\omega}_{bh}^{nh}|_{t=0}
= (\partial_t^{\ell}\mathcal{Z}^{\alpha}\hat{\omega}_{b0}^1, \partial_t^{\ell}\mathcal{Z}^{\alpha}\hat{\omega}_{b0}^2)^{\top},
\end{array}\right.
\end{equation}

To give the $L^2$ estimate of $\partial_t^{\ell}\mathcal{Z}^{\alpha}\hat{\omega}_{vh}^{nhom}$
and $\partial_t^{\ell}\mathcal{Z}^{\alpha}\hat{\omega}_{bh}^{nhom}$,
one has
\begin{equation*}
\begin{array}{ll}
\frac{\mathrm{d}}{\mathrm{d}t} (\|\partial_t^{\ell}\mathcal{Z}^{\alpha}\hat{\omega}_{vh}^{nh}\|_{L^2}^2
+\|\partial_t^{\ell}\mathcal{Z}^{\alpha}\hat{\omega}_{bh}^{nh}\|_{L^2}^2)
+ 2\e \|\nabla^{\varphi^{\e}}\partial_t^{\ell}\mathcal{Z}^{\alpha}\hat{\omega}_{vh}^{nh}\|_{L^2}^2
+ 2\e \|\nabla^{\varphi^{\e}}\partial_t^{\ell}\mathcal{Z}^{\alpha}\hat{\omega}_{bh}^{nh}\|_{L^2}^2
\\[10pt]

\lem \|\partial_t^{\ell}\mathcal{Z}^{\alpha}\hat{\omega}_{vh}^{nh}\|_{L^2}^2
+\|\partial_t^{\ell}\mathcal{Z}^{\alpha}\hat{\omega}_{bh}^{nh}\|_{L^2}^2
+ \e \int_{\mathbb{R}^3_{-}}\nabla^{\varphi^{\e}}\cdot [\partial_t^{\ell}\mathcal{Z}^{\alpha}, \nabla^{\varphi}]\hat{\omega}_{vh}^{nh}
\partial_t^{\ell}\mathcal{Z}^{\alpha}\hat{\omega}_{vh}^{nh} \,\mathrm{d}\mathcal{V}_t
\\[10pt]\quad

+ \|\partial_t^{\ell}\mathcal{Z}^{\alpha} \mathcal{I}_{v6}\|_{L^2}^2
+ \|\partial_t^{\ell}\mathcal{Z}^{\alpha} \mathcal{I}_{b6}\|_{L^2}^2
+ \e \int_{\mathbb{R}^3_{-}}\nabla^{\varphi^{\e}}\cdot [\partial_t^{\ell}\mathcal{Z}^{\alpha}, \nabla^{\varphi}]\hat{\omega}_{bh}^{nh}
\partial_t^{\ell}\mathcal{Z}^{\alpha}\hat{\omega}_{bh}^{nh} \,\mathrm{d}\mathcal{V}_t
\\[10pt]\quad

+ \|[\partial_t^{\ell}\mathcal{Z}^{\alpha}, V_z\partial_z]\hat{\omega}_{vh}^{nh}\|_{L^2}^2
+ \e \int_{\mathbb{R}^3_{-}}[\partial_t^{\ell}\mathcal{Z}^{\alpha}, \nabla^{\varphi}\cdot]\nabla^{\varphi^{\e}}\hat{\omega}_{vh}^{nhom}
\partial_t^{\ell}\mathcal{Z}^{\alpha}\hat{\omega}_{vh}^{nhom} \,\mathrm{d}\mathcal{V}_t
\\[10pt]\quad

+ \|[\partial_t^{\ell}\mathcal{Z}^{\alpha}, V_z\partial_z]\hat{\omega}_{bh}^{nh}\|_{L^2}^2
+ \e \int_{\mathbb{R}^3_{-}}[\partial_t^{\ell}\mathcal{Z}^{\alpha}, \nabla^{\varphi}\cdot]\nabla^{\varphi^{\e}}\hat{\omega}_{bh}^{nhom}
\partial_t^{\ell}\mathcal{Z}^{\alpha}\hat{\omega}_{bh}^{nhom} \,\mathrm{d}\mathcal{V}_t
\\[10pt]\quad

+ \|[\partial_t^{\ell}\mathcal{Z}^{\alpha}, B_z\partial_z]\hat{\omega}_{vh}^{nh}\|_{L^2}^2
+ \|[\partial_t^{\ell}\mathcal{Z}^{\alpha}, B_z\partial_z]\hat{\omega}_{bh}^{nh}\|_{L^2}^2
\end{array}
\end{equation*}
then
\begin{equation}\label{5.30}
\begin{array}{ll}
\frac{\mathrm{d}}{\mathrm{d}t} (\|\partial_t^{\ell}\mathcal{Z}^{\alpha}\hat{\omega}_{vh}^{nh}\|_{L^2}^2
+\|\partial_t^{\ell}\mathcal{Z}^{\alpha}\hat{\omega}_{bh}^{nh}\|_{L^2}^2)
+ 2\e \|\nabla^{\varphi^{\e}}\partial_t^{\ell}\mathcal{Z}^{\alpha}\hat{\omega}_{vh}^{nh}\|_{L^2}^2
+ 2\e \|\nabla^{\varphi^{\e}}\partial_t^{\ell}\mathcal{Z}^{\alpha}\hat{\omega}_{bh}^{nh}\|_{L^2}^2
\\[10pt]
\lem \|\partial_t^{\ell}\mathcal{Z}^{\alpha}\hat{\omega}_{vh}^{nh}\|_{L^2}^2
+\|\partial_t^{\ell}\mathcal{Z}^{\alpha}\hat{\omega}_{bh}^{nh}\|_{L^2}^2
+ \|\hat{\omega}_{vh}\|_{X^{k-1}}^2+ \|\hat{\omega}_{bh}\|_{X^{k-1}}^2
+ \|\hat{v}\|_{X^{k-1,1}}^2 + \|\hat{b}\|_{X^{k-1,1}}^2
  \\[10pt]\quad

+ \|\nabla\hat{\eta}\|_{X^{k-1}}^2
+ \|\partial_t^k\hat{\eta}\|_{L^2}^2
+ \sum\limits_{\ell_1+|\alpha_1|>0}(
\|\frac{1-z}{z}\partial_t^{\ell}\mathcal{Z}^{\alpha} V_z \cdot \partial_t^{\ell}\mathcal{Z}^{\alpha} \frac{z}{1-z}\partial_z\hat{\omega}_{vh}^{nh}\|_{L^2}^2
 \\[10pt]\quad

+\|\frac{1-z}{z}\partial_t^{\ell}\mathcal{Z}^{\alpha} V_z \cdot \partial_t^{\ell}\mathcal{Z}^{\alpha} \frac{z}{1-z}\partial_z\hat{\omega}_{bh}^{nh}\|_{L^2}^2+
\|\frac{1-z}{z}\partial_t^{\ell}\mathcal{Z}^{\alpha} B_z \cdot \partial_t^{\ell}\mathcal{Z}^{\alpha} \frac{z}{1-z}\partial_z\hat{\omega}_{vh}^{nh}\|_{L^2}^2
\\[10pt]\quad

+\|\frac{1-z}{z}\partial_t^{\ell}\mathcal{Z}^{\alpha} B_z \cdot \partial_t^{\ell}\mathcal{Z}^{\alpha} \frac{z}{1-z}\partial_z\hat{\omega}_{bh}^{nh}\|_{L^2}^2)
- \e \int_{\mathbb{R}^3_{-}}[\partial_t^{\ell}\mathcal{Z}^{\alpha}, \NN\partial_z^{\varphi}]\hat{\omega}_{vh}^{nh}
\cdot \nabla^{\varphi^{\e}}\partial_t^{\ell}\mathcal{Z}^{\alpha}\hat{\omega}_{vh}^{nh} \,\mathrm{d}\mathcal{V}_t
\\[10pt]\quad

+ \e \int_{\mathbb{R}^3_{-}} \sum\limits_{\ell_1+|\alpha_1|>0}\big[(\partial_z^{\varphi})^{-1}
\partial_t^{\ell_1}\mathcal{Z}^{\alpha_1}(\frac{\NN}{\partial_z\varphi})
\partial_t^{\ell_2}\mathcal{Z}^{\alpha_2}\partial_z\big]\cdot
\nabla^{\varphi^{\e}}\hat{\omega}_{vh}^{nh} \,
\partial_z^{\varphi}\partial_t^{\ell}\mathcal{Z}^{\alpha}\hat{\omega}_{vh}^{nh} \,\mathrm{d}\mathcal{V}_t,

\\[10pt]\quad
+ \e \int\limits_{\mathbb{R}^3_{-}} \sum\limits_{\ell_1+|\alpha_1|>0}\big[(\partial_z^{\varphi})^{-1}
\partial_t^{\ell_1}\mathcal{Z}^{\alpha_1}(\frac{\NN}{\partial_z\varphi})
\partial_t^{\ell_2}\mathcal{Z}^{\alpha_2}\partial_z\big]\cdot
\nabla^{\varphi^{\e}}\hat{\omega}_{bh}^{nh} \,
\partial_z^{\varphi}\partial_t^{\ell}\mathcal{Z}^{\alpha}\hat{\omega}_{bh}^{nh} \,\mathrm{d}\mathcal{V}_t,
\\[10pt]\quad

- \e \int_{\mathbb{R}^3_{-}}[\partial_t^{\ell}\mathcal{Z}^{\alpha}, \NN\partial_z^{\varphi}]\hat{\omega}_{bh}^{nh}
\cdot \nabla^{\varphi^{\e}}\partial_t^{\ell}\mathcal{Z}^{\alpha}\hat{\omega}_{bh}^{nh} \,\mathrm{d}\mathcal{V}_t .
\end{array}
\end{equation}
Here the notation $(\partial_z^{\varphi})^{-1}$ satisfying $(\partial_z^{\varphi})^{-1}(\partial_z^{\varphi}) =1$.

Integrating $(\ref{5.30})$ in time,
applying the Gronwall's inequality,
we get
\begin{equation}\label{5.31}
\begin{array}{ll}
\|\hat{\omega}_{vh}^{nh}\|_{X^{k-1}}^2+\|\hat{\omega}_{bh}^{nh}\|_{X^{k-1}}^2
+ 2\e \int_0^t\|\nabla\hat{\omega}_{vh}^{nh}\|_{X^{k-1}}^2\,\mathrm{d}t
+ 2\e \int_0^t\|\nabla\hat{\omega}_{bh}^{nh}\|_{X^{k-1}}^2 \,\mathrm{d}t
\\[10pt]

\leq \|\hat{\omega}_{v0,h}\|_{X^{k-1}}^2 +\|\hat{\omega}_{b0,h}\|_{X^{k-1}}^2 + \int_0^t\|\hat{\omega}_{vh}\|_{X^{k-1}}^2+\|\hat{\omega}_{bh}\|_{X^{k-1}}^2\,\mathrm{d}t
\\[10pt]\quad

+ \|\hat{h}\|_{X^{k-1,1}}^2\,\mathrm{d}t + \|\partial_t^k\hat{h}\|_{L^2}^2\,\mathrm{d}t + O(\e).
\end{array}
\end{equation}

Therefore, one has
\begin{equation}\label{5.32}
\begin{array}{ll}
\|\hat{\omega}_{vh}^{nh}\|_{L^4([0,T],X^{k-1})}^2+\|\hat{\omega}_{bh}^{nh}\|_{L^4([0,T],X^{k-1})}^2
\lem \sqrt{T}\big\|\hat{\omega}_{v0,h}\big\|_{X^{k-1}}^2
+ \sqrt{T}\big\|\hat{\omega}_{b0,h}\big\|_{X^{k-1}}^2
\\[10pt]\quad

+ T \|\hat{\omega}_{vh}\|_{L^4([0,T],X^{k-1})}^2
+ T \|\hat{\omega}_{bh}\|_{L^4([0,T],X^{k-1})}^2
+ \sqrt{T}\int_0^T\|\hat{v}\|_{X^{k-1,1}}^2 \,\mathrm{d}t
+ \sqrt{T}\int_0^T\|\hat{b}\|_{X^{k-1,1}}^2 \,\mathrm{d}t
\\[10pt]\quad

+\sqrt{T}\int_0^T|\hat{h}|_{X^{k-1,1}}^2 \,\mathrm{d}t
+ \sqrt{T}|\partial_t^k \hat{h}|_{L^4([0,T],L^2)}^2 + O(\e).
\end{array}
\end{equation}

For the homogeneous equations \eqref{5.27},
 we have
\begin{equation}\label{5.33}
\begin{array}{ll}
\|\partial_t^{\ell}\mathcal{Z}^{\alpha}\hat{\omega}_{vh}^{h}\|_{L^4([0,T],L^2(\mathbb{R}^3_{-}))}^2
+\|\partial_t^{\ell}\mathcal{Z}^{\alpha}\hat{\omega}_{bh}^{h}\|_{L^4([0,T],L^2(\mathbb{R}^3_{-}))}^2
\\[10pt]

\lem \|\partial_t^{\ell}\mathcal{Z}^{\alpha}\hat{\omega}_{vh}^{h}\|_{H^{\frac{1}{4}}([0,T],L^2(\mathbb{R}^3_{-}))}^2
+\|\partial_t^{\ell}\mathcal{Z}^{\alpha}\hat{\omega}_{bh}^{h}\|_{H^{\frac{1}{4}}([0,T],L^2(\mathbb{R}^3_{-}))}^2 \\[10pt]
\lem \sqrt{\e}\int_0^T\big|\hat{\omega}_{vh}^{h}|_{z=0}\big|_{X^{k-1}(\mathbb{R}^2)}^2 \,\mathrm{d}t
+\sqrt{\e}\int_0^T\big|\hat{\omega}_{bh}^{h}|_{z=0}\big|_{X^{k-1}(\mathbb{R}^2)}^2 \,\mathrm{d}t
\\[10pt]

\lem +\sqrt{\e}\int_0^T\big|\varsigma_1\Theta_v^1 + \varsigma_2\Theta_v^2
+\varsigma_3\Theta_v^3\big|_{X^{k-1}(\mathbb{R}^2)}^2 \,\mathrm{d}t
+\sqrt{\e}\int_0^T\big|\varsigma_1\Theta_b^1 + \varsigma_2\Theta_b^2
+\varsigma_3\Theta_b^3\big|_{X^{k-1}(\mathbb{R}^2)}^2 \,\mathrm{d}t
\\[8pt]\quad

+ \sqrt{\e}\int_0^T\big|\varsigma_4\Theta_v^4 + \varsigma_5\Theta_v^5
+ \varsigma_6\Theta_v^6\big|_{X^{k-1}(\mathbb{R}^2)}^2 \,\mathrm{d}t
+ \sqrt{\e}\int_0^T\big|\varsigma_4\Theta_b^4 + \varsigma_5\Theta_b^5
+ \varsigma_6\Theta_b^6\big|_{X^{k-1}(\mathbb{R}^2)}^2 \,\mathrm{d}t
\\[8pt]\quad

\sqrt{\e}\int_0^T\big|\textsf{F}_v^{1,2}[\nabla\varphi^{\e}](\partial_j v^{\e,i},\partial_j b^{\e,i}) -\textsf{F}_v^{1,2}[\nabla\varphi](\partial_j v^i,\partial_j b^i)
\big|_{X^{k-1}(\mathbb{R}^2)}^2 \,\mathrm{d}t
\\[8pt]\quad

+\sqrt{\e}\int_0^T\big|\textsf{F}_b^{1,2}[\nabla\varphi^{\e}](\partial_j v^{\e,i},\partial_j b^i) -\textsf{F}_b^{1,2}[\nabla\varphi](\partial_j v^i,\partial_j b^i)
\big|_{X^{k-1}(\mathbb{R}^2)}^2 \,\mathrm{d}t
\lem O(\sqrt{\e}),
\end{array}
\end{equation}
where $\varsigma_i$, $\Theta_v^i$ and $\Theta_b^i$
are defined in the proof of Lemma \ref{lemma4.1}.

By $(\ref{5.32})$ and $(\ref{5.33})$,
we have
\begin{equation}\label{5.34}
\begin{array}{ll}
\|\hat{\omega}_{vh}\|_{L^4([0,T],X^{k-1})}^2+\|\hat{\omega}_{bh}\|_{L^4([0,T],X^{k-1})}^2
\\[9pt]
\lem \|\hat{\omega}_{vh}^{nh}\|_{L^4([0,T],X^{k-1})}^2
+ \|\hat{\omega}_{vh}^{h}\|_{L^4([0,T],X^{k-1})}^2
+ \|\hat{\omega}_{bh}^{nh}\|_{L^4([0,T],X^{k-1})}^2
+ \|\hat{\omega}_{bh}^{h}\|_{L^4([0,T],X^{k-1})}^2  \\[9pt]

\lem \big\|\hat{\omega}_{v0,h}\big\|_{X^{k-1}}^2
+\|\hat{\omega}_{b0,h}\big\|_{X^{k-1}}^2
+ \int_0^T\|\hat{v}\|_{X^{k-1,1}}^2
+\|\hat{b}\|_{X^{k-1,1}}^2 \,\mathrm{d}t
\\[9pt]\quad

+ |\partial_t^k\hat{h}|_{L^4([0,T],L^2)}^2
+ \int_0^T|\hat{h}|_{X^{k-1,1}}^2 \,\mathrm{d}t
+ O(\sqrt{\e}).
\end{array}
\end{equation}
Lemma $\ref{5.5}$ is proved.
\end{proof}

\begin{remark}\label{remark5.1}
If $\Pi\mathcal{S}^{\varphi} v\nn|_{z=0} = 0$, $\Pi\mathcal{S}^{\varphi} v\nn|_{z=0} = 0$
 then $\Theta_v^i =0$, $\Theta_b^i =0$ where $i=1,\cdots,6$,
and then the estimate $(\ref{5.33})$ is reduced into the following estimate:
\begin{equation}
\begin{array}{ll}
\|\partial_t^{\ell}\mathcal{Z}^{\alpha}\hat{\omega}_{vh}^{h}\|_{L^4([0,T],L^2(\mathbb{R}^3_{-}))}^2
+\|\partial_t^{\ell}\mathcal{Z}^{\alpha}\hat{\omega}_{bh}^{h}\|_{L^4([0,T],L^2(\mathbb{R}^3_{-}))}^2
\\[9pt]

\lem \sqrt{\e}\int_0^T\big|\textsf{F}_v^{1,2}[\nabla\varphi^{\e}](\partial_j v^{\e,i}) -\textsf{F}_v^{1,2}[\nabla\varphi](\partial_j v^i)
\big|_{X^{k-1}(\mathbb{R}^2)}^2 \,\mathrm{d}t
\\[9pt]\quad

+\sqrt{\e}\int_0^T\big|\textsf{F}_b^{1,2}[\nabla\varphi^{\e}](\partial_j b^{\e,i}) -\textsf{F}_b^{1,2}[\nabla\varphi](\partial_j b^i)
\big|_{X^{k-1}(\mathbb{R}^2)}^2 \,\mathrm{d}t
\\[9pt]
\lem O(\sqrt{\e}).
\end{array}
\end{equation}
We can not improve the convergence rates of $\|\omega_v\|_{L^4([0,T],X^{k-1})}^2$ and
$\|\omega_b\|_{L^4([0,T],X^{k-1})}^2$,
since we do not have the convergence rates of $|\partial_j v^{\e,i} - \partial_j v^i|_{X^{k-1}(\mathbb{R}^2)}$,
$|\partial_j b^{\e,i} - \partial_j b^i|_{X^{k-1}(\mathbb{R}^2)}$.
However, we can improve the convergence rates of $\|\omega_v\|_{L^4([0,T],X^{k-2})}^2$,
see subsection $5.4$.
\end{remark}

Note that when $\Pi\mathcal{S}^{\varphi}v\nn|_{z=0} \neq 0$,
$\Pi\mathcal{S}^{\varphi}b\nn|_{z=0} \neq 0$,
then, not only
$\big|\nabla^{\varphi^{\e}} \times \partial_t^{\ell}\mathcal{Z}^{\alpha}(v^{\e} -v)|_{z=0}\big|_{L^2} \neq 0$, $\big|\nabla^{\varphi^{\e}} \times \partial_t^{\ell}\mathcal{Z}^{\alpha}(b^{\e} -b)|_{z=0}\big|_{L^2} \neq 0$,
 but also
$\big|\NN^{\e}\times(\nabla^{\varphi^{\e}} \times \partial_t^{\ell}\mathcal{Z}^{\alpha}(v^{\e} -v))|_{z=0}\big|_{L^2} \neq 0$,
$\big|\NN^{\e}\times(\nabla^{\varphi^{\e}} \times \partial_t^{\ell}\mathcal{Z}^{\alpha}(b^{\e} -b))|_{z=0}\big|_{L^2} \neq 0$.

Next, we prove the following lemma to estimates for $\|\partial_z \hat{v}\|_{L^{\infty}([0,T],X^{m-4})}$,
$\|\partial_z \hat{b}\|_{L^{\infty}([0,T],X^{m-4})}$,
$\|\hat{\omega}_v\|_{L^{\infty}([0,T],X^{m-4})}$
 and $\|\hat{\omega}_b\|_{L^{\infty}([0,T],X^{m-4})}$.
\begin{lemma}\label{Lemma5.6}
Assume $0\leq k\leq m-2$,
$\hat{\omega}_{vh} =\omega_{vh}^{\e} -\omega_{vh}$,
$\hat{\omega}_{bh} =\omega_{bh}^{\e} -\omega_{bh}$,
$\partial_z\hat{v} =\partial_z v^{\e} -\partial_z v$
and $\partial_z\hat{b} =\partial_z v^{\e} -\partial_z b$. We have, for
$\hat{\omega}_{vh}$, $\hat{\omega}_{bh}$, $\partial_z\hat{v}$
and $\partial_z\hat{b}$, that
\begin{equation}\label{5.36}
\begin{array}{ll}
\|\hat{\omega}_v\|_{X^{k-2}}^2 +\|\hat{\omega}_b\|_{X^{k-2}}^2
+ \|\partial_z\hat{v}\|_{X^{k-2}}^2
+ \|\partial_z\hat{b}\|_{X^{k-2}}^2
\\[9pt]

\lem \|\hat{\omega}_{v0}\|_{X^{k-2}}^2+ \|\hat{\omega}_{b0}\|_{X^{k-2}}^2
+ \int_0^t\|\hat{v}\|_{X^{k-2}} +\|\hat{b}\|_{X^{k-2}}
+\|\partial_z \hat{v}\|_{X^{k-2}}
 \\[9pt]\quad

+\|\partial_z \hat{b}\|_{X^{k-2}}
+ \|\nabla\hat{q} \|_{X^{k-2}} + \|\hat{h}\|_{X^{k-1}}\,\mathrm{d}t + O(\e).
\end{array}
\end{equation}
\end{lemma}

\begin{proof}
We rewrite \eqref{difference1} as
\begin{equation}\label{5.37}
\begin{array}{ll}
\partial_t^{\varphi^{\e}}\hat{v}-\partial_z^{\varphi} v \partial_t^{\varphi^{\e}}\hat{\eta}
+ v^{\e} \cdot\nabla^{\varphi^{\e}} \hat{v} - v^{\e}\cdot \nabla^{\varphi^{\e}}\hat{\eta}\, \partial_z^{\varphi} v + \hat{v}\cdot\nabla^{\varphi} v
- b^{\e} \cdot\nabla^{\varphi^{\e}} \hat{b}
\\[7pt]\quad

+ b^{\e}\cdot \nabla^{\varphi^{\e}}\hat{\eta}\, \partial_z^{\varphi} b
- \hat{b}\cdot\nabla^{\varphi} b
+ \nabla^{\varphi^{\e}} \hat{q} - \partial_z^{\varphi} q\nabla^{\varphi^{\e}}\hat{\eta}
= -\e\nabla^{\varphi^{\e}}\times \hat{\omega}_v -\e\nabla^{\varphi^{\e}}\times \omega_v,
\\[9pt]

\partial_t^{\varphi^{\e}}\hat{b}-\partial_z^{\varphi} b \partial_t^{\varphi^{\e}}\hat{\eta}
+ v^{\e} \cdot\nabla^{\varphi^{\e}} \hat{b} - v^{\e}\cdot \nabla^{\varphi^{\e}}\hat{\eta}\, \partial_z^{\varphi} b + \hat{v}\cdot\nabla^{\varphi} b
- b^{\e} \cdot\nabla^{\varphi^{\e}} \hat{v}
\\[7pt]\quad

+ b^{\e}\cdot \nabla^{\varphi^{\e}}\hat{\eta}\, \partial_z^{\varphi} v
- \hat{b}\cdot\nabla^{\varphi} v
= -\e\nabla^{\varphi^{\e}}\times \hat{\omega}_b -\e\nabla^{\varphi^{\e}}\times \omega_b.
\end{array}
\end{equation}

First, we give $L^2$ estimates of $\hat{\omega}_v$, $\hat{\omega}_b$.
 Multiplying \eqref{5.37}
with $\nabla^{\varphi^{\e}}\times(\nabla^{\varphi^{\e}}\times\hat{v})$,
$\nabla^{\varphi^{\e}}\times(\nabla^{\varphi^{\e}}\times\hat{b})$, respectively,
integrating in $\mathbb{R}^3_{-}$, using the integration by parts formula,
we get
\begin{equation}\label{5.38}
\begin{array}{ll}
\int_{\mathbb{R}^3_{-}}
\nabla^{\varphi^{\e}}\times\big(\partial_t^{\varphi^{\e}}\hat{v}
+ v^{\e} \cdot\nabla^{\varphi^{\e}} \hat{v}- b^{\e} \cdot\nabla^{\varphi^{\e}} \hat{b}
 + \nabla^{\varphi^{\e}} \hat{q}\big) \cdot\hat{\omega}_v\,\mathrm{d}\mathcal{V}_t
+ \e\int_{\mathbb{R}^3_{-}} |\nabla^{\varphi^{\e}}\times(\nabla^{\varphi^{\e}}\times\hat{v})|^2 \,\mathrm{d}\mathcal{V}_t\\[8pt]\quad

+\int_{\mathbb{R}^3_{-}}
\nabla^{\varphi^{\e}}\times\big(\partial_t^{\varphi^{\e}}\hat{b}
+ v^{\e} \cdot\nabla^{\varphi^{\e}} \hat{b}- b^{\e} \cdot\nabla^{\varphi^{\e}} \hat{b}
 \big)\cdot\hat{\omega}_b \,\mathrm{d}\mathcal{V}_t
+ \e\int\limits_{\mathbb{R}^3_{-}} |\nabla^{\varphi^{\e}}\times(\nabla^{\varphi^{\e}}\times\hat{b})|^2 \,\mathrm{d}\mathcal{V}_t\\[8pt]\quad

= \int_{\mathbb{R}^3_{-}}
\nabla^{\varphi^{\e}}\times\big(\partial_z^{\varphi} v \partial_t^{\varphi^{\e}}\hat{\eta}
+ v^{\e}\cdot \nabla^{\varphi^{\e}}\hat{\eta}\, \partial_z^{\varphi} v
-b^{\e}\cdot \nabla^{\varphi^{\e}}\hat{\eta}\, \partial_z^{\varphi} b
- \hat{v}\cdot\nabla^{\varphi} v+ \hat{b}\cdot\nabla^{\varphi}b
+ \partial_z^{\varphi} q\nabla^{\varphi^{\e}}\hat{\eta}\big)\cdot \hat{\omega}_v
\,\mathrm{d}\mathcal{V}_t \\[8pt]\qquad

- \int_{z=0}
\big(\partial_t\hat{v} + v^{\e}_y \cdot\nabla_y \hat{v} + \hat{v}\cdot\nabla^{\varphi} v
-\partial_z^{\varphi} v \partial_t\hat{\eta}
- v^{\e}_y\cdot \nabla_y \hat{\eta}\, \partial_z^{\varphi} v
+ \nabla^{\varphi^{\e}} \hat{q}
- b_y^{\e} \cdot\nabla^{\varphi^{\e}} \hat{b}_y
- \hat{b}\cdot\nabla^{\varphi} b

 \\[11pt]\qquad
+ b^{\e}\cdot \nabla^{\varphi^{\e}}\hat{\eta}\, \partial_z^{\varphi} b
- \partial_z^{\varphi} q\nabla^{\varphi^{\e}}\hat{\eta}
\big)\cdot \NN^{\e}\times (\nabla^{\varphi^{\e}}\times\hat{v}) \,\mathrm{d}y
- \e\int_{\mathbb{R}^3_{-}}
\nabla^{\varphi^{\e}}\times \omega_v \cdot \nabla^{\varphi^{\e}}\times (\nabla^{\varphi^{\e}}\times\hat{v}) \,\mathrm{d}\mathcal{V}_t
\\[8pt]\qquad

+\int_{\mathbb{R}^3_{-}}
\nabla^{\varphi^{\e}}\times\big(\partial_z^{\varphi} b \partial_t^{\varphi^{\e}}\hat{\eta}
+ v^{\e}\cdot \nabla^{\varphi^{\e}}\hat{\eta}\, \partial_z^{\varphi} b
-b^{\e}\cdot \nabla^{\varphi^{\e}}\hat{\eta}\, \partial_z^{\varphi} v
- \hat{v}\cdot\nabla^{\varphi} b+ \hat{b}\cdot\nabla^{\varphi}v
)\cdot \hat{\omega}_b \,\mathrm{d}\mathcal{V}_t \\[8pt]\qquad

- \int_{z=0}
\big(\partial_t\hat{b} + v^{\e}_y \cdot\nabla_y \hat{b} + \hat{v}\cdot\nabla^{\varphi} b
-\partial_z^{\varphi} b \partial_t\hat{\eta}
- v^{\e}_y\cdot \nabla_y \hat{\eta}\, \partial_z^{\varphi} b
- b_y^{\e} \cdot\nabla^{\varphi^{\e}} \hat{v}_y
- \hat{b}\cdot\nabla^{\varphi} v

 \\[11pt]\qquad
+ b^{\e}\cdot \nabla^{\varphi^{\e}}\hat{\eta}\, \partial_z^{\varphi} v
\big)\cdot \NN^{\e}\times (\nabla^{\varphi^{\e}}\times\hat{b}) \,\mathrm{d}y
- \e\int_{\mathbb{R}^3_{-}}
\nabla^{\varphi^{\e}}\times \omega_b \cdot \nabla^{\varphi^{\e}}\times (\nabla^{\varphi^{\e}}\times\hat{v}) \,\mathrm{d}\mathcal{V}_t :=J,
\end{array}
\end{equation}
where
\begin{equation}
\begin{array}{ll}
|J|\lem
\|\hat{\omega}_v\|_{L^2}^2 +\|\hat{\omega}_b\|_{L^2}^2
+ |\hat{h}|_{X^{1,\frac{1}{2}}}^2
+ \frac{\e}{2}\int_{\mathbb{R}^3_{-}} |\nabla^{\varphi^{\e}}\times\hat{\omega}_v|^2 +|\nabla^{\varphi^{\e}}\times\hat{\omega}_b|^2 \,\mathrm{d}\mathcal{V}_t + O(\e)\\[8pt]
\qquad

+ \|\hat{v}\|_{X^1}^2
+ \big|\NN^{\e}\times (\nabla^{\varphi^{\e}}\times\hat{v})|_{z=0}\big|_{L^2} \big(\big|\hat{v}|_{z=0}\big|_{X_{tan}^1}+\big|\hat{b}|_{z=0}\big|_{X_{tan}^1}
+ \big|\hat{h}|_{z=0}\big|_{X^1} \big)
\\[8pt]\qquad

+\|\hat{b}\|_{X^1}^2
+ \big|\NN^{\e}\times (\nabla^{\varphi^{\e}}\times\hat{b})|_{z=0}\big|_{L^2} \big(\big|\hat{v}|_{z=0}\big|_{X_{tan}^1}+\big|\hat{b}|_{z=0}\big|_{X_{tan}^1}
+ \big|\hat{h}|_{z=0}\big|_{X^1} \big)
\\[8pt]\qquad

+\|\partial_z \hat{v}\|_{L^2}^2+\|\partial_z \hat{b}\|_{L^2}^2
+ \big|\NN^{\e}\times (\nabla^{\varphi^{\e}}\times\hat{v})|_{z=0}\big|_{\frac{1}{2}} \big|\nabla \hat{q}|_{z=0}\big|_{-\frac{1}{2}}.
\end{array}
\end{equation}

Since $\nabla^{\varphi^{\e}}\times \nabla^{\varphi^{\e}} \hat{q} =0$,
$\nabla^{\varphi^{\e}}\times \omega_v$ and
$\nabla^{\varphi^{\e}}\times \omega_b$ is bounded,
$\big|\NN^{\e}\times (\nabla^{\varphi^{\e}}\times\hat{v})|_{z=0}\big|_{L^2}\neq 0$
and $\big|\NN^{\e}\times (\nabla^{\varphi^{\e}}\times\hat{b})|_{z=0}\big|_{L^2}\neq 0$, we have
\begin{equation}\label{5.40}
\begin{array}{ll}
\|\hat{\omega}_v\|_{L^2}+\|\hat{\omega}_b\|_{L^2}
+ \e\int_0^t \|\nabla\hat{\omega}_v\|^2
+ \e\int_0^t \|\nabla\hat{\omega}_b\|^2 \,\mathrm{d}t
\\[8pt]

\lem \|\hat{\omega}_{v0}\|_{L^2}^2 +\|\hat{\omega}_{b0}\|_{L^2}^2 +
 \int_0^t |\hat{h}|_{X^{1,1}}^2 + \|\hat{v}\|_{X^1}^2
 + \|\hat{b}\|_{X^1}^2 + \|\partial_z \hat{v}\|_{L^2}^2
\\[8pt]\quad

+ \|\partial_z \hat{b}\|_{L^2}^2
+ \big|\hat{v}|_{z=0}\big|_{X_{tan}^1}+ \big|\hat{b}|_{z=0}\big|_{X_{tan}^1}
 + \big|\hat{h}|_{z=0}\big|_{X^1}
+ \big|\nabla \hat{q}|_{z=0}\big|_{-\frac{1}{2}} \,\mathrm{d}t  + O(\e) \\[9pt]

\lem \|\hat{\omega}_{v0}\|_{L^2}^2 +\|\hat{\omega}_{b0}\|_{L^2}^2
+ \int_0^t |\hat{h}|_{X^{1,1}}^2 + \|\hat{v}\|_{X^1}^2+ \|\hat{b}\|_{X^1}^2
+ \|\partial_z \hat{v}\|_{L^2}^2+ \|\partial_z \hat{b}\|_{L^2}^2
\\[8pt]\quad

+ \|\hat{v}\|_{X^1}+ \|\hat{b}\|_{X^1}
+ \|\partial_z\hat{v}\|_{X_{tan}^1} + \|\partial_z\hat{b}\|_{X_{tan}^1}
+ \|\hat{h}\|_{X^1} + \|\nabla \hat{q}\|_{L^2} \,\mathrm{d}t  + O(\e).
\end{array}
\end{equation}

When $\ell+|\alpha|\leq k-2$, we estimate the quantity
$\nabla^{\varphi^{\e}}\times \partial_t^{\ell}\mathcal{Z}^{\alpha}\hat{v}$
and $\nabla^{\varphi^{\e}}\times \partial_t^{\ell}\mathcal{Z}^{\alpha}\hat{b}$.

Applying $\partial_t^{\ell}\mathcal{Z}^{\alpha}$ to the equations \eqref{difference1},
one has
\begin{equation}\label{5.41}
\begin{array}{ll}
\partial_t^{\varphi^{\e}} \partial_t^{\ell}\mathcal{Z}^{\alpha}\hat{v}
+ v^{\e} \cdot\nabla^{\varphi^{\e}} \partial_t^{\ell}\mathcal{Z}^{\alpha}\hat{v}
-b^{\e} \cdot\nabla^{\varphi^{\e}} \partial_t^{\ell}\mathcal{Z}^{\alpha}\hat{b}
+ \nabla^{\varphi^{\e}} \partial_t^{\ell}\mathcal{Z}^{\alpha}\hat{q}
+ \e\nabla^{\varphi^\e}\times \nabla^{\varphi^\e} \partial_t^{\ell}\mathcal{Z}^{\alpha} \hat{v}
\\[6pt]\quad
= \e \, \mathcal{I}_{v7,1} + \mathcal{I}_{v7,2},
\\[8pt]
\partial_t^{\varphi^{\e}} \partial_t^{\ell}\mathcal{Z}^{\alpha}\hat{b}
+ v^{\e} \cdot\nabla^{\varphi^{\e}} \partial_t^{\ell}\mathcal{Z}^{\alpha}\hat{b}
-b^{\e} \cdot\nabla^{\varphi^{\e}} \partial_t^{\ell}\mathcal{Z}^{\alpha}\hat{v}
+ \e\nabla^{\varphi^\e}\times \nabla^{\varphi^\e} \partial_t^{\ell}\mathcal{Z}^{\alpha} \hat{b}
\\[6pt]\quad
= \e \, \mathcal{I}_{b7,1} + \mathcal{I}_{b7,2},
\end{array}
\end{equation}
where
\begin{equation}\label{5.42}
\begin{array}{ll}
\mathcal{I}_{v7,1} = -[\partial_t^{\ell}\mathcal{Z}^{\alpha}, \nabla^{\varphi^\e}\times]\nabla^{\varphi^\e}\times\hat{v}
- \nabla^{\varphi^\e}\times[\partial_t^{\ell}\mathcal{Z}^{\alpha}, \nabla^{\varphi^\e}\times] \hat{v}
+ \partial_t^{\ell}\mathcal{Z}^{\alpha}\triangle^{\varphi^{\e}}v,
\\[9pt]

\mathcal{I}_{b7,1} = -[\partial_t^{\ell}\mathcal{Z}^{\alpha}, \nabla^{\varphi^\e}\times]\nabla^{\varphi^\e}\times\hat{b}
- \nabla^{\varphi^\e}\times[\partial_t^{\ell}\mathcal{Z}^{\alpha}, \nabla^{\varphi^\e}\times] \hat{b}
+ \partial_t^{\ell}\mathcal{Z}^{\alpha}\triangle^{\varphi^{\e}}b,
\end{array}
\end{equation}
\begin{equation}
\begin{array}{ll}
\mathcal{I}_{v7,2} := \partial_z^{\varphi} v (\partial_t + v^{\e}_y\cdot \nabla_y + V_z^{\e}\partial_z
-b^{\e}_y\cdot \nabla_y + B_z^{\e}\partial_z)
\partial_t^{\ell}\mathcal{Z}^{\alpha}\hat{\varphi}
- \partial_t^{\ell}\mathcal{Z}^{\alpha}\hat{v}\cdot\nabla^{\varphi} v
\\[8pt]\quad

+\partial_t^{\ell}\mathcal{Z}^{\alpha}\hat{b}\cdot\nabla^{\varphi} b
+ \partial_z^{\varphi} q\nabla^{\varphi^{\e}}\partial_t^{\ell}\mathcal{Z}^{\alpha}\hat{\varphi}
- [\partial_t^{\ell}\mathcal{Z}^{\alpha},\partial_t +v^{\e}\partial_y + V_z^{\e}\partial_z)]\hat{v}
\\[8pt]\quad

+ [\partial_t^{\ell}\mathcal{Z}^{\alpha}, b^{\e}\partial_y + B_z^{\e}\partial_z]\hat{b}
+ [\partial_t^{\ell}\mathcal{Z}^{\alpha}, \partial_z^{\varphi} v (\partial_t + v^{\e}_y\cdot \nabla_y + V_z^{\e}\partial_z)]\hat{\varphi}
\\[8pt]\quad

+ [\partial_t^{\ell}\mathcal{Z}^{\alpha}, \partial_z^{\varphi} b (b^{\e}_y\cdot \nabla_y + B_z^{\e}\partial_z]\hat{\varphi}
- [\partial_t^{\ell}\mathcal{Z}^{\alpha}, \nabla^{\varphi} v\cdot]\hat{v}
\\[8pt]\quad

- [\partial_t^{\ell}\mathcal{Z}^{\alpha}, \nabla^{\varphi} b\cdot]\hat{b}
- [\partial_t^{\ell}\mathcal{Z}^{\alpha},\nabla^{\varphi^{\e}}] \hat{q}
+ [\partial_t^{\ell}\mathcal{Z}^{\alpha},\partial_z^{\varphi} q\nabla^{\varphi^{\e}}]\hat{\varphi},
\end{array}
\end{equation}
and
\begin{equation}
\begin{array}{ll}
\mathcal{I}_{b7,2} := \partial_z^{\varphi} b (\partial_t + v^{\e}_y\cdot \nabla_y + V_z^{\e}\partial_z
-b^{\e}_y\cdot \nabla_y + B_z^{\e}\partial_z)
\partial_t^{\ell}\mathcal{Z}^{\alpha}\hat{\varphi}
- \partial_t^{\ell}\mathcal{Z}^{\alpha}\hat{v}\cdot\nabla^{\varphi} b
\\[8pt]\quad

+\partial_t^{\ell}\mathcal{Z}^{\alpha}\hat{b}\cdot\nabla^{\varphi} v
- [\partial_t^{\ell}\mathcal{Z}^{\alpha},\partial_t +v^{\e}\partial_y + V_z^{\e}\partial_z)]\hat{b}
\\[8pt]\quad

+ [\partial_t^{\ell}\mathcal{Z}^{\alpha}, b^{\e}\partial_y + B_z^{\e}\partial_z]\hat{v}
+ [\partial_t^{\ell}\mathcal{Z}^{\alpha}, \partial_z^{\varphi} v (\partial_t + v^{\e}_y\cdot \nabla_y + V_z^{\e}\partial_z)]\hat{\varphi}
\\[8pt]\quad

+ [\partial_t^{\ell}\mathcal{Z}^{\alpha}, \partial_z^{\varphi} v (b^{\e}_y\cdot \nabla_y + B_z^{\e}\partial_z]\hat{\varphi}
- [\partial_t^{\ell}\mathcal{Z}^{\alpha}, \nabla^{\varphi} b\cdot]\hat{v}
- [\partial_t^{\ell}\mathcal{Z}^{\alpha}, \nabla^{\varphi} b\cdot]\hat{v}.
\end{array}
\end{equation}

Note that $\nabla^{\varphi^{\e}}\times \nabla^{\varphi^{\e}} \partial_t^{\ell}\mathcal{Z}^{\alpha}\hat{q} =0$
and $(\partial_t^{\varphi^{\e}} + v^{\e} \cdot\nabla^{\varphi^{\e}})|_{z=0} = (\partial_t + v_y^{\e}\cdot\nabla_y)$.
 Multiplying $(\ref{5.41})$ with
$\nabla^{\varphi^{\e}}\times (\nabla^{\varphi^{\e}}\times\partial_t^{\ell}\mathcal{Z}^{\alpha}\hat{v})$
and $\nabla^{\varphi^{\e}}\times (\nabla^{\varphi^{\e}}\times\partial_t^{\ell}\mathcal{Z}^{\alpha}\hat{b})$,
respectively, integrating in $\mathbb{R}^3_{-}$,
we get
\begin{equation}\label{5.45}
\begin{array}{ll}
\frac{1}{2}\frac{\mathrm{d}}{\mathrm{d}t} \int_{\mathbb{R}^3_{-}}
|\nabla^{\varphi^{\e}}\times\partial_t^{\ell}\mathcal{Z}^{\alpha}\hat{v}|^2
+|\nabla^{\varphi^{\e}}\times\partial_t^{\ell}\mathcal{Z}^{\alpha}\hat{b}|^2
 \,\mathrm{d}\mathcal{V}_t^{\e}
\\[10pt]\quad

+ \e\int_{\mathbb{R}^3_{-}}|\nabla^{\varphi^\e}\times \nabla^{\varphi^\e} \partial_t^{\ell}\mathcal{Z}^{\alpha} \hat{v}|^2
+|\nabla^{\varphi^\e}\times \nabla^{\varphi^\e} \partial_t^{\ell}\mathcal{Z}^{\alpha} \hat{b}|^2
\,\mathrm{d}\mathcal{V}_t^{\e}
\\[10pt]

= - \int_{z=0}(\partial_t + v_y^{\e}\cdot\nabla_y)
\partial_t^{\ell}\mathcal{Z}^{\alpha}\hat{v} \cdot
\NN^{\e}\times (\nabla^{\varphi^{\e}}\times
\partial_t^{\ell}\mathcal{Z}^{\alpha}\hat{v}) \,\mathrm{d}y
+ \int_{z=0}\mathcal{I}_{v7,2} \cdot \NN^{\e}\times (\nabla^{\varphi^{\e}}\times
\partial_t^{\ell}\mathcal{Z}^{\alpha}\hat{v}) \,\mathrm{d}y
\\[10pt]\quad

- \int_{z=0}(\partial_t + v_y^{\e}\cdot\nabla_y) \partial_t^{\ell}\mathcal{Z}^{\alpha}\hat{b} \cdot
\NN^{\e}\times (\nabla^{\varphi^{\e}}\times
\partial_t^{\ell}\mathcal{Z}^{\alpha}\hat{b}) \,\mathrm{d}y
+ \int_{z=0}\mathcal{I}_{b7,2} \cdot \NN^{\e}\times (\nabla^{\varphi^{\e}}\times
\partial_t^{\ell}\mathcal{Z}^{\alpha}\hat{b}) \,\mathrm{d}y
\\[10pt]\quad

- \int_{\mathbb{R}^3_{-}} [(\sum\limits_{i=1}^3 \nabla^{\varphi^{\e}}
v^{\e,i} \cdot\partial_i^{\varphi^{\e}}) \times \partial_t^{\ell}\mathcal{Z}^{\alpha}\hat{v}
-(\sum_{i=1}^3 \nabla^{\varphi^{\e}}
b^{\e,i} \cdot\partial_i^{\varphi^{\e}}) \times \partial_t^{\ell}\mathcal{Z}^{\alpha}\hat{b}]
\cdot (\nabla^{\varphi^{\e}}\times
\partial_t^{\ell}\mathcal{Z}^{\alpha}\hat{v}) \,\mathrm{d}\mathcal{V}_t^{\e}
\\[10pt]\quad

- \int_{\mathbb{R}^3_{-}} [(\sum\limits_{i=1}^3 \nabla^{\varphi^{\e}}
v^{\e,i} \cdot\partial_i^{\varphi^{\e}}) \times \partial_t^{\ell}\mathcal{Z}^{\alpha}\hat{b}
-(\sum_{i=1}^3 \nabla^{\varphi^{\e}}
b^{\e,i} \cdot\partial_i^{\varphi^{\e}}) \times \partial_t^{\ell}\mathcal{Z}^{\alpha}\hat{v}]
\cdot (\nabla^{\varphi^{\e}}\times
\partial_t^{\ell}\mathcal{Z}^{\alpha}\hat{b}) \,\mathrm{d}\mathcal{V}_t^{\e}
\\[10pt]\quad

+ \e\int_{\mathbb{R}^3_{-}}\mathcal{I}_{v7,1} \cdot \nabla^{\varphi^{\e}}\times (\nabla^{\varphi^{\e}}\times
\partial_t^{\ell}\mathcal{Z}^{\alpha}\hat{v}) \,\mathrm{d}\mathcal{V}_t^{\e}
+ \e\int_{\mathbb{R}^3_{-}}\mathcal{I}_{b7,1} \cdot \nabla^{\varphi^{\e}}\times (\nabla^{\varphi^{\e}}\times
\partial_t^{\ell}\mathcal{Z}^{\alpha}\hat{b}) \,\mathrm{d}\mathcal{V}_t^{\e}
\\[11pt]\quad

+ \int_{\mathbb{R}^3_{-}} \nabla^{\varphi^{\e}}\times \mathcal{I}_{v7,2} \cdot (\nabla^{\varphi^{\e}}\times
\partial_t^{\ell}\mathcal{Z}^{\alpha}\hat{v}) \,\mathrm{d}\mathcal{V}_t^{\e}
+ \int_{\mathbb{R}^3_{-}} \nabla^{\varphi^{\e}}\times \mathcal{I}_{b7,2} \cdot (\nabla^{\varphi^{\e}}\times
\partial_t^{\ell}\mathcal{Z}^{\alpha}\hat{b}) \,\mathrm{d}\mathcal{V}_t^{\e}
\\[10pt]\quad

- \int_{z=0}\nabla^{\varphi^{\e}} \partial_t^{\ell}\mathcal{Z}^{\alpha}\hat{q} \cdot
\NN^{\e}\times (\nabla^{\varphi^{\e}}\times
\partial_t^{\ell}\mathcal{Z}^{\alpha}\hat{v}) \,\mathrm{d}y:=G.
\end{array}
\end{equation}
For $G$, one has
\begin{equation}
\begin{array}{ll}
G\lem \|\nabla^{\varphi^{\e}}\times \partial_t^{\ell}\mathcal{Z}^{\alpha}\hat{v}\|_{L^2}^2
+\|\nabla^{\varphi^{\e}}\times \partial_t^{\ell}\mathcal{Z}^{\alpha}\hat{b}\|_{L^2}^2
+ |\NN^{\e}\times (\nabla^{\varphi^{\e}}\times \partial_t^{\ell}\mathcal{Z}^{\alpha}\hat{v})|_{\frac{1}{2}}
\big|\nabla^{\varphi^{\e}} \partial_t^{\ell}\mathcal{Z}^{\alpha}\hat{q}|_{z=0} \big|_{-\frac{1}{2}} \\[10pt]\quad

+ |\NN^{\e}\times (\nabla^{\varphi^{\e}}\times \partial_t^{\ell}\mathcal{Z}^{\alpha}\hat{v})|_{L^2}
\big(\big|\mathcal{I}_{v7,2}|_{z=0}\big|_{L^2}
+ \big|\partial_t^{\ell}\mathcal{Z}^{\alpha}\hat{v}|_{z=0} \big|_{X_{tan}^1}
+ \big|\partial_t^{\ell}\mathcal{Z}^{\alpha}\hat{b}|_{z=0} \big|_{X_{tan}^1} \big)
 \\[10pt]\quad

 + |\NN^{\e}\times (\nabla^{\varphi^{\e}}\times \partial_t^{\ell}\mathcal{Z}^{\alpha}\hat{b})|_{L^2}
\big(\big|\mathcal{I}_{b7,2}|_{z=0}\big|_{L^2}
+ \big|\partial_t^{\ell}\mathcal{Z}^{\alpha}\hat{v}|_{z=0} \big|_{X_{tan}^1}
+ \big|\partial_t^{\ell}\mathcal{Z}^{\alpha}\hat{b}|_{z=0} \big|_{X_{tan}^1} \big)
 \\[10pt]\quad

+ \|\partial_t^{\ell}\mathcal{Z}^{\alpha}\hat{v}\|_{X^1}^2
+ \|\partial_z \partial_t^{\ell}\mathcal{Z}^{\alpha}\hat{v}\|_{L^2}^2
+ \e\|\mathcal{I}_{v7,1}\|_{L^2}^2
+ \|\nabla^{\varphi^{\e}}\times \mathcal{I}_{v7,2}\|_{L^2}^2
 \\[10pt]\quad

+ \|\partial_t^{\ell}\mathcal{Z}^{\alpha}\hat{b}\|_{X^1}^2
+ \|\partial_z \partial_t^{\ell}\mathcal{Z}^{\alpha}\hat{b}\|_{L^2}^2
+ \e\|\mathcal{I}_{b7,1}\|_{L^2}^2
+ \|\nabla^{\varphi^{\e}}\times \mathcal{I}_{b7,2}\|_{L^2}^2.
\end{array}
\end{equation}
It follows that
\begin{equation*}
\begin{array}{ll}
\big|\mathcal{I}_{v7,2}|_{z=0}\big|_{L^2}+\big|\mathcal{I}_{b7,2}|_{z=0}\big|_{L^2}
 \lem |\hat{h}|_{X^{k-1}} + \big|\hat{v}|_{z=0}\big|_{X^{k-2}}
 + \big|\hat{b}|_{z=0}\big|_{X^{k-2}}
+ \big|\nabla\hat{q}|_{z=0}\big|_{X^{k-3}}
 \\[6pt]\hspace{1.85cm}

\lem |\hat{h}|_{X^{k-1}} + \|\hat{v}\|_{X^{k-2}}+ \|\hat{b}\|_{X^{k-2}}
 + \|\partial_z\hat{v}\|_{X^{k-2}}
+ \|\nabla\hat{q}\|_{X^{k-2}},
\end{array}
\end{equation*}
and
\begin{equation}
\begin{array}{ll}
\e\|\mathcal{I}_{v7,1}\|_{L^2}^2+\e\|\mathcal{I}_{b7,1}\|_{L^2}^2\\[6pt]
 \lem \e\sum\limits_{\ell+|\alpha|\leq k-2} (\|\nabla^{\varphi^\e}\times \nabla^{\varphi^\e} \partial_t^{\ell}\mathcal{Z}^{\alpha} \hat{v}\|_{L^2}^2
 +\|\nabla^{\varphi^\e}\times \nabla^{\varphi^\e} \partial_t^{\ell}\mathcal{Z}^{\alpha} \hat{b}\|_{L^2}^2) + O(\e),
\\[13pt]

\|\nabla^{\varphi^{\e}}\times \mathcal{I}_{v7,2}\|_{L^2}^2
+\|\nabla^{\varphi^{\e}}\times \mathcal{I}_{b7,2}\|_{L^2}^2
\lem \|\hat{\eta}\|_{X^{k-1,1}}^2
+ \|\nabla^{\varphi^{\e}}\times \hat{v}\|_{X_{tan}^{k-2}}^2
+ \|\nabla^{\varphi^{\e}}\times \hat{b}\|_{X_{tan}^{k-2}}^2
\\[6pt]\quad

+ \|\nabla^{\varphi^{\e}}\times [\partial_t^{\ell}\mathcal{Z}^{\alpha},V_z^{\e}\partial_z]\hat{v}\|_{L^2}^2
+ \|\nabla^{\varphi^{\e}}\times [\partial_t^{\ell}\mathcal{Z}^{\alpha},V_z^{\e}\partial_z]\hat{b}\|_{L^2}^2
+ \|\nabla^{\varphi^{\e}}\times[\partial_t^{\ell}\mathcal{Z}^{\alpha},\NN^{\e}\partial_z^{\varphi^{\e}}] \hat{q}\|_{L^2}^2,
\end{array}
\end{equation}
where the estimates for the last two terms is similar to the arguments \eqref{Hady} by using Hardy's inequality and thus we omit the details.

Integrating \eqref{5.45} in time, applying the Gronwall's inequality,
one has
\begin{equation}\label{5.48}
\begin{array}{ll}
\|\nabla^{\varphi^{\e}}\times \partial_t^{\ell}\mathcal{Z}^{\alpha}\hat{v}\|_{L^2}^2
+ \|\nabla^{\varphi^{\e}}\times \partial_t^{\ell}\mathcal{Z}^{\alpha}\hat{b}\|_{L^2}^2
+ \e\int_{\mathbb{R}^3_{-}}|\nabla^{\varphi^\e}\times \nabla^{\varphi^\e} \partial_t^{\ell}\mathcal{Z}^{\alpha} \hat{v}|^2
+ |\nabla^{\varphi^\e}\times \nabla^{\varphi^\e} \partial_t^{\ell}\mathcal{Z}^{\alpha} \hat{b}|^2
\,\mathrm{d}\mathcal{V}_t^{\e}
\\[9pt]

\lem \big\|\nabla^{\varphi^{\e}}\times \partial_t^{\ell}\mathcal{Z}^{\alpha}\hat{v}|_{t=0}\big\|_{L^2}^2
+\big\|\nabla^{\varphi^{\e}}\times \partial_t^{\ell}\mathcal{Z}^{\alpha}\hat{b}|_{t=0}\big\|_{L^2}^2
\\[9pt]\quad

+ \int_0^t \|\hat{v}\|_{X^{k-2}}
+ \|\partial_z \hat{v}\|_{X^{k-2}}
+\|\hat{b}\|_{X^{k-2}}
+ \|\partial_z \hat{b}\|_{X^{k-2}}
+ \|\nabla\hat{q} \|_{X^{k-2}}
+ \|\hat{h}\|_{X^{k-1}} \,\mathrm{d}t
\\[9pt]\quad

+ \int_0^t \|\hat{v}\|_{X^{k-1}}^2 + \|\partial_z \hat{v}\|_{X^{k-1}}^2 +
 \|\hat{b}\|_{X^{k-1}}^2 + \|\partial_z \hat{b}\|_{X^{k-1}}^2 +
 \|\hat{h}\|_{X^{k-1,1}}^2 \,\mathrm{d}t
+ O(\e) \\[9pt]

\lem \big\|\partial_t^{\ell}\mathcal{Z}^{\alpha}\hat{\omega}_v|_{t=0}\big\|_{L^2}^2
+\big\|\partial_t^{\ell}\mathcal{Z}^{\alpha}\hat{\omega}_b|_{t=0}\big\|_{L^2}^2
+ \int_0^t\|\hat{v}\|_{X^{k-2}} + \|\partial_z \hat{v}\|_{X^{k-2}}
\\[9pt]\quad

+\|\hat{b}\|_{X^{k-2}} + \|\partial_z \hat{b}\|_{X^{k-2}}
+ \|\nabla\hat{q} \|_{X^{k-2}} + \|\hat{h}\|_{X^{k-1}}\,\mathrm{d}t + O(\e).
\end{array}
\end{equation}
Since $\hat{\omega}_v=\nabla^{\varphi^{\e}} \times \hat{v}
- \nabla^{\varphi^{\e}}\hat{\eta} \times \partial_z^{\varphi} v$ and
$\hat{\omega}_b = \nabla^{\varphi^{\e}} \times \hat{b}
- \nabla^{\varphi^{\e}}\hat{\eta} \times \partial_z^{\varphi} b$, we have
\begin{equation}
\begin{array}{ll}
\|\hat{\omega}_v\|_{X^{k-2}}^2+\|\hat{\omega}_b\|_{X^{k-2}}^2
+ \|\partial_z\hat{v}\|_{X^{k-2}}^2+ \|\partial_z\hat{b}\|_{X^{k-2}}^2
\lem \|\hat{\omega}_v|_{t=0}\|_{X^{k-2}}^2+\|\hat{\omega}_b|_{t=0}\|_{X^{k-2}}^2
\\[11pt]\quad

+ \int_0^t\|\hat{v}\|_{X^{k-2}} + \|\partial_z \hat{v}\|_{X^{k-2}}
+\|\hat{b}\|_{X^{k-2}} + \|\partial_z \hat{b}\|_{X^{k-2}}
+ \|\nabla\hat{q} \|_{X^{k-2}} + \|\hat{h}\|_{X^{k-1}}\,\mathrm{d}t + O(\e).
\end{array}
\end{equation}

Lemma $\ref{Lemma5.6}$ is proved.
\end{proof}

\subsection{Estimates for the Normal Derivatives when
$\Pi\mathcal{S}^{\varphi} v\nn|_{z=0} = 0$ and $\Pi\mathcal{S}^{\varphi} b\nn|_{z=0} = 0$}

In this subsection, we give the estimates for the normal derivatives if the ideal MHD equations satisfies
$\Pi\mathcal{S}^{\varphi} v\nn|_{z=0} = 0$
and $\Pi\mathcal{S}^{\varphi} v\nn|_{z=0} = 0$.

The following lemma shows the estimate of the normal derivatives.
\begin{lemma}\label{Lemma5.7}
Assume $k\leq m-2$, if $\Pi\mathcal{S}^{\varphi} v \nn|_{z=0} = 0$ and
$\Pi\mathcal{S}^{\varphi} b \nn|_{z=0} = 0$.
We have
\begin{equation}\label{5.51}
\begin{array}{ll}
\|\partial_z\hat{v}_h\|_{L^4([0,T],X^{k-2})}^2 +
\|\partial_z\hat{b}_h\|_{L^4([0,T],X^{k-2})}^2 +
 \|\hat{\omega}_{vh}\|_{L^4([0,T],X^{k-2})}^2+
 \|\hat{\omega}_{bh}\|_{L^4([0,T],X^{k-2})}^2
\\[6pt]
\lem \big\|\hat{\omega}_{v0}\big\|_{X^{k-2}}^2
+\big\|\hat{\omega}_{b0}\big\|_{X^{k-2}}^2
+ \int_0^T\|\hat{v}\|_{X^{k-1,1}}^2+\|\hat{b}\|_{X^{k-1,1}}^2 \,\mathrm{d}t
+ \int_0^T|\hat{h}|_{X^{k-2,1}}^2 \,\mathrm{d}t
 \\[8pt]\quad

+ \|\partial_t^{k-1}\hat{h}\|_{L^4([0,T],L^2)}^2
+ \sqrt{\e}\|\partial_z\hat{v}\|_{L^4([0,T],X^{k-1})}^2
+ \sqrt{\e}\|\partial_z\hat{b}\|_{L^4([0,T],X^{k-1})}^2
+ O(\e).
\end{array}
\end{equation}
\end{lemma}

\begin{proof}
If $\Pi\mathcal{S}^{\varphi} v\nn|_{z=0} = 0$
, $\Pi\mathcal{S}^{\varphi} b\nn|_{z=0} = 0$, then $\Theta_v^i =0$
 and $\Theta_b^i =0$ where $i=1,\cdots,6$.

Assume $\ell+|\alpha|\leq k-2$. We study the equations $(\ref{1.25})$
by decomposing $\hat{\omega}_{vh} = \hat{\omega}_{vh}^{nh} + \hat{\omega}_{vh}^{h}$,
$\hat{\omega}_{bh} = \hat{\omega}_{bh}^{nh} + \hat{\omega}_{bh}^{h}$, such that
$\hat{\omega}_{vh}^{nh}$, $\hat{\omega}_{bh}^{nh}$
satisfy the nonhomogeneous equations $(\ref{5.26})$
and $\hat{\omega}_{vh}^{h}$, $\hat{\omega}_{bh}^{h}$
satisfy the homogeneous equations $(\ref{5.27})$.

For $\hat{\omega}_{vh}^{nh}$ and $\hat{\omega}_{bh}^{nh}$, we get
\begin{equation}\label{5.52}
\begin{array}{ll}
\|\hat{\omega}_{vh}^{nh}\|_{L^4([0,T],X^{k-2})}^2
+\|\hat{\omega}_{bh}^{nh}\|_{L^4([0,T],X^{k-2})}^2
\\[9pt]

\lem \sqrt{T}\big\|\hat{\omega}_{v0,h}\big\|_{X^{k-2}}^2 +
 \big\|\hat{\omega}_{b0,h}\big\|_{X^{k-2}}^2 +
 T \|\hat{\omega}_{vh}\|_{L^4([0,T],X^{k-2})}^2
 + T \|\hat{\omega}_{bh}\|_{L^4([0,T],X^{k-2})}^2 \\[9pt]\quad

+ \sqrt{T}\int_0^T\|\hat{v}\|_{X^{k-2,1}}^2+\|\hat{b}\|_{X^{k-2,1}}^2 \,\mathrm{d}t
+ \sqrt{T}\int_0^T|\hat{h}|_{X^{k-2,1}}^2 \,\mathrm{d}t
+ \sqrt{T}|\partial_t^{k-1} \hat{h}|_{L^4([0,T],L^2)}^2 + O(\e).
\end{array}
\end{equation}

When $\ell+|\alpha|\leq k-2$, the estimate $(\ref{5.33})$ is reduced to
\begin{equation}\label{5.53}
\begin{array}{ll}
\|\partial_t^{\ell}\mathcal{Z}^{\alpha}\hat{\omega}_{vh}^{h}\|_{L^4([0,T],L^2(\mathbb{R}^3_{-}))}^2
+\|\partial_t^{\ell}\mathcal{Z}^{\alpha}\hat{\omega}_{bh}^{h}\|_{L^4([0,T],L^2(\mathbb{R}^3_{-}))}^2
 \\[9pt]

\lem \sqrt{\e}\int_0^T\big|\textsf{F}_v^{1,2}[\nabla\varphi^{\e}](\partial_j v^{\e,i},\partial_j b^{\e,i}) -\textsf{F}_v^{1,2}[\nabla\varphi](\partial_j v^i,\partial_j b^i)
\big|_{X^{k-2}(\mathbb{R}^2)}^2
\,\mathrm{d}t
\\[9pt]

+\sqrt{\e}\int_0^T\big|\textsf{F}_b^{1,2}[\nabla\varphi^{\e}](\partial_j b^{\e,i},\partial_j b^{\e,i}) -\textsf{F}_v^{1,2}[\nabla\varphi](\partial_j v^i,\partial_j b^i)
\big|_{X^{k-2}(\mathbb{R}^2)}^2
\,\mathrm{d}t
\\[9pt]

\lem \sqrt{\e}\int_0^T\big|\hat{v}|_{z=0}\big|_{X^{k-1}(\mathbb{R}^2)}^2
+\big|\hat{b}|_{z=0}\big|_{X^{k-1}(\mathbb{R}^2)}^2 \,\mathrm{d}t
+ \sqrt{\e}\int_0^T|\hat{h}|_{X^{k-2,1}(\mathbb{R}^2)}^2 \,\mathrm{d}t
\\[9pt]

\lem \sqrt{\e}\int_0^T\|\hat{v}\|_{X^{k-1,1}(\mathbb{R}^2)}^2
+\|\hat{b}\|_{X^{k-1,1}(\mathbb{R}^2)}^2 \,\mathrm{d}t
+ \sqrt{\e}\sqrt{T}\|\partial_z\hat{v}\|_{L^4([0,T],X^{k-1})}^2
\\[9pt]\quad

+ \sqrt{\e}\sqrt{T}\|\partial_z\hat{b}\|_{L^4([0,T],X^{k-1})}^2
+ \sqrt{\e}\int_0^T|\hat{h}|_{X^{k-2,1}(\mathbb{R}^2)}^2 \,\mathrm{d}t.
\end{array}
\end{equation}

By $(\ref{5.52})$ and $(\ref{5.53})$, we have
\begin{equation}
\begin{array}{ll}
\|\partial_z\hat{v}_{vh}\|_{L^4([0,T],X^{k-2})}^2 +
\|\partial_z\hat{v}_{bh}\|_{L^4([0,T],X^{k-2})}^2 +
 \|\hat{\omega}_{vh}\|_{L^4([0,T],X^{k-2})}^2+
 \|\hat{\omega}_{bh}\|_{L^4([0,T],X^{k-2})}^2
\\[10pt]

\lem \|\hat{\omega}_{vh}^{nh}\|_{L^4([0,T],X^{k-2})}^2
+ \|\hat{\omega}_{bh}^{nh}\|_{L^4([0,T],X^{k-2})}^2
+ \|\partial_t^{\ell}\mathcal{Z}^{\alpha}\hat{\omega}_{vh}^{h}\|_{L^4([0,T],L^2(\mathbb{R}^3_{-}))}^2
+ \|\partial_t^{\ell}\mathcal{Z}^{\alpha}\hat{\omega}_{bh}^{h}\|_{L^4([0,T],L^2(\mathbb{R}^3_{-}))}^2
 \\[9pt]

\lem \big\|\hat{\omega}_{v0}\big\|_{X^{k-2}}^2
+\big\|\hat{\omega}_{b0}\big\|_{X^{k-2}}^2
+ \int_0^T\|\hat{v}\|_{X^{k-1,1}}^2+\|\hat{b}\|_{X^{k-1,1}}^2 \,\mathrm{d}t
+ \int_0^T|\hat{h}|_{X^{k-2,1}}^2 \,\mathrm{d}t
\\[9pt]\quad

+ \|\partial_t^{k-1}\hat{h}\|_{L^4([0,T],L^2)}^2
+ \sqrt{\e}\|\partial_z\hat{v}\|_{L^4([0,T],X^{k-1})}^2
+ \sqrt{\e}\|\partial_z\hat{b}\|_{L^4([0,T],X^{k-1})}^2
+ O(\e).
\end{array}
\end{equation}

Lemma $\ref{Lemma5.7}$ is proved.
\end{proof}

\subsection{Proof of Theorem 1.2}
In this subsection, we prove Theorem \ref{Theorem1.2}.

{\rm \textbf{Proof of Theorem \ref{Theorem1.2}:}}
If $\Pi\mathcal{S}^{\varphi} v\nn|_{z=0}\neq 0$
and $\Pi\mathcal{S}^{\varphi} b\nn|_{z=0}\neq 0$,
we have the following convergence estimates by Lemmas $\ref{Lemma5.3}, \ref{Lemma5.4},
\ref{Lemma5.5}$
\begin{equation}\label{5.55}
\begin{array}{ll}
\|\hat{v}\|_{X^{k-1,1}}^2 +\|\hat{b}\|_{X^{k-1,1}}^2
+|\hat{h}|_{X^{k-1,1}}^2

\lem \|\hat{v}_0\|_{X^{k-1,1}}^2 +\|\hat{b}_0\|_{X^{k-1,1}}^2 +
 |\hat{h}_0|_{X^{k-1,1}}^2 + |\hat{\omega}_{v0}|_{X^{k-1}}^2
 \\[10pt]\quad
+ |\hat{\omega}_{b0}|_{X^{k-1}}^2+ \int_0^t\|\hat{v}\|_{X^{k-1,1}}^2
+ \int_0^t\|\hat{b}\|_{X^{k-1,1}}^2
+ \int_0^t\|\hat{h}\|_{X^{k-1,1}}^2 \,\mathrm{d}t + O(\sqrt{\e}).
\end{array}
\end{equation}

Applying the Gronwall's inequality to $(\ref{5.55})$, we get
\begin{equation}
\begin{array}{ll}
\|\hat{v}\|_{X^{k-1,1}}^2 +\|\hat{b}\|_{X^{k-1,1}}^2 + |\hat{h}|_{X^{k-1,1}}^2
\\[9pt]
\lem \|\hat{v}_0\|_{X^{k-1,1}}^2 +\|\hat{b}_0\|_{X^{k-1,1}}^2 +
 |\hat{h}_0|_{X^{k-1,1}}^2
+ \big\|\hat{\omega}_{v0}\big\|_{X^{k-1}}^2
+ \big\|\hat{\omega}_{b0}\big\|_{X^{k-1}}^2 + O(\sqrt{\e})
\\[7pt]

\lem O(\e^{\min\{\frac{1}{2}, 2\lambda^v, 2\lambda^b, 2\lambda^h, 2\lambda^{\omega_v}_1, 2\lambda^{\omega_b}_1\}}).
\end{array}
\end{equation}

By Lemma $\ref{Lemma5.5}$, we have
\begin{equation}
\begin{array}{ll}
\|\partial_z\hat{v}_h\|_{L^4([0,T],X^{k-1})}^2 +
\|\partial_z\hat{b}_h\|_{L^4([0,T],X^{k-1})}^2 +
 \|\hat{\omega}_{vh}\|_{L^4([0,T],X^{k-1})}^2+
 \|\hat{\omega}_{bh}\|_{L^4([0,T],X^{k-1})}^2
 \\[9pt]

\lem O(\e^{\min\{\frac{1}{2}, 2\lambda^v, 2\lambda^b, 2\lambda^h, 2\lambda^{\omega_v}_1, 2\lambda^{\omega_b}_1\}}).
\end{array}
\end{equation}

By Lemmas $\ref{Lemma5.4}, \ref{Lemma5.5}, \ref{Lemma5.6}$, we have
\begin{equation}\label{5.58}
\begin{array}{ll}
\|\hat{\omega}_v\|_{X^{k-2}}^2 +\|\hat{\omega}_b\|_{X^{k-2}}^2 +
 \|\partial_z\hat{v}\|_{X^{k-2}}^2+ \|\partial_z\hat{b}\|_{X^{k-2}}^2
\\[6pt]

\lem \|\hat{\omega}_{v0}\|_{X^{k-2}}^2+\|\hat{\omega}_{b0}\|_{X^{k-2}}^2
+ \int_0^t\|\hat{v}\|_{X^{k-2}} +\|\hat{b}\|_{X^{k-2}} +
 \|\partial_z \hat{v}\|_{X^{k-2}}

\\[6pt]\quad
+\|\partial_z \hat{b}\|_{X^{k-2}}
+ \|\nabla\hat{q} \|_{X^{k-2}} + \|\hat{h}\|_{X^{k-1}}\,\mathrm{d}t + O(\e)
\\[6pt]

\lem O(\e^{\min\{\frac{1}{4}, \lambda^v,  \lambda^b, \lambda^h, \lambda^{\omega_v}_1, \lambda^{\omega_b}_1\}}).
\end{array}
\end{equation}

If $\Pi\mathcal{S}^{\varphi} v\nn|_{z=0}= 0$ and $\Pi\mathcal{S}^{\varphi} b\nn|_{z=0}= 0$,
we have the convergence
estimates by Lemmas $\ref{Lemma5.7}$
\begin{equation}\label{5.59}
\begin{array}{ll}
\|\partial_z\hat{v}_h\|_{L^4([0,T],X^{k-2})}^2
+\|\partial_z\hat{b}_h\|_{L^4([0,T],X^{k-2})}^2
+ \|\hat{\omega}_{vh}\|_{L^4([0,T],X^{k-2})}^2
+ \|\hat{\omega}_{bh}\|_{L^4([0,T],X^{k-2})}^2
\\[9pt]
\lem \big\|\hat{\omega}_{v0}\big\|_{X^{k-2}}^2
+\big\|\hat{\omega}_{b0}\big\|_{X^{k-2}}^2
+ \int_0^T\|\hat{v}\|_{X^{k-1,1}}^2+\|\hat{b}\|_{X^{k-1,1}}^2 \,\mathrm{d}t
+ \int_0^T|\hat{h}|_{X^{k-2,1}}^2 \,\mathrm{d}t
\\[9pt]\quad
+ \sqrt{\e}O(\e^{\min\{\frac{1}{2}, 2\lambda^v, 2\lambda^b, 2\lambda^h, 2\lambda^{\omega_v}_1, 2\lambda^{\omega_b}_1\}}).
\end{array}
\end{equation}

On the other hand, one has the tangential estimates
\begin{equation}\label{5.60}
\begin{array}{ll}
\|\hat{v}\|_{X^{k-2,1}}^2 +\|\hat{b}\|_{X^{k-2,1}}^2
+ |\hat{h}|_{X^{k-2,1}}^2
\\[8pt]

\lem \|\hat{v}_0\|_{X^{k-2,1}}^2 +\|\hat{b}_0\|_{X^{k-2,1}}^2 + |\hat{h}_0|_{X^{k-2,1}}^2
+\|\partial_z \hat{v}\|_{L^4([0,T],X^{k-2})}^2
+\|\partial_z \hat{b}\|_{L^4([0,T],X^{k-2})}^2
\\[8pt]\quad

+ \int_0^t\|\hat{v}\|_{X^{k-2,1}}^2+ \int\limits_0^t\|\hat{b}\|_{X^{k-2,1}}^2
+ \int_0^t\|\hat{h}\|_{X^{k-2,1}}^2 \,\mathrm{d}t + O(\e).
\end{array}
\end{equation}

Coupling \eqref{5.59} with \eqref{5.60} and
applying the Gronwall's inequality, we get
\begin{equation}
\begin{array}{ll}
\|\hat{v}\|_{X^{k-2,1}}^2 +\|\hat{b}\|_{X^{k-2,1}}^2 + |\hat{h}|_{X^{k-2,1}}^2 \\[6pt]
\lem \big\|\hat{\omega}_{v0}\big\|_{X^{k-2}}^2+\big\|\hat{\omega}_{b0}\big\|_{X^{k-2}}^2
+ \|\hat{v}_0\|_{X^{k-2,1}}^2+ \|\hat{b}_0\|_{X^{k-2,1}}^2
 + |\hat{h}_0|_{X^{k-2,1}}^2
\\[8pt]\quad

 + \sqrt{\e}O(\e^{\min\{\frac{1}{2}, 2\lambda^v, 2\lambda^b, 2\lambda^h, 2\lambda^{\omega_v}_1,
2\lambda^{\omega_b}_1\}})
\\[8pt]

\lem O(\e^{\min\{1, 2\lambda^v,2\lambda^b, 2\lambda^h, 2\lambda^{\omega_v}_2, 2\lambda^{\omega_b}_2, 2\lambda^{\omega_v}_1+1,2\lambda^{\omega_b}_1+1\}})
\\[8pt]

= O(\e^{\min\{1, 2\lambda^v, 2\lambda^b, 2\lambda^h, 2\lambda^{\omega_v}_2, 2\lambda^{\omega_b}_2\}}).
\end{array}
\end{equation}

Similarly, we have
\begin{equation}
\begin{array}{ll}
\|\hat{\omega}_v\|_{X^{k-3}}^2 +\|\hat{\omega}_b\|_{X^{k-3}}^2 +
\|\partial_z\hat{v}\|_{X^{k-3}}^2+\|\partial_z\hat{b}\|_{X^{k-3}}^2
\\[8pt]

\lem O(\e^{\min\{\frac{1}{2}, \lambda^v,\lambda^b, \lambda^h, \lambda^{\omega_v}_2, \lambda^{\omega_b}_2\}}).
\end{array}
\end{equation}

Now, the estimates for $\NN^{\e}\cdot \partial_z^{\varphi^{\e}} v^{\e} - \NN \cdot \partial_z^{\varphi} v$ and
 $\NN^{\e}\cdot \partial_z^{\varphi^{\e}} b^{\e} - \NN \cdot \partial_z^{\varphi} b$ follow from the above estimates for tangential derivatives since
$\NN\cdot\partial_z^{\varphi} v = - (\partial_1 v^1 + \partial_2 v^2)$
and $\NN\cdot\partial_z^{\varphi} b = - (\partial_1 b^1 + \partial_2 b^2)$.

Similarly, the estimate for $\NN^{\e}\cdot {\omega}_v^{\e} -\NN\cdot \omega_v$ follows from the following equality
\begin{equation}
\begin{array}{ll}
\NN\cdot\omega_v
= -\partial_1\varphi(\partial_2 v^3 - \frac{\partial_2\varphi}{\partial_z\varphi}\partial_z v^3
- \frac{1}{\partial_z\varphi}\partial_z v^2)
-\partial_2\varphi(- \partial_1 v^3 + \frac{\partial_1\varphi}{\partial_z\varphi}\partial_z v^3
+ \frac{1}{\partial_z\varphi}\partial_z v^1)
\\[8pt]\hspace{1.3cm}
+ \partial_1 v^2 - \frac{\partial_1\varphi}{\partial_z\varphi}\partial_z v^2 - \partial_2 v^1 + \frac{\partial_2\varphi}{\partial_z\varphi}\partial_z v^1
\\[8pt]\hspace{0.9cm}

= -\partial_1\varphi\partial_2 v^3 +\partial_2\varphi\partial_1 v^3 + \partial_1 v^2 - \partial_2 v^1,
\end{array}
\end{equation}
and the estimate for $\NN^{\e}\cdot {\omega}_b^{\e} -\NN\cdot \omega_b$ follows from
$\NN\cdot\omega_b
= -\partial_1\varphi\partial_2 b^3 +\partial_2\varphi\partial_1 b^3 + \partial_1 b^2 - \partial_2 b^1$.

Theorem $\ref{Theorem1.2}$ is proved.

\section{Regularity Structure of MHD Solutions for Fixed $\sigma>0$}
In this section,
we establish the regularity with $\sigma>0$.
It should be pointed out that the estimates of normal derivatives are similar to $\sigma=0$ case.
Hence, we only focus on the estimates of the pressure and
the tangential derivatives.
For simplicity, we omit the superscript ${}^{\e}$ in this section.

\subsection{Estimates for the Pressure}
 We have the elliptic equation for the pressure with the Neumann boundary condition as follows
\begin{equation*}
\left\{
\begin{array}{ll}
\Delta^{\varphi} q = -\partial_j^{\varphi} v^i \partial_i^{\varphi} v^j
+\partial_j^{\varphi} b^i \partial_i^{\varphi} b^j,
\\[5pt]

\nabla^{\varphi} q\cdot\NN|_{z=0} = - \partial_t^{\varphi} v\cdot\NN
- v\cdot\nabla^{\varphi} v\cdot\NN
+b\cdot\nabla^{\varphi} b\cdot\NN
+ \e\Delta^{\varphi} v\cdot\NN.
\end{array}\right.
\end{equation*}

The following lemma concerns the estimates for the gradient of pressure.
\begin{lemma}\label{Lemma6.1}
Assume the pressure $q$ satisfies the elliptic equation with Neumann boundary condition, we have
\begin{equation}\label{6.1}
\begin{array}{ll}
\|\nabla q\|_{X^{m-1}}
\lem \|\partial_t^m v\|_{X^m} + \|v\|_{X^{m-1,1}} + \|\partial_z v\|_{X^{m-1}} +  |h|_{X^{m,1}} \\[6pt]\quad
+ \e \|\nabla_y v\|_{X^{m-1,1}} + \e \|\partial_z v\|_{X^{m-1}} + \e |h|_{X^{m-1,\frac{3}{2}}}.
\end{array}
\end{equation}
\end{lemma}

\begin{proof}
For the $L^2$ estimate of the pressure, one has
\begin{equation}\label{6.2}
\begin{array}{ll}
\|\nabla q\|_{L^2} \lem \|v\cdot \nabla^{\varphi} v\|_{L^2}
+\|b\cdot \nabla^{\varphi} b\|_{L^2}
+ \big|\nabla^{\varphi} q\cdot\NN|_{z=0} \big|_{-\frac{1}{2}} \\[8pt]

\lem \|v\|_{X^{0,1}}+\|b\|_{X^{0,1}}+ \|\partial_z v\|_{L^2} + \|\partial_z b\|_{L^2}
+ |h|_{X^{0,1}}
+ \big|\partial_t^{\varphi} v\cdot\NN|_{z=0}\big|_{-\frac{1}{2}} \\[6pt]\quad
+ \big|v\cdot\nabla^{\varphi} v\cdot\NN|_{z=0} \big|_{-\frac{1}{2}}
+ \big|b\cdot\nabla^{\varphi} b\cdot\NN|_{z=0} \big|_{-\frac{1}{2}}
+ \e\big|\Delta^{\varphi} v\cdot\NN|_{z=0}\big|_{-\frac{1}{2}} \\[8pt]

\lem \|v\|_{X^{0,1}}+\|b\|_{X^{0,1}}
+ \|\partial_z v\|_{L^2} + |h|_{X^{0,1}}+ \|\partial_z b\|_{L^2} + |h|_{X^{0,1}}
+ \|\partial_t^{\varphi} v\|_{L^2}
\\[6pt]\quad

+ \|\nabla^{\varphi}\cdot\partial_t^{\varphi} v\|_{L^2}
+ \|v\cdot\nabla^{\varphi} v\|_{L^2} + \|\nabla\cdot(v\cdot\nabla^{\varphi} v)\|_{L^2}
+ \|b\cdot\nabla^{\varphi} b\|_{L^2}
\\[6pt]\quad

+ \|\nabla\cdot(b\cdot\nabla^{\varphi} b)\|_{L^2}
+ \e\big|v|_{z=0}\big|_{\frac{3}{2}} + \e|h|_{\frac{3}{2}},
\end{array}
\end{equation}
where $|v\cdot\NN|_{-\frac{1}{2}} \lem \|v\| + \|\nabla^{\varphi}\cdot v\|$.

Similarly, we have
\begin{equation}\label{6.3}
\begin{array}{ll}
\|\nabla q\|_{X^{m-1}} &\lem \|v\cdot \nabla^{\varphi} v\|_{X^{m-1}}
+\|b\cdot \nabla^{\varphi} b\|_{X^{m-1}}
+ \big|\nabla^{\varphi} q\cdot\NN|_{z=0} \big|_{X^{m-1, -\frac{1}{2}}} \\[6pt]

&\lem \|v\|_{X^{m-1,1}}+\|b\|_{X^{m-1,1}}
+ \|\partial_z v\|_{X^{m-1}}+ \|\partial_z b\|_{X^{m-1}}
\\[6pt]
&\quad+ |h|_{X^{m-1,1}}
+ \big|\nabla^{\varphi} q\cdot\NN|_{z=0} \big|_{X^{m-1, -\frac{1}{2}}}.
\end{array}
\end{equation}

For $\big|\nabla^{\varphi} q\cdot\NN|_{z=0} \big|_{X^{m-1, -\frac{1}{2}}}$,
we have the following estimates as in \cite{Wu16},
\begin{equation}\label{6.4}
\begin{array}{ll}
\e \big|\Delta^{\varphi} v\cdot\NN|_{z=0} \big|_{X^{m-1,-\frac{1}{2}}}
\lem \e \|\nabla_y v\|_{X^{m-1,1}}
 + \e \|\partial_z v\|_{X^{m-1}} + \e |h|_{X^{m-1,\frac{3}{2}}}.
\end{array}
\end{equation}
\begin{equation}\label{6.5}
\begin{array}{ll}
\big|(\partial_t^{\varphi} v + v\cdot\nabla^{\varphi} v)\cdot\NN|_{z=0} \big|_{X^{m-1, -\frac{1}{2}}}
= \big|(\partial_t v + v_y\cdot\nabla_y v)\cdot\NN|_{z=0} \big|_{X^{m-1, -\frac{1}{2}}} \\[8pt]
\lem \|v\|_{X^m} + \|\partial_z v\|_{X^{m-1}} +  |h|_{X^{m,1}}.
\end{array}
\end{equation}

Lemma $\ref{Lemma6.1}$ is proved.
\end{proof}

\subsection{Proof of Proposition 1.2}
First, we give the following lemma to estimate $\partial_t^{\ell} h$
by the kinetical boundary condition.
\begin{lemma}[\cite{Wu16}]\label{Lemma6.2}
Assume $0\leq\ell\leq m-1$. We have
\begin{equation}\label{6.6}
\begin{array}{ll}
\int\limits_{\mathbb{R}^2} |\partial_t^{\ell}h|^2 \,\mathrm{d}y

\lem |h_0|_{X^{m-1}}^2
+ \int_0^t |h|_{X^{m-1,1}}^2 + \|v\|_{X^{m-1,1}}^2 \,\mathrm{d}t
+ \|\partial_z v\|_{L^4([0,T],X^{m-1})}^2.
\end{array}
\end{equation}
\end{lemma}

Next,
we establish the a priori estimates for the tangential derivatives as well as the time derivatives.
We note that the proof does not require the Alinhac'c good unknown.

\begin{lemma}\label{Lemma6.3}
Under the assumptions of Proposition $\ref{Proposition1.2}$,
we have
\begin{equation}\label{6.7}
\begin{array}{ll}
\|v\|_{X^{m-1,1}}^2 +\|b\|_{X^{m-1,1}}^2 +
 |h|_{X^{m-1,1}}^2 + \e |h|_{X^{m-1,\frac{3}{2}}}^2 + \sigma |h|_{X^{m-1,2}}^2
\\[9pt]\quad
+ \e\int_0^t \|\nabla v\|_{X^{m-1,1}}^2\,\mathrm{d}t
+ \e\int_0^t \|\nabla b\|_{X^{m-1,1}}^2 \,\mathrm{d}t \\[9pt]

\lem \|v_0\|_{X^{m-1,1}}^2 +\|b_0\|_{X^{m-1,1}}^2 + |h_0|_{X^{m-1,1}}^2 + \e |h_0|_{X^{m-1,\frac{3}{2}}}^2 + \sigma |h_0|_{X^{m-1,2}}^2
\\[9pt]\quad

+ \|\partial_z v\|_{L^4([0,T],X^{m-1})}^2
+ \|\partial_z b\|_{L^4([0,T],X^{m-1})}^2 + \|\partial_t^m v\|_{L^4([0,T],L^2)}^2
\\[9pt]\quad
 + \|\partial_t^m b\|_{L^4([0,T],L^2)}^2+ \|\partial_t^m h\|_{L^4([0,T],X^{0,1})}^2.
\end{array}
\end{equation}
\end{lemma}

\begin{proof}
Applying $\partial_t^{\ell}\mathcal{Z}^{\alpha}$ to $(\ref{MHDFS})$,
we deduce that $\partial_t^{\ell}\mathcal{Z}^{\alpha}v$,
 $\partial_t^{\ell}\mathcal{Z}^{\alpha}b$
 and $\partial_t^{\ell}\mathcal{Z}^{\alpha}q$ satisfy the following equations for fixed $\sigma>0$
\begin{equation}\label{6.8}
\left\{\begin{array}{ll}
\partial_t^{\varphi} \partial_t^{\ell}\mathcal{Z}^{\alpha}v
+ v\cdot\nabla^{\varphi} \partial_t^{\ell}\mathcal{Z}^{\alpha} v
- b\cdot\nabla^{\varphi} \partial_t^{\ell}\mathcal{Z}^{\alpha} b
+ \nabla^{\varphi} \partial_t^{\ell}\mathcal{Z}^{\alpha} q
- 2\e\nabla^{\varphi}\cdot \mathcal{S}^{\varphi} \partial_t^{\ell}\mathcal{Z}^{\alpha}v
\\[8pt]\quad
= \partial_t^{\ell+1}\mathcal{Z}^{\alpha}\eta \partial_z^{\varphi}v
+ \partial_t^{\ell}\mathcal{Z}^{\alpha}\nabla \eta \cdot v \partial_z^{\varphi} v
- \partial_t^{\ell}\mathcal{Z}^{\alpha}\nabla \eta \cdot b \partial_z^{\varphi} b
+ \partial_t^{\ell}\mathcal{Z}^{\alpha}\nabla \eta \cdot \partial_z^{\varphi} q
\\[8pt]\quad
+ 2\e\nabla^{\varphi}\cdot [\partial_t^{\ell}\mathcal{Z}^{\alpha}, \mathcal{S}^{\varphi}]v
+ 2\e[\partial_t^{\ell}\mathcal{Z}^{\alpha}, \nabla^{\varphi}\cdot] \mathcal{S}^{\varphi}v + b.t.,
\\[10pt]

\partial_t^{\varphi} \partial_t^{\ell}\mathcal{Z}^{\alpha}b
+ v\cdot\nabla^{\varphi} \partial_t^{\ell}\mathcal{Z}^{\alpha} b
- b\cdot\nabla^{\varphi} \partial_t^{\ell}\mathcal{Z}^{\alpha} v
- 2\e\nabla^{\varphi}\cdot \mathcal{S}^{\varphi} \partial_t^{\ell}\mathcal{Z}^{\alpha}b
\\[8pt]\quad
= \partial_t^{\ell+1}\mathcal{Z}^{\alpha}\eta \partial_z^{\varphi}b
+ \partial_t^{\ell}\mathcal{Z}^{\alpha}\nabla \eta \cdot v \partial_z^{\varphi} b
- \partial_t^{\ell}\mathcal{Z}^{\alpha}\nabla \eta \cdot b \partial_z^{\varphi} v
\\[8pt]\quad
+ 2\e\nabla^{\varphi}\cdot [\partial_t^{\ell}\mathcal{Z}^{\alpha}, \mathcal{S}^{\varphi}]b
+ 2\e[\partial_t^{\ell}\mathcal{Z}^{\alpha}, \nabla^{\varphi}\cdot] \mathcal{S}^{\varphi}b +{\text b.t.},
\\[10pt]

\nabla^{\varphi}\cdot \partial_t^{\ell}\mathcal{Z}^{\alpha}v = \partial_t^{\ell}\mathcal{Z}^{\alpha}\nabla\eta \cdot \partial_z^{\varphi}v + {\text b.t.},
\\[10pt]
\nabla^{\varphi}\cdot \partial_t^{\ell}\mathcal{Z}^{\alpha}b = \partial_t^{\ell}\mathcal{Z}^{\alpha}\nabla\eta \cdot \partial_z^{\varphi}b + {\text b.t.},
\end{array}\right.
\end{equation}
with the following initial boundary conditions
\begin{equation}
\left\{\begin{array}{ll}
\partial_t \partial_t^{\ell}\mathcal{Z}^{\alpha} h
+ v_y \cdot\nabla_y \partial_t^{\ell}\mathcal{Z}^{\alpha} h
= \partial_t^{\ell}\mathcal{Z}^{\alpha}v \cdot \NN
+ [\partial_t^{\ell}\mathcal{Z}^{\alpha}, v,\NN],
\\[10pt]

\partial_t^{\ell}\mathcal{Z}^{\alpha}q \NN
-2\e \mathcal{S}^{\varphi} \partial_t^{\ell}\mathcal{Z}^{\alpha}v\,\NN
= g \partial_t^{\ell}\mathcal{Z}^{\alpha}h\NN
+ 2\e [\partial_t^{\ell}\mathcal{Z}^{\alpha},\mathcal{S}^{\varphi}] v\,\NN
\\[8pt]\quad

- \sigma \nabla_y\cdot\frac{1}{\sqrt{1+|\nabla_y h|^2}}
\big(\nabla_y\partial_t^{\ell}\mathcal{Z}^{\alpha} h
- \frac{\nabla_y h(\nabla_y h\cdot \nabla_y\partial_t^{\ell}\mathcal{Z}^{\alpha}h)}{1+|\nabla_y h|^2}\big) \NN
\\[10pt]\quad

+ (2\e \mathcal{S}^{\varphi}v - (q-g h))\,\partial_t^{\ell}\mathcal{Z}^{\alpha}\NN
- [\partial_t^{\ell}\mathcal{Z}^{\alpha},q-g h,\NN]
+2\e [\partial_t^{\ell}\mathcal{Z}^{\alpha},\mathcal{S}^{\varphi}v, \NN]
\\[8pt]\quad

-\sigma [\partial_t^{\ell}\mathcal{Z}^{\alpha}, \NN]\mathcal{M}
- \sigma \nabla_y\cdot [\partial_t^{\ell}\mathcal{Z}^{\alpha}, \nabla_y h, \frac{1}{\sqrt{1+|\nabla_y h|^2}}] \NN, \\[12pt]

\partial_t^{\ell}\mathcal{Z}^{\alpha}b|_{z=0}=0, \\[12pt]

(\partial_t^{\ell}\mathcal{Z}^{\alpha}v,\partial_t^{\ell}\mathcal{Z}^{\alpha}b, \partial_t^{\ell}\mathcal{Z}^{\alpha}h)|_{t=0}
= (\partial_t^{\ell}\mathcal{Z}^{\alpha}v_0, \partial_t^{\ell}\mathcal{Z}^{\alpha}b_0, \partial_t^{\ell}\mathcal{Z}^{\alpha}h_0).
\end{array}\right.
\end{equation}

Step one: When $|\alpha|\geq 1, \, 1\leq \ell+|\alpha|\leq m, 0\leq\ell\leq m-1$,
we have the $L^2$ estimate of
$\partial_t^{\ell}\mathcal{Z}^{\alpha} v$,
$\partial_t^{\ell}\mathcal{Z}^{\alpha} b$
 and $\partial_t^{\ell}\mathcal{Z}^{\alpha} h$
\begin{equation}\label{6.10}
\begin{array}{ll}
\frac{1}{2}\frac{\mathrm{d}}{\mathrm{d}t}\int_{\mathbb{R}^3_{-}} |\partial_t^{\ell}\mathcal{Z}^{\alpha}v|^2
+|\partial_t^{\ell}\mathcal{Z}^{\alpha}b|^2 \,\mathrm{d}\mathcal{V}_t
- \int_{\mathbb{R}^3_{-}} \partial_t^{\ell}\mathcal{Z}^{\alpha}q \, \nabla^{\varphi}\cdot \partial_t^{\ell}\mathcal{Z}^{\alpha}v \,\mathrm{d}\mathcal{V}_t
\\[14pt]

+ 2\e \int_{\mathbb{R}^3_{-}} |\mathcal{S}^{\varphi}\partial_t^{\ell}\mathcal{Z}^{\alpha}v|^2 \,\mathrm{d}\mathcal{V}_t
+ 2\e \int_{\mathbb{R}^3_{-}} |\mathcal{S}^{\varphi}\partial_t^{\ell}\mathcal{Z}^{\alpha}b|^2 \,\mathrm{d}\mathcal{V}_t \\[14pt]

\leq \int_{z=0} (2\e \mathcal{S}^{\varphi}\partial_t^{\ell}\mathcal{Z}^{\alpha}v\NN
- \partial_t^{\ell}\mathcal{Z}^{\alpha}q\NN)\cdot \partial_t^{\ell}\mathcal{Z}^{\alpha}v \mathrm{d}y
+ \|\partial_z v\|_{X^{m-1}}^2+ \|\partial_z b\|_{X^{m-1}}^2
 \\[11pt]\quad

+ \|\nabla q\|_{X^{m-1}}^2
+ |h|_{X^{m-1,2}}^2 + |\partial_t^m h|_{L^2}^2
+ \text{b.t.}
\\[7pt]

 \leq -\int_{z=0} \big[
g \partial_t^{\ell}\mathcal{Z}^{\alpha}h
- \sigma \nabla_y\cdot\frac{1}{\sqrt{1+|\nabla_y h|^2}}
\big(\nabla_y\partial_t^{\ell}\mathcal{Z}^{\alpha} h
- \frac{\nabla_y h(\nabla_y h\cdot \nabla_y\partial_t^{\ell}\mathcal{Z}^{\alpha}h)}{1+|\nabla_y h|^2}\big)
\big]\NN \\[13pt]\quad

\cdot \partial_t^{\ell}\mathcal{Z}^{\alpha}v \mathrm{d}y
+ \|\partial_z v\|_{X^{m-1}}^2+ \|\partial_z b\|_{X^{m-1}}^2 + \|\nabla q\|_{X^{m-1}}^2
+ |h|_{X^{m-1,2}}^2 + |\partial_t^m h|_{L^2}^2 + \text{b.t.} \\[8pt]

\leq \sigma\int_{z=0} \nabla_y\cdot\frac{1}{\sqrt{1+|\nabla_y h|^2}}
\big(\nabla_y\partial_t^{\ell}\mathcal{Z}^{\alpha} h
- \frac{\nabla_y h(\nabla_y h\cdot \nabla_y\partial_t^{\ell}\mathcal{Z}^{\alpha}h)}{1+|\nabla_y h|^2}\big)

\cdot(\partial_t \partial_t^{\ell}\mathcal{Z}^{\alpha} h \\
\end{array}
\end{equation}
which implies that
\begin{equation}\label{6.11}
\begin{array}{ll}
\frac{1}{2}\frac{\mathrm{d}}{\mathrm{d}t}\int_{\mathbb{R}^3_{-}} |\partial_t^{\ell}\mathcal{Z}^{\alpha}v|^2
+|\partial_t^{\ell}\mathcal{Z}^{\alpha}b|^2 \,\mathrm{d}\mathcal{V}_t
+ \frac{g}{2}\frac{\mathrm{d}}{\mathrm{d}t}
\int_{z=0} |\partial_t^{\ell}\mathcal{Z}^{\alpha}h|^2 \mathrm{d}y
+ 2\e \int_{\mathbb{R}^3_{-}} |\mathcal{S}^{\varphi}\partial_t^{\ell}\mathcal{Z}^{\alpha}v|^2 \,\mathrm{d}\mathcal{V}_t \\[10pt]\quad

+ 2\e \int_{\mathbb{R}^3_{-}} |\mathcal{S}^{\varphi}\partial_t^{\ell}\mathcal{Z}^{\alpha}b|^2 \,\mathrm{d}\mathcal{V}_t
+ \frac{\sigma}{2}\frac{\mathrm{d}}{\mathrm{d}t} \int_{z=0} \frac{1}{\sqrt{1+|\nabla_y h|^2}}
\big(|\nabla_y\partial_t^{\ell}\mathcal{Z}^{\alpha} h|^2
- \frac{|\nabla_y h\cdot \nabla_y\partial_t^{\ell}\mathcal{Z}^{\alpha}h|^2}{1+|\nabla_y h|^2}\big) \mathrm{d}y \\[12pt]

\leq \|\partial_z v\|_{X^{m-1}}^2 +\|\partial_z b\|_{X^{m-1}}^2 + \|\nabla q\|_{X^{m-1}}^2
+ |h|_{X^{m-1,2}}^2 + |\partial_t^m h|_{L^2}^2 + \text{b.t.}.
\end{array}
\end{equation}

Applying the Gronwall's inequality to $(\ref{6.11})$,
we get
\begin{equation}\label{6.12}
\begin{array}{ll}
\|\partial_t^{\ell}\mathcal{Z}^{\alpha}v\|^2 +\|\partial_t^{\ell}\mathcal{Z}^{\alpha}b\|^2 + |\partial_t^{\ell}\mathcal{Z}^{\alpha} h|^2 + \e |\partial_t^{\ell}\mathcal{Z}^{\alpha} h|_{\frac{1}{2}}^2
+ \frac{\sigma}{4} |\partial_t^{\ell}\mathcal{Z}^{\alpha} h|_{1}^2
\\[9pt]
+ \e\int_0^t \|\nabla \partial_t^{\ell}\mathcal{Z}^{\alpha}v\|^2 \,\mathrm{d}t
+ \e\int_0^t \|\nabla \partial_t^{\ell}\mathcal{Z}^{\alpha}b\|^2 \,\mathrm{d}t
 \\[9pt]

\lem \|v_0\|_{X^{m-1,1}}^2 +\|b_0\|_{X^{m-1,1}}^2
+ |h_0|_{X^{m-1,1}}^2 + \e |h_0|_{X^{m-1,\frac{3}{2}}}^2
+ \sigma |h_0|_{X^{m-1,2}}^2 \\[7pt]\quad
+ \int_0^T \|\partial_z v\|_{X^{m-2,1}}^2
+\|\partial_z b\|_{X^{m-2,1}}^2 + \|\nabla q\|_{X^{m-2,1}}^2
+ |\partial_t^m h|_{L^2}^2\,\mathrm{d}t.
\end{array}
\end{equation}

And, the surface tension term in (\ref{6.11}) is estimated by
\begin{equation}\label{6.13}
\begin{array}{ll}
\frac{\sigma}{2}\int_{z=0} \frac{1}{\sqrt{1+|\nabla_y h|^2}}
\big(|\nabla_y\partial_t^{\ell}\mathcal{Z}^{\alpha} h|^2
- \frac{|\nabla_y h\cdot \nabla_y\partial_t^{\ell}\mathcal{Z}^{\alpha}h|^2}{1+|\nabla_y h|^2}\big) \mathrm{d}y \\[15pt]
\geq \frac{\sigma}{2}\int_{z=0}\frac{1}{2(1+|\nabla_y h|^2)^{\frac{3}{2}}} |\nabla_y\partial_t^{\ell}\mathcal{Z}^{\alpha} h|^2 \,\mathrm{d}y

\geq \frac{\sigma}{4}\int_{z=0}|\nabla_y\partial_t^{\ell}\mathcal{Z}^{\alpha} h|^2 \,\mathrm{d}y.
\end{array}
\end{equation}

Step two:
When $|\alpha|=0$ and $0\leq\ell\leq m-1$,
we have the $L^2$ estimate
\begin{equation}\label{6.14}
\begin{array}{ll}
\frac{1}{2}\frac{\mathrm{d}}{\mathrm{d}t}\int_{\mathbb{R}^3_{-}} |\partial_t^{\ell} v|^2
+ |\partial_t^{\ell} b|^2 \,\mathrm{d}\mathcal{V}_t
+ 2\e \int_{\mathbb{R}^3_{-}} |\mathcal{S}^{\varphi}\partial_t^{\ell} v|^2 \,\mathrm{d}\mathcal{V}_t
+ 2\e \int_{\mathbb{R}^3_{-}} |\mathcal{S}^{\varphi}\partial_t^{\ell} b|^2 \,\mathrm{d}\mathcal{V}_t
\\[12pt]

\leq - \int_{\mathbb{R}^3_{-}} \partial_t^{\ell}\nabla^{\varphi} q \cdot \partial_t^{\ell} v \,\mathrm{d}\mathcal{V}_t
+ \int_{z=0} 2\e \mathcal{S}^{\varphi}\partial_t^{\ell} v\NN \cdot \partial_t^{\ell} v \mathrm{d}y
+ \|\partial_z v\|_{X^{m-2}}^2+ \|\partial_z b\|_{X^{m-2}}^2
\\[12pt]\quad

+ \|\partial_z q\|_{X^{m-2}}^2
+ \sum\limits_{\ell=0}^{m-1}|\partial_t^{\ell+1} h|_{L^2}^2 + \text{b.t.},
\\[12pt]

\lem \|\nabla q\|_{X^{m-1}}^2
+ \|\partial_t^{\ell} v\|_{L^2}^2+ \|\partial_t^{\ell} b\|_{L^2}^2
+ \|\partial_z v\|_{X^{m-2}}^2+ \|\partial_z b\|_{X^{m-2}}^2
\\[12pt]\quad

+ \e\|\nabla_y\partial_t^{\ell} v\|_{X^{0,1}}^2
+ \e\|\partial_z\nabla_y\partial_t^{\ell} v\|_{L^2}^2
+ \sum\limits_{\ell=0}^m|\partial_t^{\ell} h|_{L^2}^2 + \text{b.t.},
\end{array}
\end{equation}
where $\partial_t^{\ell}b|_{z=0}=0$.
Combining $(\ref{6.14})$ and $(\ref{6.6})$, we have
\begin{equation}\label{6.15}
\begin{array}{ll}
\|\partial_t^{\ell} v\|^2 +\|\partial_t^{\ell} b\|^2 +
 \|\partial_t^{\ell} h\|^2 + \e\int_0^t \|\nabla \partial_t^{\ell} v\|^2 \,\mathrm{d}t
\\[12pt]

\lem \|v_0\|_{X^{m-1}}^2 +\|b_0\|_{X^{m-1}}^2 +
 \|\nabla q\|_{L^4([0,T],X^{m-1})}^2
 + \|\partial_z v\|_{L^4([0,T],X^{m-1})}^2
 \\[12pt]\quad

+ \|\partial_z b\|_{L^4([0,T],X^{m-1})}^2
+ |\partial_t^m h|_{L^4([0,T],L^2)}^2 + \text{b.t.}.
\end{array}
\end{equation}

Summing $\ell$ and $\alpha$, using $(\ref{6.12}),(\ref{6.15})$ and Lemma $\ref{Lemma6.2}$,
we get $(\ref{6.7})$.

Lemma $\ref{Lemma6.3}$ is proved.
\end{proof}

In order to close our estimates,
we also need to bound $\|\partial_t^m v\|_{L^4([0,T],L^2)}^2$,
$\|\partial_t^m b\|_{L^4([0,T],L^2)}^2$ and
$\|\partial_t^m h\|_{L^4([0,T],X^{0,1})}^2$ on the right hand of \eqref{6.7}.
\begin{lemma}\label{Lemma6.4}
We have the following estimate for $\partial_t^m v, \partial_t^m b, \partial_t^m h, \partial_t^{m+1}h$
\begin{equation}\label{6.16}
\begin{array}{ll}
\|\partial_t^m v\|_{L^4([0,T],L^2)}^2 +\|\partial_t^m b\|_{L^4([0,T],L^2)}^2 +
 |\partial_t^m h|_{L^4([0,T],X^{0,1})}^2 + |\partial_t^{m+1}\nabla h|_{L^4([0,T],L^2)}^2
\\[12pt]

\lem \|\partial_t^m v_0\|_{L^2}^2 +\|\partial_t^m b_0\|_{L^2}^2 +
 g|\partial_t^m h_0|_{L^2}^2 + \sigma|\partial_t^m\nabla h_0|_{L^2}^2
+ \|\partial_z\partial_t^{m-1} v_0\|_{L^2}^2
\\[12pt]\quad

+ \|\partial_z\partial_t^{m-1} b_0\|_{L^2}^2
+ \|\partial_z v\|_{L^4([0,T],X^{m-1})}^2
+ \|\partial_z b\|_{L^4([0,T],X^{m-1})}^2
+ \text {b.t.}.
\end{array}
\end{equation}
\end{lemma}

\begin{proof}
Assuming $\alpha=0$ and $\ell=m$,
multiplying $(\ref{6.8})$ with $\partial_t^m v$, $\partial_t^m b$,
integrating in $\mathbb{R}^3_{-}$ and
using $\int_{\mathbb{R}^3_{-}}b\cdot\nabla^{\varphi} \partial_t^{\ell}\mathcal{Z}^{\alpha} b\cdot\partial_t^{\ell}\mathcal{Z}^{\alpha} v
+b\cdot\nabla^{\varphi} \partial_t^{\ell}\mathcal{Z}^{\alpha} v\cdot\partial_t^{\ell}\mathcal{Z}^{\alpha} b=0$,
we get
\begin{equation}\label{6.17}
\begin{array}{ll}
\frac{1}{2}\frac{\mathrm{d}}{\mathrm{d}t}\int_{\mathbb{R}^3_{-}} |\partial_t^m v|^2
+|\partial_t^m b|^2\,\mathrm{d}\mathcal{V}_t
-\int_{\mathbb{R}^3_{-}} \partial_t^m q \, \nabla^{\varphi}\cdot \partial_t^m v \,\mathrm{d}\mathcal{V}_t
+ 2\e \int_{\mathbb{R}^3_{-}} |\mathcal{S}^{\varphi}\partial_t^m v|^2
+ |\mathcal{S}^{\varphi}\partial_t^m b|^2 \,\mathrm{d}\mathcal{V}_t
\\[10pt]

\leq \int_{z=0} (2\e \mathcal{S}^{\varphi}\partial_t^m v\NN
- \partial_t^m q\NN)\cdot \partial_t^m v \mathrm{d}y
+ \|\partial_z v\|_{X^{m-1}}^2 + \|\partial_z b\|_{X^{m-1}}^2
\\[10pt]\quad

+ \|\nabla q\|_{X^{m-1}}^2
+ |h|_{X^{m-1,2}}^2 + |\partial_t^m h|_{X^{0,1}}^2  + |\partial_t^{m+1} h|_{L^2}^2
+ \text{b.t.}
\\[10pt]

\leq -\int_{z=0} \big[
g \partial_t^m h
- \sigma \nabla_y\cdot\frac{1}{\sqrt{1+|\nabla_y h|^2}}
\big(\nabla_y\partial_t^m  h
- \frac{\nabla_y h(\nabla_y h\cdot \nabla_y\partial_t^m h)}{1+|\nabla_y h|^2}\big)
\big]\NN \cdot \partial_t^m v \mathrm{d}y
\\[10pt]\quad

+ \|\partial_z v\|_{X^{m-1}}^2 + \|\partial_z b\|_{X^{m-1}}^2 + \|\nabla q\|_{X^{m-1}}^2
+ |h|_{X^{m-1,2}}^2
 + |\partial_t^m h|_{X^{0,1}}^2
+ |\partial_t^{m+1} h|_{L^2}^2 + \text{b.t.}
\\[10pt]

\leq \sigma\int_{z=0} \nabla_y\cdot\frac{1}{\sqrt{1+|\nabla_y h|^2}}
\big(\nabla_y\partial_t^m  h
- \frac{\nabla_y h(\nabla_y h\cdot \nabla_y\partial_t^m h)}{1+|\nabla_y h|^2}\big)
\cdot(\partial_t \partial_t^m  h
\\[10pt]\quad

+ v_y \cdot\nabla_y \partial_t^m  h) \mathrm{d}y
-\int_{z=0}g \partial_t^m h\cdot(\partial_t \partial_t^m  h
+ v_y \cdot\nabla_y \partial_t^m  h) \mathrm{d}y
+ \|\partial_z v\|_{X^{m-1}}^2 \\[10pt]\quad
+ \|\partial_z b\|_{X^{m-1}}^2 + \|\nabla q\|_{X^{m-1}}^2
+ |h|_{X^{m-1,2}}^2 + |\partial_t^m h|_{X^{0,1}}^2  + |\partial_t^{m+1} h|_{L^2}^2 + \text{b.t.}
\\
\end{array}
\end{equation}

We have the $L^4([0,T],L^2)$ type estimate by integrating in time twice.
After the first integration in time,
one has
\begin{equation}\label{6.18}
\begin{array}{ll}
\|\partial_t^m v\|_{L^2}^2 +\|\partial_t^m b\|_{L^2}^2 +
 g|\partial_t^m h|_{L^2}^2 + \frac{\sigma}{2}|\partial_t^m\nabla h|_{L^2}^2
\\[10pt]

+ 4\e \int_{\mathbb{R}^3_{-}} |\mathcal{S}^{\varphi}\partial_t^m v|^2 \,\mathrm{d}\mathcal{V}_t
+ 4\e \int_{\mathbb{R}^3_{-}} |\mathcal{S}^{\varphi}\partial_t^m b|^2 \,\mathrm{d}\mathcal{V}_t
\\[10pt]

\lem \|\partial_t^m v_0\|_{L^2}^2 +\|\partial_t^m b_0\|_{L^2}^2 +
 g|\partial_t^m h_0|_{L^2}^2 + \sigma|\partial_t^m\nabla h_0|_{L^2}^2
\\[10pt]\quad

+ \int_0^t\int_{\mathbb{R}^3_{-}} \partial_t^m q \,
\nabla^{\varphi}\cdot \partial_t^m v \,\mathrm{d}\mathcal{V}_t\mathrm{d}t
+ \|\partial_z v\|_{L^4([0,T],X^{m-1})}^2
+ \|\partial_z b\|_{L^4([0,T],X^{m-1})}^2
\\[10pt]\quad

+ |\partial_t^m h|_{L^4([0,T],X^{0,1})}^2
+ |\partial_t^m v|_{L^4([0,T],L^2)}^2
+ |\partial_t^m b|_{L^4([0,T],L^2)}^2 + \text{b.t.}.
\end{array}
\end{equation}
in which the pressure estimates is similar to that in the Lemma 6.4 in \cite{Wu16}, i.e
\begin{equation}\label{6.19}
\begin{array}{ll}
\int_0^t\int\limits_{\mathbb{R}^3_{-}} \partial_t^m q \,
 \nabla^{\varphi}\cdot \partial_t^m v
\,\mathrm{d}\mathcal{V}_t\mathrm{d}t
 \\[10pt]

\lem \|\partial_t^m v_0\|_{L^2}^2 + \|\partial_z\partial_t^{m-1} v_0\|_{L^2}^2
+ |\partial_t^m h_0|_{X^{0,1}}^2
+ \|\partial_t^m v\|_{L^2}^2 + \|\partial_z\partial_t^{m-1} v\|_{L^2}^2
\\[10pt]\quad

+ |\partial_t^m h|_{X^{0,1}}^2
+ \|\partial_t^m v\|_{L^4([0,T],L^2)}^2
+ \|\partial_z\partial_t^{m-1} v\|_{L^4([0,T],L^2)}^2
+ |\partial_t^m h|_{L^4([0,T],L^2)}^2.
\end{array}
\end{equation}

Combining $(\ref{6.18})$ with $(\ref{6.19})$, we get
\begin{equation}\label{6.20}
\begin{array}{ll}
\|\partial_t^m v\|_{L^2}^2 +\|\partial_t^m b\|_{L^2}^2 +
 g|\partial_t^m h|_{L^2}^2 + \frac{\sigma}{2}|\partial_t^m\nabla_y h|_{L^2}^2
\\[10pt]

\lem \|\partial_t^m v_0\|_{L^2}^2 +\|\partial_t^m b_0\|_{L^2}^2
+g|\partial_t^m h_0|_{L^2}^2
+ \sigma|\partial_t^m\nabla_y h_0|_{L^2}^2
 + \|\partial_z\partial_t^{m-1} v_0\|_{L^2}^2
\\[10pt]\quad

+ \|\partial_z\partial_t^{m-1} b_0\|_{L^2}^2
+ \|\partial_t^m v\|_{L^2}^2
+ \|\partial_t^m b\|_{L^2}^2
+ \|\partial_z\partial_t^{m-1} v\|_{L^2}^2
+ \|\partial_z\partial_t^{m-1} b\|_{L^2}^2
\\[10pt]\quad

+ |\partial_t^m h|_{X^{0,1}}^2
+ \|\partial_t^m v\|_{L^4([0,T],L^2)}^2
+ \|\partial_t^m b\|_{L^4([0,T],L^2)}^2
+ \|\partial_z v\|_{L^4([0,T],X^{m-1})}^2
\\[10pt]\quad

+ \|\partial_z b\|_{L^4([0,T],X^{m-1})}^2
+ |\partial_t^m h|_{L^4([0,T],X^{0,1})}^2 + \text{b.t.}.
\end{array}
\end{equation}

Squaring $(\ref{6.20})$ and integrating in time again,
applying the Gronwall's inequality,
we have
\begin{equation}\label{6.21}
\begin{array}{ll}
\|\partial_t^m v\|_{L^4([0,T],L^2)}^2 +\|\partial_t^m b\|_{L^4([0,T],L^2)}^2 +
 g|\partial_t^m h|_{L^4([0,T],L^2)}^2 + \frac{\sigma}{2}|\partial_t^m\nabla_y h|_{L^4([0,T],L^2)}^2
\\[10pt]

\lem \|\partial_t^m v_0\|_{L^2}^2 +\|\partial_t^m b_0\|_{L^2}^2 +
 g|\partial_t^m h_0|_{L^2}^2 + \sigma|\partial_t^m\nabla_y h_0|_{L^2}^2
+ \|\partial_z\partial_t^{m-1} v_0\|_{L^2}^2
\\[10pt]\quad

+ \|\partial_z\partial_t^{m-1} b_0\|_{L^2}^2
+ \|\partial_z v\|_{L^4([0,T],X^{m-1})}^2
+ \|\partial_z b\|_{L^4([0,T],X^{m-1})}^2
+ \text{b.t.}
\end{array}
\end{equation}

Lemma $\ref{Lemma6.4}$ is proved.
\end{proof}

The estimates for the normal derivatives
\begin{equation}\label{6.22}
\begin{array}{ll}
\partial_z v, \omega_v\, \in L^4([0,T],X^{m-1}) \cap L^{\infty}([0,T],X^{m-2}),
\\[8pt]
\partial_z b,  \omega_b\, \in L^4([0,T],X^{m-1}) \cap L^{\infty}([0,T],X^{m-2})
\end{array}
\end{equation}
are similar to the case of $\sigma=0$ since the proof is not required dynamical boundary conditions.

At last, putting Lemmas $\ref{Lemma6.1},
\ref{Lemma6.3}, \ref{Lemma6.4}$ and $(\ref{6.22})$ together,
one completes the proof of Proposition $\ref{Proposition1.2}$.

\section{Proof of Theorem 1.3}
In this section,
we establish the convergence rates estimates of the inviscid limits for $\sigma>0$ case.

Denote that $\hat{v} =v^{\e} -v, \hat{b} =v^{\e} -b, \hat{q} =q^{\e} -q, \hat{h} =h^{\e} -h$,
and denote the $i-$th components of $f^{\e}$ and $f$ by $f^{\e,i}$ and $f^i$.
We deduce that $\hat{v}$, $\hat{h}$, $\hat{q}$ satisfy the following equations
\begin{equation}\label{7.1}
\left\{\begin{aligned}
&\partial_t^{\varphi^{\e}}\hat{v}
+ v^{\e} \cdot\nabla^{\varphi^{\e}} \hat{v}
- b^{\e} \cdot\nabla^{\varphi^{\e}} \hat{b}
+ \nabla^{\varphi^{\e}} \hat{q}
- 2\e \nabla^{\varphi^{\e}} \cdot\mathcal{S}^{\varphi^{\e}} \hat{v}
= \partial_z^{\varphi} v \partial_t^{\varphi^{\e}}\hat{\eta}&
\\[3pt]
&\quad
+ v^{\e}\cdot \nabla^{\varphi^{\e}}\hat{\eta}\, \partial_z^{\varphi} v
- b^{\e}\cdot \nabla^{\varphi^{\e}}\hat{\eta}\, \partial_z^{\varphi} b
- \hat{v}\cdot\nabla^{\varphi} v
+ \hat{b}\cdot\nabla^{\varphi} b
+ \partial_z^{\varphi} q\nabla^{\varphi^{\e}}\hat{\eta}
+ \e \Delta^{\varphi^{\e}} v, & x\in\mathbb{R}^3_{-},
\\[3pt]
&\partial_t^{\varphi^{\e}}\hat{b}
+ v^{\e} \cdot\nabla^{\varphi^{\e}} \hat{b}
- b^{\e} \cdot\nabla^{\varphi^{\e}} \hat{v}
- 2\e \nabla^{\varphi^{\e}} \cdot\mathcal{S}^{\varphi^{\e}} \hat{b}
= \partial_z^{\varphi} b \partial_t^{\varphi^{\e}}\hat{\eta}&
\\[3pt]
&\quad
+ v^{\e}\cdot \nabla^{\varphi^{\e}}\hat{\eta}\, \partial_z^{\varphi} b
- b^{\e}\cdot \nabla^{\varphi^{\e}}\hat{\eta}\, \partial_z^{\varphi} v
- \hat{v}\cdot\nabla^{\varphi} b
+ \hat{b}\cdot\nabla^{\varphi} v
+ \e \Delta^{\varphi^{\e}} v,
&  x\in\mathbb{R}^3_{-},
\\[3pt]
&\nabla^{\varphi^{\e}}\cdot \hat{v}=\partial_z^{\varphi}v \cdot\nabla^{\varphi^{\e}}\hat{\eta},
\nabla^{\varphi^{\e}}\cdot \hat{b}=\partial_z^{\varphi}b \cdot\nabla^{\varphi^{\e}}\hat{\eta},
& x\in\mathbb{R}^3_{-},
\\[3pt]
&\partial_t \hat{h} + v_y\cdot \nabla \hat{h} = \hat{v}\cdot\NN^{\e},
&\{z=0\},
\\[3pt]
&\hat{q}\NN^{\e} -2\e \mathcal{S}^{\varphi^\e} \hat{v} \,\NN^\e
= g \hat{h} \NN^{\e} -\sigma \nabla_y \cdot \big( \mathfrak{M}_1\nabla_y\hat{h}
\\[3pt]
&\quad+ \mathfrak{M}_2 \nabla_y\hat{h}\cdot\nabla_y(h^{\e}+h) \nabla_y (h^{\e} +h)\big)\NN^{\e}
+ 2\e \mathcal{S}^{\varphi^\e}v\,\NN^\e,
&\{z=0\},
\\[3pt]
&\hat{b}=0,
&\{z=0\},
\\[3pt]
&(\hat{v},\hat{b},\hat{h})|_{t=0} = (v_0^\e -v_0,b_0^\e -b_0,h_0^\e -h_0),
\end{aligned}\right.
\end{equation}
where $\mathfrak{M}_1$ and $\mathfrak{M}_2$ are defined as
\begin{equation}\label{7.2}
\begin{array}{ll}
\mathfrak{M}_1=\frac{1}{2\sqrt{1+|\nabla_y h^{\e}|^2}}+ \frac{1}{2\sqrt{1+|\nabla_y h|^2}},
\\[10pt]

\mathfrak{M}_2=\frac{-1}
{2\sqrt{1+|\nabla_y h^{\e}|^2}\sqrt{1+|\nabla_y h|^2}(\sqrt{1+|\nabla_y h^{\e}|^2}
+\sqrt{1+|\nabla_y h|^2})}.
\end{array}
\end{equation}

The following lemma shows the estimate of $\nabla\hat{q} = \nabla q^{\e} -\nabla q$:
\begin{lemma}\label{Lemma7.1}
Assume $0\leq s\leq k-1,\, k \leq m-1$. We have
\begin{equation}\label{7.3}
\begin{aligned}
\|\nabla \hat{q}\|_{X^s}
\lem &\|\hat{v}\|_{X^{s,1}} +\|\hat{b}\|_{X^{s,1}} + \|\partial_z\hat{v}\|_{X^s}
 + \|\partial_z\hat{b}\|_{X^s}
\\[5pt]
& + \|\partial_t^{s+1}\hat{v}\|_{L^2} + |\partial_t^s\hat{h}\|_{X^{0,\frac{1}{2}}}
+ |\hat{h}|_{X^{s,\frac{3}{2}}} +O(\e).
\end{aligned}
\end{equation}
\end{lemma}

\begin{proof}
The pressure $q^{\e}$ satisfies the elliptic equations $(\ref{1.40})$,
and the pressure for ideal MHD satisfies the following equations
\begin{equation}\label{7.4}
\left\{
\begin{array}{ll}
\Delta^{\varphi} q = -\partial_j^{\varphi} v^i \partial_i^{\varphi} v^j
+\partial_j^{\varphi} b^i \partial_i^{\varphi} b^j,
\\[6pt]

\nabla^{\varphi} q\cdot\NN|_{z=0} = - \partial_t^{\varphi} v\cdot\NN
- v\cdot\nabla^{\varphi} v\cdot\NN+b\cdot\nabla^{\varphi} b\cdot\NN.
\end{array}\right.
\end{equation}

Then the difference between boundary values is
\begin{equation}\label{7.5}
\begin{array}{ll}
\nabla^{\varphi^{\e}} q^{\e}\cdot\NN^{\e} |_{z=0} - \nabla^{\varphi} q\cdot\NN |_{z=0}
= \e\Delta^{\varphi^{\e}} v^{\e}\cdot\NN^{\e} - (\partial_t^{\varphi^{\e}}
v^{\e}\cdot\NN^{\e}-\partial_t^{\varphi} v\cdot\NN)
\\[5pt]

- (v^{\e}\cdot\nabla^{\varphi^{\e}} v^{\e}\cdot\NN^{\e} - v\cdot\nabla^{\varphi} v\cdot\NN)
+ (b^{\e}\cdot\nabla^{\varphi^{\e}} b^{\e}\cdot\NN^{\e} - b\cdot\nabla^{\varphi} b\cdot\NN).
\end{array}
\end{equation}
It follows that
\begin{equation*}
\begin{array}{ll}
\nabla^{\varphi^{\e}} \hat{q} \cdot\NN^{\e} |_{z=0}
=\partial_z^{\varphi}q \nabla^{\varphi^{\e}}\hat{\eta} \cdot\NN^{\e} |_{z=0}
- \nabla^{\varphi} q\cdot \hat{\NN}|_{z=0} \\[7pt]\quad

+ \e\triangle^{\varphi^{\e}} v^{\e}\cdot\NN^{\e}
- (\partial_t \hat{v} + v^{\e}_y\cdot\nabla_y \hat{v})\cdot\NN^{\e}
- (\partial_t v + v_y\cdot\nabla_y v) \cdot\hat{\NN} \\[7pt]\quad
+ [(\partial_t\hat{\eta} + v_y^{\e}\cdot\nabla_y \hat{\eta})\, \partial_z^{\varphi}v
- \hat{v}\cdot\nabla^{\varphi} v]\cdot\NN^{\e} :=\mathcal{I}_9.
\end{array}
\end{equation*}

Hence, $\hat{q}$ satisfies the following elliptic equation
\begin{equation}\label{7.6}
\left\{\begin{array}{ll}
\nabla\cdot(\textsf{E}^{\e}\nabla \hat{q})
= -\nabla\cdot((\textsf{E}^{\e} -\textsf{E}) \nabla q)
- \nabla\cdot [(\textsf{P}^{\e} - \textsf{P}) (v \cdot\nabla^{\varphi} v)]
+\nabla\cdot [(\textsf{P}^{\e} - \textsf{P}) (b \cdot\nabla^{\varphi} b)] \\[5pt]\hspace{2.07cm}

-\nabla\cdot [\textsf{P}^{\e} (v^{\e} \cdot\nabla^{\varphi^{\e}} \hat{v} - v^{\e}\cdot\nabla^{\varphi^{\e}}\hat{\varphi}\partial_z^{\varphi} v
+ \hat{v}\cdot\nabla^{\varphi} v)], \\[5pt]\hspace{2.07cm}

+\nabla\cdot [\textsf{P}^{\e} (b^{\e} \cdot\nabla^{\varphi^{\e}} \hat{b} - b^{\e}\cdot\nabla^{\varphi^{\e}}\hat{\varphi}\partial_z^{\varphi} b
+ \hat{b}\cdot\nabla^{\varphi} b)], \\[7pt]

\nabla^{\varphi^{\e}} \hat{q} \cdot\NN^{\e} |_{z=0} = \mathcal{I}_9.
\end{array}\right.
\end{equation}

It is standard to prove that $\hat{q}$ satisfies the following gradient estimate since the matrix $\textsf{E}^{\e}$ is definitely positive
\begin{equation}\label{7.7}
\begin{array}{ll}
\|\nabla \hat{q}\|_{X^s} \lem \|(\textsf{E}^{\e} -\textsf{E}) \nabla q \|_{X^s}
+ \|(\textsf{P}^{\e} - \textsf{P}) (v \cdot\nabla^{\varphi} v)\|_{X^s}
 \\[5pt]\quad

+ \|\textsf{P}^{\e} (v^{\e} \cdot\nabla^{\varphi^{\e}} \hat{v} - v^{\e}\cdot\nabla^{\varphi^{\e}}\hat{\varphi}\partial_z^{\varphi} v
+ \hat{v}\cdot\nabla^{\varphi} v)\|_{X^s}
+ |\mathcal{I}_9|_{X^{s,-\frac{1}{2}}}
 \\[5pt]\quad

+ \|(\textsf{P}^{\e} - \textsf{P}) (b \cdot\nabla^{\varphi} b)\|_{X^s}
+ \|\textsf{P}^{\e} (b^{\e} \cdot\nabla^{\varphi^{\e}} \hat{b} - b^{\e}\cdot\nabla^{\varphi^{\e}}\hat{\varphi}\partial_z^{\varphi} b
+ \hat{b}\cdot\nabla^{\varphi} b)\|_{X^s}
\\[9pt]

\lem \|\textsf{E}^{\e} -\textsf{E}\|_{X^s} + \|\textsf{P}^{\e} - \textsf{P}\|_{X^s}
+ \|\hat{v}\|_{X^s} + \|\hat{b}\|_{X^s}+ \|\nabla\hat{v}\|_{X^s} + \|\nabla\hat{b}\|_{X^s}
+ \|\nabla\hat{\varphi}\|_{X^s}

+ |\mathcal{I}_9|_{X^{s,-\frac{1}{2}}} \\[9pt]

\lem \|\hat{v}\|_{X^{s,1}} +\|\hat{b}\|_{X^{s,1}} +
 \|\partial_z\hat{v}\|_{X^s} +\|\partial_z\hat{b}\|_{X^s} +
  |\hat{h}|_{X^{s,\frac{1}{2}}}
+ |\mathcal{I}_9|_{X^{s,-\frac{1}{2}}}.
\end{array}
\end{equation}

Now we estimate
\begin{equation}\label{7.8}
\begin{array}{ll}
|\mathcal{I}_9|_{X^{s,-\frac{1}{2}}}
\lem \big|\partial_t\hat{v}\cdot\NN^{\e}|_{z=0}\big|_{X^{s,-\frac{1}{2}}}
+ \big|v_y^{\e}\cdot\nabla_y \hat{v}\cdot\NN^{\e}|_{z=0}\big|_{X^{s,-\frac{1}{2}}}
+ \big|\hat{v}\cdot\NN^{\e}|_{z=0}\big|_{X^{s,-\frac{1}{2}}}
\\[9pt]\quad

+ \big|\partial_t\hat{\eta}\cdot\NN^{\e}|_{z=0}\big|_{X^{s,-\frac{1}{2}}}
+ \big|\nabla_y \hat{\eta}\cdot\NN^{\e}|_{z=0}\big|_{X^{s,-\frac{1}{2}}}
+|\hat{h}|_{X^{s,\frac{1}{2}}} +O(\e)
\\[10pt]

\lem |\hat{h}|_{X^{s,\frac{1}{2}}}
+ \|\partial_t\hat{v}\|_{X^s} + \|\hat{v}\|_{X^{s,1}}
+ \|\partial_t\hat{\eta}\|_{X^s} + \|\partial_t\hat{\eta}\|_{X^{s,1}}
+ \|\nabla\hat{\eta}\|_{X^{s,1}} +O(\e)
\\[10pt]

\lem \|\partial_t^{s+1}\hat{v}\|_{L^2}
+ \|\hat{v}\|_{X^{s,1}}
+ |\partial_t\hat{h}\|_{X^{s,\frac{1}{2}}}
+ |\hat{h}|_{X^{s,\frac{3}{2}}} +O(\e),
\end{array}
\end{equation}
where the estimate of $\e\Delta^{\varphi^{\e}}v^{\e}\cdot\NN^{\e}$ is in $(\ref{6.4})$.

By $(\ref{7.7})$ and $(\ref{7.8})$, we obtain $(\ref{7.3})$.

Lemma $\ref{Lemma7.1}$ is proved.
\end{proof}

Before estimating tangential derivatives of $\hat{v}$, $\hat{b}$,
 we have to give the following estimate of $\partial_t^{\ell} \hat{h}$
 whose proof is similar to that of Lemma $\ref{Lemma5.1}$.
\begin{lemma}\label{Lemma7.2}
Assume $0\leq k\leq m-2$, $0\leq\ell\leq k-1$.
We have the estimates
\begin{equation}\label{7.9}
\begin{array}{ll}
\int_{\mathbb{R}^2} |\partial_t^{\ell}\hat{h}|^2 \,\mathrm{d}y

\lem |\hat{h}_0|_{X^{k-1}}^2
+ \int_0^t |\hat{h}|_{X^{k-1,1}}^2 + \|\hat{v}\|_{X^{k-1,1}}^2 \,\mathrm{d}t
+ \|\partial_z\hat{v}\|_{L^4([0,T],X^{k-1})}^2.
\end{array}
\end{equation}
\end{lemma}

Next,
we establish the estimates for tangential derivatives of
$\hat{v}$, $\hat{b}$, $\hat{h}$.
\begin{lemma}\label{Lemma7.3}
Assume $0\leq k\leq m-2$. We have
\begin{equation}\label{7.10}
\begin{array}{ll}
\|\hat{v}\|_{X^{k-1,1}}^2 +\|\hat{b}\|_{X^{k-1,1}}^2 + |\hat{h}|_{X^{k-1,1}}^2
\lem \|\hat{v}_0\|_{X^{k-1,1}}^2 +\|\hat{b}_0\|_{X^{k-1,1}}^2 +
|\hat{h}_0|_{X^{k-1,1}}^2

\\[10pt]\quad
+ \int\limits_0^t \|\partial_z \hat{v}\|_{X^{k-1}}^2
+\|\partial_z \hat{b}\|_{X^{k-1}}^2
+ \|\hat{v}\|_{X^{k-1,1}}^2+ \|\hat{b}\|_{X^{k-1,1}}^2
+ \|\hat{h}\|_{X^{k-1,1}}^2 \,\mathrm{d}t + O(\e).
\end{array}
\end{equation}
\end{lemma}

\begin{proof}
It is easy to deduce that
$(\partial_t^{\ell}\mathcal{Z}^{\alpha}\hat{v},\partial_t^{\ell}\mathcal{Z}^{\alpha}\hat{b}, \partial_t^{\ell}\mathcal{Z}^{\alpha}\hat{h}, \partial_t^{\ell}\mathcal{Z}^{\alpha}\hat{q})$
satisfy the following equations:
\begin{equation}\label{7.11}
\left\{\begin{array}{ll}
\partial_t^{\varphi^{\e}} \partial_t^{\ell}\mathcal{Z}^{\alpha}\hat{v}
+ v^{\e} \cdot\nabla^{\varphi^{\e}} \partial_t^{\ell}\mathcal{Z}^{\alpha}\hat{v}
- b^{\e} \cdot\nabla^{\varphi^{\e}} \partial_t^{\ell}\mathcal{Z}^{\alpha}\hat{b}
+ \nabla^{\varphi^{\e}} \partial_t^{\ell}\mathcal{Z}^{\alpha}\hat{q}
- 2\e \nabla^{\varphi^{\e}}\cdot\mathcal{S}^{\varphi^{\e}} \partial_t^{\ell}\mathcal{Z}^{\alpha}\hat{v} \\[8pt]\quad

= \e\partial_t^{\ell}\mathcal{Z}^{\alpha}\triangle^{\varphi^\e} v
+ 2\e [\partial_t^{\ell}\mathcal{Z}^{\alpha},\nabla^{\varphi^{\e}}\cdot]\mathcal{S}^{\varphi^{\e}} \hat{v}
+ 2\e \nabla^{\varphi^{\e}}\cdot[\partial_t^{\ell}\mathcal{Z}^{\alpha}, \mathcal{S}^{\varphi^{\e}}] \hat{v}+ \partial_z^{\varphi} v \partial_t^{\varphi^{\e}} \partial_t^{\ell}\mathcal{Z}^{\alpha}\hat{\varphi}
\\[8pt]\quad

+ \partial_z^{\varphi} v \, v^{\e}\cdot \nabla^{\varphi^{\e}}\partial_t^{\ell}\mathcal{Z}^{\alpha}\hat{\varphi}
- \partial_t^{\ell}\mathcal{Z}^{\alpha}\hat{v}\cdot\nabla^{\varphi} v
+ \partial_z^{\varphi} q\nabla^{\varphi^{\e}}\partial_t^{\ell}\mathcal{Z}^{\alpha}\hat{\varphi}
- [\partial_t^{\ell}\mathcal{Z}^{\alpha},\partial_t^{\varphi^{\e}}]\hat{v}
\\[8pt]\quad

+ [\partial_t^{\ell}\mathcal{Z}^{\alpha}, \partial_z^{\varphi} v \partial_t^{\varphi^{\e}}]\hat{\varphi}
- [\partial_t^{\ell}\mathcal{Z}^{\alpha}, v^{\e} \cdot\nabla^{\varphi^{\e}}] \hat{v}
- [\partial_t^{\ell}\mathcal{Z}^{\alpha}, \nabla^{\varphi} v\cdot]\hat{v}
+ [\partial_t^{\ell}\mathcal{Z}^{\alpha}, \partial_z^{\varphi} v \, v^{\e}\cdot \nabla^{\varphi^{\e}}]\hat{\varphi}
\\[8pt]\quad

- [\partial_t^{\ell}\mathcal{Z}^{\alpha},\nabla^{\varphi^{\e}}] \hat{q}
+ [\partial_t^{\ell}\mathcal{Z}^{\alpha},\partial_z^{\varphi} q\nabla^{\varphi^{\e}}]\hat{\varphi}
- \partial_z^{\varphi} b \, b^{\e}\cdot \nabla^{\varphi^{\e}}\partial_t^{\ell}\mathcal{Z}^{\alpha}\hat{\varphi}
+ \partial_t^{\ell}\mathcal{Z}^{\alpha}\hat{b}\cdot\nabla^{\varphi}b
\\
\end{array}\right.
\end{equation}
where
\begin{equation}\label{7.12}
\begin{array}{ll}
\mathcal{I}_{11,1}
:= 2\e \mathcal{S}^{\varphi^{\e}} \partial_t^{\ell}\mathcal{Z}^{\alpha}v\,\NN^{\e}
+ 2\e(\mathcal{S}^{\varphi^{\e}}v^{\e} - \mathcal{S}^{\varphi^{\e}}v^{\e}\nn^{\e}\cdot\nn^{\e})\,\partial_t^{\ell}\mathcal{Z}^{\alpha}\NN^{\e}
\\[4pt]\hspace{1.25cm}
+ 2\e[\partial_t^{\ell}\mathcal{Z}^{\alpha}, \mathcal{S}^{\varphi^{\e}}v^{\e} - \mathcal{S}^{\varphi^{\e}}v^{\e}\nn^{\e}\cdot\nn^{\e}, \NN^{\e}]
- 2\e [\partial_t^{\ell}\mathcal{Z}^{\alpha},\mathcal{S}^{\varphi^{\e}}] v^{\e}\,\NN^{\e}, \\[10pt]

\mathcal{I}_{11,2}
:= - \sigma\nabla_y \cdot \big([\partial_t^{\ell}\mathcal{Z}^{\alpha}, \mathfrak{H}_1\nabla_y] \hat{h}\big) \NN^{\e} \\[4pt]\hspace{1.25cm}
- \sigma\nabla_y \cdot \big([\partial_t^{\ell}\mathcal{Z}^{\alpha},
\mathfrak{H}_2 \nabla_y (h^{\e} +h) \nabla_y(h^{\e}+h)\cdot \nabla_y] \hat{h} \big)\NN^{\e}.
\end{array}
\end{equation}

Step one:
When $|\alpha|\geq 1$ and $1\leq \ell+|\alpha|\leq k$, we have the $L^2$ estimate for $\partial_t^{\ell}\mathcal{Z}^{\alpha}\hat{v}$ and
$\partial_t^{\ell}\mathcal{Z}^{\alpha}\hat{b}$
\begin{equation}\label{7.13}
\begin{array}{ll}
\frac{1}{2}\frac{\mathrm{d}}{\mathrm{d}t} \int_{\mathbb{R}^3_{-}} |\partial_t^{\ell}\mathcal{Z}^{\alpha}\hat{v}|^2
+|\partial_t^{\ell}\mathcal{Z}^{\alpha}\hat{b}|^2 \,\mathrm{d}\mathcal{V}_t^{\e}
- \int_{\mathbb{R}^3_{-}} \partial_t^{\ell}\mathcal{Z}^{\alpha}\hat{q} \nabla^{\varphi^{\e}}\cdot \partial_t^{\ell}\mathcal{Z}^{\alpha}\hat{v} \,\mathrm{d}\mathcal{V}_t^{\e}
\\[11pt]\quad

+ 2\e\int_{\mathbb{R}^3_{-}} |\mathcal{S}^{\varphi^{\e}} \partial_t^{\ell}\mathcal{Z}^{\alpha}\hat{v}|^2 \,\mathrm{d}\mathcal{V}_t^{\e}
+ 2\e\int_{\mathbb{R}^3_{-}} |\mathcal{S}^{\varphi^{\e}} \partial_t^{\ell}\mathcal{Z}^{\alpha}\hat{b}|^2 \,\mathrm{d}\mathcal{V}_t^{\e}
\\[11pt]

\lem -\int_{z=0} \big(\partial_t^{\ell}\mathcal{Z}^{\alpha}\hat{q} \NN^{\e}
- 2\e \mathcal{S}^{\varphi^{\e}}\partial_t^{\ell}\mathcal{Z}^{\alpha} \hat{v} \NN^{\e} \big)
\cdot \partial_t^{\ell}\mathcal{Z}^{\alpha}\hat{v} \,\mathrm{d}y
\\[11pt]\quad

+ \int_{\mathbb{R}^3_{-}} \mathcal{I}_{v10} \cdot \partial_t^{\ell}\mathcal{Z}^{\alpha}\hat{v}
\,\mathrm{d}\mathcal{V}_t^{\e}
+ \int_{\mathbb{R}^3_{-}} \mathcal{I}_{b10} \cdot \partial_t^{\ell}\mathcal{Z}^{\alpha}\hat{b} \,
\mathrm{d}\mathcal{V}_t^{\e}
\\[11pt]

\lem \int_{z=0} \big[- g\partial_t^{\ell}\mathcal{Z}^{\alpha}\hat{h}
+ \sigma \nabla_y \cdot \big( \mathfrak{M}_1\nabla_y \partial_t^{\ell}\mathcal{Z}^{\alpha} \hat{h} \big)
- \int_{z=0} \mathcal{I}_{11,1} \cdot \partial_t^{\ell}\mathcal{Z}^{\alpha}\hat{v} \,\mathrm{d}y
 \\[11pt]\quad

- \int_{z=0} \mathcal{I}_{11,2} \cdot \partial_t^{\ell}\mathcal{Z}^{\alpha}\hat{v} \,\mathrm{d}y
+ \sigma\nabla_y \cdot \big( \mathfrak{M}_2 \nabla_y \partial_t^{\ell}\mathcal{Z}^{\alpha} \hat{h}\cdot\nabla_y(h^{\e}+h) \nabla_y (h^{\e} +h)\big)\big]
\\
\end{array}
\end{equation}

As in the proof of lemma 7.3 in \cite{Wu16},
we have the boundary estimates in $(\ref{7.13})$ as follows
\begin{equation}\label{7.14}
\begin{array}{ll}
\int_{z=0} \big(\partial_t \partial_t^{\ell}\mathcal{Z}^{\alpha}\hat{h} + v_y^{\e}\cdot \nabla_y \partial_t^{\ell}\mathcal{Z}^{\alpha}\hat{h}\big)
\big(- g\partial_t^{\ell}\mathcal{Z}^{\alpha}\hat{h}
+ \sigma \nabla_y \cdot \big( \mathfrak{M}_1\nabla_y \partial_t^{\ell}\mathcal{Z}^{\alpha} \hat{h} \big)
\\[11pt]\quad

+\sigma\nabla_y \cdot \big( \mathfrak{M}_2 \nabla_y \partial_t^{\ell}\mathcal{Z}^{\alpha} \hat{h}\cdot\nabla_y(h^{\e}+h) \nabla_y (h^{\e} +h)\big)\big)
 \,\mathrm{d}y
\\[11pt]

\lem-\frac{g}{2} \frac{\mathrm{d}}{\mathrm{d}t}
\int_{z=0} |\partial_t^{\ell}\mathcal{Z}^{\alpha}\hat{h}|^2 \,\mathrm{d}y
- \frac{\sigma}{2}\frac{\mathrm{d}}{\mathrm{d}t}
\int_{z=0} \mathfrak{M}_1|\nabla_y \partial_t^{\ell}\mathcal{Z}^{\alpha} \hat{h}|^2
\\[11pt]\quad

+ \mathfrak{M}_2 |\nabla_y \partial_t^{\ell}\mathcal{Z}^{\alpha} \hat{h}\cdot\nabla_y(h^{\e}+h)|^2
 \,\mathrm{d}y
+ \|\nabla_y \partial_t^{\ell}\mathcal{Z}^{\alpha} \hat{h}\|_{L^2}^2 + \|\partial_t^{\ell}\mathcal{Z}^{\alpha} \hat{h}\|_{L^2}^2.
\end{array}
\end{equation}
and
\begin{equation}\label{7.15}
\begin{array}{ll}
- \int_{z=0} \mathcal{I}_{11,1} \cdot \partial_t^{\ell}\mathcal{Z}^{\alpha}\hat{v} \,\mathrm{d}y =O(\e).
\end{array}
\end{equation}
In addition
\begin{equation}\label{7.16}
\begin{array}{ll}
- \int_{z=0} \mathcal{I}_{11,2} \cdot \partial_t^{\ell}\mathcal{Z}^{\alpha}\hat{v} \,\mathrm{d}y
= \sigma\int_{z=0} \nabla_y \cdot \Big( [\partial_t^{\ell}\mathcal{Z}^{\alpha}, \mathfrak{M}_1\nabla_y] \hat{h}
\\[12pt]\quad

+ [\partial_t^{\ell}\mathcal{Z}^{\alpha},
\mathfrak{M}_2 \nabla_y (h^{\e} +h) \nabla_y(h^{\e}+h)\cdot \nabla_y] \hat{h} \Big)
\NN^{\e}\cdot \partial_t^{\ell}\mathcal{Z}^{\alpha}\hat{v}\,\mathrm{d}y
\\[10pt]

\lem \sigma |\hat{h}|_{X^{k-1,2}}^2 + \sigma |\partial_t^k \hat{h}|_{X^{0,1}}^2
+ \|\hat{v}\big|_{X^{k-1,1}}^2 + \|\partial_z\hat{v}\big|_{X^{k-1}}^2 .
\end{array}%
\end{equation}
where $4|\mathfrak{M}_2| \geq \delta_{\sigma}>0$.

Pluging \eqref{7.18} into \eqref{7.17},
applying the Gronwall's inequality to \eqref{7.17} yield
\begin{equation}\label{7.19}
%
\end{equation}

Step two:
When $|\alpha|=0, \, 0\leq\ell\leq k-1$,
we have no bounds of $\hat{q}$ and $\partial_t^{\ell} \hat{q}$.
The equations for $\partial_t^{\ell}\hat{v}$,
$\partial_t^{\ell}\hat{b}$ and their initial and boundary conditions
are
\begin{equation}\label{7.20}
\left\{%
\right.
\end{equation}

It follows that
\begin{equation}\label{7.22}
%
\end{equation}

Pluging $(\ref{7.23})$ into $(\ref{7.22})$, integrating in time and applying
the Gronwall's inequality, we have
\begin{equation}\label{7.24}
%
\end{equation}

Summing $\ell$ and $\alpha$, using $(\ref{7.19})$, $(\ref{7.24})$ and Lemma $\ref{Lemma7.2}$,
we have $(\ref{7.10})$.

Lemma $\ref{Lemma7.3}$ is proved.
\end{proof}

In order to close our estimates of tangential derivatives, we need to bound
$\|\partial_t^k \hat{v}\|_{L^4([0,T],L^2)}^2$, $\|\partial_t^k \hat{b}\|_{L^4([0,T],L^2)}^2$
 and $\|\partial_t^k \hat{h}\|_{L^4([0,T],X^{0,1})}^2$.
Thus, we estimate $\partial_t^k \hat{v}$, $\partial_t^k \hat{b}$ and $\partial_t^k \hat{h}$.
\begin{lemma}\label{Lemma7.4}
Let $\hat{v} =v^{\e} -v, \hat{b} =v^{\e} -b, \hat{h} =h^{\e} -h$.
We have, for $\partial_t^k \hat{v}, \partial_t^k \hat{b}, \partial_t^k \hat{h}, \partial_t^{k+1}\hat{h}$, that
\begin{equation}\label{7.26}
\begin{array}{ll}
\|\partial_t^k \hat{v}\|_{L^4([0,T],L^2)}^2 +\|\partial_t^k \hat{b}\|_{L^4([0,T],L^2)}^2
+ |\partial_t^k \hat{h}|_{L^4([0,T],X^{0,1})}^2 + |\partial_t^{k+1}\nabla \hat{h}|_{L^4([0,T],L^2)}^2
\\[6pt]

\lem \|\partial_t^k \hat{v}_0\|_{L^2}^2 +\|\partial_t^k \hat{b}_0\|_{L^2}^2
+ g|\partial_t^k \hat{h}_0|_{L^2}^2 + \sigma|\partial_t^k\nabla \hat{h}_0|_{L^2}^2
+ \|\partial_z\partial_t^{k-1} \hat{v}_0\|_{L^2}^2
\\[6pt]\quad

+ \|\partial_z\partial_t^{k-1} \hat{b}_0\|_{L^2}^2
+ \|\partial_z \hat{v}\|_{L^4([0,T],X^{k-1})}^2
+ \|\partial_z \hat{b}\|_{L^4([0,T],X^{k-1})}^2 + O(\e).
\end{array}
\end{equation}
\end{lemma}

\begin{proof}
It follows that
$(\partial_t^k \hat{v},\partial_t^k \hat{h},\partial_t^k \hat{q})$
satisfy the following equations:
\begin{equation}\label{7.27}
\left\{\begin{array}{ll}
\partial_t^{\varphi^{\e}} \partial_t^k\hat{v}
+ v^{\e} \cdot\nabla^{\varphi^{\e}} \partial_t^k\hat{v}
- b^{\e} \cdot\nabla^{\varphi^{\e}} \partial_t^k\hat{b}
+ \nabla^{\varphi^{\e}} \partial_t^k\hat{q}
- 2\e \nabla^{\varphi^{\e}}\cdot\mathcal{S}^{\varphi^{\e}} \partial_t^k\hat{v}
= \mathcal{I}_{v10}|_{\ell=k,|\alpha|=0},
\\[9pt]

\partial_t^{\varphi^{\e}} \partial_t^k\hat{b}
+ v^{\e} \cdot\nabla^{\varphi^{\e}} \partial_t^k\hat{b}
- b^{\e} \cdot\nabla^{\varphi^{\e}} \partial_t^k\hat{v}
- 2\e \nabla^{\varphi^{\e}}\cdot\mathcal{S}^{\varphi^{\e}} \partial_t^k\hat{b}
= \mathcal{I}_{b10}|_{\ell=k,|\alpha|=0},
\\[9pt]

\nabla^{\varphi^{\e}}\cdot \partial_t^k\hat{v}
= \partial_z^{\varphi}v \cdot\nabla^{\varphi^{\e}}\partial_t^k\hat{\eta}
-[\partial_t^k,\nabla^{\varphi^{\e}}\cdot] \hat{v}
+ [\partial_t^k,\partial_z^{\varphi}v \cdot\nabla^{\varphi^{\e}}]\hat{\eta},
\\[11pt]

\nabla^{\varphi^{\e}}\cdot \partial_t^k\hat{b}
= \partial_z^{\varphi}b \cdot\nabla^{\varphi^{\e}}\partial_t^k\hat{\eta}
-[\partial_t^k,\nabla^{\varphi^{\e}}\cdot] \hat{b}
+ [\partial_t^k,\partial_z^{\varphi}b \cdot\nabla^{\varphi^{\e}}]\hat{\eta},
\\[11pt]

\partial_t \partial_t^k\hat{h} + v_y^{\e}\cdot \nabla_y \partial_t^k\hat{h}
- \NN^{\e}\cdot \partial_t^k\hat{v}
 = - \hat{v}_y \cdot \nabla_y \partial_t^k h
- \partial_y \hat{h}\cdot \partial_t^k v_y
\\[5pt]\quad
 + [\partial_t^k, \hat{v},\NN^{\e}]
 - [\partial_t^k, v_y, \partial_y \hat{h}],
 \\[11pt]

\partial_t^k \hat{q}\NN^{\e}
-2\e \mathcal{S}^{\varphi^{\e}} \partial_t^k \hat{v}\,\NN^{\e}

- g\partial_t^k\hat{h}\NN^{\e}
+ \sigma \nabla_y \cdot \big( \mathfrak{M}_1\nabla_y \partial_t^k \hat{h} \big) \NN^{\e} \\[5pt]\quad
+ \sigma\nabla_y \cdot \big( \mathfrak{M}_2 \nabla_y \partial_t^k \hat{h}\cdot\nabla_y(h^{\e}+h) \nabla_y (h^{\e} +h)\big)\NN^{\e}
\\[5pt]\quad

= \mathcal{I}_{11,1}|_{\ell=k,|\alpha|=0} + \mathcal{I}_{11,2}|_{\ell=k,|\alpha|=0},
\\[11pt]

\partial_t^k \hat{b}=0,\\[11pt]

(\partial_t^k\hat{v},\partial_t^k\hat{b},\partial_t^k\hat{h})|_{t=0}
= (\partial_t^k v_0^\e -\partial_t^k v_0,
\partial_t^k b_0^\e -\partial_t^k b_0,
\partial_t^k h_0^\e -\partial_t^k h_0),
\end{array}\right.
\end{equation}

Multiplying $(\ref{7.27})$ by $\partial_t^k \hat{v}$,
$\partial_t^k \hat{b}$ respectively,
integrating over $\mathbb{R}^3_{-}$, we get
\begin{equation}\label{7.28}
\begin{array}{ll}
\frac{1}{2}\frac{\mathrm{d}}{\mathrm{d}t}\int_{\mathbb{R}^3_{-}} |\partial_t^k \hat{v}|^2
+|\partial_t^k \hat{b}|^2 \,\mathrm{d}\mathcal{V}_t^{\e}
- \int_{\mathbb{R}^3_{-}} \partial_t^k \hat{q} \, \nabla^{\varphi^{\e}}\cdot \partial_t^k \hat{v} \,\mathrm{d}\mathcal{V}_t^{\e}
+ 2\e\int_{\mathbb{R}^3_{-}} |\mathcal{S}^{\varphi^{\e}} \partial_t^k\hat{v}|^2 + |\mathcal{S}^{\varphi^{\e}} \partial_t^k\hat{b}|^2 \,\mathrm{d}\mathcal{V}_t^{\e}
 \\[14pt]

\leq \int_{z=0} (2\e \mathcal{S}^{\varphi}\partial_t^k \hat{v} \NN^{\e}
- \partial_t^k \hat{q}\NN^{\e})\cdot \partial_t^k \hat{v} \mathrm{d}y
+ \|\partial_z \hat{v}\|_{X^{k-1}}^2+ \|\partial_z \hat{b}\|_{X^{k-1}}^2
 \\[11pt]\quad

+ \|\nabla \hat{q}\|_{X^{k-1}}^2
+ |\hat{h}|_{X^{k-1,2}}^2 + |\partial_t^k \hat{h}|_{X^{0,1}}^2  + |\partial_t^{k+1} \hat{h}|_{L^2}^2
+ O(\e)
\\[11pt]

\leq - \frac{g}{2}\frac{\mathrm{d}}{\mathrm{d}t}
\int_{z=0} |\partial_t^k \hat{h}|^2 \mathrm{d}y

- \frac{\sigma}{2}\frac{\mathrm{d}}{\mathrm{d}t} \int_{z=0}
\big(\mathfrak{M}_1|\nabla_y\partial_t^k \hat{h}|^2
+ \mathfrak{M}_2 |\nabla_y \partial_t^{\ell} \hat{h}\cdot\nabla_y(h^{\e}+h)|^2 \big) \mathrm{d}y
\\[12pt]\quad

+ \|\partial_z \hat{v}\|_{X^{k-1}}^2+ \|\partial_z \hat{b}\|_{X^{k-1}}^2
 + \|\nabla \hat{q}\|_{X^{k-1}}^2
+ |\hat{h}|_{X^{k-1,2}}^2 + |\partial_t^k \hat{h}|_{X^{0,1}}^2
+ |\partial_t^{k+1} \hat{h}|_{L^2}^2 + O(\e).
\end{array}
\end{equation}

Similar to $(\ref{7.19})$, we have
\begin{equation}\label{7.29}
\begin{array}{ll}
\int_{z=0} \mathfrak{M}_1|\nabla_y \partial_t^k \hat{h}|^2
+ \mathfrak{M}_2 |\nabla_y \partial_t^k \hat{h}\cdot\nabla_y(h^{\e}+h)|^2\big]
 \,\mathrm{d}y
 \\[10pt]

\geq \int_{z=0} |\nabla_y \partial_t^k \hat{h}|^2
(\mathfrak{M}_1 - |\mathfrak{M}_2| |\nabla_y(h^{\e}+h)|^2) \,\mathrm{d}y
\geq \int_{z=0} 4|\mathfrak{M}_2| |\nabla_y \partial_t^k \hat{h}|^2 \,\mathrm{d}y.
\end{array}
\end{equation}
where $4|\mathfrak{H}_2| \geq \delta_{\sigma}>0$,
since $|\nabla_y h^{\e}|_{\infty}$ and $|\nabla_y h|_{\infty}$ are bounded.

We will integrate in time twice,
we get the $L^4([0,T],L^2)$ type estimate.
After the integration in time, we have
\begin{equation}\label{7.30}
\begin{array}{ll}
\|\partial_t^k \hat{v}\|_{L^2}^2 +\|\partial_t^k \hat{b}\|_{L^2}^2
+ g|\partial_t^k \hat{h}|_{L^2}^2 + \sigma|\partial_t^k\nabla_y \hat{h}|_{L^2}^2
+\e\int_0^t \|\nabla\partial_t^k \hat{v}\|_{L^2}^2 \,\mathrm{d}t
+\e\int_0^t \|\nabla\partial_t^k \hat{b}\|_{L^2}^2 \,\mathrm{d}t
\\[11pt]

\lem \|\partial_t^k \hat{v}_0\|_{L^2}^2 +\|\partial_t^k \hat{b}_0\|_{L^2}^2
+ g|\partial_t^k \hat{h}_0|_{L^2}^2 + \sigma|\partial_t^k\nabla_y \hat{h}_0|_{L^2}^2
+ \int_0^t\int_{\mathbb{R}^3_{-}} \partial_t^k \hat{q} \, \nabla^{\varphi^{\e}}\cdot \partial_t^k \hat{v} \,\mathrm{d}\mathcal{V}_t^{\e}\mathrm{d}t
\\[11pt]\quad

+ \int_0^t\|\partial_z \hat{v}\|_{X^{k-1}}^2 +\|\partial_z \hat{b}\|_{X^{k-1}}^2
+ \|\nabla \hat{q}\|_{X^{k-1}}^2
+ |\hat{h}|_{X^{k-1,2}}^2 + |\partial_t^k \hat{h}|_{X^{0,1}}^2
+ |\partial_t^{k+1} \hat{h}|_{L^2}^2\mathrm{d}t + O(\e)
\\[11pt]

\lem \|\partial_t^k \hat{v}_0\|_{L^2}^2 +\|\partial_t^k \hat{b}_0\|_{L^2}^2
+ g|\partial_t^k \hat{h}_0|_{L^2}^2 + \sigma|\partial_t^k\nabla \hat{h}_0|_{L^2}^2
+ \int_0^t\int_{\mathbb{R}^3_{-}} \partial_t^k \hat{q} \, \nabla^{\varphi^{\e}}\cdot \partial_t^k \hat{v} \,\mathrm{d}\mathcal{V}_t^{\e}\mathrm{d}t
\\[11pt]\quad

+ \|\partial_z \hat{v}\|_{L^4([0,T],X^{k-1})}^2+ \|\partial_z \hat{b}\|_{L^4([0,T],X^{k-1})}^2
 + |\partial_t^k \hat{h}|_{L^4([0,T],X^{0,1})}^2
+ |\partial_t^k \hat{v}|_{L^4([0,T],L^2)}^2 + O(\e). \hspace{1cm}
\end{array}
\end{equation}

We deal with the pressure term $\int_0^t\int_{\mathbb{R}^3_{-}} \partial_t^k \hat{q} \, \nabla^{\varphi^{\e}}\cdot \partial_t^k \hat{v} \,\mathrm{d}\mathcal{V}_t^{\e}\mathrm{d}t$ by using Hardy's inequality.  Denote
\begin{equation}\label{7.31}
\begin{array}{ll}
\mathcal{I}_{13} := \|\partial_z \partial_t^{k-1} \hat{q}\|_{L^2}^2 + \big|\partial_t^{k-1} \hat{q}|_{z=0}\big|_{L^2}^2
+ |\partial_t^k \hat{h}|_{X^{0,1}}^2
+ \|\partial_t^{k-1}\hat{v}\|_{L^2}^2 + \big|\partial_t^{k-1}\hat{v}|_{z=0}\big|_{L^2}^2
\\[11pt]\hspace{0.7cm}

\lem \|\partial_t^k \hat{v}\|_{L^2}^2 + \|\partial_z\partial_t^{k-1} \hat{v}\|_{L^2}^2
+ |\partial_t^k \hat{h}|_{X^{0,1}}^2 + O(\e),
\end{array}
\end{equation}
then it yields
\begin{equation}\label{7.32}
\begin{array}{ll}
\int_0^t\int\limits_{\mathbb{R}^3_{-}} \partial_t^k \hat{q} \, \nabla^{\varphi}\cdot \partial_t^k \hat{v}
\,\mathrm{d}\mathcal{V}_t^{\e}\mathrm{d}t
\lem \mathcal{I}_{13}|_{t=0} + \mathcal{I}_{13}
+ \int_0^T \mathcal{I}_{13} \,\mathrm{d}s
\\[11pt]

\lem \|\partial_t^k \hat{v}_0\|_{L^2}^2
+ \|\partial_z\partial_t^{k-1} \hat{v}_0\|_{L^2}^2
+ |\partial_t^k \hat{h}_0|_{X^{0,1}}^2
+ \|\partial_t^k \hat{v}\|_{L^2}^2
+ \|\partial_z\partial_t^{k-1} \hat{v}\|_{L^2}^2
\\[11pt]\quad

+ |\partial_t^k \hat{h}|_{X^{0,1}}^2
+ \|\partial_t^k \hat{v}\|_{L^4([0,T],L^2)}^2
+ \|\partial_z\partial_t^{k-1} \hat{v}\|_{L^4([0,T],L^2)}^2
+ |\partial_t^k \hat{h}|_{L^4([0,T],L^2)}^2.
\end{array}
\end{equation}

By $(\ref{7.30})$ and $(\ref{7.32})$, we get
\begin{equation}\label{7.33}
\begin{array}{ll}
\|\partial_t^k \hat{v}\|_{L^2}^2 +\|\partial_t^k \hat{b}\|_{L^2}^2
+ g|\partial_t^k \hat{h}|_{L^2}^2 + \sigma|\partial_t^k\nabla_y \hat{h}|_{L^2}^2
\\[11pt]

\lem \|\partial_t^k \hat{v}_0\|_{L^2}^2 +\|\partial_t^k \hat{b}_0\|_{L^2}^2
+ g|\partial_t^k \hat{h}_0|_{L^2}^2 + \sigma|\partial_t^k\nabla_y \hat{h}_0|_{L^2}^2
 + \|\partial_z\partial_t^{k-1} \hat{v}_0\|_{L^2}^2
+ \|\partial_z\partial_t^{k-1} \hat{b}_0\|_{L^2}^2
\\[11pt]\quad

+ \|\partial_t^k \hat{v}\|_{L^2}^2
+ \|\partial_t^k \hat{b}\|_{L^2}^2
+ \|\partial_z\partial_t^{k-1} \hat{v}\|_{L^2}^2
+ \|\partial_z\partial_t^{k-1} \hat{b}\|_{L^2}^2
+ |\partial_t^k \hat{h}|_{X^{0,1}}^2
+ \|\partial_t^k \hat{v}\|_{L^4([0,T],L^2)}^2
\\[11pt]\quad

+ \|\partial_t^k \hat{b}\|_{L^4([0,T],L^2)}^2
+ \|\partial_z \hat{v}\|_{L^4([0,T],X^{k-1})}^2
+ \|\partial_z \hat{b}\|_{L^4([0,T],X^{k-1})}^2
+ |\partial_t^k \hat{h}|_{L^4([0,T],X^{0,1})}^2 + O(\e).
\end{array}
\end{equation}

Squaring $(\ref{7.33})$ and integrating in time again,
applying the Gronwall's inequality,
we have
\begin{equation}\label{7.34}
\begin{array}{ll}
\|\partial_t^k \hat{v}\|_{L^4([0,T],L^2)}^2 +\|\partial_t^k \hat{b}\|_{L^4([0,T],L^2)}^2
 + g|\partial_t^k \hat{h}|_{L^4([0,T],L^2)}^2 + \sigma|\partial_t^k\nabla_y \hat{h}|_{L^4([0,T],L^2)}^2
\\[11pt]

\lem \|\partial_t^k \hat{v}_0\|_{L^2}^2 +\|\partial_t^k \hat{b}_0\|_{L^2}^2
 + g|\partial_t^k \hat{h}_0|_{L^2}^2 + \sigma|\partial_t^k\nabla_y \hat{h}_0|_{L^2}^2
+ \|\partial_z\partial_t^{k-1} \hat{v}_0\|_{L^2}^2
\\[11pt]\quad

+ \|\partial_z\partial_t^{k-1} \hat{b}_0\|_{L^2}^2
+ \|\partial_z \hat{v}\|_{L^4([0,T],X^{k-1})}^2
+ \|\partial_z \hat{b}\|_{L^4([0,T],X^{k-1})}^2
+ O(\e).
\end{array}
\end{equation}

Lemma $\ref{Lemma7.4}$ is proved.
\end{proof}

Based on Lemmas $\ref{Lemma7.2}$, $\ref{Lemma7.3}$ and $\ref{Lemma7.4}$,
the estimates of the tangential derivatives can be closed.
The estimates of the normal derivatives are the same as those for the case of $\sigma=0$.
Finally, it is standard to complete the proof of Theorem \ref{Theorem1.3}.

\end{document}